\documentclass[11pt,reqno,draft,oneside,]{amsart}
\allowdisplaybreaks 

\usepackage{color}
\usepackage{amstext}
\usepackage{amsmath}
\usepackage{amssymb}
\usepackage{amsthm}
\usepackage{amsfonts}
\usepackage{latexsym}
\usepackage{hyperref}
\usepackage{tikz}
\usetikzlibrary{matrix,arrows}
\usepackage{enumerate}
\usepackage{esint}
\usepackage{lscape}
\usepackage{ulem}
\usepackage{mathtools}
\usepackage{gensymb}
 
\newtheorem{theorem}{Theorem}[section]

\newtheorem*{theoremA*}{theorem A}
 
\newtheorem{proposition}[theorem]{Proposition}
\newtheorem{lemma}[theorem]{Lemma}
\newtheorem{corollary}[theorem]{Corollary}
\newtheorem{definition}[theorem]{Definition}

\newtheorem{remark}[theorem]{Remark}
\newtheorem{remarks}[theorem]{Remarks}

\numberwithin{equation}{section}

\theoremstyle{definition}

\begin{document}

\def\bfQ{\mathbf Q}

\def \bA {\mathbb A}
\def \bB {\mathbb B}
\def \bC {\mathbb C}
\def \bD {\mathbb D}
\def \bE {\mathbb E}
\def \bF {\mathbb F}
\def \bG {\mathbb G}
\def \bH {\mathbb H}
\def \bI {\mathbb I}
\def \bJ {\mathbb J}
\def \bK {\mathbb K}
\def \bL {\mathbb L}
\def \bM {\mathbb M}
\def \bN {\mathbb N}
\def \bO {\mathbb O}
\def \bP {\mathbb P}
\def \bQ {\mathbb Q}
\def \R {\mathbb R}
\def \bS {\mathbb S}
\def \R {\mathbb R}
\def \bR {\mathbb R}
\def \bT {\mathbb T}
\def \bU {\mathbb U}
\def \bV {\mathbb V}
\def \bW {\mathbb W}
\def \bX {\mathbb X}
\def \bY {\mathbb Y}
\def \bZ {\mathbb Z}

\def \cA {\mathcal A}
\def \cB {\mathcal B}
\def \cC {\mathcal C}
\def \cD {\mathcal D}
\def \cE {\mathcal E}
\def \cF {\mathcal F}
\def \cG {\mathcal G}
\def \cH {\mathcal H}
\def \cI {\mathcal I}
\def \cJ {\mathcal J}
\def \cK {\mathcal K}
\def \LL {\mathcal L}
\def \cM {\mathcal M}
\def \cN {\mathcal N}
\def \cO {\mathcal O}
\def \cP {\mathcal P}
\def \cQ {\mathcal Q}
\def \cR {\mathcal R}
\def \cS {\mathcal S}
\def \cR {\mathcal R}
\def \cR {\mathcal R}
\def \cT {\mathcal T}
\def \cU {\mathcal U}
\def \cV {\mathcal V}
\def \cW {\mathcal W}
\def \cX {\mathcal X}
\def \cY {\mathcal Y}
\def \cZ {\mathcal Z}

\def \fa {\mathfrak a}
\def \fb {\mathfrak b}
\def \fc {\mathfrak c}
\def \fd {\mathfrak d}
\def \fe {\mathfrak e}
\def \ff {\mathfrak f}
\def \fg {\mathfrak g}
\def \fh {\mathfrak h}
\def \mfi {\mathfrak i}
\def \fj {\mathfrak j}
\def \fk {\mathfrak k}
\def \fl {\mathfrak l}
\def \fm {\mathfrak m}
\def \fn {\mathfrak n}
\def \fo {\mathfrak o}
\def \fp {\mathfrak p}
\def \fq {\mathfrak q}
\def \fr {\mathfrak r}
\def \fs {\mathfrak s}
\def \ft {\mathfrak t}
\def \fu {\mathfrak u}
\def \fv {\mathfrak v}
\def \fw {\mathfrak w}
\def \fx {\mathfrak x}
\def \fy {\mathfrak y}
\def \fz {\mathfrak z}

\def \fA {\mathfrak A}
\def \fB {\mathfrak B}
\def \fC {\mathfrak C}
\def \fD {\mathfrak D}
\def \fE {\mathfrak E}
\def \fF {\mathfrak F}
\def \fG {\mathfrak G}
\def \fH {\mathfrak H}
\def \fI {\mathfrak I}
\def \fJ {\mathfrak J}
\def \fK {\mathfrak K}
\def \fL {\mathfrak L}
\def \fM {\mathfrak M}
\def \fN {\mathfrak N}
\def \fO {\mathfrak O}
\def \fP {\mathfrak P}
\def \fQ {\mathfrak Q}
\def \fR {\mathfrak R}
\def \fS {\mathfrak S}
\def \fR {\mathfrak R}
\def \fR {\mathfrak R}
\def \fT {\mathfrak T}
\def \fU {\mathfrak U}
\def \fV {\mathfrak V}
\def \fW {\mathfrak W}
\def \fX {\mathfrak X}
\def \fY {\mathfrak Y}
\def \fZ {\mathfrak Z}

\def\osum{\mathop{{\sum}^\oplus}}
\def \Sp {\text{\rm Sp}}
\def \sp {\mathfrak {sp}}
\def \su {\mathfrak {su}}
\def \RE {\text{\rm Re}\,}
\def \IM {\text{\rm Im}\,}
\def \al {\alpha}
\def \la {\lambda}
\def \ph {\varphi}
\def \del {\delta}
\def \eps {\varepsilon}
\def \lan {\langle}
\def \ran {\rangle}
\def \de {\partial}
\def \trans{\,{}^t\!}
\def \half{\frac{1}{2}}
\def \inv{^{-1}}
\def \rinv#1{^{(#1)}}
\def \supp {\text{\rm supp\,}}
\def \inter {\overset \text{\rm o} \to}
\def \deg {\text{\rm deg\,}}
\def \dim {\text{\rm dim\,}}
\def \span {\text{\rm span\,}}
\def \rad {\text{\rm rad\,}}
\def \tr {\text{\rm tr\,}}
\def \bl {[\![}
\def \br {]\!]}
\def \bfE {\mathbf E}
\def\[{[\![}
\def\]{]\!]}

\def\mbe{{\mathbf e}}
\def\mbh{{\mathbf h}}
\def\mbi{{\mathbf i}}
\def\mbj{{\mathbf j}}
\def\mbk{{\mathbf k}}
\def\x {{\mathbf x}}
\def\mby{{\mathbf y}}
\def\mbz{{\mathbf z}}
\def\mbw{{\mathbf w}}
\def\mbu{{\mathbf u}}
\def\mbv{{\mathbf v}}
\def\mbs{{\mathbf s}}
\def\mbt{{\mathbf t}}
\def\x{{\mathbf X}}
\def\mbE{{\mathbf E}}
\def\mbL{{\mathbf L}}
\def\mb0{{\mathbf 0}}
\def\bxi{{\boldsymbol\xi}}
\def\beps{{\boldsymbol\eps}}
\def\beeta{{\boldsymbol\eta}}

\def\Rho{{\rm P}}
\def\bv{\big\vert}

\def\be{\begin{equation}}
\def\ee{\end{equation}}
\def\bes{\begin{equation*}}
\def\ees{\end{equation*}}
\def\bea{\begin{equation}\begin{aligned}}
\def\eea{\end{aligned}\end{equation}}
\def\beas{\begin{equation*}\begin{aligned}}
\def\eeas{\end{aligned}\end{equation*}}

\def\a{\mathbf a}
\def\b{\mathbf b}
\def\c{\mathbf c}
\def\d{\mathbf d}
\def\e{\mathbf e}
\def\f{\mathbf f}
\def\g{\mathbf g}
\def\h{\mathbf h}
\def\i{\mathbf i}
\def\j{\mathbf j}
\def\k{\mathbf k}

\def\m{\mathbf m}
\def\n{\mathbf n}
\def\o{\mathbf o}
\def\p{\mathbf p}
\def\q{\mathbf q}
\def\r{\mathbf r}
\def\s{\mathbf s}
\def\t{\mathbf t}
\def\u{\mathbf u}
\def\v{\mathbf v}
\def\w{\mathbf w}
\def\x{\mathbf x}
\def\y{\mathbf y}
\def\z{\mathbf z}

\def\0{{\mathbf 0}}
\def\1{{\mathbf 1}}
\def\2{{\mathbf 2}}
\def\3{{\mathbf 3}}
\def\4{{\mathbf 4}}
\def\5{{\mathbf 5}}
\def\6{{\mathbf 6}}
\def\7{{\mathbf 7}}
\def\8{{\mathbf 8}}
\def\9{{\mathbf 9}}

\def\A{\mathbb A}
\def\B{\mathbb B}
\def\C{\mathbb C}
\def\D{\mathbb D}
\def\E{\mathbb E}
\def\F{\mathbb F}
\def\G{\mathbb G}
\def\H{\mathbb H}
\def\I{\mathbb I}
\def\J{\mathbb J}
\def\K{\mathbb K}
\def\L{\mathbb L}
\def\M{\mathbb M}
\def\N{\mathbb N}
\def\O{\mathbb O}
\def\P{\mathbb P}
\def\Q{\mathbb Q}
\def\R{\mathbb R}
\def\S{\mathbb S}
\def\T{\mathbb T}
\def\U{\mathbb U}
\def\V{\mathbb V}
\def\W{\mathbb W}
\def\X{\mathbb X}
\def\Y{\mathbb Y}
\def\Z{\mathbb Z}

\def\AA{\mathcal A}
\def\BB{\mathcal B}
\def\CC{\mathcal C}
\def\DD{\mathcal D}
\def\EE{\mathcal E}
\def\FF{\mathcal F}
\def\GG{\mathcal G}
\def\HH{\mathcal H}
\def\II{\mathcal I}
\def\JJ{\mathcal J}
\def\KK{\mathcal K}
\def\LL{\mathcal L}
\def\MM{\mathcal M}
\def\NN{\mathcal N}
\def\OO{\mathcal O}
\def\PP{\mathcal P}
\def\QQ{\mathcal Q}
\def\RR{\mathcal R}
\def\SS{\mathcal S}
\def\TT{\mathcal T}
\def\UU{\mathcal U}
\def\VV{\mathcal V}
\def\WW{\mathcal W}
\def\XX{\mathcal X}
\def\YY{\mathcal Y}
\def\ZZ{\mathcal Z}

\def \bfA {\mathbf A}
\def \bfB {\mathbf B}
\def \bfC {\mathbf C}
\def \bfD {\mathbf D}
\def \bfE {\mathbf E}
\def \bfF {\mathbf F}
\def \bfG {\mathbf G}
\def \bfH {\mathbf H}
\def \bfI {\mathbf I}
\def \bfJ {\mathbf J}
\def \bfK {\mathbf K}
\def \bfL {\mathbf L}
\def \bfM {\mathbf M}
\def \bfN {\mathbf N}
\def \bfO {\mathbf O}
\def \bfP {\mathbf P}
\def \bfQ {\mathbf Q}
\def \bfR {\mathbf R}
\def \bfS {\mathbf S}
\def \bfR {\mathbf R}
\def \bfR {\mathbf R}
\def \bfT {\mathbf T}
\def \bfU {\mathbf U}
\def \bfV {\mathbf V}
\def \bfW {\mathbf W}
\def \bfX {\mathbf X}
\def \bfY {\mathbf Y}
\def \bfZ {\mathbf Z}

\def\CZ{Calder\'on-Zygmund }

\def\alphab{{\alpha\!\!\!\!\alpha}}
\def\ellb{\boldsymbol{\ell}}
\def\betab{\boldsymbol{\beta}}
\def\gammab{\boldsymbol{\gamma}}
\def\deltab{\boldsymbol{\delta}}
\def\epsilonb{\boldsymbol{\epsilon}}
\def\zetab{\boldsymbol{\zeta}}
\def\etab{\boldsymbol{\eta}}
\def\thetab{\boldsymbol{\theta}}
\def\iotab{\boldsymbol{\iota}}
\def\kappab{\boldsymbol{\kappa}}
\def\lambdab{\boldsymbol{\lambda}}
\def\mub{\boldsymbol{\mu}}
\def\nub{\boldsymbol{\nu}}
\def\xib{{\boldsymbol\xi}}
\def\pib{\boldsymbol{\pi}}
\def\rhob{{\boldsymbol\rho}}
\def\omegab{\boldsymbol{\omega}}
\def\taub{\boldsymbol{\tau}}
\def\upsilonb{\boldsymbol{\upsilon}}
\def\phib{\boldsymbol{\phi}}
\def\varphib{\boldsymbol{\varphi}}
\def\chib{\boldsymbol{\chi}}
\def\psib{\boldsymbol{\psi}}
\def\omegab{\boldsymbol{\omega}}
\def\varthetab{\boldsymbol{\vartheta}}
\def\Sigmab{{\boldsymbol\Sigma}}

\def\bigkappa{\text{\large $\kappa$}}
\def\dbar{\bar\partial}
\def\bx{\square}
\def \RE {\Re\text{\rm e}}
\def \IM {\Im\text{\rm m}}
\def\[{\left[\!\!\left[}
\def\]{\right]\!\!\right]}

\def\interior{\text{int\,}}
\newcommand\bkt[2]{\left\langle #1,#2\right\rangle}

\def\be{\begin{equation}}
\def\ee{\end{equation}}
\def\bes{\begin{equation*}}
\def\ees{\end{equation*}}
\def\bea{\begin{equation}\begin{aligned}}
\def\eea{\end{aligned}\end{equation}}
\def\beas{\begin{equation*}\begin{aligned}}
\def\eeas{\end{aligned}\end{equation*}}

\def\AAA{\mathbf A}
\def\BBB{\mathbf B}
\def\CCC{\mathbf C}
\def\DDD{\mathbf D}
\def\EEE{\mathbf E}
\def\EEEt{\tidetilde{\EEE}}
\def\FFF{\mathbf F}
\def\tI{\tilde I}
\def\xt{\tilde{\x}}
\def\cdott{\,\tilde\cdot\,}
\def\INT{\text{\rm Int\,}}
\def\VOL{\text{\rm Vol}}
\def\NRSW{\cite{MR3862599} }

\newtheorem*{theorem*}{Theorem}
\newtheorem*{corollary*}{Corollary}
\newtheorem*{lemma*}{Lemma}
\newtheorem*{proposition*}{Proposition}
\newtheorem*{definition*}{Definition}
\newtheorem*{remark*}{Remark}
\thispagestyle{empty}

\allowdisplaybreaks[0]
\def\NRSW{\cite{MR3862599} }
\def\Emb{{\rm Emb}}
\def\Enl{{\rm Enl}}
\def\Dil{{\rm Dil}}
\def\sh{{\rm sh}}
\def\[{[\![}
\def\]{]\!]}
\def\bigllbracket{\big[\!\big[}
\def\bigrrbracket{\big]\!\big]}
\def\llbracket{[\![}
\def\rrbracket{]\!]}

\def\bmp{\marginpar{\color{brown}XX}}
\def\bsout#1{{\color{brown}\sout{#1}}}
\def\badd#1{{\color{brown}{#1}}}

\title[Square functions and local Hardy space]{Littlewood-Paley square functions\\ and the local Hardy space\\for Multi-Norm Structures on $\R^{d}$}
\author{Agnieszka Hejna$^{1}$}\thanks{$^{1}$Instytut Matematyczny, Uniwersytet Wroc\l awski, Pl. Grunwaldzki 2, 50-384 Wroc\l aw,\\ \texttt{\phantom{xxxxx} agnieszka.hejna@math.uni.wroc.pl}}
\author{Alexander Nagel$^{2}$}\thanks{$^{2}$University of Wisconsin-Madison, Madison, WI 53706,\,\,\,\texttt{ajnagel@\,me.com}}
\author{Fulvio Ricci $^{3}$}\thanks{$^{3}$Scuola Normale Superiore, Piazza dei Cavalieri 7, 56126 Pisa,\,\,\, \texttt{fulvio.ricci@sns.it}}
\maketitle

\setcounter{tocdepth}{1}
\tableofcontents

\thispagestyle{empty}

\begin{abstract}
Multi-norm singular integrals and Fourier multipliers were introduced in \cite{MR3862599}, and one application of these notions was a precise description of the composition of convolution operators with Calder\'on-Zygmund kernels adapted to $n$ different families of dilations. The description of the resulting operators was given in terms of differential inequalities specified by a matrix $\bfE$, and in terms of dyadic decompositions of the kernels and multipliers. In this paper we extend the analysis of multi-norm structures on $\R^d$ by studying the induced Littlewood-Paley decomposition of the frequency space and various associated square functions. After establishing their $L^1$-equivalence, we use these square functions to define a \textit{local multi-norm Hardy space} $\h^1_\bfE(\R^d)$. We give several equivalent descriptions of this space, including an atomic characterization. There has been recent work, limited to the $2$-dilation case, by other authors. The general $n$-dilation case treated here is considerably harder and requires new ideas and a more systematic approach. 
\end{abstract}

\tableofcontents

\thispagestyle{empty}

\section{Introduction}\label{Sec1}

This paper studies Littlewood-Paley theory, square functions, and the associated local Hardy space in the context of the  multi-norm theory formally introduced, though not with this name,  in \cite{MR3862599}. 
The main objective of \cite{MR3862599} was to identify a class of singular integral operators that includes the composition of several given one-parameter Calder\'on-Zygmund operators having different homogeneities. This class is characterized by differential inequalities on the  multipliers or equivalent differential inequalities and cancellation conditions on the kernels.  These kernels and multipliers also have dyadic decompositions, and the class is closed under composition. The result is a multi-parameter theory that is intermediate between product theory and each of the one-parameter theories corresponding to the given homogeneities. In this paper we proceed with this type of analysis, introducing new ideas and techniques or adapting those available in product theory.

The contents of this paper is discussed in detail in Section \ref{Sec2}, but  it may be useful to begin with a general  description of the place of multi-norm theory in relation with other well established theories in Euclidean harmonic analysis having to do with dilations of different kinds. It is customary to distinguish between one-parameter and multi-parameter theories, the first being essentially based on Calder\'on-Zygmund theory and its extensions, while the second ones require more sophisticated techniques. 

A one-parameter structure on $\R^d$ is determined by a family of dilations, 
\be\label{1.1}
\delta_{t}(\x)=\delta_{t}(x_1,\dots,x_d)= (t^{\la_1}x_1,\dots,t^{\la_d}x_d)
\ee
 depending on a single parameter $t>0$.
 The exponents $\la_1,\dots,\la_d$  are assumed to be strictly positive, and we  use the notation $\delta_{r}(\x)=r\x$ when no ambiguity with scalar multiplication is possible. The relevant ingredients of the theory are Calder\'on-Zygmund singular integrals, Fourier multipliers of Mihlin-H\"ormander type, Hardy-Littlewood maximal functions and their smoothened versions, Littlewood-Paley decompositions, square functions, and Hardy and BMO spaces. Each of these structures also admits a ``local'' sub-theory, which includes Calder\'on-Zygmund kernels with compact support and maximal functions restricted to small scaling parameters. On the Fourier transform side, the local theory includes Mihlin-H\"ormander multipliers that are only singular at infinity, and Littlewood-Paley decompositions limited to high frequencies. The local Hardy spaces are the $\h^p$ spaces introduced by D.~Goldberg \cite{ MR545262, MR523600}.

The prototypical examples of multi-parameter theory are the {\it product structures}.  In this paper, an $n$-parameter product structure on $\R^d$, endowed with dilations of the form \eqref{1.1}, is obtained by decomposing $\R^d$ as product of dilation-invariant subspaces $\bR^{d_1},\ldots,\bR^{d_n}$ and then allowing independent scaling in each component as in \cite{MR0664621}: 
\be\label{1.2}
\delta_{t_1, \ldots, t_n}(\x _{1}, \ldots, \x _{n})=\big(t_1\x _{1},\ldots, t_n\x _{n})\ ,\qquad (t_1,\dots,t_n)\in\bR^n_+\ .
\ee
The corresponding product theory includes singular integral operators of product type (such as tensor products of Calder\'on-Zygmund kernels on each component), strong-type maximal functions, Marcinkiewicz multipliers, and the product Hardy and $BMO$ spaces of A.~Chang and R.~Fefferman \cite{MR584078, MR658542, MR766959}. Product theory also admits {\it local} versions; see \cite{MR4026464, MR3925053} for the relevant local Hardy spaces. In these theories it is also of interest to determine the precise order of regularity of integral kernels and multipliers required to obtain boundedness results. However in this paper we will only consider kernels and multipliers that are smooth away from the appropriate singular set (the origin in the one-parameter setting, the coordinate hyperplanes in the product setting, etc.) and satisfy the appropriate differential inequalities for derivatives of any order. 

A basic aspect of these theories that identifies the structure at hand and serves as a paradigm for extensions to new ones is the \textit{dyadic Littlewood-Paley decomposition}. This is the \textit{spectral partition} of the frequency space $\R^d_\xib$ into bounded sets $B_{L}$ depending on a scale parameter $L$ varying in an appropriate subset $\LL\subseteq\Z^n$. In the $n$-parameter product theory defined by the dilations \eqref{1.2}, these sets $B_L$ depend on  $L=(\ell_1,\dots,\ell_n)\in\Z^n$ and one can take $B_L=\big\{\xib=(\xib_1,\dots,\xib_n)\in \R^{d_{1}}\times\cdots\times\R^{d_{n}}:2^{\ell_{i}-1}<|\xib_i|\leq 2^{\ell_i}\ \forall\,i\big\}$. (With $n=1$ this also covers the one-parameter case.) For the local product theory, the parameter $L$ varies in $\N^n$ and $B_L=\big\{\xib:1+|\xib_i|\sim2^{\ell_i}\ \forall\,i\big\}$. The Littlewood-Paley decomposition $\R^d_\xib=\bigcup_{L\in\LL}B_L$ associated to a given structure is essentially unique and identifies the structure. The following analytic objects can be defined in terms of it.
\begin{enumerate}[(i)]
\item The associated {\it Littlewood-Paley square functions}, of the form $S f(\x)=\big(\sum_{L\in\LL}\big|P_Lf(\x)\big|^2\big)^{\frac{1}{2}}$ where $\sum_{L}P_{L}=I$ and each $P_L$ is a smoothened spectral projection on $B_L$. The simplest example is that of an operator $\widehat{P_Lf}(\xib)=\widehat f(\xib)\eta_L(\xib)$, where $\{\eta_L\}$ is a smooth partition of unity with each $\eta_L$ supported in a sufficiently small neighborhood of $B_L$ and satisfying appropriate uniformity conditions.

\smallskip

\item The algebra (under point-wise product) of Fourier multipliers $m(\xib)$ adapted to the given structure, which can be expressed as $\sum_{L\in\LL}m_L(\xib)$, with smooth functions $m_L$ satisfying the same support and uniformity condition as the $\eta_L$ in (i).

\smallskip

\item The convolution algebra of singular kernels $\KK$ adapted to the given structure, consisting of the inverse Fourier transforms of the multipliers in (ii).

\smallskip

\item The Hardy spaces $\bfH^{p}$ (or $\h^p$ if the structure is local) with $p\le1$, consisting of those distributions $f$ whose square functions $Sf$, as defined in (i), is in $L^p$.
\end{enumerate}

For one-parameter and product structures, (i)-(iv) describe well-known objects. Multi-norm structures depend on the decomposition of $\R^d$ as the product of $n$ subspaces $\R^{d_i}$ and on the choice of an $n\times n$ {\it standard matrix} $\bfE$, as defined below in \eqref{2.3}. The most interesting property of this structure, and the reason for its study, is that it describes, under appropriate non-degeneracy conditions, the singular integral operators arising as compositions of Calder\'on-Zygmund operators adapted to $n$ different dilation families. A brief history of this study is the following.
 A parametrix for the $\overline\partial$-Neumann problem on strictly pseudo-convex domain was constructed in \cite{MR0499319}, and the resulting operators involve compositions of two kinds of operators that reflect the co-existence of two geometries on the boundary of these domains: the ``isotropic'' Riemannian geometry induced from the ambient space and the ``parabolic'' sub-Riemannian geometry induced by the CR-structure. These operators have different homogeneities, and one of the ``open questions'' from 1979 raised by N. Rivi\`ere (collected in \cite[p. xvii]{MR0545233}) concerns the behavior near $p=1$ of the composition of two Calder\'on-Zygmund operators on $\R^d$, one elliptic and the other parabolic. Some years later this question was studied by D.H.~Phong and E.M.~Stein in \cite{MR648484, MR0853446} (questions related to this have been studied recently by Han, Krantz, and Tan \cite{MR4760280}). In \cite{MR3862599}, A.~Nagel, F.~Ricci, E.M.~Stein and S.~Wainger described the classes of kernels and Fourier multipliers that are obtained by composing $n$ local Calder\'on-Zygmund operators for $n$ dilation families depending on a standard matrix. The differential inequalities satisfied by these operators involve $n$ different homogeneous norms, generated by the assignment of the matrix $\bfE$. The two-factor situation is simpler than those with a higher number of factors. For instance, in the definition of the standard matrix $\bfE$ in \eqref{2.3} below, if $n=2$ the conditions in the second line  are essentially vacuous and the strict inequality in the third line is verified as soon as the two chosen dilations are different, modulo a possible interchange. This is no longer true for $n\ge3$, as shown by a comparison between the cases $n=2$ and $n=3$ described in Section \ref{Sec3.7}.
 
Multi-norm structures are defined as local structures. We stress from the beginning that this is not an arbitrary choice, but arises naturally from direct analysis of compositions of singular kernels or Littlewood-Paley decompositions. 
Each multi-norm structure is, in any reasonable sense, intermediate between the local product structure for the decomposition of $\R^d$ as product of the $\R^{d_i}$ and any of the $n$ local one-parameter structures induced by the dilations dictated by the matrix $\bfE$. For instance, the multi-norm Littlewood-Paley decomposition of the frequency space is coarser than the product decomposition and finer than each of the $n$ one-parameter ones. Similarly, the space of multi-norm Fourier multipliers is included in the space of local Marcinkiewicz multipliers and includes the $n$ spaces of local Mihlin-H\"ormander multipliers, etc.

Another kind of structure,  {\it flag structure}, has been introduced in recent years, and this  is also intermediate between one-parameter and product structures. 
Multi-norm and flag structures have many points in common, but there are also several differences. The common setting is $\R^d$, endowed with the dilations \eqref{1.1} and decomposed as product of dilation invariant subspaces $\R^{d_i}$. The typical example of a {\it flag singular kernel} is a convolution $\KK=\KK_1*\cdots*\KK_n$, where $\KK_1$ is a smooth Calder\'on-Zygmund kernel on $\R^d$ adapted to the given dilations and, for $i\ge2$, $\KK_i=\del_\0\otimes \HH_i$, $\del_\0$ denoting the Dirac delta on $\R^{d_1}\times\cdots\times\R^{d_{i-1}}$ and $\HH_i$ a smooth Calder\'on-Zygmund kernel on $\R^{d_i}\times\cdots\times\R^{d_n}$, also adapted to the given dilations. See \cite{MR1312498}, \cite{MR1376298}, where flag kernels appear for the first time (though not with this name), \cite{MR1818111}, \cite{MR2949616} for formal definitions and analysis of singular kernels and multipliers, and \cite{MR3293443, MR3192289, MR4458535, arXiv:2102.07371} for introduction and study of flag Hardy spaces. In the last set of references, however, the analysis is limited to the two-factor case. See also \cite{MR2312951, MR2602167} for a different approach to flag theory on nilpotent groups. To the best of our knowledge the literature on multi-norm structures consists, in addition to the already mentioned papers, of \cite{MR3148598, MR3270788}, both limited to two factors. 

The objective of this paper is to develop and extend parts of multi-norm theory not treated in~\cite{MR3862599}, concentrating our attention on the definition of a {\it multi-norm Littlewood-Paley decomposition} of the frequency space, on the equivalence of different forms of Littlewood-Paley square functions in defining a local Hardy space for $p=1$, and finally on the atomic structure of this space. 
Our approach to multi-norm Littlewood-Paley theory is new, even in the two-factor case. Nonetheless, we prove that the Hardy space defined in this paper coincides with that, previously defined in \cite{MR3148598, MR3270788}, with $n=2$ and by means of other (non-multinorm) square functions.

\section{Basic definitions and outline of the paper}\label{Sec2} 

The underlying space in this paper is $\bR^d$, endowed with the dilations $\delta_{t}$ in \eqref{1.1}. $\R^{d}$ is decomposed as the product of $n$ subspaces: $\bR^d=\bR^{d_1}\times\cdots\times\bR^{d_n}$, and each $\R^{d_{i}}$ is invariant under $\delta_{t}$. There are two `dual' copies of $\R^{d}$: convolution kernels live on the `physical space' $\R_{\x}^{d}$ where variables are denoted by Latin letters $\x$, $\y, \ldots$, while the corresponding multipliers live on the  dual `frequency space' $\R_{\xib}^{d}$ where the variables are denoted by Greek letters $\xib$, $\etab, \ldots$. If $\x, \xib\in \R^{d}$ write $\x=(\x_{1}, \ldots, \x_{n})$ and $\xib=(\xib_{1}, \ldots, \xib_{n})$ with $\x_{i}, \xib_{i}\in \R^{d_{i}}$. \label{E_{i}} To identify individual scalar coordinates, let $\big\{E_{1}, \ldots, E_{n}\big\}$ be the partition of $\{1, \ldots, d\}$ so that $\x_{i}=(x_h)_{h\in E_i}$.\footnote{At a first stage the reader may assume that the subspaces $\R^{d_i}$ are one-dimensional and that the dilations $\del_t$ are just scalar multiplication by $t$, but later on must be prepared to meet coarser decompositions with non-isotropic dilations coming into the picture.} Thus
\bea\label{2.1}
E_{i}=\Big\{h\in \{1,\ldots,d\}:\text{$x_{h}$ is a coordinate in $\R^{d_{i}}$}\Big\}.
\eea
If $G\subseteq\{1,\dots,n\}$, we write $\R^{G}$, $\x_{G}$, $\xib_G$ for $\bigoplus_{i\in G}\bR^{d_i}$, $(\x_{i})_{i\in G}$, $(\xib_i)_{i\in G}, \textit{etc.}$. We denote by $G^{c}$ the complement of $G$ in $\{1,\dots,n\}$.  For any $d$-tuple $ \alphab=(\alpha_{1}, \ldots, \alpha_{d})$ we write $\alphab_{i}=(\alpha_{h})_{h\in E_{i}}$. Denote the homogeneous dimensions of $\R^{d_{i}}$  by 
 \bea\label{2.2ww}
 q_{i}&=\sum\nolimits_{h\in E_{i}}\lambda_{h},
 \eea
and by $|\x_{i}|$ the homogeneous norm on $\R^{d_{i}}$ given by 
 \bea\label{2.2}
 |\x_{i}|= \sum\nolimits_{h\in E_{i}}|x_{h}|^{\frac{1}{\lambda_{h}}}.
 \eea
An $n\times n$ matrix $\bfE$ with positive entries $\{e(j,k)\}$ is called  \textit{standard} if
\bea\label{2.3}
e(j,j) &=1 &&\text{for $1\leq j\leq n$},\\
e(j,\ell)&\leq e(j,k)e(k,\ell) &&\text{for $1 \leq j,k,\ell\leq n$},\\
1&<e(j,k)e(k,j)&&\text{for $1\leq j, k\leq n$, $j\neq k$}.
\eea
In \cite{MR3862599}, a matrix satisfying the first two conditions in \eqref{2.3} is said to satisfy the \textit{basic assumptions}. Notice that the first two conditions imply that $1\le e(i,j)e(j,i)$ for all $i,j$. The purpose of the added third condition is to exclued a degeneracy  as explained in  Remarks \ref{Rem3.11}. Fixing a standard matrix $\bfE$, we define families of dilations $\delta_{t}^{i}(\x)$ and $\widehat\delta_{t}^{i}(\xib)$ and corresponding homogeneous norms  $N_{i}(\x)$ and $\widehat N_{i}(\xib)$ by
\bea\label{2.5a}
\delta_{t}^{i}(\x)&=\big(t^{\frac{1}{e(i,1)}}\x_{1}, \ldots, t^{\frac{1}{e(i,n)}}\x_{n}\big),&
\widehat\delta_{t}^{i}(\xib)&=\big(t^{e(1,i)}\xib_1,\dots, t^{e(n,i)}\xib_n\big),\\
N_{i}(\x)&=\max\big\{|\x_{j}|^{e(i,j)}:1\leq j \leq n\big\}, &\widehat N_{i}(\xib)&=\max\big\{|\xib_j|^{\frac{1}{e(j,i)}}:1\leq j \leq n\big\}.
\eea 
This notation is taken from \cite{MR3862599}.\footnote{The symbols $N$ and $\widehat N$ are motivated in \cite{MR3862599} by the fact that the norms $N_{i}(\x)$ control the size of the singular multi-norm kernels and their derivatives,  while the norms $\widehat N_{i}(\xib)$ control the size of derivatives of Fourier multipliers. We will see that the same occurs in this paper with the dyadic decompositions in the $\x$- and the $\xib$-space.}  The choice of the $\ell^{\infty}$-norm for $N_{i}(\x)$ and $\widehat N_{i}(\xib)$ produces various technical simplifications. For instance, all the norms in \eqref{2.5a} have the same unit ball given by $B(1)=\prod_{i=1}^n B_i(1)=\big\{(\x_{1}, \ldots, \x_{n})\in \R^{d}:\text{$|\x_{j}|\leq 1$ for $1 \leq j \leq n$}\big\}$, and similarly in the $\xib$-variables.

\subsection{Littlewood-Paley decomposition of $\R^d_\xib$}\label{Sec2.1}\quad

\smallskip

 An important focus of this paper is the multi-norm Littlewood-Paley decomposition of frequency space $\R_{\xib}^{d}$, and  this is discussed in detail in the first four sections of Section \ref{Sec3}. Here we give a brief introduction. The starting point is a decomposition of the complement $B(1)^{c}$ of the unit ball in the frequency space $\R_{\xib}^{d}$.  For $1\leq i \leq n$, the local one-parameter Littlewood-Paley partition of the frequency space associated to the $i$-th dilation $\widehat \delta^{i}_{r}$ in \eqref{2.5a} consists of the unit ball $B(1)=B_{\0}^{i}$ and the dyadic annuli $B_{\ell}^{i}=\big\{\bxi\in\R^{d}:\widehat N_i(\bxi)\sim 2^\ell\big\}$ with $\ell\geq 1$ so that $\R_{\xib}^{d}=B(1)\cup\bigcup_{\ell=1}^{\infty}B_{\ell}^{i}$. Then, roughly speaking, the {\it multi-norm Littlewood-Paley decomposition} is the minimal decomposition that is finer than all the $\widehat\delta_{t}^{i}$  Littlewood-Paley decompositions for $1\leq i \leq n$. It consists of the \textit{non-empty} intersections 
\bea\label{2.5}
 B_L= \bigcap\nolimits_{j=1}^n B^j_{\ell_j}=\Big\{\bxi\in \R_{\xib}^{d}:\widehat N_{i}(\bxi)\sim 2^{\ell_{i}} \text{ if }\ell_i\ge1\ ,\ \widehat N_{i}(\bxi)\le 1 \text{ if }\ell_i=0\Big\},
\eea
with $L=(\ell_{1}, \ldots, \ell_{n})\in \N^{n}$.
Thus the multi-norm Littlewood-Paley decomposition is indexed by the set $\LL=\big\{L\in \N^{n}:B_{L}\neq\emptyset\big\}$ 
 of {\it admissible scales}. 

To better understand the decomposition we introduce the notions of dominance and marked partitions. 
We say that a component $\xib_{i}$ is {\it dominant} in the norm $\widehat N_{j}(\xib)$ if $\widehat N_{i}(\xib)=|\xib_{j}|^{\frac{1}{e(j,i)}}$.\footnote{Two different components are both dominant only on a negligible set of $\xib$.} Then  the inequalities \eqref{2.3} for a standard matrix imply that if $\xib_{i}$ is dominant in $\widehat N_{j}(\xib)$ then it is also dominant in $\widehat N_{i}(\xib)$. It follows that dominant components in $\widehat N_{j}(\xib)$ are described by a {\it marked partition}, {\it i.e.} a partition $\{1,\ldots, n\}=A_{1}\cup\cdots\cup A_{s}$ together with a ``marked entry'' $k_{r}\in A_{r}$ for $1\leq r\leq s$ so that $\xib_{k_{r}}$ is dominant in $\widehat N_{i}(\xib)$ for all $i\in A_{r}$. We write such a marked partition as $S=\big\{(A_{1},k_{1}), \ldots, (A_{s},k_{s})\big\}$. Given a marked partition $S$ there is then the associated set of $\xib$ such that the dominant components of $\widehat N_{i}(\xib)$ are given by $S$.

The {\it principal marked partition} $S_{p}$ has $s=n$ and $A_{r}=\{r\}$. The associated {\it principal cone} is 
\bea\label{2.8}
 \Gamma(\bfE)=\big\{(t_{1}, \ldots, t_{n})\in \R_{+}^{n}:\text{$ t_{i}\leq e(i,j)t_{j}$ for $1\leq i,j\leq n$}\big\}\subseteq\R_{+}^{n}.
 \eea
The elements $\t\in\Gamma(\bfE)$ are characterized by the following property: if $\xib$ is such that $|\xib_i|=2^{t_i}$ for $i=1,\dots,n$, then every $\xib_i$ is dominant in $\widehat N_{i}(\xib)$. It turns out that the index set $\LL$ consists of those lattice points $L\in \N^{n}$ at distance at most $1$ from $\Gamma(\bfE)$. More generally, associated to each marked partition $S=\big\{(A_{1},k_{1}), \ldots, (A_{s},k_{s})\big\}$ there is a cone $\Gamma_{S}\subseteq\R_{+}^{n}$ (see \eqref{3.2rr}).  

The explicit forms of the sets $B_L$ with $L\in\LL\setminus\{0\}$ can also be described in terms of  marked partitions. Roughly speaking, if $L=(\ell_{1}, \ldots, \ell_{n})$ is close to $\Gamma_{S}$ then  $B_L$ can be expressed as the set of $\xib$ with $|\xib_{k_r}|\sim2^{\ell_{k_r}}$ for $r=1,\dots,s$ and, if $i\ne k_1,\dots,k_r$, then $|\xib_i|\lesssim 2^{\ell_i}$.
This indicates that a bump function supported in a small neighborhood of $B_L$ has an inverse Fourier transform with cancellation in the marked variables $\x_{k_r}$, but not in the others. 

Thus the set $\LL\setminus\{0\}$ naturally decomposes, in an essentially unique way, as a disjoint union of subsets $\LL_S$ indexed by marked partitions $S$. However,  only  those $S$ such that $\Gamma_{S}$ has nonempty interior will play a role. We denote by $\SS_n$ the set of all marked partitions and by $\SS_\bfE$ the set for which $\Gamma_{S}$ has nonempty interior. A coarser decomposition of $\LL$  is obtained as the union of the $\LL_S$ for all the marked partitions $S$ with the same set $D$ of marked elements. In this case we use the notation $\LL(D)$. It is then natural to set $\LL(\emptyset)=\LL_\emptyset=\{0\}$. We define $\SS^*_\bfE=\SS_\bfE\cup \{\emptyset\}$, so that the most frequently used partition of $\LL$ takes the form
\be\label{L=cupL_S}
\LL=\bigcup\nolimits_{S\in\SS^*_\bfE}\LL_S.
\ee
Note that this Littlewood-Paley decomposition is \textit{local} in nature; the sets $B_{L}$ are all ``large'', the associated singular integral kernels are singular only at the origin, and the associated multipliers are smooth at the origin. See Question 1 in Section \ref{Sec12} for more comments on this issue.

In Section \ref{Sec3.7} we describe in detail the Littlewood-Paley decomposition in the cases $n=2$ or $n=3$ since this may help the reader understand the geometry involved in the general construction.

\subsection{Littlewood-Paley and Calder\'on reproducing formulas}\label{Sec2.3}\quad 

\smallskip

For $L=(\ell_{1}, \ldots, \ell_{n})\in \N^{n}$ we set 
\bea\label{2.12}
f^{(-L)}(\x_{1}, \ldots, \x_{n})=2^{\sum_{1}^{n}\ell_{j}q_{j}}f(2^{\ell_{1}}\x_{1}, \ldots, 2^{\ell_{n}}\x_{n}).
\eea
A {\it multi-norm Littlewood-Paley family} $\{\Xi_L:L\in\LL\}\subset\SS(\R^{d})$ is a set of uniformly bounded Schwartz functions such that $\sum_{L\in\LL}\Xi_L^{(-L)}=\del_\0$ in the distribution sense. A (discrete) \textit{multi-norm Calder\'on family} associated to $\{\Xi_L:L\in\LL\}$ is then a set $\Xi=\big\{\Xi_{L}, \,\widetilde\Xi_{L}:L\in\LL\big\}\subset\SS(\R^{d})$, uniformly bounded in every Schwartz norm, such that $\sum\nolimits_{L\in \LL}\widetilde\Xi_{L}^{(-L)}*\Xi_{L}^{(-L)}=\del_\0$. We use the expression {\it Littlewood-Paley reproducing formula} and {\it Calder\'on reproducing formulas} for the identities $f=\sum\nolimits_{L\in\LL}f*\Xi_L^{(-L)}$ and $f=\sum\nolimits_{L\in\LL}f*\Xi_L^{(-L)}*\widetilde\Xi_L^{(-L)}$ respectively.  In Section \ref{Sec4} we construct associated Calder\'on families whose tilded elements $\widetilde\Xi_{L}$ have vanishing moments (depending on the type and on $L$) up to any preassigned order. Our proofs of Propositions \ref{Prop4.3} and \ref{Prop4.4} are based on the construction due to V.~Rychkov \cite{MR1821243} in the one-parameter case. These are based on two different types of Littlewood-Paley-families adapted to the multi-norm structure.

\smallskip

\noindent (i) \textit{Tensor  type.}

\smallskip
For each $1\leq i \leq n$ we take a local one-parameter reproducing formula $f(\x_{i})=\sum_{\ell_{i}=0}^{\infty}f*\psi_{i,\ell_{i}}^{(-\ell_{j})}(\x_{i})$ on $\R^{d_{i}}$ where $\psi_{i,\ell_{i}}$ are uniformly bounded in $\SS(\R^{d_{i}})$ and have integral zero for $\ell\ge1$. Taking the tensor product $\Psi_{L}^{\otimes}(\x_{1}, \ldots, \x_{n})=\prod_{\ell=1}^{n}\psi_{i,\ell_{i}}(\x_{i})$ we obtain a reproducing formula $f(\x )=\sum_{L\in\N^{n}}f*(\Psi_{L}^{\otimes})^{(-L)}(\x )$ on $\R^{d}$, but this is not a Littlewood-Paley reproducing formula in our sense since we are summing over $L\in \N^{n}$ rather than $L\in \LL$. However, we can easily obtain a Littlewood-Paley reproducing formula $f(\x )=\sum_{L\in \LL}f*\Psi_{L}^{(-L)}(\x )$ by appropriately grouping together finite sets of indices. Each $\Psi_L$ will have integral zero in any variable $\x_i$ with $i\in D_L$. The same can be applied to the $n$ Calder\'on reproducing formulas, producing a multi-norm Calder\'on reproducing formula $f(\x )=\sum_{L\in \LL}f*\Psi_{L}^{(-L)}*\widetilde\Psi_{L}^{(-L)}(\x )$, where each $\Psi_{L}$ and $\widetilde\Psi_{L}$ is a tensor product and each $\widetilde\Psi_{L}$ has vanishing moments  in any variable $\x_i$ with $i\in D_L$. 

\smallskip

\noindent (ii)  \textit{Convolution type.}

\smallskip
For each $1 \leq i \leq n$ we construct a reproducing formula on all of $\R^{d}$ using the $i^{th}$ family  $\widehat \delta_{t}^{i}$ of dilations defined in \eqref{2.5a}. With the notation $f^{(-\ell)_{i}}(\x)=2^{\ell\sum_{j}e(j,i)q_{j}}f(\widehat\delta^{i}_{2^{\ell}}\x)$, we fix functions $\sigma_{i,\ell}$ with $\ell\in \N$, uniformly bounded in $\SS(\R^{d})$, with integral zero if $\ell\ge1$ and giving a reproducing formula $f(\x )=\sum_{\ell=0}^{\infty}f*\sigma_{i,\ell}^{(-\ell)_{i}}(\x )$. By composition, we obtain a reproducing formula $f(\x )=\sum_{L\in\N^n}f*\sigma_{1,\ell_{1}}^{(-\ell_{1})_{1}}*\cdots*\sigma_{n,\ell_{n}}^{(-\ell_{n})_{n}}(\x )$ which, once again, is not a multi-norm Littlewood-Paley decomposition. However, the same grouping of terms with $L\not\in\LL$ used in (i) produces a multi-norm reproducing formula $f(\x )=\sum_{L\in\LL}f*\Sigma_L^{(-L)}(\x )$. It is proved in Corollary \ref{Cor4.6} that the terms $\Sigma_L^{(-L)}$ produced by this grouping can be expressed as convolutions $\sigma_{1,\ell_{1}}^{(-\ell_{1})_{1}}*\cdots*\sigma_{n,\ell_{n}}^{(-\ell_{n})_{n}}$, where the $\sigma_{i,\ell_{i}}$ are uniformly bounded in $\SS(\R^{d})$ and have integral zero if $i\in D_L$. The same operations can be performed on Calder\'on reproducing formulas relative to the $n$ dilation families to obtain a multi-norm Calder\'on reproducing formula of convolution type, with the $\widetilde\sigma_{i,\ell_{i}}$ uniformly bounded in $\SS(\R^{d})$ and with vanishing moments up to any preassigned order in each variable $i\in D_L$.

\subsection{Square functions}\label{Sec2.4}\quad

\smallskip

Given a Littlewood-Paley family $\Xi=\big\{\Xi_{L}:L\in \LL\big\}$, we define the associated square function
\be\label{2.13}
\SS_\Xi f=\Big(\sum\nolimits_{L\in\LL}\big|f*\Xi_L^{(-L)}\big|^2\Big)^{\half}.
\ee
One of the main themes  of this paper is the {\it $L^{1}$-equivalence} of certain multi-norm square functions, and this is the subject of Sections \ref{Sec5}, \ref{Sec6}, and \ref{Sec7}. Two square functions $S_\Xi$ and $S_{\Xi'}$ are $L^{1}$-equivalent if the condition $S_\Xi f\in L^{1}(\bR^d)$ is equivalent to $S_{\Xi'} f\in L^{1}(\bR^d)$ for every tempered distribution $f$. All the different versions of square functions will play a role in the description of the multi-norm Hardy space $\h_{\bfE}^{1}(\R^d)$. In Section \ref{Sec5}, Theorem \ref{Thm5.1}, we prove $L^{1}$-equivalence of square functions of tensor type. In Section \ref{Sec6.1}, Theorem \ref{Thm6.2}, we prove $L^{1}$-equivalence of square functions of convolution type and then, in Theorem \ref{Thm6.4}, between these and those of tensor type. In Section \ref{Sec6.3} we study the operator $S''_\sigma f=\big(\sum_{L\in\N^{n}}\big|f*\big(\mathop*_{i=1}^n\sigma_{i,\ell_i}^{(-\ell_i)_i}\big)\big|^2\big)^{\frac{1}{2}}$ which arises from the reproducing formula $f(\x )=\sum_{L\in \Z^{n}}f*\big(\mathop*_{i=1}^n\sigma_{i,\ell_i}^{(-\ell_i)_i}\big)$ mentioned above and, in Theorem \ref{Thm6.11}, prove its $L^{1}$-equivalence with the previous square functions.\footnote{This is not a Littlewood-Paley decomposition because it is not necessarily true that each term is an approximate spectral projection. This type of square function, however, is a local analogue of the ones used in previous work on Hardy spaces related to the two-norm situation \cite{arXiv:2102.07371}. It is also useful in some proofs where an inductive argument on the number $n$ of factors is needed.}  In Theorem \ref{Thm7.1} in Section \ref{Sec7} we establish $L^{1}$-equivalence of \textit{Plancherel-P\'olya} variants of the square functions of tensor type. The definition of these requires the notion of multi-norm dyadic rectangles, which is the subject of Section \ref{Sec7.1}.

\subsection{Dyadic structure of $\R^d_\x$ and Plancherel-P\'olya square functions}\label{Sec2.2}\quad

\smallskip

In Section \ref{Sec7.1} we shift our attention to the $\x$-space and define multi-norm dyadic cubes and rectangles. This requires some care since several different families of dilations are involved. These definitions are  ``local'' in the sense that we are only concerned with sets at scales smaller than $1$, and are adapted the the norms $\{N_{i}:1\leq i \leq n\}$. These notions are used in Section \ref{Sec7}  to define Plancherel-P\'olya type square functions, and in Section \ref{Sec13}, in connection with the atomic theory of the Hardy space $\h^1_\bfE(\R^d)$. Dyadic cubes $\bfQ_{i}^{\ell_{i}}$ at scales $2^{-\ell_{i}}$ are defined in each factor $\R^{d_i}$, $i=1,\dots,n$, and dyadic rectangles $\bfR$ in $\R^d$ are products of dyadic cubes $\bfQ_{i}^{\ell_{i}}$ in $\R^{d_i}$ at scales $2^{-\ell_i}$ subject to the admissibility condition $L=(\ell_1,\dots,\ell_n)\in\LL$.

\subsection{Local Hardy space}\label{Sec2.5}\quad

\smallskip

In Section \ref{Sec8} we introduce the multi-norm Hardy space $\h_{\bfE}^{1}(\bR^d)$ as the space of distributions $f$ such that, for any of the square functions $S_\Xi$ introduced in the previous sections, $S_\Xi f\in L^{1}(\bR^d)$. We first prove that $\h_{\bfE}^{1}(\bR^d)$ is contained in $L^{1}(\bR^d)$ (Theorem \ref{Thm8.1}). Proposition \ref{Prop8.4} shows that it is invariant under smooth cutoffs and consequently that every $f\in \h_{\bfE}^{1}(\bR^d)$ can be decomposed as a sum of $\h_{\bfE}^{1}$-functions supported in a mesh of unit cubes with control of their norms. Theorem \ref{Thm8.6} shows that the multi-norm singular integral operators (defined in Section \ref{Sec3.5} ) are bounded from $\h_{\bfE}^{1}(\bR^d)$ to itself. Theorem \ref{Thm8.9} shows that $\h_{\bfE}^{1}(\bR^d)$ can also be defined via {\it multi-norm local Riesz tranforms} $R_J=\cK^\flat_{1,j_1}*\cdots*\cK^\flat_{n,j_n}$, where for each fixed $i=1,\dots,n$, the kernels $\cK^\flat_{i,j}$ characterize the local Hardy space $\h^1_{\widehat\del^i}(\bR^d)$ relative to the dilations $\widehat\del^i$.

\subsection{Atomic characterization of $\h^1_\bfE(\R^d)$}\label{Sec2.6}\quad

\smallskip

 In Section \ref{Sec9} we obtain an atomic decomposition of $\h_{\bfE}^{1}$-functions. Atoms depend on a parameter $\tau>0$ and an element $S\in\SS_{\bfE}^{*}$.  There are three types, depending on the number $s$ of marked indices in $S$. 
The precise definition of an $(S,\tau)$-atom is given in Section \ref{Sec9.1}, the following being only a rough description. 
\begin{enumerate}[1.]
\item If $s=0$ then $S=\emptyset$ and an $(\emptyset,\tau)$-atom is a function supported in a translate of $2^{\tau} \B(1)$ and normalized in~$L^2$, with no required cancellation. 

\smallskip
\item If $s=1$ with dotted entry $k$, an $(S,\tau)$-atom $a$ is a Coifman-Weiss atom for the local Hardy space $\h^1_{\del^k}(\bR^d)$ associated to the dilations $\del^k$ in \eqref{2.5a}, supported on the $2^{\tau}$-dilate of a dyadic cube at scale $2^{-L}$, with $L\in\LL_S$, satisfying the stronger cancellation $\int_{\bR^{d_k}}a(\x )\,d\x _k=0$.

\smallskip

\item If $S=\big\{(A_r,k_r):r=1,\dots,s\big\}$ with $s\ge2$, an $(S,\tau)$-atom is a product (Chang-Fefferman) atom relative to the coarser decomposition $\bR^d=\bR^{A_1}\times\cdots \bR^{A_s}$ with dilations $\delta^{k_{r}}$ on each $A_{r}$, supported in a set $\Omega$ contained in a translate of $\tau \B(1)$, with pre-atoms $b_j$, each supported on the $2^{\tau}$-dilate of a dyadic cube at scale $2^{-L_j}$, with $L_j\in\LL_S$, with the stronger cancellations $\int_{\bR^{d_{k_r}}}b_j(\x )\,d\x _{k_r}=0$ for every $r=1,\dots,s$.
\end{enumerate}
Theorem \ref{Thm9.3} gives an atomic decomposition of a function $f\in \h_{\bfE}^{1}(\R^d)$: for any $\tau>0$, there is a constant $C_\tau$ so that $f=\sum_{j=0}^\infty\la_ja_j$, 
where each $a_j$ is a $(S_j,\tau)$-atom for some $S_j\in\SS_{\bfE}^{*}$ and $\sum_{j=0}^{\infty}|\la_j|\le C_\tau \|f\|_{\h^1_\bfE}$. 
More precisely, rewriting the Calder\'on reproducing formula as
$$
 f=f*\Psi_0*\widetilde\Psi_0+\sum\nolimits_{S\in\cS'_\bfE}\Big(\sum\nolimits_{L\in\LL_S}f*\Psi_L^{(-L)}*\widetilde\Psi_L^{(-L)}\Big)=f_\emptyset+\sum\nolimits_{S\in\cS'_\bfE}f_S,
 $$
 the proof shows that for given $\tau$, each term $f_S$ decomposes as a sum of $(S,\tau)$-atoms. 
 
 To complete the discussion of atomic decompositions, it remains to show that every $(S,\tau)$-atom belongs to the space $\h_{\bfE}^{1}(\R^d)$. The proof  is given in Theorem \ref{Thm11.7} and Proposition \ref{Prop11.9}, and requires a preliminary estimate on convolutions with multi-norm kernels proved in Section \ref{Sec12gg}  as well as a version of Journ\'e's covering lemma in $n$-parameters. The required statements are stated and proved in Section \ref{journe}, and follows the arguments in  \cite{MR1072104}.

\section{Multi-norm structures}\label{Sec3}

In this section we review the basic multi-norm theory developed in \cite{MR3862599}. Write $\R^{d}=\R^{d_{1}}\oplus\cdots\oplus\R^{d_{n}}$ and let  $\bfE=\{e(i,j)\}$ be a standard $n\times n$ matrix as defined in \eqref{2.3}. The family of dilations $\delta_{r}$ is defined in \eqref{1.1}, and $q_{i}=\sum_{h\in E_{i}}\la_i$ is the homogeneous dimensions of $\R^{d_{i}}$ as in \eqref{2.2ww}. For $1\leq i \leq n$,   $|\xib_i|$ is the $\delta_{r}$--homogeneous norm on $\R^{d_i}$ as in \eqref{2.2}. 
Let $\R_{+}^{n}=\big\{\t\in\R^{n}:\text{$t_{j}\geq 0$ for $1\leq j \leq n$}\}$. For $\x=(\x_{1}, \ldots, \x_{n}), \,\xib=(\xib_{1}, \ldots, \xib_{n})\in \R^{d}$, recall from \eqref{2.5a} that
\bea\label{2.5aa}
\delta_{t}^{i}(\x)&=\big(t^{\frac{1}{e(i,1)}}\x_{1}, \ldots, t^{\frac{1}{e(i,n)}}\x_{n}\big),&
\widehat\delta_{t}^{i}(\xib)&=\big(t^{e(1,i)}\xib_1,\dots, t^{e(n,i)}\xib_n\big),\\
N_{i}(\x)&=\max\nolimits_{1\leq j \leq n}|\x_{j}|^{e(i,j)}, &\widehat N_{i}(\xib)&=\max\nolimits_{1\leq j \leq n}|\xib_j|^{\frac{1}{e(j,i)}}.
\eea

To obtain the explicit description of the sets $\LL$ and  $B_L$ in \eqref{2.5}, it is convenient to introduce logarithmic parameters $\t=(t_1,\dots,t_n)\in\R^n_+=\big\{\t=(t_{1}, \ldots, t_{n})\in \R^{n}:t_i\ge0\ \forall i\big\}$ and logarithmic norms $\widehat n_i(t)$ by setting
\be\label{2.10}
\t=\pi(\xib)\overset{\rm def}=\left(\log_{2}^{+}|\xib_{1}|, \ldots, \log_{2}^{+}|\xib_{n}|\right)\text{ and }
\widehat n_{i}(\t)=\widehat n_{i}(t_{1}, \ldots, t_{n})=\max\nolimits_{1\leq j \leq n}|t_{j}|e(j,i)^{-1}.
\ee
The above formula for $\widehat n_i$ is such that $\widehat N_i(\xib)=2^{\widehat n_i(\pi(\xib))}$ for all $\xib \in B(1)^c$,
The conditions \eqref{2.3} guarantee that the principal cone $\Gamma(\bfE)$ in \eqref{2.8} has nonempty interior. In the next section we privilege definitions in the $\t$ variables, transporting them to the $\xib$ variables via $\pi$.

\subsection{Marked partitions and dominance}\label{Sec3.1}\quad

\smallskip

A {\it marked partition} $S=\big\{(A_{1},k_{1}), \ldots, (A_{s},k_{s})\big\}$ is a  partition of $\{1, \ldots, n\}$ with the entries $k_{r}\in A_{r}$ for $1\leq r\leq s$. We denote the set of {\it marked}, or {\it dotted}, elements of $S$ as
\bea\label{3.3qq}
D(S)=\{k_{1}, \ldots, k_{s}\}.
\eea
 When convenient we can use a different notation for marked partitions. For instance, instead of writing $S=\big\{(\{1,3\},1),(\{2\},2)\big\}$   we write $S=\{\overset. 1,3\}\{\overset.2\}$. The {\it principal marked partition}  is 
 \bea\label{3.4qq}
 S_{p}=\{\overset . 1\}\{\overset . 2\}\cdots\{\overset . n\}.
 \eea
 The variable $t_{k}$ is called {\it dominant} (respectively {\it strictly dominant}) in $\widehat n_{i}$ at the point $\t\in\R^n_+$, or shortly ``in $\widehat n_{i}(\t)$'', if $\frac{t_{k}}{e(j,k)}\geq \frac{t_{j}}{e(j,i)}$ for all $1\leq j \leq n$ (respectively $\frac{t_{k}}{e(j,k)}> \frac{t_{j}}{e(j,i)}$ for all $j\neq k$). If $S=\big\{(A_{1},k_{1}), \ldots, (A_{s},k_{s})\big\}$ is a marked partition, let
\bea\label{3.2rr}
\Gamma_{S}=\Big\{(t_{1}, \ldots, t_{n})\in \R_{+}^{n}:\frac{t_{j}}{e(j,i)}\leq \frac{t_{k_{r}}}{e(k_{r},i)} \text{ for all $1\leq r\leq s$, all $i\in A_{r}$ and all $1\leq j \leq n$}\Big\}
\eea
denote the closed convex cone of points $\t\in\R_{+}^{n}$ such that $t_{k_{r}}$ is dominant in $\widehat n_{i}(\t)$ for all $i\in A_{r}$ and all $1\leq r\leq s$. Then $\Gamma_{S_p}=\Gamma(\bfE)$. The symbols $\SS_n$, $\SS_\bfE$ and $\SS^*_{\bfE}$ are as defined just before equation \eqref{L=cupL_S}. The notions of dominance and strict dominance are extended to the $\xib$ variable in terms of the norms $\widehat N_j$ as follows.
\be\label{3.1}
\widehat E_{S}=\pi^{-1}(\Gamma_S)\cap B(1)^c=\big\{\xib\in B(1)^c:\text{ $j\in A_{r} \Longrightarrow |\xib_\ell|^{1/e(\ell,j)}\leq |\xib_{k_r}|^{1/e(k_r,j)}$ for $1\leq \ell \leq n$}\big\}.
\ee
It is possible for more than one variable to be dominant in {the same norm. This occurs on the sets 
\beas
\Lambda_{\R^{n}}&=\Big\{\t\in \R^{n}:\big(\exists j\big)\big(\exists k\big)(\exists\ell\big)\big(\text{$\ell\neq k$ and $t_k e(k,j)^{-1}=t_{\ell}e(\ell,j)^{-1}$}\Big\},\\
\Lambda_{\R^{d}}&=\big\{\xib\in\R^{d}\setminus B(1):\big(\exists j\big)\big(\exists k\neq \ell\big)\big(|\xib_k|^{1/e(k,j)}=|\xib_{\ell}|^{1/e(\ell,j)}\big)\big\}.
\eeas
 These are finite unions of proper closed submanifolds and so have measure zero. Thus if $\t\notin \Lambda_{\R^{n}}$ (respectively $\xib\notin \Lambda_{\R^{d}}$), for every $j$ there is a unique index $k$ so that $t_{k}$ is strictly dominant in $\widehat n_{j}(\t)$ (respectively $\xib_{k}$ is strictly dominant in $\widehat N_{j}(\xib)$), and the same is true on a full neighborhood of $\t$ or $\xib$. Lemma \ref{Lem3.1} below summarizes the main consequences of these notions. The proofs can be extracted from Sections 3.1-3.3 in \cite{MR3862599}.\footnote{Parts \eqref{Lem3.1i}--\eqref{Lem3.1iv} only require  the weaker hypothesis that the matrix satisfies the basic assumptions, the first two lines of \eqref{2.3}.} 

\goodbreak

\begin{lemma}\label{Lem3.1}
Let $\bfE$ satisfy \eqref{2.3} and let $S$ be a marked partition. 
\begin{enumerate}[\rm(1)]
\item\label{Lem3.1i}
 If $\t\in \R_{+}^{n}$ and $t_{k}$ is (strictly) dominant in $\widehat n_{j}(\t)$ for some $j\ne k$,  then $t_{k}$ is also (strictly) dominant in $\widehat n_{k}(\t)$.
 
 \smallskip

\item  If $\xib\in B(1)^c$ and $\xib_k$ is (strictly) dominant in $\widehat N_j(\xib)$  then $\xib_{k}$ is  (strictly) dominant in $\widehat N_k(\xib)$. 

\smallskip

\item\label{Lem3.1iii} If $\t\notin \Lambda_{\R^{n}}$ (respectively $\xib\notin\Lambda_{\R^{d}}$), there is a unique marked partition $S=\big\{(A_1,k_1),\dots,(A_s,k_s)\big\}$ such that, for every $r$, $t_{k_{r}}$ is strictly dominant in $\widehat n_{j}(\t)$ for $j\in A_{r}$ (respectively $\xib_{k_r}$ is strictly dominant in all the norms $\widehat N_j(\xib)$ with $j\in A_r$).

\smallskip
\item \label{Lem3.1iv}
 $\xib\in\widehat E_S$  if and only if  $\xib_{k_r}$ is dominant in  $\widehat N_j(\xib)$ for all $j\in A_r$ and every $1\leq r\leq s$.

\smallskip

\item \label{Lem3.1v}If $\bfE$ is a standard matrix and $S_p=\{\overset . 1\}\{\overset . 2\}\cdots\{\overset . n\}$  then $\text{\rm Int}(\Gamma_{S_{p}})\neq \emptyset$.
\end{enumerate}
\end{lemma}

Let $S= \big\{(A_1,k_1),\dots,(A_s,k_s)\big\}\in \SS_{\bfE}$  and let $1\le q,r\le s$. Define
\bea\label{3.3rr}
\tau_{S}(k_q,k_{r})=\min\nolimits_{\ell\in A_{r}}\frac{e(k_q,\ell)}{e(k_{r},\ell)}.
\eea
The following result is contained in \cite{MR3862599}, Corollary 3.10. 
\begin{proposition}\label{Prop3.2}
If $S= \big\{(A_1,k_1),\dots,(A_s,k_s)\big\}\in \SS_{\bfE}$ then
\bea\label{3.3}
\Gamma_S&=\Big\{\t\in \R_{+}^{n}:t_j\leq\frac {t_{k_{r}}}{e(k_{r},j)}\text{ for all }j \in A_{r} \text{ and } t_{k_{p}}\leq \tau_{S}(k_{p}, k_{q})t_{k_{q}}\text{ for all }1\leq p,q \leq s\Big\},
\eea
and
 \bea\label{3.2}
 \widehat E_{S}=\Bigg\{\xib\in B(1)^c:
 &
\begin{cases}
 |\xib_j|\leq |\xib_{k_r} |^{\frac1{e(k_r,j)}}&\text{ for }1\le r\le s\text{ and }j\in A_r,\\ 
 |\xib_{k_{p}}|\leq |\xib_{k_{q}}|^{\tau_{S}(k_{p}, k_{q})}&\text{ for }1\leq p,q \leq s\qquad
 \end{cases}\Bigg\}.
\eea
\end{proposition}

\begin{remark}
{\rm Notice that the conditions in \eqref{3.2} describing the set $\widehat E_S$ can also be expressed in the following equivalent form:
\be\label{3.2new}
 \widehat E_{S}=\Bigg\{\xib\in B(1)^c:
\begin{cases}
N_{k_r}(\xib_{A_r})= |\xib_{k_r} |&\text{ for }1\le r\le s,\\ 
 N_{k_p}(\xib_{A_p})\leq N_{k_q}(\xib_{A_q})^{\tau_{S}(k_{p}, k_{q})}&\text{ for }1\leq p,q \leq s
 \end{cases}\Bigg\},
\ee
where $\xib_{A_r}=(\xib_i)_{i\in A_r}$. Together with the analogous reformulation of \eqref{3.3}, this is interesting in two ways: 
\begin{enumerate}[$\bullet$]
\item it emphasizes the presence of a coarser multi-norm structure on $\R^d$ associated to $S$, with $\R^d$ decomposed as $\prod_{r=1}^s\R^{A_r}$, homogeneous norms $N_{k_r}$, and principal cone $\big\{(t_{k_1},\dots,t_{k_s}):(t_1,\dots,t_n)\in\Gamma_S\big\}\subset\R^s_+$;
\item  it gives a simple example of the interplay between the two kinds of norms in \eqref{2.5aa}: starting from the dominance principle in the $\widehat N$-norms, the resulting sets $\widehat E_{S}$ are described in terms of the $N$-norms. 
\end{enumerate}
The most relevant form of this interplay is the correspondence, via Fourier transform, described in Section \ref{Sec3.5} between kernels satisfying estimates involving the $N$-norms and multipliers satisfying estimates involving the $\widehat N$-norms.
}
\end{remark}

\subsection{Partition of $\R_{+}^{n}$ in terms of dotted elements}\label{Sec3.3}\quad

\smallskip

At this point it is preferable to classify points $\t\in\R^n_+\setminus\{\0\}$ in terms of sets $D$ of dotted elements rather than  in terms of marked partitions.  By Lemma \ref{Lem3.4} below, it turns out that this can be done in a natural way without excluding the singular points in $\Lambda_{\R^n}$. In fact we show that for every $\t\in\R_{+}^{n}\setminus\{\0\}$ the set of the $D(S)\subseteq\{1, \ldots, n\}$ with $S\in\SS_\bfE$ such that $\t\in\Gamma_S$ admits a minimum element under inclusion.

\begin{lemma}\label{Lem3.4}
If $\0\neq \t\in \R_{+}^{n}$, let
$\SS(\t)=\big\{S\in \SS_{\bfE}:\t\in\Gamma_{S}\big\}$. There exists a (possibly non-unique)\footnote{Non-uniqueness is quite possible if $n\ge3$: Figure 2 shows that there are points belonging to $\Gamma_S\cap\Gamma_{S'}$ with $S=\{\overset\cdot 1\}\{\overset\cdot 2,3\}$ and $S'=\{\overset\cdot1,3\}\{\overset\cdot2\}$, with $D=\{1,2\}$.} $S_{0}=\big\{(A_1,k_1),\dots,(A_{s_{0}},k_{s_{0}})\big\}\in\SS(\t)$ such that 
\begin{enumerate}[(a)]
\item\label{Lem3.4a} $t_{k_p}<e(k_p,k_q)t_{k_q}$ for all  $1\leq p\neq q\leq s_{0}$;
\smallskip
\item\label{Lem3.4b} if $T=\big\{(B_{1},m_{1}), \ldots, (B_{\sigma}, m_{\sigma})\big\}\in \SS(\t)$ then $\{k_{1}, \ldots, k_{s_{0}}\}\subseteq \{m_{1}, \ldots, m_{\sigma}\big\}$.
\end{enumerate}
\end{lemma}

\begin{proof}
Let $s_{0}=\min\big\{\big|(D(S)\big|:S\in \cS(\mbt)\big\}$, let $S_0=\big\{(A_1,k_1),\dots,(A_{s_0},k_{s_0})\big\}\in\SS(\t)$, and suppose that $t_{k_p}\ge e(k_p,k_q)t_{k_q}$ for some $p\neq q$. Now if $j\in A_q$, $\frac{t_\ell}{e(\ell,j)}\leq \frac{t_{k_q}}{e(k_q,j)}$ for every $\ell$ since  $\t\in S_{0}$ and so 
\beas
\frac{t_\ell}{e(\ell,j)}\leq \frac{t_{k_q}}{e(k_q,j)}\leq \frac{t_{k_p}}{e(k_p,k_q)e(k_q,j)}\leq \frac{t_{k_p}}{e(k_p,j)}.
\eeas 
It then follows from \eqref{3.2rr}  that $\t\in \Gamma_{\widetilde S_{0}}$ where $\widetilde S_{0}$ is obtained from $S_{0}$ by replacing the two elements $(A_p,k_p),(A_q,k_q)$ of $S_0$ with the single element $(A_p\cup A_q,k_p)$. This would contradict the minimality of $s_{0}$, proving \eqref{Lem3.4a}.

Now let $T=\big\{(B_1,m_1),\dots,(B_s,m_s)\big\}\in\cS(\mbt)$. To prove \eqref{Lem3.4b}
suppose that $k_1\notin\{m_{1}, \ldots, m_{\sigma}\}$. Without loss of generality, we may suppose that $k_1\in B_1\setminus\{m_1\}$. Let $p$ be such that $m_1\in A_p$. If $p=1$ then since $k_{1}\in B_{1}$, this means that $t_{m_1}\leq \frac{t_{k_1}}{e(k_1,m_1)}\leq \frac{t_{m_1}}{e(k_1,m_1)e(m_1,k_1)}$, implying that $e(k_1,m_1)e(m_1,k_1)\le1$. This contradicts the hypothesis that the matrix $\bfE$ is standard, and so $p\neq 1$. It remains to consider the case $p\ge2$. Suppose that $m_1\in A_2$. For every $i\in A_1$ we have
$$
t_i\leq \frac{t_{k_1}}{e(k_1,i)}\leq \frac{t_{m_1}}{e(m_1,k_1)e(k_1,i)}\leq \frac{t_{k_2}}{e(k_2,m_1)e(m_1,k_1)e(k_1,i)} \leq \frac{t_{k_2}}{e(k_2,i)}\ .
$$
This implies that $S'=\big\{(A_1\cup A_2,k_2),(A_3,k_3),\dots,(A_{s_0},k_{s_0})\big\}\in\cS(\mbt)$, contradicting once again minimality of $s_0$. This completes the proof. 
\end{proof}

We denote by $D(\t)$ the set $D(S_0)$, $S_0$ being as in Lemma \ref{Lem3.4}. For $D\subseteq\{1,\dots,n\}$ we define
$$
\Gamma(D)=\begin{cases}\{0\}&\text{ if }D=\emptyset\\
\big\{\t\ne0:D(\t)=D\big\}& \text{ otherwise. }\end{cases}
$$
Obviously the sets $\Gamma(D)$ partition $\R^n_+$. The next result  follows directly from Lemma \ref{Lem3.4}.

\begin{proposition}\label{Prop3.5}For $\emptyset\ne D\subseteq\{1,\dots,n\}$, let $\SS(D)$ be the set of $S\in\SS_\bfE$ such that $D(S)=D$. Then
$$
\Gamma(D)=\Big(\bigcup\nolimits_{S\in \SS(D)}\Gamma_{S}\Big)\cap\Big\{\t\in\R_{+}^{n}:\text{$0<t_{k_{p}}<t_{k_{q}}e(k_{p},k_{q})$ for all $k_{p},k_{q}\in D$ with $p\neq q$}\Big\}.
$$
\end{proposition}

The sets $\Gamma(D)$ are cones but are not necessarily convex.  If $S=\big\{(A_{1},k_{1}), \ldots, (A_{s},k_{s})\big\}$ and $D=\{k_{1}, \ldots,k_{s}\}$ let $F_{S}=\Gamma_{S}\cap\Gamma(\bfE)$ and  $F(D)=\Gamma(D)\cap \Gamma(\bfE)$ so that
\bea\label{3.6*}
F_{S}&=\Big\{\t\in\R_{+}^{n}:\text{$t_{j}=\frac{t_{k_{p}}}{e(k_{p},j)}$ if $j\in A_{r}$ and $t_{k_{p}}\leq \tau_{S}(k_{p},k_{q})t_{k_{q}}$ for $1\leq p,q\leq s$}\Big\},\\
F(D)&=
\Big(\bigcup\nolimits_{S\in \SS(D)}F_S\Big)\cap\left\{\t\in \R_{+}^{n}:t_{k_{p}}<e(k_{p},k_{q})t_{k_{q}} \text{ for all\ } k_{p}\neq k_{q}\in D\right\}.
\eea
In particular if $S_{p}$ is the principal marked partition from \eqref{3.4qq} then $F_{S_{\rm p}}=\Gamma(\bfE)$.  The sets $\big\{F(D):D\subseteq\{1, \ldots, n\}\big\}$ are disjoint by Proposition \ref{Prop3.5}, and partition the cone $\Gamma(\bfE)$. More precisely, if $D=\{1,\dots,n\}$, $F(D)=\INT\big(\Gamma(\bfE)\big)$, and the faces $F(D)$ with $D\subsetneqq\{1,\dots,n\}$ partition $\de\Gamma(\bfE)$, with $F(\emptyset)=\{0\}$.\footnote{It can be proved that each set $F_{S}$ is a full {\it face} of  $\Gamma(\bfE)$, but we do not prove this since it is not needed. If $n\ge3$ not all faces of $\Gamma(\bfE)$ are included.}

\subsection{ Multi-norm dyadic decomposition of $\R_{\xib}^{d}$}\label{Sec3.4}\quad
\smallskip

If $L=(\ell_{1}, \ldots, \ell_{n})\in \N^{n}$ and $D\subseteq \{1, \ldots, n\}$ let 
\bea
Q_{L}&=\Big\{\s=(s_{1}, \ldots, s_{n})\in\R^{n}:\ell_{j}-\frac{1}{2}\leq s_{j}\leq \ell_{j}+\frac{1}{2}\Big\}&&\text{ and }&
\QQ_{D}&=\Big\{Q_{L}:L\in \Gamma(D)\Big\}.
\eea 
The sets $Q_{L}$ are cubes of unit size with disjoint interiors and are entirely contained in $\R_{+}^{n}$ if all $\ell_{j}>0$.
Since the sets $\big\{\Gamma(D):D\subseteq\{1, \ldots, n\}\big\}$ partition $\R_{+}^{n}$, the sets $\big\{\QQ_{D}:D\subseteq\{1, \ldots, n\}\big\}$ partition the set of cubes $\{Q_{L}:L\in\N^{n}\}$.  Let $\{\e_{1}, \ldots, \e_{n}\}$ be the standard basis of $\R^{n}$, and let  $D\subseteq\{1, \ldots, n\}$ be non-empty.
The set $Q_{L}\in\QQ_{D}$ is {\it maximal in $\QQ_{D}$} if $Q_{L+\e_{j}}\notin \QQ_{D}$ for every $j\not\in D$.  Note that if $D=\{1, \ldots, n\}$ every $Q_{L}\in \QQ_{D}$ is automatically maximal in $\QQ_{D}$ and $L\in \INT\Gamma(\bfE)$. Also if $D=\emptyset$ then $\QQ(\emptyset)$ contains exactly the unit cube centered at the origin, which is obviously  maximal in $\QQ_{\emptyset}$. For $\emptyset\neq D\subsetneqq \{1, \ldots, n\}$ the situation is more complicated.

\begin{lemma}\label{Lem3.7xyz}
Let $\emptyset\neq D\subsetneqq \{1, \ldots, n\}$, $L\in\Gamma(D)$ and $S=\big\{(A_{1},k_{1}), \ldots, (A_{s},k_{s})\big\}$ be a marked partition with $D(S)=D$ and $L\in\Gamma(D)\cap\Gamma_S$.  Let $\lfloor x\rfloor$ denote the largest integer less than or equal to $x$. With $\tau_S(k_{p},k_{q})$  defined in \eqref{3.3rr}, if $L\in \Gamma_S$ then $Q_{L}$ is maximal in $\QQ_{D}$ if and only if
\bea\label{3.6}
\ell_{i}&=\left\lfloor\frac{\ell_{k_{r}}}{e(k_{r},i)}\right\rfloor&&\text{for all $i\in A_{r}$}&&\text{and}&
\ell_{k_{p}}&< \tau_S(k_{p},k_{q})\ell_{k_{q}}&&\text{for all $1\leq p\neq q\leq s$}.
\eea
\end{lemma}

\begin{proof}The conditions \eqref{3.6} are necessary in general and also sufficient is there is no other $S'$ such that $D(S')=D$ and $L\in\Gamma(D)\cap\Gamma_{S'}$. Assuming instead that this occurs with $S'=\big\{(B_{1},k_{1}), \ldots, (B_{s},k_{s})\big\}$, suppose $i\in B_{q}\cap A_{r}$.  Then both $\ell_{k_{r}}$ and $\ell_{k_{q}}$ are dominant in $\widehat n_{i}(\ell_{1}, \ldots, \ell_{n})$ and it follows from \eqref{2.5aa} that $e(k_{r},i)^{-1}\ell_{k_{r}}=e(k_{q},i)^{-1}\ell_{k_{q}}$. So the conditions \eqref{3.6} are the same for $S$ and $S'$.
\end{proof}

\medskip
Let $D\subseteq\{1, \ldots, n\}$ and put
\bea\label{3.16rr}
\LL(D)&=\Big\{L\in \N^{n}:\text{$Q_{L}$ is maximal in $\QQ_{D}$}\Big\}&&\text{ and }&
\LL&=\bigcup\nolimits_{D\subseteq\{1, \ldots, n\}}\LL(D)\subset\N^{n}.
\eea
Since $\{\QQ_{D}\}$ partitions the set of  cubes, $\{\LL(D):D\subseteq\{1, \ldots, n\}\big\}$ is a partition of the index set $\LL$.  If $L\in \LL$ let $D_{L}\subseteq\{1, \ldots, n\}$ be the unique subset such that $L\in \Gamma(D_{L})$. We enlarge each cube which is maximal in $\QQ_{D}$ to a larger \textit{tube} in $\R_{+}^{n}$. If $L=(\ell_{1}, \ldots, \ell_{n})\in \LL$ put
\bea\label{3.8}
\TT_{L}&=
\Big\{L'=(\ell_{1}', \ldots, \ell_{n}')\in \N^{n}:\text{$\ell_{j}'=\ell_{j}$ if $j\in D_{L}$ and $0\leq \ell_{j}'
\leq \ell_{j}$ if $j \notin D_{L}$}\Big\}&&\text{and}\\
T_{L}&=
\Bigg\{\t\in \R_{+}^{n}:
\begin{cases}
\ell_{j}-\frac{1}{2}\leq t_{j}\leq \ell_{j}+\frac{1}{2}&\text{if $j\in D_{L}$}\\
0\leq t_{j}\leq \ell_{j}+\frac{1}{2}&\text{if $j\notin D_{L}$}
\end{cases}\Bigg\}
=\bigcup\nolimits_{L'\in \TT_{L}}Q_{L'}
\eea
Thus if $D_{L}=\{1, \ldots, n\}$ or $D_{L}=\emptyset$ then $T_{L}=Q_{L}$ is a cube and $\TT_{L}=L$. If $\emptyset\neq D\subsetneq\{1, \ldots, n\}$ then $T_{L}$ is a tube of side length $1$ in the directions of $\e_{j}$ if $j\in D_{L}$ and length $\ell_{j}+\frac{1}{2}$ in the direction of $\e_{j}$ if $j\notin D_{L}$. The following lemma says that the sets $\{\TT_{L}:L\in \LL\}$ are a partition of $\N^{n}$ and the tubes $\{T_{L}:L\in \LL\}$ cover $\R_{+}^{n}$ with disjoint interiors.

\begin{lemma}\label{Lem3.7}\quad

\begin{enumerate}[\rm(a)]
\item \label{Lem3.7a}
If $Q_{M}\in \QQ_{D}$ there is a unique maximal element $Q_{L}\in \QQ_{D}$ with $Q_{M}\subseteq T_{L}$.

\smallskip

\item \label{Lem3.7c}
$\displaystyle \R_{+}^{n}= \bigcup\nolimits_{D\subseteq\{1, \ldots, n\}}\Big(\bigcup\nolimits_{L\in \LL(D)}T_{L}\Big)$.

\smallskip

\item \label{Lem3.7d}
If $L_{1}\neq L_{2}$ are in $\LL$, then $T_{L_{1}}\cap T_{L_{2}}$ is closed with measure zero.

\end{enumerate}

\end{lemma}

\begin{proof} 
Let $D=\{k_{1}, \ldots, k_{s}\}$. If $M=(m_{1}, \ldots, m_{n})$ and $Q_{M}\in \QQ_{D}$ then $M\in \Gamma_S$ for a marked partition $S=\big\{(A_{1},k_{1}), \ldots, (A_{s},k_{s})\big\}$ and from \eqref{3.3} we have
\beas
0&\leq m_{j}\leq \frac{1}{e(k_{r},j)}m_{k_{r}} &&\text{for all $j \in A_{r}$, \quad and \quad}&
0&<m_{k_{p}}\leq\tau_{S}(k_{p}, k_{q})m_{k_{q}}&\text{for all $1\leq p,q \leq s$.}
\eeas
Put
$
\ell_{j}=
\begin{cases}
m_{j}&\text{if $j=k_{r}\in D$}\\
\lfloor e(k_{r},j)^{-1}m_{k_{r}}\rfloor&\text{if $j\in A_{r}$ and $j\neq k_{r}$}
\end{cases}$. 
Then $L$ satisfies the inequalities of \eqref{3.6} and so $Q_{L}\in \QQ_{D}$ and $Q_{L}$ is maximal. Since $m_{j}\leq \ell_{j}$ for all $1 \leq j \leq n$ it follows that $M\in \TT_{L}$ and so $Q_{M}\in T_{L}$. Since cube maximal in $\LL(D)$ are determined by the dotted entries, this proves \eqref{Lem3.7a}. 
If $\t=(t_{1}, \ldots, t_{n})\in \R_{+}^{n}$ then $\t\in Q_{M}$ for some $M\in \N^{n}$, and so \eqref{Lem3.7c} follows from \eqref{Lem3.7a}. Finally if two distinct sets $\TT_{L}$ intersect in a set of positive measure then the intersection contains a cube $Q_{L}$, and so \eqref{Lem3.7d} follows from \eqref{Lem3.7a}.
\end{proof}
\label{LLS}
Although the sets $\LL(D)$ partition $\LL$, we will require a finer partition  indexed by marked partitions  $S\in\SS_{\bfE}^{*}=\SS_\bfE\cup\{\emptyset\}$ rather than by sets of dotted entries. This partition involves a certain, though irrelevant, amount of arbitrariness. We choose pairwise disjoint subsets $\LL_{S}\subseteq \LL(D)$ such that $\LL_\emptyset=\{0\}$, $\LL_S$ contains all  $L\in\LL\cap \INT\Gamma(S)$, and $\LL(D)=\bigcup\nolimits_{D(S)=D}\LL_S$. Thus if $L=(\ell_{1}, \ldots, \ell_{n})\in \LL_{S}$ it follows from \eqref{3.6} that  $\ell_{i}=\left\lfloor\frac{\ell_{k_{r}}}{e(k_{r},i)}\right\rfloor$ for all $i\in A_{r}$ and $\ell_{k_{p}}< \tau_S(k_{p},k_{q})\ell_{k_{q}}$ for all $1\leq p \neq q\leq s$. In addition to the  decomposition in Lemma \ref{Lem3.7} \eqref{Lem3.7c} we also have  the alternative decomposition
\be\label{sum LL_S}
\R_{+}^{n}= \bigcup\nolimits_{S\in\SS_{\bfE}^{*}}\Big(\bigcup\nolimits_{L\in \LL_S}T_{L}\Big).
\ee

This now leads to decompositions of the frequency space $\R^{d}$. With $\pi:\R^{d}\to\R_{+}^{n}$ defined in \eqref{2.10} and setting $B_L=\pi^{-1}(T_L)$ for $L\in \LL$, we have the decompositions  $ \R^{d}= \bigcup\nolimits_{D\subseteq\{1, \ldots, n\}}\big(\bigcup\nolimits_{L\in \LL(D)}B_L\big)$ and $\R^{d}= \bigcup\nolimits_{S\in\SS_{\bfE}^{*}}\big(\bigcup\nolimits_{L\in \LL_{S}}B_L\big)$.
The sets $B_L$ are the basic dyadic blocks in $\R^{d}$ for the multi-norm structure generated by $\bfE$.
We prove that the sets $B_L$ are comparable to products of dyadic balls and annuli in coordinate subspaces and $N_i$-norms depending on the marked partition $S$ such that $L\in\LL_S$.
\begin{proposition}\label{Prop3.8.5}
Let $L\in\LL_S$ with $S=\big\{(A_r,k_r):r=1,\dots,s\big\}$. Then there are constants $c,C>0$ depending only on the matrix $\bfE$ such that
\bea\label{3.8.5}
&\prod\nolimits_{r=1}^s\big\{\xib_{A_r}\in\R^{A_r}:N_{k_r}(\xib_{A_r})\le c2^{\ell_{k_r}}\,,\,|\xib_{k_r}|>(c/2)2^{\ell_{k_r}}\big\}\\
&\qquad\subset B_L\subset\prod\nolimits_{r=1}^s\big\{\xib_{A_r}\in\R^{A_r}:N_{k_r}(\xib_{A_r})\le C2^{\ell_{k_r}}\,,\,|\xib_{k_r}|>(C/2)2^{\ell_{k_r}}\big\}\ .
\eea
\end{proposition}

\begin{proof}
By \eqref{3.8},
\beas
B_L=\pi^{-1}(T_L)&=\Big\{\xib:\frac1{\sqrt2}2^{\ell_j}\le|\xib_j|\le\sqrt2\,2^{\ell_j}\text{ if }j\in D_L\text{ and }|\xib_j|\le\sqrt2\,2^{\ell_j}\text{ if }j\not\in D_L\Big\}\\
&=\prod\nolimits_{r=1}^s\Big\{\xib_{A_r}:\frac1{\sqrt2}2^{\ell_{k_r}}\le|\xib_{k_r}|\le\sqrt2\,2^{\ell_{k_r}}\ ,\ |\xib_j|\le\sqrt2\,2^{\ell_j}\text{ if }j\ne k_r\Big\}.
\eeas

Given $\xib\in B_L$ and $r\in\{1,\dots,s\}$, we have $N_{k_r}(\xib_{A_r})=\max_{i\in A_r}|\xib_i|^{e(k_r,i)}\le\max_{i\in A_r}2^{e(k_r,i)\ell_i}$. Since $\ell_i=\big\lfloor\frac{\ell_{k_r}}{e(k_r,i)}\big\rfloor$, $N_{k_r}(\xib_{A_r})\le 2^{\ell_{k_r}}$. The other inequality is obvious and this proves the second inclusion. The first inclusion is proved similarly.
\end{proof}

\subsection{Distances between tubes and $\Gamma(\bfE)$}\label{Sec3.3.3}\quad

\smallskip
We use the $\ell^{\infty}$-distance $|\t-\u|=\max_{i=1\dots,n}|t_{i}-u_{i}|$ on $\R_{+}^{n}$. 

\begin{lemma}\label{Lem3.8}
There is a constant $\kappa>0$ depending only on $\mbE$ such that
\begin{enumerate}[\rm(i)]
\item \label{Lem3.8a}for every $L\in\LL$, the distance from $L$ to $\Gamma(\mbE)$ is smaller than $\kappa$; more precisely, if $L\in\Gamma_S$, then the distance from $L$ to $F_S$ is smaller than $\kappa$;
\item \label{Lem3.8b}for every point $\t\in\Gamma(\mbE)$, the distance from $\t$ to $\LL$ is smaller than $\kappa$;
\item \label{Lem3.8c} given $L,L'\in\LL$, the distance from $\bar L=L\wedge L'\in\N^{n}$ to $\LL$ is smaller than $\kappa$;
\item \label{Lem3.8d}if $L,L'\in\LL$ and $T_L\cap T_{L'}\ne\emptyset$, then $|L-L'|\leq \kappa$.
\end{enumerate}
\end{lemma}

\begin{proof}

For $\delta >0$ consider the set
$
\Gamma_\del=\Big\{\t\in\bR^n_+: \frac{t_i}{e(i,j)}\leq t_j+\del\ ,\ \text{ for all } i,j\Big\} .
$
We show that points in $\Gamma_\del$ have uniformly bounded distances from $\Gamma(\mbE)$.
Clearly $\Gamma(\mbE)\subset\Gamma_\del$. Let $\t\in\Gamma_\del\setminus\Gamma(\mbE)$ and fix a marked partition $S=\big\{(A_r,k_r):r=1,\dots,s\big\}$ such that $\t\in\Gamma_S$. Consider the point $\u\in F_S$ with entries
$u_i=t_{k_r}/e(k_r,i)$ if $i\in A_r$. By \eqref{3.3}, $|\u- \t|<\del$. It suffices to prove that, for some $\del>0$, every $L\in\LL$ is in $\Gamma_\del$. This is obvious if $L\in\LL_{S_{\rm p}}$. If $S$ is non-principal, we use \eqref{3.6} to obtain an upper bound for the differences $d_{i,j}=\ell_i/e(i,j)-\ell_j$.
For $i\in A_p$ and $j\in A_q$ we have
$
d_{i,j}\leq \frac{\ell_{k_p}}{e(k_p,i)e(i,j)}-\big(\frac{\ell_{k_q}}{e(k_q,j)}-1\big)
\leq \big(\frac{e(k_p,k_q)}{e(k_p,i)e(i,j)}-\frac1{e(k_q,j)}\big)\ell_{k_q}+1
\leq \frac1{e(i,j)e(k_q,j)} \big(\frac{e(k_p,j)}{e(k_p,i)}-e(i,j)\big)\ell_{k_q}+1
\leq 1\ .
$
It follows that $\LL\subset \Gamma_\del$ with $\del=\max_{i,j}e(i,j)$, and this proves \eqref{Lem3.8a}.
To prove (ii), let $L\in\N^{n}$ such that $\t\in Q_L$. If $L\in\LL$, we are finished. If $L\not\in L$, then there is $L'\in\LL$ such that $Q_M\subset T_L$, as defined in \eqref{3.8}. The same formula easily implies that $L=L'-\sum_{i\not\in D_{L'}}p_i\e_i$ with $p_i\in\bN$. Then (ii) will be proved if we show that, since $Q_L$ intersects $\Gamma(\mbE)$, then $\max_{i\not\in D_{L'}}p_i=|L-L'|$ must have an upper bound independent of $\t$. Write $\t$ as $\t=L+\u=L'-\sum_{i\not\in D_{L'}}p_i\e_i+\u$, with $|u_i|\leq \frac{1}{2}$ for every $i$. Since $\t\in\Gamma(\mbE)$, setting $p_i=0$ for $i\in D_{L'}$, the conditions
$
\ell'_i-p_i+u_i\leq e(i,j)(\ell'_j-p_j+u_j)
$
are satisfied for every $i,j$. If $L'\in\Gamma_S$ with $S=\big\{(A_r,k_r):r=1,\dots,s\big\}$, we take $i=k_r$, $j\in A_r\setminus\{k_r\}$. Then \eqref{3.6} gives
$\ell'_{k_r}+u_{k_r}\leq e(k_r,j)(\ell'_j-p_j+u_j)$,
so that 
$
p_j\leq \ell'_j -\frac{\ell'_{k_r}}{e(k_r,j)}+u_j-\frac{u_{k_r}}{e(k_r,j)}\leq \half\Big(1+\frac1{e(k_r,j)}\Big),
$
and \eqref{Lem3.8b} follows easily.
To prove \eqref{Lem3.8c}, observe that $\Gamma_\del$ is closed under $\wedge$. Then, for $L,L'\in\LL$, also $\bar L\in\Gamma_\del$ and the proof of \eqref{Lem3.8a} shows that $d\big(\bar L,\Gamma(\mbE)\big)\leq C\del$ and the conclusion follows easily. To prove \eqref{Lem3.8d}, consider $L\in\Gamma_S$ and $L'\in\Gamma_{S'}$ such that there exists $\mbt\in T_L\cap T_{L'}$. We let $S=\big\{(A_r,k_r):r=1,\dots,s\big\}$, $S'=\big\{(A'_r,k'_r):r=1,\dots,s'\big\}$., and $D=D_S$, $D'=D_{S'}$. There exist cubes $Q_{L-P}\subset T_L$ and $Q_{L'-P'}\subset T_{L'}$, with
$P=\sum_{i\not\in D}p_i\mbe_i$, $P'=\sum_{i\not\in D'}p'_i\mbe_i$ in $\N^{n}$ such that $\mbt\in Q_{L-P}\cap Q_{L'-P'}$. We may assume that $p_i=0$ for $i\in D$ and $p'_i=0$ for $i\in D'$. We remark that, since $|L-L'|\leq |L-\mbt|+|\mbt-L'|\leq |P|+|P'|+n$, we must find uniform upper bounds for $p_i$ and $p'_i$. This is obvious for $i\in D\cap D'$. Assume now that $i\in D\setminus D'$. Since $p_i=0$ and $p'_i\ge0$, we have $|t_i-\ell_i|\leq \frac{1}{2}$, and $|t_i-\ell'_i+p'_i|\leq \frac{1}{2}$. In particular, $\ell_i\leq \ell'_i+1$. Assume that $i\in A'_r$ and $k'_r\in A_q$. Then, using \eqref{3.6}, $p'_i\leq \ell'_i-t_i+\half \leq \ell'_i-\ell_i+1
\leq \frac{\ell'_{k'_r}}{e(k'_r,i)}-\frac{\ell_{k_q}}{e(k_q,i)}\leq \frac{\ell'_{k'_r}}{e(k'_r,i)}-\frac{e(k_q,k'_r)}{e(k_q,i)}\ell_{k'_r}
\leq \frac{\ell'_{k'_r}-\ell_{k'_r}}{e(k'_r,i)}$. Since $k'_r\in D'$, it follows from the previous remark that $\ell'_{k'_r}\leq \ell_{k'_r}+1$, so that $p'_i\leq 1/e(k'_r,i)$. If $i\in D'\setminus D$ the argument is similar. It remains to consider the case $i\not\in D\cup D'$.
We first observe that, if both $p_i$ and $p'_i$ are strictly positive, say $1\leq p_i\leq p'_i$, then also the translated cubes $Q_{L-P+p_i\mbe_i}$ and $Q_{L'-P'+p_i\mbe_i}$ have nonempty intersection and are contained in $T_L$, $T_{L'}$ respectively. We may then assume that $p_i=0$. Then the previous proof can be repeated.
\end{proof}

\subsection{Multi-norm multipliers and kernels}\label{Sec3.5}\quad

\smallskip

The paper \cite{MR3862599} studies algebras $\PP_{0}(\bfE)$ of singular integral operators associated with a standard matrix~$\bfE$. The sizes of the Schwartz kernels of operators in $\PP_{0}(\bfE)$ are  controlled by the norms $N_{i}$ and the corresponding multipliers in $\MM_{\infty}(\bfE)$ are controlled by the norms $\widehat N_{i}$ defined in \eqref{2.5a}. To give the precise definitions we need the following notation. 
\begin{enumerate}[$\circ$]
\item
$\SS(\R^{d})$ is the space of Schwartz functions and $\SS'(\R^{d})$ is the space of tempered distributions on $\R^{d}$. If $\varphi\in \SS(\R^{d})$ then the quantities 
\bea\label{3.18yy}
\|\varphi\|_{(M)}= \sup\big\{\big\vert\x^{\alphab}\partial^{\betab}\varphi(\x)\big\vert:\text{$|\alphab|+|\betab|\leq M$ and $\x\in \R^{d}$}\big\}
\eea
are the basic semi-norms defining the topology on $\SS(\R^{d})$. 

\smallskip

\item
The families of dilations $\delta_{t}^{i}$ and $\widehat\delta_{t}^{i}$ on $\R^{d}=\bigoplus_{i=1}^{n}\R^{d_{i}}$ were defined in \eqref{2.5a}. The homogeneous dimensions of $\R^{d}$ are given by
\bea\label{3.16ww}
Q_{i}&=\sum\nolimits_{j=1}^{n}q_{j}e(i,j)^{-1}&&\text{for the dilations $\delta_{t}^{i}$, and}\\
\widehat Q_{i}&=\sum\nolimits_{j=1}^{n}q_{j}e(j,i)&&\text{for the dilations $\widehat\delta_{t}^{i}$}.
\eea
For $\alphab=(\alpha_{1}, \ldots, \alpha_{d})\in\bN^d$, set $\alphab^j=(\al_k)_{k\in E_j}\in\N^{d_j}$, so that $\alphab=(\alphab^1,\dots,\alphab^n)$. Then
\bea \label{3.17ww}
\bl\alphab^{j}\br &=\sum\nolimits_{k\in E_j}\la_k\alpha_k &&\text{is the homogeneous length of $\alphab^j$ relative to the dilations \eqref{1.1},} \\
\bl\alphab\br_i&= \sum\nolimits_{j=1}^{n}e(j,i)\bl\alphab^j\br&&\text{is the homogeneous length of $\alpha$ relative to } \widehat\delta^{i}_{r}.
\eea

\item
Given $\cK\in\cS'(\bR^d)$ and $\omega\in\cS(\bR^G)$, denote by $\cK_\omega\in\cS'(\bR^{G^c})$ the distribution such that
$\lan \cK_\omega,\ph\ran=\lan \cK, \omega\otimes\ph\ran$ for all $\ph\in\cS(\bR^{G^c})$. With an abuse of language we  write $\cK_\omega=\int_{\bR^G}\cK(\x _G,\cdot)\omega(\x _G)\,d\x _G$.

\smallskip

\item
If $f:\R^{d}\to\C$, the Fourier transform is denoted by $\FF[f]$ or $\widehat f$.

\end{enumerate}
The spaces of multi-norm kernels and multipliers are defined as follows. The symbol $\|\ \|$ will be used to denote any (homogeneous) norm on the appropriate space, whose choice is irrelevant.

\begin{definition}\label{Def3.1}\quad

\begin{enumerate}[A.]
\item
The space of \textit{multi-norm kernels} $\PP_{0}(\bfE)$ is the set $\KK\in\SS'(\R^{d})$ satisfying the following conditions. 
\begin{enumerate}[i)]
\item On $\R^{d}\setminus\{\0\}$ $\KK$ is given by integration against a function $K\in \CC^{\infty}(\R^{d}\setminus\{\0\})$, and
for $\alphab\in\bN^d$ and $N\in\bN$ there is a constant $C_{\alphab,N}$ so that for all $\x\neq \0$
\be\label{3.9}
\big|\de^\alphab K(\x )\big|\leq C_{\alphab,N}\Big(\prod\nolimits_{i=1}^nN_i(\x )^{-q_i-\bl\alphab_i\br}\Big)\big(1+\|\x \|\big)^{-N}\ .
\ee

\item
For every $F\subseteq\{1,\dots,n\}$, $R\in\bR^F_+$, and $\eta\in C^\infty(\bR^F)$ compactly supported on the unit ball $B_{F}(1)$ with $\|\eta\|_{\CC^n}\le1$, the distribution $\cK_{R,\eta}=\int_{\bR^F}\cK(\x _F,\cdot)\eta(R\x _F)\,d\x _F\in\cS'(\bR^{F^c})$ is given on $\bR^{F^c}\setminus\{0\}$ by a smooth function $K_{R,\eta}(\xib_{F^c})$ satisfying inequalities analogous to \eqref{3.9}; precisely,
\bea\label{3.9.5}
\big|\de^\al K_{R,\eta}(\x _{F^c})\big|\leq C'_{\alphab,N}\Big(\prod\nolimits_{i\in F^c}N_i(\x _{F^c})^{-q_i-\bl\alphab_i\br}\Big)\big(1+\|\x _{F^c}\|\big)^{-N}
\eea
for every $\al,N$, with constants $C'_{\al,N}$ independent of $R,\eta$.
\end{enumerate}
\item
The space of \textit{multi-norm multipliers} $\MM_{\infty}(\bfE)$ \ is the set of $m\in\CC^{\infty}(\R^{d})$ such that every $\alphab\in\bN^d$ there is a constant $C_{\alphab}$ so that
\bea\label{3.10}
\big|\de^{\alphab} m(\xib)\big|\leq C_{\mathbf\alphab}\prod\nolimits_{i=1}^{n}\big(1+\widehat N_{i}(\xib)\big)^{-\bl\alphab^i\br}.
\eea 
\end{enumerate}

\noindent We refer to the constants $C_{\alphab,N}$ in \eqref{3.9} and $C'_{\al,N}$ in \eqref{3.9.5} as the {\rm kernel constants} of $\KK$ and to the constants $C_{\alphab}$ in \eqref{3.10} as the {\rm multiplier constants} of $m$.
\end{definition}

\goodbreak

\begin{theorem}\label{Thm3.9}
If $\bfE$ is a standard matrix and $m\in \CC^{\infty}(\R^{d})$, the following properties are equivalent. 
\begin{enumerate}[\rm(1)]
\item $m\in \MM_{\infty}(\bfE)$.

\smallskip

\item $m=\FF\KK$ with $\KK\in \PP_{0}(\bfE)$.
\smallskip

\item
There exists $\big\{m_{L}:L\in\LL\big\}\subseteq\SS(\R^{d})$, uniformly bounded in every Schwartz norm, so that
$ m(\xib)=\sum\nolimits_{L\in\LL}m_L(2^{-L}\xib)$ and, if $L\in\LL_S$ and $D_S=\{k_1,\ldots,k_{s}\}$, then $m_L(\xib)=0$ on the subspace $V_{k_r}=\{\xib:\xib_{k_r}=\0\}$ for $1 \leq r\leq s$.

\smallskip

\item \label{Thm3.9.(4)}
There exist $\big\{\eta_{L}:L\in\LL\big\}\subseteq\SS(\R^{d})$, uniformly bounded in every Schwartz norm, so that $ \cK=\sum\nolimits_{L\in\LL}\eta_L^{(-L)}$  \footnote{The notation $f^{(-L)}$ is defined in \eqref{2.12}.} with the sum converging in the sense of distributions and, if $L\in\LL_S$ with $S=\big\{(A_r,k_r)\,,\, r=1,\dots,s\big\}$, then $\int_{\bR^{k_r}} \eta_L(\x )\,d\x _{k_r}=0$ for every $r=1,\dots,s$.
\end{enumerate}
\end{theorem}

\begin{proof}The proof is essentially contained in \cite{MR3862599}, though the statements are formulated with the following weaker conditions:
\begin{enumerate}[(3')]
\item[($3^{\prime}$)]: same as (3) but with $m_L$ vanishing on the smaller subspaces $V_{A_r}=\{\xib:\xib_{A_r}=\0\}$;

\smallskip

\item[($4^{\prime}$)]: same as (4) but with $\eta_L$ having integral 0 in each block of coordinates $\x_{A_r}$.
\end{enumerate}

The equivalence of (1), (2), (3'), (4') is a summary of Theorems 4.1, 4.2, 6.13, 6.14 in \cite{MR3862599}.
However, we will see in Section \ref{Sec4} below that $m(\xib)=\sum_{L\in\LL}m(\xib)\widehat{\Psi_L^{(-L)}}(\xib)$ where $\Psi_{L}$ is defined in \eqref{3.4ee}. If $L\in \LL(D)$ and $k\in D(S)$ then $\Psi_{L}\in\SS(\R^{d})$ has cancellation of order $m$ in each set of variables $\x_{k}$. This gives the implications (1)$\Rightarrow$(3) and (2)$\Rightarrow$(4).
The implications $(3)\Rightarrow (3')$ and $(4)\Rightarrow (4')$ are obvious. Thus the equivalence of all six conditions is established.
\end{proof} 

Convolutions of local smooth \CZ kernels adapted to different dilations are contained in the classes $\PP_{0}(\bfE)$ for an appropriate matrix $\mbE$. Thus for $1\leq i \leq n$, let $m_{i}\in \CC^{\infty}(\R^{d})$ be a local smooth {\it Mihlin-H\"ormander multiplier} for the family of dilations $\{\widehat\delta_{r}^{i}:r> 0\}$. This means that $\big|\de^\alpha m_i(\xib)\big|\leq C_\alpha\big(1+\widehat N_i(\xib)\big)^{-\bl\alphab\br_i}$ for all $\alpha\in \N^{d}$. Each $m_i$ is the Fourier transforms of a local smooth {\it Calder\'on-Zygmund kernel} $\cK_i$; \textit{i.e.} $\KK_{i}$ is a distribution which, away from the origin, is given by a smooth function $K_{i}$ satisfying
$\big|\de^\alphab K_{i}(\x)\big|\leq C_{\al,N}  \widehat N_i(\x)^{- Q_i-\bl\alpha\br_i}\big(1+\|\x\|\big)^{-N}$ 
together with cancellations in the full set of variables, i.e.,
$ \big|\lan\cK_i,\eta\circ\widehat\del^i_r\ran\big|\leq C$,
 for all $r>0$ and $\eta\in C^\infty_c(\bR^d)$ supported on the unit cube $B$ and with $\|\eta\|_{C^1}\le1$. Then the following follows from \cite[Theorem 10.2]{MR3862599}.

\begin{theorem}\label{Thm3.10}
Suppose that $\mbE$ is a standard matrix. For $1\leq i \leq n$, let $\cK_{i}$ be a local smooth \CZ kernel relative to the dilations $\widehat\delta^{i}_{r}$ given in \eqref{2.5a} and let $m_{i}=\widehat \cK_{i}$ be the corresponding local multiplier. Then $m=\prod_{i=1}^{n}m_{i}$ is in $\MM_\infty(\mbE)$ and the kernel $\cK_{1}*\cdots*\cK_{n}$ belongs to  $\PP_{0}(\EEE)$.\end{theorem}

\begin{remarks}\label{Rem3.11}\quad{\rm
\begin{enumerate}[1.]
\item
The strict inequality $1<e(j,k)e(k,j)$, which is the third requirement in \eqref{2.3} for a standard matrix, is not required either in \cite[Theorem 10.2]{MR3862599} or in Theorem \ref{Thm3.10}, but as explained in \cite[Section 14.4]{MR3862599}, if equality occurs for some $j\ne k$, then the two families of dilations $\{\delta_{r}^{j}\},\{\delta_{r}^{k}\}$ coincide, up to a re-parametrization, and the same holds for the dilations $\{\widehat\delta_{r}^{j}\}$ and $\{\widehat\delta_{r}^{k}\}$. Consequently the norms $N_j$ and $N_k$, as well as $\widehat N_j$ and $\widehat N_k$, are paired by the property of being one a power of the other. At the same time, the two quantities $|\bxi_j|$, $|\bxi_k|$ appear with proportional exponents in each norm $\widehat N_i$, so that the product $\bR^{d_j}\times\bR^{d_k}$ can be regarded as a single factor, reducing by one the dimension of the matrix $\mbE$. Hence the strict inequality is a non-reducibility condition and it will be assumed throughout this paper.

\item
If $\mbE$ is standard it can be proved that $\cM_\infty(\mbE)$ is minimal in the following sense: it is the smallest class of smooth multipliers that is closed under pointwise multiplication and contains the local smooth Mihlin-H\"ormander multipliers for all dilations $\{\widehat\del^i_r\}$, $i=1,\dots,n$.
Theorem 10.2 in \cite{MR3862599} also proves that, given any number of dilations $\widetilde\delta_{r}^\nu$ with exponents $\la(1,\nu),\dots,\la(n,\nu)$, and corresponding smooth local Mihlin-H\"ormander multipliers $m_\nu$, then the product $m=\prod_\nu m_\nu$ belongs to $\cM_\infty(\mbE)$ where $\mbE$ has entries $e(i,j)=\max_{\nu}\frac{\la(i,\nu)}{\la(j,\nu)}$.
This matrix $\mbE$ satisfies the first two conditions in \eqref{2.3} but is not necessarily standard. It may also happen that $\mbE$ is standard but $\cM_\infty(\mbE)$ does not satisfy the minimality condition relative to the dilations $\widetilde\delta_{r}^{\nu}$ defined above. See \cite[Section 10.3]{MR3862599}.
\end{enumerate}}
\end{remarks}

\subsection{The cases $n=2$  and $n=3$}\label{Sec3.7}\quad

\smallskip

The dyadic decomposition $\R_{+}^{n}=\bigcup_{L\in \LL}T_{L}$ and the corresponding Littlewood-Paley decomposition $\R^{d}= B(1)\cup\bigcup_{L\in \LL}\pi^{-1}(T_{L})$ are hard to visualize, and it may help to see what they mean when $n$ is small. 
\smallskip
\subsubsection{Case $n=2$}\quad

Consider $\R^{2}=\R\times\R$ and the standard matrix $\bfE={\left[\begin{matrix}1&a\\b&1\end{matrix}\right]}$ with $ab>1$.\footnote{With 2 factors any pair of distinct dilations can be reduced to the present form. This is no longer the case with $n\ge3$ factors. The assumption $d_1=d_2=1$ is just for notational simplicity and is not essential at all.} In this case $\Gamma(\bfE)=\big\{(t_{1},t_{2})\in \R^{2}_{+}:t_{1}/a\leq t_{2}\leq bt_{1}\big\}$ with dilations $\widehat\delta^{1}_{r}(\xi_{1},\xi_{2})=(r\xi_{1}, r^b\xi_{2})$ and  $\widehat\delta^{2}_{r}(\xi_{1},\xi_{2})=(r^a\xi_{1}, r\xi_{2})$. The relative homogeneous norms are $\widehat N_{1}(\xi_{1},\xi_{2})=\max\big\{|\xi_{1}|, |\xi_{2}|^{\frac{1}{b}}\big\}$ and $\widehat N_{2}(\xi_{1},\xi_{2})=\max\big\{|\xi_{1}|^{1/a}, |\xi_{2}|\big\}$. 

There are three marked partitions of $\{1, 2\}$: $S_p=\{\overset.1\}\{\overset.2\}$ (the principal partition), $S_{1}=\{\overset .1,2\}$, $S_{2}=\{1,\overset . 2\}$. The frequency space $\R^{2}$ is the union of four essentially disjoint regions, $\R^{2} = B(1)\cup \widehat E_{S_{1}} \cup \widehat E_{1}\cup \widehat E_{2}$ where $B(1)=\big\{(\xi_{1},\xi_{2}):|\xi_{1}|\leq 1,\,|\xi_{2}|\leq 1\big\}$ and the other sets are as follows.
\begin{enumerate}[1.]

\item
$\widehat E_{p}$ is the set where $1\leq |\xi_{1}|$ is dominant in $\widehat N_{1}(\xib)$ and $1\leq |\xi_{2}|$ is dominant in $\widehat N_{2}(\xib)$, \textit{i.e.} where $ |\xi_1|^{1/a}\le|\xi_2|\leq |\xi_1|^b$. Here $\widehat N_{1}(\xib)=|\xi_1|$ and $\widehat N_{2}(\xib)=|\xi_{2}|$. This is the ``product'' region.

\smallskip
\item 
$\widehat E_{1}$ is the set where $1\leq |\xi_{1}|$ is dominant in both $\widehat N_{1}(\xib)$ and $\widehat N_{2}(\xib)$; \textit{i.e.} where $|\xi_{1}|\geq |\xi_{2}|^a$  (notice that $|\xi_2|$ can be small). This is the region where $N_1(\xi_1,\xi_2)=\max\{|\xi_1|,|\xi_2|^a\}=|\xi_1|\ge1$.

\smallskip
 
\item 
$\widehat E_{2}$ is the set where $1\leq |\xi_{2}|$ is dominant in both $\widehat N_{1}(\xib)$ and $\widehat N_{2}(\xib)$; \textit{i.e.} where $ |\xib_2|\leq |\xib_1|^b$ and $|\xib_{2}|\geq 1$. This is the region where $N_2(\xi_1,\xi_2)=\max\{|\xi_1|^b,|\xi_2|\}=|\xi_2|\ge1$.
\end{enumerate}

The dyadic decomposition is done accordingly in each of the four regions as shown in  Figure~3.1 in the case $a=1$, $b=2$ .

\begin{center}
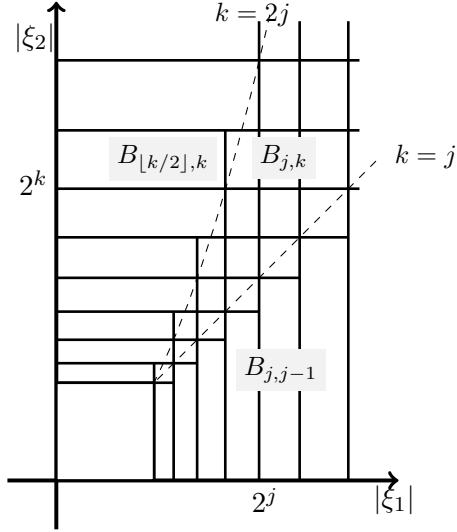
\begin{figure}

\begin{tikzpicture}[scale=1.3] [>=stealth]

\draw[line width=1] (0,1) -- (1.2,1);
\draw[line width=1] (0,1.2) -- (1.44,1.2);
\draw[line width=1] (0,1.44) -- (1.728,1.44);
\draw[line width=1] (0,1.728) -- (2.0736,1.728);
\draw[line width=1] (0,2.0736) -- (2.48832,2.0736);
\draw[line width=1] (0,2.48832) -- (2.985984,2.48832);
\draw[line width=1] (0,2.985984) -- (3.1,2.985984);
\draw[line width=1] (0,3.5831808) -- (3.1,3.5831808);
\draw[line width=1] (0,4.29981696) -- (3.1,4.29981696);

\draw[line width=1] (1,0) -- (1,1.2);
\draw[line width=1] (1.2,0) -- (1.2,1.728);
\draw[line width=1] (1.44,0) -- (1.44,2.48832);
\draw[line width=1] (1.728,0) -- (1.728,3.5831808);
\draw[line width=1] (2.0736,0) -- (2.0736,4.7);
\draw[line width=1] (2.48832,0) -- (2.48832,4.7);
\draw[line width=1] (2.985984,0) -- (2.985984,4.7);

\draw [->] [line width=1.5](-.5,0) -- (3.5,0);
\draw [->] [line width=1.5](0,-.5) -- (0,4.9);

\draw[dashed] (0,0) --(3.3,3.3) ;
\draw[dashed] (0,0) parabola (2.2,4.84);

\put(120,-10){$|\xi_1|$}
\put(-16,165){$|\xi_2|$}
\put(128,120){\small $k=j$}
\put(60,174){\small $k=2j$}
\put(74,-12){$2^j$}
\put(-14,108){$2^k$}
\small
\put(74,120){\colorbox{gray!10}{${B_{j,k}}$}}
\put(68,40){\colorbox{gray!10}{${B_{j,j-1}}$}}
\put(20,120){\colorbox{gray!10}{${B_{\lfloor k/2\rfloor,k}}$}}

\draw[fill=white , line width=1] (1,1) rectangle (0,0);

\end{tikzpicture}

\caption{Dyadic decomposition of the frequency space $\R_{\xib}^{2}$ (case $a=1$, $b=2$).}

\end{figure}

\end{center}

\subsubsection{Case $n=3$}\quad

For the case $n=3$, let $\bfE=\big\{e(j,k)\big\}$ be a standard $3\times 3$ matrix. For simplicity we assume $\bR^{d_i}=\bR$ for $i=1,2,3$. For $1\leq i \leq 3$ we define families of dilations $
\widehat\delta^{i}_{r}(\xi_{1},\xi_{2},\xi_{3})=(r^{e(1,i)}\xi_{1}, r^{e(2,i)}\xi_{2}, r^{e(3,i)}\xi_{3})$ for $1\leq i \leq 3$. The corresponding homogeneous norms are $\widehat N_{i}(\xi_{1},\xi_{2},\xi_{3})=\max_{1\leq j\leq 3}|\xi_{j}|^{\frac{1}{e(j,i)}}$ for $1 \leq i \leq 3$. The set $\SS_3$ of marked partitions with three terms consists of 10 elements. Given a $3\times3$ standard matrix $\bfE$, all marked partitions give a cone $\Gamma_S$ with nonempty interior if and only if the inequalities $e(i,k)\le e(i,j)e(j,k)$ are strict for all triples of distinct $i,j,k$.
Then the frequency space decomposes into at most $11$ non-trivial regions, one of them being $B(1)$ and the others associated to elements of $\SS_{\bfE}$. 

 Besides the principal marked partition $S_{p}=\{\overset.1\}\{\overset.2\}\{\overset.3\}$, we have two other types, those of the form $S=\{\overset.a,b\}\{\overset.c\}$ and $S=\{\overset.a,b,c\}$.
We describe them with $(a,b,c)=(1,2,3)$.
\begin{enumerate}[1.]

\item 
If $S_{p}=\{\overset.1\}\{\overset.2\}\{\overset.3\}$ is the principal partition and $\xib\in \widehat E_{S_p}$ then, for $1\leq i \leq 3$, $|\xi_{i}|$ is dominant in $\widehat N_{i}(\xib)$, hence $\ge1$. This is the ``product'' region and $\widehat E_{S_{p}}$ is defined by the six inequalities $|\xi_i|\le |\xi_j|^\frac1{e(j,i)}$ with $i\ne j$.
\smallskip

\item
If $S=\{\overset.1,2\}\{\overset.3\}$ and $\xib\in \widehat E_{S}$ then $|\xi_{1}|\ge1$ is dominant in $\widehat N_{1}(\xib)$ and $\widehat N_{2}(\xib)$, while $|\xi_{3}|\ge1$ is dominant in $\widehat N_{3}(\xib)$. The cone $\Gamma_{S}$ in $\R^3_\t$ is defined by the inequalities $\frac{e(1,2)}{e(3,2)}t_3\le t_1\le e(1,3)t_3$ and $t_2\le\frac{1}{e(1,2)}t_1$. The set
$\widehat E_S$ is defined by the conditions $|\xi_3|^{\frac{e(1,2)}{e(3,2)}}\le |\xi_1|\le |\xi_3|^{e(1,3)}$ and $N_1(\xi_1,\xi_2)=|\xi_1|$.

\smallskip

\item If $S=\{\overset.1,2,3\}$ and $\xib\in \widehat E_{S}$ then $|\xi_{1}|\ge1$ is dominant in $\widehat N_{i}(\xib)$ for $1 \leq i \leq 3$. The cone $\Gamma_S$ is defined by the inequalities $t_2\le\frac{t_1}{e(1,2)}$ and $t_3\le\frac{t_1}{e(1,3)}$
and $\widehat E_S$ by the single condition $N_1(\xib)=|\xi_1|\ge1$.
\end{enumerate}
\smallskip

 Figure 2 shows a transversal section of $\R^3_+$ with the $\t$-coordinates. The various polygons are sections of the cones $\Gamma_S$, assuming that all inequalities in \eqref{2.3} $e(j,\ell)<e(j,k)e(k,\ell)$ with $j,k,\ell$ all different are strict.

{\begin{center}
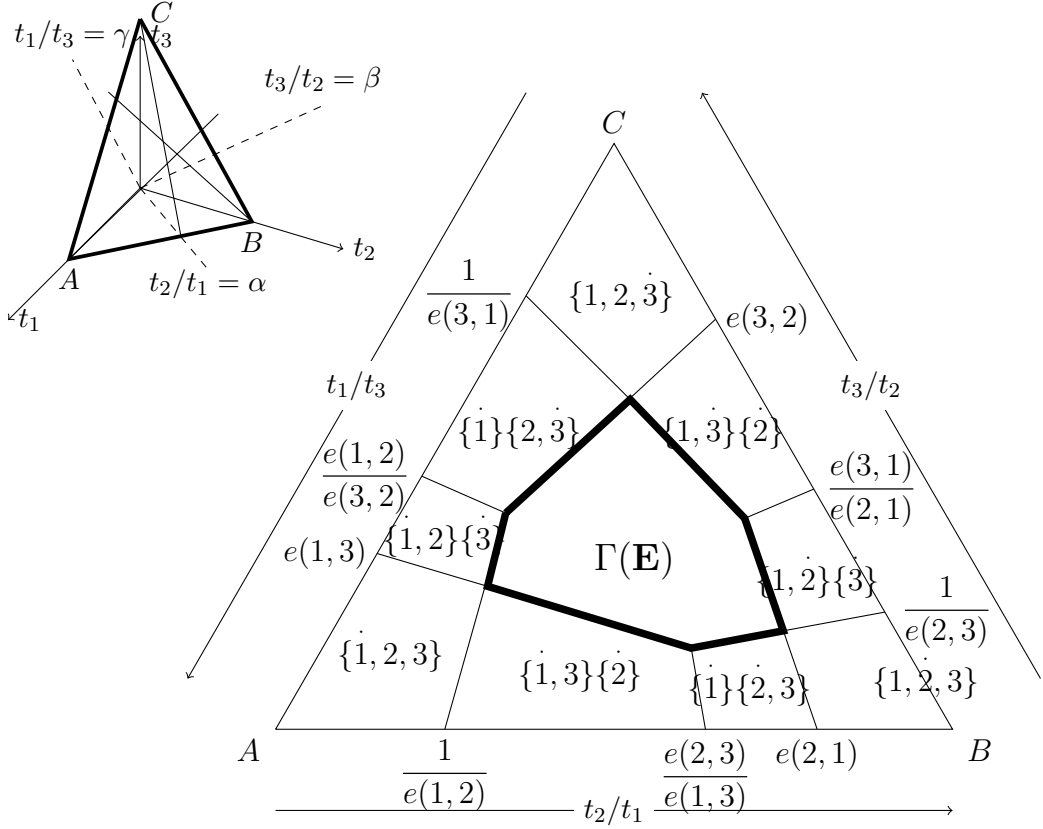
\begin{figure}

\begin{tikzpicture}[scale=.9] [>=stealth]

\draw (0,0)--(10,0)--({10*cos(60)},{10*sin(60)})--(0,0);

\draw (5.238,{5.63*sin(60)})--(3.7,{3.7*tan(60)});
\draw (3.1315,{2.44*sin(60)})--(1.5,{1.5*tan(60)});
\draw (3.4,{3.69*sin(60)})--({6.5},{3.5*tan(60)});
\draw (7.5,{1.68*sin(60)})--(9,{tan(60)});
\draw (3.4,{3.69*sin(60)})--(2.5,0);
\draw (6.93,{3.6*sin(60)})--(8,0);

\draw[line width=0.1cm] (3.4,{3.69*sin(60)})--(5.238,{5.63*sin(60)})--(6.93,{3.6*sin(60)})--(7.5,{1.68*sin(60)})--(6.14,{1.385*sin(60)}) --(3.1315,{2.44*sin(60)}) -- (3.4,{3.69*sin(60)});

\draw (6.14,{1.385*sin(60)})--(6.35,0);
\draw (6.93,{3.6*sin(60)})-- (10-2.05,{2.05*tan(60)});
\draw (3.4,{3.69*sin(60)})--(2.16,{2.16*tan(60)});

\node at (5.3,2.5) {\Large $\Gamma(\mathbf E )$};
\node at (1.7,1.2) {\large $\{\overset\cdot1,2,3\}$};
\node at (9.6,.8) {\large $\{1,\overset\cdot2,3\}$};
\node at (5.1,6.5) {\large $\{1,2,\overset\cdot3\}$};
\node at (2.5,2.9) {\large $\{\overset\cdot1,2\}\{\overset\cdot3\}$};
\node at (3.6,4.5) {\large $\{\overset\cdot1\}\{2,\overset\cdot3\}$};
\node at (6.6,4.5) {\large $\{1,\overset\cdot3\}\{\overset\cdot2\}$};
\node at (8,2.3) {\large $\{1,\overset\cdot2\}\{\overset\cdot3\}$};
\node at (7,.7) {\large $\{\overset\cdot1\}\{\overset\cdot2,3\}$};
\node at (4.5,.9) {\large $\{\overset\cdot1,3\}\{\overset\cdot2\}$};


\node[below] at (-.4,0) {\large$A$};
\node[below] at (10.4,0) {\large$B$};
\node[above] at ({10*cos(60)},{10*sin(60)}) {\large$C$};


\node[below] at (2.5,0) {\large $\displaystyle\frac1{e(1,2)}$} ;
\node[below] at (6.35,0) {\large $\displaystyle\frac{e(2,3)}{e(1,3)}$} ;
\node[below] at (8,0) {\large $e(2,1)$} ;


\node[left] at (3.67,{3.7*tan(60)}) {\large $\displaystyle\frac1{e(3,1)}$} ;
\node[left] at (2.13,{2.16*tan(60)}) {\large $\displaystyle\frac{e(1,2)}{e(3,2)}$} ;
\node[left] at (1.5,{1.5*tan(60)}) {\large $e(1,3)$} ;

 
\node[right] at ({10-.9},{tan(60)}) {\large $\displaystyle\frac1{e(2,3)}$} ;
\node[right] at (10-2.01,{2.05*tan(60)}) {\large $\displaystyle\frac{e(3,1)}{e(2,1)}$} ;
\node[right] at ({10-3.5},{3.5*tan(60)}) {\large $e(3,2)$} ;

\draw [->] (0,-1.2)--(10,-1.2) node [midway,fill=white] {$t_2/t_1$};
\draw [<-] ({-1.5*cos(30)},{1.5*sin(30)})--({10*cos(60)-1.5*cos(30)},{10*sin(60)+1.5*sin(30)}) node [midway,fill=white] {$t_1/t_3$};
\draw [->] ({10+1.5*cos(30)},{1.5*sin(30)})--({10*cos(60)+1.5*cos(30)},{10*sin(60)+1.5*sin(30)}) node [midway,fill=white] {$t_3/t_2$};

\draw [->] (-2,8)--(-3.95,6.05) ;
\draw [->] (-2,8)--(1,7.1) ;
\draw [->] (-2,8)--(-2,10.25) ;
\draw [line width=0.5mm] (-2,10.5)--(-3.05,6.95)--(-.35,7.505)--(-2,10.5);
\draw [dashed] (-2,8)-- (-1,6.8);
\draw (-2,10.5)-- (-1.4,7.3);

\draw [dashed] (-2,8)-- (.7,9.2205);
\draw (-3.05,6.95)-- (-.845,9.0975);

\draw [dashed] (-2,8)-- (-3,9.9);
\draw (-.35,7.505)-- (-2.47,9.42);

\node[right] at (-3.95,6.05) {$t_1$};
\node[right] at (1,7.1) {$t_2$};
\node[right] at (-2,10.25) {$t_3$};
\node[below] at (-1,6.95){$t_2/t_1=\al $};
\node[above] at (.7,9.24){$t_3/t_2=\beta$};
\node[above] at (-3,9.9){$t_1/t_3=\gamma$};
\node[below] at (-3.05,6.95) {$A$};
\node[below] at (-.35,7.505) {$B$};
\node[right] at (-2,10.6) {$C$};
\end{tikzpicture}
\caption{The case $n=3$ (in the $\t$ coordinates). The big triangle on the right gives a projective view of the positive orthant (as its section on the plane $t_1+t_2+t_3=1$) and its subdivision into the cones $\Gamma_S$. All edges belong to lines issued from one of the three vertices $A,B,C$.}
\end{figure}
\end{center}

\section{Multi-norm Calder\'on reproducing formulas}\label{Sec4}

In this section we develop a multi-norm version of the Calder\'on reproducing formula. Working in a vector space $\R^{\nu}$ we use the following notation and terminology.  

\begin {enumerate}[\rm(i)]
\item Let $r\x= (r^{\lambda_{1}}x_1,\dots,r^{\lambda_{\nu}}x_{\nu})$ be a one-parameter family of dilations so that $\R^{\nu}$ has homogeneous dimension $Q=\sum_{i=1}^{\nu}\lambda_{i}$. If $\psi:\R^{\nu}\to\C$ and $p\in \R$ then $\psi^{(p)}(\x)=2^{-pQ}\psi\big(2^{-p}\x\big)$.

\smallskip

\item   The Schwartz norms $\|f\|_{(M)}$ were defined in \eqref{3.18yy}. Let $\FF=\{f_{\sigma}:\sigma\in\Sigma\big\}\subseteq\SS(\R^{\nu})$ be a family of Schwartz functions such that $\|f_{\sigma}\|_{(M)}\leq B_{M}$ for all $\sigma\in \Sigma$.  Then a second family $\GG=\{g_{\sigma'}:\sigma\in\Sigma'\big\}\subseteq\SS(\R^{d})$  is  \textit{normalized relative to $f$} if for every $M$ there exists $N$ so that $\|g_{\sigma'}\|_{(M)}\leq C_MB_{N}$ for all $\sigma'\in \Sigma'$.  ({\it Cf}. Definition 4.1 in \cite{MR2949616}.)

\smallskip

\item  A function $\psi:\R^{\nu}\to\C$ has {\it cancellation of order $m$}
if it has vanishing moments up to order $m-1$; \textit{i.e.} $\int_{\R^{\nu}}\x^{\alphab}\psi(\x)\,d\x=0$ for all $|\alphab|=\sum_{j}\alpha_{j}\leq m-1$. For $\psi\in\SS(\R^{\nu})$ this is the same as saying that the Fourier transform $\widehat\psi(\xib)$ has a zero of order $m$ at $\xib=\0$. Thus cancellation of order 1 means $\int_{\R^{\nu}}\psi(\x)\,d\x=0$ and cancellation of order 0 is a vacuous condition. 
\end{enumerate}

\subsection{The one-parameter case}\label{Sec4.1}\quad

\smallskip

Let $\big\{\ph_\ell:\ell\in\N\big\}\subseteq\SS(\R^{\nu})$ be a sequence that is uniformly bounded in every Schwartz norm  and satisfies $\int_{\R^\nu}\ph_\ell(\x)\,d\x=1$ for $\ell\geq 0$. Put $\varphi_{-1}(\x)\equiv 0$.} Then define
\bea\label{4.1}
\psi_{\ell}(\x)=
\varphi_{\ell}(\x)-\varphi_{\ell-1}^{(1)}(\x)\qquad\text{so that}\qquad
\int_{\R^{\nu}}\psi_{\ell}(\x)\,d\x =
\begin{cases}
1&\text{if $\ell=0$}\\
0&\text{if $\ell\neq 0$}
\end{cases}.
\eea
Then\beas
\psi_{\ell}^{(-\ell)}(\x)&=2^{\ell Q}\psi_{\ell}(2^{\ell}\x)- 2^{(\ell-1)Q}\psi_{\ell-1}(2^{\ell-1}\x)&&\text{so that }&  \varphi_{N}^{(-N)}(\x)&=\sum_{\ell=0}^{N}\psi_{\ell}^{(-\ell)}(\x) \quad\forall N\geq 0.
\eeas 
Then in the sense of distributions,
\beas
\delta_{\0}&=\sum\nolimits_{\ell\geq 0}\psi_{\ell}^{(-\ell)}&&\text{ and }& \delta_{\0}&=\varphi_{M}^{(-M)}+\sum\nolimits_{\ell>M}\psi_{\ell}^{(-\ell)}&&\text{ for any $M\geq 0$}.
\eeas
As a consequence of Hadamard's lemma we have

\begin{lemma}\label{Lem4.1}
A function $\psi\in \SS(\R^{\nu})$ has cancellation of order $m$ if and only if there exist functions $\big\{\psi_{\betab}\in \SS(\R^{\nu}):|\betab|=m\big\}$ such that $\psi=\sum_{|\betab|=m}\partial^{\betab}\psi_{\betab}$ with $\big\{\psi_{\betab}\big\}$ normalized relative to $\psi$. 
\end{lemma}

The following version of the Calder\'on reproducing formula is taken from \cite{MR1821243}. The proof given below is simplified by the use of Fourier transform and requires weaker assumptions. Let $B(\0,r)$ be the open ball centered at the origin with radius $r$. Denote by $\widehat g$ or $\FF[g]$ the Fourier transform of a function $g$.

\begin{lemma}\label{Lem4.2}
Let $\{\varphi_{\ell}\}\subseteq \SS(\R^{\nu})$ be uniformly bounded with $\int_{\R^{\nu}}\varphi_{\ell}(\x)\,d\x=1$, and let $\psi_{\ell}$ be defined in \eqref{4.1}. Given $m>0$ there are two families of Schwartz functions, $\{\widetilde\ph_\ell\}_{\ell\in\bN}$, $\{\widetilde\psi_\ell\}_{\ell\in\bN}$,
such that
\begin{enumerate}[\rm(a)]
\item the families $\{\widetilde\psi_\ell\}$ and $\{\widetilde\ph_\ell\}$ are normalized relative to $\{\ph_\ell\}$;
\smallskip

\item $\widetilde\psi_0=\widetilde\ph_0$;
\smallskip

\item \label{Lem4.2b}
for $\ell>0$, $\widetilde\psi_\ell$ has cancellation of order $m$;

\smallskip

\item \label{Lem4.2c}
if $\ell\geq 0$ then $\displaystyle \del_\0=\ph_\ell^{(-\ell)}*\widetilde\ph_\ell^{(-\ell)}+\sum\nolimits_{\ell'>\ell}\psi_{\ell'}^{(-\ell')}*\widetilde\psi_{\ell'}^{(-\ell')}$;

\smallskip

\item \label{Lem4.2d}
if $\ell\geq 0$ then $\displaystyle \ph_\ell^{(-\ell)}*\widetilde\ph_\ell^{(-\ell)}=\sum\nolimits_{0\leq \ell'\leq \ell}\psi_{\ell'}^{(-\ell')}*\widetilde\psi_{\ell'}^{(-\ell')}$;
\smallskip

\item \label{Lem4.2e} if the Fourier transforms $\FF[\ph_\ell]$ are supported on $B(\0,r)$ 
and are identically equal to $1$ on $B\left(\0,\frac{r}{2}\right)$ so that $\{\psi_{\ell}\}$ are supported on $ B(\0,2r)\setminus B\left(\0,\frac{r}{2}\right)$, then the Fourier transforms $\FF[\widetilde\ph_\ell]$ are supported on $B(\0,r)$ and $\FF[\widetilde\psi_\ell]$ are  supported on $B(\0,2r)\setminus B\left(\0,\frac{r}{2}\right)$;
\smallskip

\item \label{Lem4.2f} if the functions $\{\ph_\ell\}$ are supported on a compact neighborhood $K$ of $\0$, then the functions $\widetilde\ph_\ell,\widetilde\psi_\ell$ are supported on a dilate $c_mK$ where $c_{m}$ depends only on $m$.

\end{enumerate}
\end{lemma}

\begin{proof}

The hypothesis on the $\ph_\ell$ imply that $\lim_{\ell\to\infty}\FF[\ph_\ell^{(-\ell)}]=1$ point-wise and boundedly, so that $\lim_{\ell\to\infty}\ph_\ell^{(-\ell)}=\del_\0$. The same is true with $\ph_\ell$ replaced by $g_\ell=\ph_\ell*\ph_\ell$. Let $\nu_{\ell}=\FF[\varphi_{\ell}]$ and $\mu_{\ell}=\FF[\psi_{\ell}]$, so that $\mu_0=\nu_0$ and $\mu_\ell(\xi)=\FF[\psi_\ell](\xi)=\FF[\ph_\ell](\xi)-\FF[\ph_{\ell-1}](2\xi)$ for $\ell>0$. By dominated convergence, we have the following identities in the sense of distributions for all $M\geq 1$:

\beas
1&=\lim_{\ell\to\infty} \FF[g_\ell^{(-\ell)}](\xib)=\lim_{\ell\to\infty} \nu_\ell^2(2^{-\ell}\xib)=\nu_0^2(\xib)+\sum\nolimits_{\ell=1}^\infty \big(\nu_\ell^2(2^{-\ell}\xib)-\nu_{\ell-1}^2(2^{-\ell+1}\xib)\big)\\
&=\nu_0^2(\xib)+\sum\nolimits_{\ell=1}^\infty \mu_\ell(2^{-\ell}\xib)\big(\nu_\ell(2^{-\ell}\xib)+\nu_{\ell-1}(2^{-\ell+1}\xib)\big)\\
&=\Big(\nu_0^2(\xib)+\sum\nolimits_{\ell=1}^\infty \mu_\ell(2^{-\ell}\xib)\big(\nu_\ell(2^{-\ell}\xib)+\nu_{\ell-1}(2^{-\ell+1}\xib)\big)\Big)^M=\big(\nu_0^2(\xib)+s_0(\xib)\big)^M\ ,
\eeas
We expand the $M$-th power into terms depending on $(\ell_1,\dots,\ell_M)\in\bN^M$ and group together all terms with the same $\bar\ell=\min\{\ell_1,\dots,\ell_M\}$.
For $\bar\ell=0$ we must take all terms containing $\nu_0$, which add up to $1-s_0(\xib)^M=1-\big(1-\nu_0^2(\xib)\big)^M=\nu_0^2(\xib)\sum_{k=0}^{M-1}\big(1-\nu_0^2(\xib)\big)^k$, since $s_0(\xib)=1-\nu_0^2(\xib)$.

We then define $\widetilde\ph_0=\widetilde\psi_0=\cF\inv\big[\nu_0(\xib)\sum_{k=0}^{M-1}\big(1-\nu_0^2(\xib)\big)^k\big]$, and for any $\bar\ell>0$, we denote by $s_{\bar\ell}$ the remainder $s_{\bar\ell}(\xib)=\sum_{\ell=\bar\ell}^\infty \big(\nu_\ell^2(2^{-\ell}\xib)-\nu_{\ell-1}^2(2^{-\ell+1}\xib)\big)=1-\nu_{\bar\ell-1}^2(2^{-\bar\ell+1}\xib)$. The terms to be grouped together are 
those appearing in the expansion of $s_{\bar\ell}^M$ but not in the expansion of $s_{\bar\ell+1}^M$. These are
\beas
s_{\bar\ell}(\xib)^M-s_{\bar\ell+1}(\xib)^M&=\big(1-\nu_{\bar\ell-1}^2(2^{-\bar\ell+1}\xib)\big)^M-\big(1-\nu_{\bar\ell}^2(2^{-\bar\ell}\xib)\big)^M\\
&=\mu_{\bar\ell}(\xib)\big(\nu_{\bar\ell}(\xib)+\nu_{{\bar\ell}-1}(2\xib)\big)\sum\nolimits_{k=0}^{M-1}\big(1-\nu_{{\bar\ell}-1}^2(2\xib)\big)^{M-1-k}\big(1-\nu_{\bar\ell}^2(\xib)\big)^k\ .
\eeas
For $\ell\ge1$ define 
\bea\label{4.2}
\widetilde\ph_\ell&=\cF\inv\Big(\nu_\ell\sum\nolimits_{k=0}^{M-1}\big(1-\nu_\ell^2\big)^k\Big)\\
\widetilde\psi_\ell&= \cF\inv\Big(\big(\nu_\ell+\nu_{\ell-1}(2\cdot)\big)\sum\nolimits_{k=0}^{M-1}\big(1-\nu_{\ell-1}^2(2\cdot)\big)^{M-1-k}\big(1-\nu_\ell^2\big)^k\Big)\ ,
\eea
so that
\bea\label{4.3}
\qquad\ph_\ell*\widetilde\ph_\ell&=\del_\0-(\del_\0-g_\ell\big)^{*M}\\
\psi_\ell*\widetilde\psi_\ell&=\big(\del_\0-g_{\ell-1}^{(1)}\big)^{*M}-\big(\del_\0-g_\ell\big)^{*M}\ ,
\eea
where the exponent $*M$ denotes $M$-th convolution power. It is quite clear that $\widetilde\ph_\ell$ and $\widetilde\psi_\ell$ are Schwartz functions.
Then \eqref{4.2} shows uniform boundedness of their Schwartz norms, as well as cancellation of order $M-1$ of $\widetilde\psi_\ell$ at $\xib=\0$, while \eqref{4.3} provides the identities $\sum_{\ell\in\bN} \mu_\ell(2^{-\ell}\xi)\widetilde\mu_\ell(2^{-\ell}\xi)=1$ and $\sum_{0\leq \ell'\leq \ell}\mu_{\ell'}(2^{-\ell'}\xi)\widetilde\mu_{\ell'}(2^{-\ell'}\xi)= \nu_\ell(2^{-\ell}\xi)\widetilde\nu_\ell(2^{-\ell}\xi)$. This establishes \eqref{Lem4.2c} and \eqref{Lem4.2d}. Then \eqref{Lem4.2e} follows from \eqref{4.2}. Also \eqref{Lem4.2f} follows from \eqref{4.2}, observing that $\widetilde\ph_\ell$ is a linear combination of positive convolution powers of $\ph_\ell$, up to degree $2M+1$. The same applies to $\widetilde \psi_\ell$.
\end{proof}

\subsection{Multi-norm Calder\'on reproducing formulas by tensor products}\label{Sec4.2}
\quad
\smallskip

Based on the formulas in Section \ref{Sec4.1}, we now construct a multi-norm Calder\'on reproducing formula adapted to the dyadic decomposition from Section \ref{Sec3.4}. Let $\R^d=\R^{d_1}\times\cdots\times\R^{d_n}$, and for $1\leq i \leq n$ let $\big\{\varphi_{i,\ell}:\ell\geq 0\big\}\subset \SS(\R^{d_{i}})$ be a family of functions uniformly bounded in every Schwartz norm satisfying $\int_{\R^{d_{i}}}\varphi_{i,\ell}(\x_{i})d\x_{i}=1$. Let $\displaystyle \psi_{i,\ell}(\x)=\varphi_{i,\ell}(\x)-\varphi_{i,\ell-1}^{(1)}(\x)$, and let $\widetilde\ph_{i,\ell},\widetilde\psi_{i,\ell}$ be the functions constructed from $\varphi_{i,\ell}$ and $\psi_{i,\ell}$ as in Lemma \ref{Lem4.2}. Let $\LL$ and $\LL(D)\subset\N^{n}$ be the index sets defined in equation \eqref{3.16rr}.
 If $L\in \LL$ recall that $D_{L}$ is the unique subset of $\{1, \ldots, n\}$ such that $L\in D_{L}$. For $L=(\ell_{1}, \ldots, \ell_{n})\in\LL$ let 
\bea\label{3.4ee}
\Phi_{L}(\x)&=\prod\nolimits_{i=1}^{n}\varphi_{i,\ell_{i}}(\x_{i}), &
\Psi_{L}(\x)&=
\prod\nolimits_{i\in D_{L}}\psi_{i,\ell_i}(\x_{i})\prod\nolimits_{i\notin D_{L}}\ph_{i,\ell_i}(\x_{i}),\\
\widetilde\Phi_{L}(\x)&=\prod\nolimits_{i=1}^{n}\widetilde\varphi_{i,\ell_{i}}(\x_{i}),
&
\widetilde\Psi_{L}(\x)&=
\prod_{i\in D_{L}}\widetilde\psi_{i,\ell_i}(\x_{i})\prod_{i\notin D_{L}}\widetilde\ph_{i,\ell_i}(\x_{i}),\\
\Psi_{i,L}(\x_{i})&=\begin{cases}\psi_{i,\ell_{i}}(\x_{i})&\text{if $i\in D_{L}$}\\
\varphi_{i,\ell_{i}}(\x_{i})&\text{if $i\notin D_{L}$}
\end{cases},
&
\widetilde \Psi_{i,L}(\x_{i})&=\begin{cases}\widetilde \psi_{i,\ell_{i}}(\x_{i})&\text{if $i\in D_{L}$}\\
\widetilde\varphi_{i,\ell_{i}}(\x_{i})&\text{if $i\notin D_{L}$}
\end{cases}.
\eea
We say  that the family $\Psi=\big\{\Phi_{L}, \Psi_{L}, :L\in \LL\big\}$ is \textit{constructed via tensor products} from $\big\{\varphi_{i,\ell}\big\}$. If $L\in\LL_D$ and $i\in D$ then $\ell_i>0$ so $\Psi_{L}(\x)$ has integral zero in $\x_{i}$, and $\widetilde\Psi_{L}(\x)$ has cancellation of order~$m$ in $\x_{i}$. On the Fourier transform side, $\FF[\Psi_{L}^{(-L)}]$ and $\FF[\widetilde\Psi_{L}^{(-L)}]$ give a Schwartz partition of unity subordinated to the sets $B_L=\tau\inv(T_L)$ where $T_L$ is the rectangle in \eqref{3.8}.  We have the following multi-norm reproducing formulas.

\goodbreak

\begin{proposition}\label{Prop4.3}\quad
\begin{enumerate}[(a)]
\item
The identities $\del_\0=\sum_{L\in \LL}\Psi_{L}^{(-L)}$ and $\del_\0=\sum_{L\in \LL}\Psi_{L}^{(-L)}*\widetilde\Psi_{L}^{(-L)}$ hold in the sense of distributions. 

\smallskip

\item If $L,L'\in\N^{n}$ write $L'\leq L$ if and only if $\ell'_i\leq \ell_i$ for all $i$ and $L'<L$ if and only if $L'\leq L$ and $L'\ne L$. Then for every $L\in \LL$, $\Phi_{L}^{(-L)}*\widetilde\Phi_{L}^{(-L)}=\sum_{L'\in \LL\,,\,L'\leq L}\Psi_{L'}^{(-L')}*\widetilde\Psi_{L'}^{(-L')}$.
\end{enumerate}
\end{proposition} 

\begin{proof}
We only prove the last identity since the first two are analogous and simpler.
According to Lemma \ref{Lem4.2}\eqref{Lem4.2d}, $\ph_{i,\ell_{i}}^{(-\ell)}=\sum\nolimits_{0\leq \ell_{i}'\leq \ell_{i}}\psi_{{i,\ell_{i}}'}^{(-\ell')}*\widetilde\psi_{{i,\ell_{i}}'}^{(-\ell')}$. Thus
\beas
\Psi_{L}^{(-L)}&*\widetilde \Psi_{L}^{(-L)}(\x)\\
&=
\Big(\prod\nolimits_{i\in D_{L}}\psi_{i,\ell_i}^{(-\ell_{i})}(\x_{i})\Big)\Big(\prod\nolimits_{i\notin D_{L}}\ph_{i,\ell_i}^{(-\ell_{i})}(\x_{i})\Big)*
\Big(\prod\nolimits_{i\in D_{L}}\widetilde\psi_{i,\ell_i}^{(-\ell_{i})}(\x_{i})\Big)\Big(\prod\nolimits_{i\notin D_{L}}\widetilde\ph_{i,\ell_i}^{(-\ell_{i})}(\x_{i})\Big)\\
&=
\Big(\prod\nolimits_{i\in D_{L}}\psi_{i,\ell_{i}}^{(-\ell_{i})}*\widetilde\psi_{i,\ell_{i}}^{(-\ell_{i})}(\x_{i})\Big)\Big(\prod\nolimits_{i\notin D_{L}}\varphi_{i,\ell_{i}}^{(-\ell_{i})}*\widetilde\varphi_{i,\ell_{i}}^{(-\ell_{i})}(\x_{i})\Big)\\
&=
\Big(\prod\nolimits_{i\in D_{L}}\psi_{i,\ell_{i}}^{(-\ell_{i})}*\widetilde\psi_{i,\ell_{i}}^{(-\ell_{i})}(\x_{i})\Big)\Big(\prod\nolimits_{i\notin D_{L}}\sum\nolimits_{0\leq \ell_{i}'\leq \ell_{i}}\psi_{{i,\ell_{i}}'}^{(-\ell')}*\widetilde\psi_{{i,\ell_{i}}'}^{(-\ell')}(\x_{i})\Big)\\
&=
\sum\nolimits_{L'\in T_L}\mathop{\otimes}_{i=1}^{n}(\psi_{i,\ell_{i}}^{(-\ell_{i})}*\widetilde\psi_{i,\ell_{i}}^{(-\ell_{i})})(\x).
\eeas
Then by Lemma \ref{Lem4.2}
\beas
\sum\nolimits_{L\in \LL}\Psi_{L}^{(-L)}*\widetilde\Psi_{L}^{(-L)}
&=
\sum\nolimits_{L'\in\N^{n}}\mathop{\otimes}_{i=1}^n\big(\psi_{i,\ell'_i}^{(-\ell'_i)}*\widetilde\psi_{i,\ell'_i}^{(-\ell'_i)}\big)
=
\mathop{\otimes}_{i=1}^n\sum\nolimits_{\ell'_i\in\bN}\big(\psi_{i,\ell'_i}^{(-\ell'_i)}*\widetilde\psi_{i,\ell'_i}^{(-\ell'_i)}\big)
=\del_\0,\\
\sum_{\substack{L'\in \LL\\L'\leq L}}\Psi_{L'}^{(-L')}*\widetilde\Psi_{L'}^{(-L')}
&=
\sum_{\substack{L'\in \N^{n}\\ L'\leq L}}\mathop{\otimes}_{i=1}^n\big(\psi_{i,\ell'_i}^{(-\ell'_i)}*\widetilde\psi_{i,\ell'_i}^{(-\ell'_i)}\big)
=
\mathop{\otimes}_{i=1}^{n}\sum_{\ell_{i}'\leq \ell_{i}}\big(\psi_{i,\ell'_i}^{(-\ell'_i)}*\widetilde\psi_{i,\ell'_i}^{(-\ell'_i)}\big)
=\Phi_{L}^{(-L)}*\widetilde\Phi_{L}^{(-L)}. \qedhere
\eeas
\end{proof}

\subsection{Multi-norm Calder\'on reproducing formulas by convolution}\label{Sec4.3}
\quad
\smallskip

In this section we construct reproducing formulas $\delta_{\0}=\sum_{L\in \LL}\Sigma_{L}^{(-L)}$ and $\delta_{\0}=\sum_{L\in \LL}\Sigma_{L}^{(-L)}*\widetilde\Sigma_{L}^{(-L)}$ where the functions $\big\{\Sigma_{L}\big\}$ are convolutions rather than tensor products. Recall from \eqref{2.5a} that $\widehat N_{i}(\xib)$ is a homogeneous norm for the dilations $\widehat\delta^{i}_{r}(\xib)=\left(r^{e(1,i)}\xib_{1}, \ldots, r^{e(n,i)}\xib_{n}\right)$. If $f:\R^{d}\to\C$ let $f^{(-t)_{i}}$ denote $f$ rescaled by $2^{-t}$ relative to the dilations $\widehat\delta^{i}_{r}$. Thus $f^{(-t)_i}(\x)=2^{t\widehat Q_i}f\big(\widehat\delta^{i}_{2^{t}}(\x)\big)$, where $\widehat Q_i$, defined in \eqref{3.16ww}, denotes the homogeneous dimension of $\R^d$ relative to these dilations.

\begin{lemma}\label{Lem4.4w} 
Let $\t=(t_{1}, \ldots, t_{n})\in \R_{+}^{n}$ and let $\widetilde \t=(\widetilde t_{1}, \ldots, \widetilde t_{n})$ where $\widetilde t_{j}=\min_{1\leq i \leq n}e(j,i)t_{i}$. Then $\widetilde t_{j}\leq t_{j}$ and $\widetilde \t\in \Gamma(\bfE)$. Moreover $\t\in \Gamma(\bfE)$ if and only if $\widetilde \t=\t$.\end{lemma}
\begin{proof}
Since $e(j,j)=1$ it follows that $\widetilde t_{j}\leq t_{j}$. For any $i,j,k$ we have $\widetilde t_{j}\leq e(j,i)t_{i}\leq e(j,k)e(k,i)t_{i}$. Taking the minimum over $i$ it follows that $\widetilde t_{j}\leq e(j,k)\widetilde t_{k}$ and so $\widetilde t\in \Gamma(\bfE)$. Finally, $\t\in \Gamma(\bfE)$ if and only if $t_{j}= \min_{1\leq i \leq n}e(j,i)t_{i}=\widetilde t_{j}$.
\end{proof}

\begin{lemma}\label{Lem4.5}
Let $f_{1}, \ldots, f_{n}\in\SS(\R^{d})$, $\t=(t_{1}, \ldots, t_{n})\in \R_{+}^{n}$ and let $h\in \SS(\R^{d})$ be such that $f_{1}^{(-t_{1})_{1}}*\cdots * f_{n}^{(-t_{n})_{n}}=h^{(-\widetilde \t)}$ where $\widetilde \t$ is given in Lemma \ref{Lem4.4w}. Then $\|h\|_{(N)}\leq C_N\prod_{i=1}^n\|f_i\|_{(N')}$ for every $N\geq 0$ where $C_N$ and $N'$ depend only on $N$.
\end{lemma}

\begin{proof}
If $h=\big(f_{1}^{(-t_{1})_{1}}*\cdots * f_{n}^{(-t_{n})_{n}}\big)^{(\widetilde\t)}$ then $\FF[h](\x)=\FF\big[f_{1}^{(-t_{1})_{1}}*\cdots * f_{n}^{(-t_{n})_{n}}\big](2^{\widetilde t_{1}}\xib_{1}, \ldots, 2^{\widetilde t_{n}}\xib_{n})=
\prod\nolimits_{i=1}^n\widehat{f_i}\big(2^{\tilde t_1-e(1,i)t_i}\xib_1,\dots,2^{\tilde t_n-e(n,i)t_i}\xib_n\big)$. We show that all derivatives of $\FF[h]$ are rapidly decaying at infinity. Since $\tilde t_j-e(j,i)t_i\le0$, using Leibniz formula it suffices to prove the inequality
\bea\label{4.7}
\prod\nolimits_{i=1}^n\big(1+2^{\tilde t_1-e(1,i)t_i}|\xib_1|+\cdots+2^{\tilde t_n-e(n,i)t_i}|\xib_n|\big)\ge 1+|\xib_1|+\cdots+|\xib_n|.
\eea
Let $E_{i}$ the set of $j$ such that $\tilde t_j=e(j,i)t_i$. Then, since every $j$ is in some $E_{i}$,
\beas
\prod_{i=1}^n\big(1+2^{\tilde t_1-e(1,i)t_i}|\xib_1|+\cdots+2^{\tilde t_n-e(n,i)t_i}|\xib_n|\big)&\ge \prod_{i=1}^n\Big(1+\sum_{j\in E_{i}}|\xib_j|\Big)
\ge1+\sum_{j=1}^n|\xib_j|.\qedhere
\eeas

\end{proof}

We give a discrete version of Lemma \ref{Lem4.5} involving only dyadic scales $2^{-L}$ with $L\in\LL$. 
If $|\t-\t'|<\delta$, if $f\in \SS(\R^{d})$, and if $f^{(-\t)}=g^{(-\t')}$ then $g\in \SS(\R^{d})$ and $\|g\|_{(N)}\leq C_{N,\del}\|f\|_{(N)}$ for $\forall N$.

%
%

\begin{corollary}\label{Cor4.6}
Let $L=(\ell_1,\dots,\ell_n)\in\LL$ and $f_1,\dots,f_n\in\cS(\R^d)$. Then $f_{1}^{(-\ell_{1})_{1}}*\cdots * f_{n}^{(-\ell_{n})_{n}}=k^{(-L)}$ where $k\in\SS(\R^d)$ and $\|k\|_{(N)}\leq C_N\prod_{i=1}^n\|f_i\|_{(N')}$ for every $N\geq 0$, with $C_N$ and $N'$ depend only on~$N$.
\end{corollary}

\begin{proof}
By Lemma \ref{Lem4.5}, $\mathop*_{i=1}^nf_i^{(-\ell_i)_i}=h^{(-\t_L)}$ where $\t_L$ has entries $t_j=\min_{i=1,\dots,n}e(j,i)\ell_i$ and $\|h\|_{(N)}\leq C_N\prod_{i=1}^n\|f_i\|_{(N')}$. By Lemma \ref{Lem3.8} \eqref{Lem3.8a} there is $\t'\in\Gamma(\mbE)$ such that $|\t'-L|<\kappa$. Given $j$, let $i$ be such that $t_j=e(j,i)\ell_i$. Then $t_j\leq \ell_j$ and
$
|\ell_j-t_j|= \ell_j-e(j,i) \ell_i\leq t'_j+\kappa-e(j,i)(t'_i-\kappa)\leq \big(1+e(j,i)\big)\kappa\ .
$
Hence the distances $|L-\t_L|$ are uniformly bounded and we can define $k=\big[h^{(-\t_L)}\big]^{(L)}$, keeping the uniform control of Schwartz norms.
\end{proof}

Now for $1\leq i \leq n$ fix a family $\{\rho_{i,\ell}\}_{\ell\in\bN}\subset\cS(\bR^d)$ uniformly bounded in every Schwartz norm and with $\int_{\R^d}\rho_{i,\ell}=1$. Following the procedure in Subsection \ref{Sec4.1}, relative to the $\widehat\delta^{i}_{r}$ dilations in \eqref{2.5a}, construct the functions $\sigma_{i,\ell}=\rho_{i,\ell}-\rho_{i,\ell-1}^{(1)_i}$ as in \eqref{4.1} and let $\widetilde\rho_{i,\ell},\widetilde\sigma_{i,\ell}$ be the functions, depending on the order $m$ of the required cancellations, given by Lemma \ref{Lem4.2}.
We define the functions $\Sigma_L$ and $\widetilde\Sigma_L$ by
\bea\label{4.5}
\Sigma_L^{(-L)}&=\Big(\mathop{*}_{i\in D_{L}}\sigma_{i,\ell_i}^{(-\ell_i)_i}\Big)*\Big(\mathop{*}_{i\not\in D_{L}}\rho_{i,\ell_i}^{(-\ell_i)_i}\Big)&&\text{and}&
\widetilde\Sigma_L^{(-L)}&=\Big(\mathop{*}_{i\in D_{L}}\widetilde\sigma_{i,\ell_i}^{(-\ell_i)_i}\Big)*\Big(\mathop{*}_{i\not\in D_{L}}\widetilde\rho_{i,\ell_i}^{(-\ell_i)_i}\Big)
\eea
where $D_L$ is such that $L\in\Gamma(D_L)$.
The following statement is the analogue of Proposition \ref{Prop4.3}. 
\begin{proposition}\label{Prop4.4} We have
$\del_\0=\sum_{L\in \LL}\Sigma_L^{(-L)}$ and $\del_\0=\sum_{L\in \LL}\Sigma_L^{(-L)}*\widetilde\Sigma_L^{(-L)}$ in the sense of distributions. Also if $\Rho_L=\mathop{*}_{i=1}^n\rho_{i,\ell_i}$ and $\widetilde\Rho_L=\mathop{*}_{i=1}^n\widetilde\rho_{i,\ell_i}$, then $\Rho_L^{(-L)}*\widetilde\Rho_L^{(-L)}=\sum_{L'\in \LL\,,\,L'\leq L}\Sigma_{L'}^{(-L')}*\widetilde\Sigma_{L'}^{(-L')}$.
\end{proposition}

\begin{proof}
The identity $\del_\0=\sum\nolimits_{L\in\N^{n}}\Big(\mathop{*}_{i=1}^n\sigma_{i,\ell_i}^{(-\ell_i)_i}\Big)*\Big(\mathop{*}_{i=1}^n\widetilde\sigma_{i,\ell_i}^{(-\ell_i)_i}\Big)$ is a trivial consequences of Lemma \ref{Lem4.2}\eqref{Lem4.2c}. By Lemma \ref{Lem3.7}\eqref{Lem3.7c}, $\N^{n}$ is the disjoint union of the sets $T_L\cap\N^{n}$, parametrized by $L\in\LL$. For $L\in\LL_D$, with $D\subseteq\{1,\dots,n\}$, $T_L\cap\N^{n}$ consists of the elements $L'\in\N^{n}$ such that
$\ell'_i=\ell_i$ if $i\in D$ and $\ell'_i\leq \ell_i$ if $i\not\in D$. It follows from Lemma \ref{Lem4.2}\eqref{Lem4.2d} that
$\Sigma_L^{(-L)}*\widetilde\Sigma_L^{(-L)}=\sum_{L'\in T_L\cap\N^{n}}\big(\mathop{*}_{i=1}^n\sigma_{i,\ell'_i}^{(-\ell'_i)_i}\big)*\big(\mathop{*}_{i=1}^n\widetilde\sigma_{i,\ell'_i}^{(-\ell'_i)_i}\big)$ for every $L\in\LL$. This gives the two reproducing formulas. The proof of the formula for $\Rho_L^{(-L)}*\widetilde\Rho_L^{(-L)}$ is similar.
\end{proof}

We say  that the family $\Sigma=\big\{\Sigma_{L}, \Rho_{L}, :L\in \LL\big\}$ is \textit{constructed via convolution} from $\big\{\rho_{i,\ell}\big\}$.

\subsection{Associated square functions}\label{Sec4.4}
\quad
\smallskip

\begin{definition} 
{\rm All reproducing formulas defined above have  corresponding square functions.
\begin{enumerate}[\rm(i)]
\item 
Let $\Psi=\big\{\Psi_{L}, \Phi_{L} :L\in \LL\big\}$ be a family constructed via tensor product. The associated {\it square function of tensor type} is
\be\label{5.1}
S_\Psi f=\Big(\sum\nolimits_{L\in\LL}\big|f*\Psi_{L}^{(-L)}\big|^2\Big)^\half.
\ee

\item Let $\Sigma=\big\{\Sigma_{L}, \Rho_{L}, :L\in \LL\big\}$ be a family constructed via convolution. The associated {\it square function of convolution type} is
\be\label{6.1}
S'_\Sigma f=\Big(\sum\nolimits_{L\in\LL}\big|f*\Sigma_L^{(-L)}\big|^2\Big)^\half\ .
\ee
\end{enumerate} }
\end{definition}

The main result of the next two section is that, for $f\in\SS(\R^d)$, the condition $Sf\in L^{1}(\R^{d})$ is  independent of the choice of 
the  square function $S$ among those defined above. We call this property {\it $L^1$-equivalence}.

\section{$L^1$-equivalence of multi-norm square functions of tensor type}\label{Sec5}

Let $\Psi=\big\{\Phi_{L}, \Psi_{L}, :L\in \LL\big\}$ and $\Psi'=\big\{\Phi'_{L}, \Psi'_{L}, :L\in \LL\big\}$ be two families constructed via tensor products from $\big\{\varphi_{i,\ell}\big\}$ and $\big\{\varphi'_{i,\ell}\big\}$ respectively and let $S_\Psi$, $S_{\Psi'}$ be the associated square functions.

\begin{theorem}\label{Thm5.1}
If $f\in \SS^{\prime}(\R^{d})$, then $S_{\Psi}f\in L^{1}(\R^{d})$ if and only if $S_{\Psi^{\prime}}\in L^{1}(\R^{d})$. More precisely, there is a constant $C>0$ independent of $\Psi$ and $\Psi'$ and an integer $M>0$ depending on the exponents $\{\lambda_{i}\}$ and the matrix $\bfE$ so that if $f\in \SS^{\prime}(\R^{d})$ and $S_{\Psi'}f\in L^{1}(\R^{d})$ then 
\be\label{5.1.5}
\big\vert\big\vert S_{\Psi^{\prime}}f\big\vert\big\vert_{L^{1}}\leq C\,\big\vert\big\vert\Psi^{\prime }\big\vert\big\vert_{(M)}\big\vert\big\vert S_{\Psi}f\big\vert\big\vert_{L^{1}},
\ee
and similarly with $\Psi$ and $\Psi'$ interchanged.
\end{theorem}

The proof is modelled upon \cite{MR3925053} and is based on the existence of associated Calder\'on reproducing formulas, cf. Proposition \ref{Prop4.3} (a), involving auxiliary functions $\widetilde\Psi_L,\widetilde\Phi_L$, which are also tensor products of function $\widetilde\psi_{i,\ell_i},\widetilde\ph_{i,\ell_i}$, and $\widetilde\Psi_L$ can be assumed, for $i\in D_L$, to have cancellations of order $m$ in the variables $\x _i$, with $m$ arbitrarily large. We also need several lemmas.
Recall that $\big\{\lambda_{i}:1\leq i \leq d\big\}$ are the exponents of the dilations in equation \eqref{1.1}. 
Let 
\bea\label{5.2WWT}
\bar\lambda_{i}&=\min_{k\in E_{i}}\lambda_{k}
&&
\text{ and }&
\bar\lambda&=\min_{1\leq i \leq d}\lambda_{i}.
\eea 
If $f,g\in \SS(\R^{d_{i}})$, $\ell\in \Z$, and $\betab\in \N^{d_{i}}$,
\beas
(f*g)^{(\ell)}&=f^{(\ell)}*g^{(\ell)}, & \partial^{\betab}\big[f*g\big]&=\big[\partial^{\betab}f\big]*g=f*\big[\partial^{\betab}g\big], &&\text{and}& \big[\partial^{\betab}f\big]^{(-\ell)}&=2^{-\ell\,\[\betab\]_{i}}\partial^{\betab}\big[f^{(-\ell)}\big].
\eeas

\begin{lemma}\label{Lem5.2}
Let $f,g\in \SS(\R^{d_{i}})$, and suppose $f$ has cancellation of order $m$ and $g$ has cancellation of order $m'$. Then for any $p,\ell \in\bN$ there exists $h\in\SS(\R^{d_{i}})$ with order of cancellation $m+m'$ and normalized relative to $f$ and $g$ so that
\beas
f^{(-p)}*g^{(-\ell)}=\begin{cases}2^{-\bar\la_i m(p-\ell)}h^{(-\ell)}&\text{ if $p\geq \ell$} \\
2^{-\bar\la_i m'(\ell-p)}h^{(-p)}&\text{ if $\ell\geq p$}
\end{cases}.
\eeas
\end{lemma}

\begin{proof}
We can assume $p\geq \ell $. Also $f^{(-p)}*g^{(-\ell)}=\big(f^{(-(p-\ell))}*g\big)^{(-\ell)}$ so by rescaling we can assume that $\ell=0$. Note that $\[\beta\]_{i}=\sum_{k\in E_{i}}\beta_{k}\lambda_{k}\geq\bar\lambda_{i}\sum_{k\in E_{i}}\beta_{k}=m\bar\lambda_{i}$ so $2^{-\ell\,[\[\betab\]_{i}-\bar\lambda_{i}m]}\leq 1$. By Lemma \ref{Lem4.1} we have $f=\sum_{|\betab|=m}\partial^{\betab}f_{\betab}$. Thus
\beas
f^{(-p)}*g &
=
\sum\nolimits_{|\betab|=m}\big[\partial^{\betab}f_{\betab}\big]^{(-p)}*g
=
\sum\nolimits_{|\betab|=m}2^{-p\,\[\betab\]_{i}}\partial^{\betab}\big[f_{\betab}^{(-p)}\big]*g\\
&=
\sum\nolimits_{|\betab|=m}2^{-p\,\[\betab\]_{i}}\partial^{\betab}\big[f_{\betab}^{(-p)}*g\big]=
2^{-\bar\lambda_{i}m\ell}\Big[\sum\nolimits_{|\betab|=m}2^{-\ell\,[\[\betab\]_{i}-\bar\lambda_{i}m]}\partial^{\betab}\big[f_{\betab}^{(-p)}*g\big]\Big]\\
&=:2^{-\bar\lambda_{i}m\ell}h.
\eeas
Clearly $h$ has vanishing moments up to order $m+m'$. To estimates the Schwartz norms of $\partial^{\betab}\big[f_{\betab}^{(-\p)}*g\big]=f_{\betab}^{(-p)}*\partial^{\betab}g$ we must estimate $(1+|\t|)^{N}\big\vert\f_{\beta}^{(-p)}*\partial^{\betab+\gammab}g(\t)\big\vert$ for $\t\in \R^{d_{i}}$. 
\beas
(1&+|\t|)^{N}\big\vert\f_{\beta}^{(-p)}*\partial^{\betab+\gammab}g(\t)\big\vert
=
(1+|\t|)^{N}\Big|\int_{\R^{d_i}}2^{q_{i} p}f_\beta(2^p \s)\de^{\betab+\gammab} g(\t-\s)\,d\s\Big|
\\
&\leq \|f_{\betab}\|_{(N)}\|\partial^{\betab+\gammab}g\|_{(N)}\int_{\R^{d_i}}\frac{\big(1+|\t|\big)^{N}2^{q_i\ell}d\s}{\big(1+\big|2^\ell \s\big|\big)^N\big(1+|\t-\s|\big)^N}\\
&\leq C_{N} \|f_{\betab}\|_{(N)}\|\partial^{\betab+\gammab}g\|_{(N)}\int_{\R^{d_i}}\frac{\big|2^{\ell}\s\big|^{N}+2^{q_i\ell}\big(1+|\t-\s|\big)^N}{\big(1+\big|2^\ell \s\big|\big)^N\big(1+|\t-\s|\big)^N}\,d\s \leq C_{N} \|f_{\betab}\|_{(N)}\|\partial^{\betab+\gammab}g\|_{(N)}.\qedhere
\eeas
\end{proof}

For $1 \leq i \leq n$ let $\{\varphi_{i,\ell}:\ell\in \N\big\}\subseteq\SS(\R^{d_{i}})$ and $\{\varphi_{i,p}^{\prime}:p\in \N\big\}\subseteq\SS(\R^{d_{i}})$ be uniformly bounded families with $\int_{\R^{d_{i}}}\varphi_{i,\ell}(\x_{i})\,d\x_{i}=\int_{\R^{d_{i}}}\varphi_{i,p}^{\prime}(\x_{i})\,d\x_{i}=1$. Let $\Psi=\big\{\Psi_{L}, \widetilde\Psi_{L}:L\in \LL\big\}$ and $\Psi^{\prime}=\big\{\Psi_{P}^{\prime},\widetilde\Psi_{P}^{\prime}:P\in \LL\big\}$ be the families of Schwartz functions on $\R^{d}$ constructed via tensor products from $\{\varphi_{i,\ell}\}$ and $\{\varphi_{i,p}^{\prime}\}$ as in Section \ref{Sec4.2}. If $P,L\in \N^{n}$ then $|P-L|=\max_{i=1,\dots,n}|p_{i}-\ell_{i}|$. Let $\kappa$ be the constant from Lemma \ref{Lem3.8} and let $N=\kappa\big(1+\max_{j,k}e(j,k)^{-1}\big)$.

\begin{lemma}\label{Lem5.3}
Let $M>0$ and let $m$ be the order of cancellation of the functions $\widetilde\psi_{i,\ell}$. With $\bar\lambda$ defined in \eqref{5.2WWT}, if $\bar\lambda(m-1)>M$ there exists $ C, \epsilon>0$ depending only on $\kappa$, $\bar\lambda$, and the matrix $\bfE$ so that for all $P=(p_{1}, \ldots, p_{n})$ and $L=(\ell_{1}, \ldots, \ell_{n})$ in $\LL$ we have $\int_{\R^{d}}\big(\prod_{i=1}^{n}(1+2^{\ell_{i}}|\x_{i}|)\big)^{M}\big\vert (\Psi_{P}^{\prime})^{(-P)}*(\widetilde\Psi_{L})^{(-L)}(\x)\big\vert\,d\x \leq C\,2^{-\epsilon|P-L|}$.
\end{lemma}

\begin{proof} 
Let $S=\big\{(A_{1},k_{1}), \ldots, (A_{s},k_{s})\big\}$ and $T=\big\{(B_{1},m_{1}), \ldots, (B_{\sigma},m_{\sigma})\big\}$ be marked partitions such that $P\in \LL_{S}$ and $L\in \LL_{T}$. Put $D_{S}=\{k_{1}, \ldots, k_{s}\}$ and $D_{T}=\{m_{1}, \ldots, m_{\sigma}\}$. Since $\Psi_{P}^{\prime}$ and $\widetilde\Psi_{L}$ are tensor products the integral in the statement equals $\prod_{i=1}^{n}I_{i}$ where 
\beas
I_{i}=\int_{\R^{d_i}}\big(1+2^{\ell_i}|\x_{i}|\big)^M\big| (\psi_{i,p_{i}}^{\prime})^{(-p_{i})}*(\widetilde\psi_{i,\ell_{i}})^{(-\ell_{i})}(\x_{i})\big|\,d\x_{i}.
\eeas
We can assume that $|p_{i}-\ell_{i}|>\kappa(1+\max_{j,k}e(j,k)^{-1})$ for otherwise $2^{-\bar\lambda|p_{i}-\ell_{i}|}$ is controlled from above and below by fixed constants and the inequality $I_i\leq C2^{-\eps|p_i-\ell_i|}$ is obvious. 
By Lemma \ref{Lem5.2} there exists $h_{i}\in \SS(\R^{d_{i}})$, normalized relative to $\Psi$ and $\Psi^{\prime}$, so that
\beas
(\Psi_{i,P}^{\prime})^{(-p_{i})}*(\widetilde\Psi_{i,L})^{(-\ell_{i})}(\x_{i})
=
\begin{cases}
2^{-\bar\lambda_{i}(p_{i}-\ell_{i})}h_{i}^{(-\ell_{i})}(\x_{i}) &\text{if $p_{i}\geq \ell_{i}$ and $i\in D_{S}$}\\
h_{i}^{(-\ell_{i})}(\x_{i}) &\text{if $p_{i}\geq \ell_{i}$ and $i\notin D_{S}$}\\
2^{-\bar\lambda_{i}m(\ell_{i}-p_{i})}h_{i}^{(-p_{i})}(\x_{i})&\text{if $\ell_{i}\geq p_{i}$ and $i\in D_{T}$}\\
h_{i}^{(-p_{i})}(\x_{i}) &\text{if $\ell_{i}\geq p_{i}$ and $i\notin D_{T}$}
\end{cases}
\eeas
and so 
\bea\label{5.2}
I_{i}\leq
\begin{cases}
C_{M}\,2^{-\bar\lambda(p_{i}-\ell_{i})}&\text{if $p_{i}\geq \ell_{i}$ and $i\in D_{S}$}\\
C_{M}&\text{if $p_{i}\geq \ell_{i}$ and $i\notin D_{S}$}\\
C_{M}\,2^{-(\bar\lambda m-M)(\ell_{i}-p_{i})}&\text{if $\ell_{i}\geq p_{i}$ and $i\in D_{T}$}\\
C_{M}\,2^{M(\ell_{i}-p_{i})}&\text{if $\ell_{i}\geq p_{i}$ and $i\notin D_{T}$}
\end{cases}.
\eea
Let 
\beas
\SS^{+}&=\Big\{i:p_{i}>\ell_{i},\, i\in D_{S}\Big\}&&\text{\qquad and \quad}& \TT^{+}&=\Big\{i:\ell_{i}>p_{i},\,i\in D_{T}\Big\},\\
\SS^{-}&=\Big\{i:p_{i}>\ell_{i},\,i\notin D_{S}\Big\}&&\text{\qquad and \quad}& \TT^{-}&=\Big\{i:\ell_{i}>p_{i},\, i\notin D_{T}\Big\}.
\eeas
The estimates \eqref{5.2} are small in $|p_{i}-\ell_{i}|$ only when $i\in \SS^{+}$, in which case $I_{i}\leq C_{M}e^{-\bar\lambda(p_{i}-\ell_{i})}$, or when $i\in \TT^{+}$, in which case $I_{i}\leq C_{M}2^{-(\bar\lambda m-M)(\ell_{i}-p_{i})}$. If $i\in \SS^{-}$ the estimate is not small and if $i\in \TT^{-}$ the estimate for $I_{i}$ is large. To prove Lemma \ref{Lem5.3} we show that we can ``borrow'' some of the gain from $I_{i}$ when $i\in \SS^{+}\cup\TT^{+}$ to obtain a gain for $I_{j}$ when $j\notin \SS^{+}\cup\TT^{+}$. We show that there are constants $c_{1},c_{2}>0$ depending only on $\kappa$ and the matrix $\bfE$ so that if $|p_{i}-\ell_{i}|>\kappa\big(1+\max_{j,k}e(j,k)^{-1}\big)$ then the following hold.
\begin{enumerate}[(1)]
\item\label{A}
If $i\in\SS^{-}$ there exists $k_{r}\in \SS^{+}$ with $p_{k_{r}}-\ell_{k_{r}}>c_{1}(p_{i}-\ell_{i})-c_{2}$, and so for $0<\alpha<\!\!<1$
\beas
I_{k_{r}}^{\alpha}I_{i}\leq C_{M}2^{-\bar\lambda \alpha(p_{k_{r}}-\ell_{k_{r}})}\leq C_{M}^{\prime}2^{-\bar\lambda c_{1}\alpha(p_{i}-\ell_{i})}.
\eeas

\item\label{B}
If $i\in\TT^{-}$ there exists $m_{q}\in \TT^{+}$ with $\ell_{m_{q}}-p_{m_{q}}\geq c_{1}(\ell_{i}-p_{i})-c_{2}$, and so
\beas
I_{m_{q}}^{\alpha}I_{i}\leq C_{M}2^{M-(\bar\lambda m-M)\alpha(\ell_{m_{q}}-p_{m_{q}})}\leq C_{M}^{\prime}2^{M-(\bar\lambda m-M-1)c_{1}\alpha(\ell_{i}-p_{i})}.
\eeas
\end{enumerate}
Assuming \eqref{A} and \eqref{B}, then choosing $\alpha$ sufficiently small and then $m$ sufficiently large, it follows that there exists $\epsilon >0$ depending on $n$, $\kappa$, and the matrix $\bfE$ so that $\prod_{i=1}^{n}I_{i}\leq C_{M}^{\prime\prime}2^{-\epsilon|P-L|}$.

Before proving \eqref{A} and \eqref{B}, it follows from Lemma \ref{Lem3.8} that since $P\in \LL_{S}$ and $L\in \LL_{T}$, there are points $\u\in F_{S}$ and $\v\in F_{T}$ with $|P-\u|\leq \kappa$ and $|L-\v|\leq \kappa$. Since $\u,\v\in \Gamma(\bfE)$ we have $u_{i}\leq e(i,j)u_{j}$ and $v_{i}\leq e(i,j)v_{j}$ for all $1 \leq i,j\leq n$. Also $ i\in A_{r}\Longrightarrow p_{i}\leq \frac{p_{k_{r}}}{e(k_{r},i)}$ and $ i\in B_{q}\Longrightarrow m_{i}\leq \frac{\ell_{m_{q}}}{e(m_{q},i)}$.

\medskip

\noindent \textit{Proof of {\rm(1)}.}
Suppose $i\in A_{r}$ so $p_{i}\leq \frac{p_{k_{r}}}{e(k_{r},i)}$. Note that $i\notin \TT^{+}$ since otherwise $\ell_{i}>p_{i}$. We have $\ell_{i}=v_{i}+(\ell_{i}-v_{i})$ with $|\ell_{i}-v_{i}|\leq \kappa$. Since $\v\in \Gamma(\bfE)$ we have $v_{k_{r}}\leq e(k_{r},i)v_{i}$ so
\beas
p_{i}-\ell_{i}&\leq \frac{p_{k_{r}}}{e(k_{r},i)}-\ell_{i}=\frac{p_{k_{r}}}{e(k_{r},i)}-v_{i}+|v_{i}-\ell_{i}|
\leq 
\frac{p_{k_{r}}-v_{k_{r}}}{e(k_{r},i)}+|v_{i}-\ell_{i}| \\
&\leq 
\frac{p_{k_{r}}-\ell_{k_{r}}}{e(k_{r},i)}+\frac{\ell_{k_{r}}-v_{k_{r}}}{e(k_{r},i)}+|v_{i}-\ell_{i}|\leq \frac{p_{k_{r}}-\ell_{k_{r}}}{e(k_{r},i)}+\frac{\kappa}{e(k_{r},i)}+\kappa.
\eeas
It follows that $p_{k_{r}}-\ell_{k_{r}}>0$ from our assumption that $|p_{i}-\ell_{i}|$ is large. Thus $k_{r}\in \SS^{+}$, completing the proof.

\smallskip

\noindent\textit{Proof of {\rm(2)}.}
Suppose $\ell_{i}>p_{i}$ and $i\notin D_{T}$. In this case $\ell_{i}\in B_{q}$ for some $1 \leq q\leq \sigma$ but $\ell_{i}\neq m_{q}$. Then $\ell_{i}=\frac{\ell_{m_{q}}}{e(m_{q},i)}$, and since $u_{m_{q}}\leq e(k_{q},i)u_{i}$ we have

\beas
0<\ell_{i}-p_{i}&=\frac{\ell_{m_{q}}}{e(m_{q},i)}-u_{i}-(p_{i}-u_{i})
\leq
\frac{\ell_{m_{q}}}{e(m_{q},i)}-\frac{u_{m_{q}}}{e(m_{q},i)}-(p_{i}-u_{i})\\
&=
\frac{\ell_{m_{q}}-p_{m_{q}}}{e(m_{q},i)}-\frac{(u_{m_{q}}-p_{m_{q}})}{e(m_{q},i)}-(p_{i}-u_{i})
\eeas
Again it follows that $\ell_{m_{q}}-p_{m_{q}}>0$ and so $m_{q}\in \TT^{+}$. This completes the proof of Lemma \ref{Lem5.3}.
\end{proof}

We will also need to introduce other Calder\'on-type reproducing formulas, with sums restricted to ``large'' elements of $\LL$. Fix $L=(\ell_{1}, \ldots, \ell_{n})\in \LL$ and let $\LL_{L}^{+}=\big\{L'=(\ell_{1}', \ldots, \ell_{n}')\in \LL:\,\ell_{i}'\geq \ell_{i}\big\}$. If $L'\in \LL_{L}^{+}$ let $A_{L,L'}=\big\{i:\ell_{i}'>\ell_{i}\big\}$, and define
\beas
\psi_{i,\ell_{i}'}^{\#}&=
\begin{cases}
\psi_{i,\ell_{i}'}&\text{if $i\in A(L,L')\cap D_{L'}$}\\
\varphi_{i,\ell_{i}'}&\text{if $i\notin A(L,L')\cap D_{L'}$}
\end{cases}
&&\text{and}&\Psi_{L,L'}^{\#}&=\bigotimes_{i=1}^{n}\psi_{i,\ell_{i}'}^{\#},\\
\widetilde \psi_{i,\ell_{i}'}^{\#}&=
\begin{cases}
\widetilde \psi_{i,\ell_{i}'}&\text{if $i\in A(L,L')\cap D_{L'}$}\\
\widetilde \varphi_{i,\ell_{i}'}&\text{if $i\notin A(L,L')\cap D_{L'}$}
\end{cases}&&\text{and}&\widetilde\Psi_{L,L'}^{\#}&=\bigotimes_{i=1}^{n}\widetilde\psi_{i,\ell_{i}'}^{\#}.
\eeas
Note that $\Psi^\#_{L,L'}$ has cancellation in the variable $\x_{i}$ and $\widetilde\Psi^\#_{L,L'}$ has cancellation of order $m$ in the variable $\x_{i}$ only when $i\in D_{L'}\cap A_{L,L'}$. In particular, $\Psi^\#_{L,L}$ and $\widetilde\Psi^\#_{L,L}$ have no cancellation. The following is then a variant of Proposition \ref{Prop4.3}, and we omit the proof.

\begin{lemma} \label{Lem5.4}If $L\in \LL$ and $L'\in \LL_{L}^{+}$ then 
$
\sum_{L'\in\LL_L^+}(\Psi_{L,L'}^\#)^{(-L')}*(\widetilde\Psi_{L,L'}^\#)^{(-L')}=\delta_{\0}.
$
\end{lemma}
\begin{proof}
Let 
$
B_{L,L'}=\Big\{L''\in \N^{n}:\begin{cases}\ell_{i}''=\ell_{i}'&\text{if $i\in A_{L,L'}\cap D_{L'}$}\\ 0\leq \ell_{i}''
\leq \ell_{i}'&\text{ if $i\notin A_{L,L'}\cap D_{L'}$}
\end{cases}\Big\}
$,
so $\bigcup_{L'\in \LL_{L}^{+}}B(L,L')=\N^{n}$. 
Then 
$$
(\Psi_{L,L'}^{\#})^{(-L')}*(\widetilde \Psi_{L,L'}^{\#})^{(-L')}=\sum\nolimits_{L''\in B_{L,L'}}\mathop{\otimes}\nolimits_{i=1}^{n}(\psi_{i,\ell_{i}''}*\widetilde\psi_{i,\ell_{i}''})
$$
and so $\sum_{L'\in\LL_L^+}(\Psi_{L,L'}^\#)^{(-L')}*(\widetilde\Psi_{L,L'}^\#)^{(-L')}=\mathop{\otimes}_{i=1}^{n}\sum_{\ell_{i}=0}^{\infty}(\psi_{i,\ell_{i}}*\widetilde\psi_{i,\ell_{i}})=\delta_{\0}$. \qedhere
\end{proof}

For $M>0$ let $\chi_M(\x)=\prod\nolimits_{i=1}^n\big(1+|\x_{i}|\big)^{-M}$.
\begin{lemma}\label{Lem5.5}
Let $f\in\SS^{\prime}(R^{d})$, let $M>0$, and let $m$ be the order of cancellation of $\widetilde\psi_{i,\ell_{i}}$. There is a constant $\delta>0$ depending only on the exponents $\lambda_{i}$ and the matrix $\bfE$, and a constant $C>0$, depending only on $M,m,\delta$ so that 
\beas
\big|f*\Psi_{L}^{(-L)}(\x)\big|\leq C \sum\nolimits_{L'\in\LL_L^+}2^{-m\delta |L-L'|}\big|f*\Psi_{L'}^{(-L')}\big|*\chi_M^{(-L)}(\x)\ ,
\eeas
\end{lemma}

\begin{proof}
We have $f*\Psi_{L}^{(-L)}=\sum_{L'\in \LL_{L}^{+}}f*\Psi_{L}^{(-L)}*(\Psi_{L,L'}^\#)^{(-L')}*(\widetilde\Psi_{L,L'}^\#)^{(-L')}$ from Lemma \ref{Lem5.4}. Now
\beas
\psi_{i,\ell_{i}}^{(-\ell_{i})}*\big(\psi_{i,\ell_{i}'}^{\#}\big)^{(-\ell_{i}')}
&=
\begin{cases}
\psi_{i,\ell_{i}}^{(-\ell_{i})}*\psi_{i,\ell_{i}'}^{(-\ell_{i}')}=\psi_{i,\ell_{i}}^{(-\ell_{i})}*\psi_{i,\ell_{i}'}^{(-\ell_{i}')}=\psi_{i,\ell_{i}'}^{(-\ell_{i}')}*\psi_{i,\ell_{i}}^{(-\ell_{i})}&\text{if $i\in A(L,L')\cap D_{L'}$}\\
\psi_{i,\ell_{i}}^{(-\ell_{i})}*\varphi_{i,\ell_{i}'}^{(-\ell_{i}')}=\psi_{i,\ell_{i}}^{(-\ell_{i})}*\varphi_{i,\ell_{i}}^{(-\ell_{i})}=\psi_{i,\ell_{i}'}^{(-\ell_{i}')}*\varphi_{i,\ell'}^{(-\ell_{i})}&\text{if $i\notin A(L,L')\cap D_{L'}$}
\end{cases}\\
&=
\psi_{i,\ell_{i}'}^{(-\ell_{i}')}*(\psi_{i,\ell_{i}}^{\#})^{(-\ell_{i})},
\eeas
so
\beas
f*\Psi_{L}^{(-L)}
&=\sum\nolimits_{L'\in \LL_{L}^{+}}f*\Psi_{L'}^{(-L')}*(\Psi_{L,L}^\#)^{(-L)}*(\widetilde\Psi_{L,L'}^\#)^{(-L')}\\
&=\sum\nolimits_{L'\in \LL_{L}^{+}}f*\Psi_{L'}^{(-L')}*\left[\Psi_{L,L}^\#*(\widetilde\Psi_{L,L'}^\#)^{(L-L')}\right]^{(-L)}.
\eeas
Now from Lemma \ref{Lem5.2}, there exists $h_{i}\in \SS(\R^{d_{i}})$, normalized relative to $\Psi$ and $\Psi'$, so that
\bea\label{5.3}
\psi_{i,\ell_{i}}^\#*(\widetilde\psi_{i,\ell_{i}'}^\#)^{(-(\ell_{i}'-\ell_{i}))}(\x_{i})=
\begin{cases}
2^{-\bar\lambda_{i}m(\ell_{i}'-\ell_{i})}h_{i}(\x_{i})&\text{if $i\in D_{L'}\cap A_{L,L'}$}\\
h_{i}(\x_{i})&\text{if $i\notin D_{L'}\cap A_{L,L'}$}
\end{cases}.
\eea

To complete the proof we claim that $\prod_{i=1}^{n}\psi_{i,\ell_{i}}^\#*(\widetilde\psi_{i,\ell_{i}'}^\#)^{(-(\ell_{i}'-\ell_{i}))}(\x_{i})=2^{-\delta m|L'-L|}\prod_{i=1}^{n}h_{i}(\x_{i})$ where $\delta>0$ depends on the exponents $\{\lambda_{i}\}$ and the matrix $\bfE$. From \eqref{5.3} we only need additional estimates for factors with $i\notin D_{L'}\cap A_{L,L'}$. If $i\in D_{L}\setminus A_{L,L'}$ then $\ell_{i}'-\ell_{}=0$ and so $\psi_{i,\ell_{i}}^{\#}*(\widetilde\psi_{i,\ell_{i}'}^\#)^{(-(L'-L))}(\x_{i})=2^{-\bar\lambda_{i}m(\ell_{i}'-\ell_{i})}h_{i}^{(-(\ell_{i}'-\ell_{i}))}(\x_{i})$. If $i\notin D_{L'}$ then $i\in A_{r}$ for some $1\leq r\leq s$ and $k_{r}\in D_{L'}\cap A_{L,L'}$. Then $|\ell_{i}-\ell_{i}'|=\ell_{i}-\ell_{i}'\leq c_{1}|\ell_{k_{r}}-\ell_{k_{r}}'|+c_{2}$ where $c_{1}, c_{2}$ where $c_{1}, c_{2}>0$ depend only on $\kappa$ and the matrix $\bfE$. Then we can ``borrow'' part of the estimate $2^{-\bar\lambda_{k_{r}}m(\ell_{k_{r}}'-\ell_{k_{r}})}$ to obtain a gain $2^{-\delta m(\ell_{i}'-\ell_{i})}$. This establishes the claim, and completes the proof.
\end{proof}

We introduce two auxiliary functions related to the function $f*\Psi_{L}^{(-L)}$. 
\bea\label{5.4}
\omega_{M,L}(\x)&=\sup\nolimits_{\x'\in\R^{d}}\frac{|f*\Psi_{L}^{(-L)}(\x-\x')|}{\prod_{i=1}^{n}(1+2^{\ell_{i}}|\x_{i}'|)^{M}}\\
&=
2^{-\sum_{i=1}^{n}q_{i}\ell_{i}}\sup\nolimits_{\x'\in \R^{d}}|f*\Psi_{L}^{(-L)}(\x-\x')|\chi_{M}^{(-L)}(\x')\\
\omega^*_{M,L}(\x)&=\sup\nolimits_{\x'\in\R^d}\sup\nolimits_{L'\in\LL_L^+}2^{-M|L-L'|}\frac{\big|f*\Psi_{L'}^{(-L')}(\x-\x')\big|}{\prod_{i=1}^{n}\big(1+2^{\ell_i}|\x'_i|\big)^M}\\
&=2^{-\sum_{i=1}^{n} q_i\ell_i}\sup\nolimits_{\x'\in\R^d}\sup\nolimits_{L'\in\LL_L^+}2^{-M|L-L'|}\big|f*\Psi_{L'}^{(-L')}(\x-\x')\big|\chi_M^{(-L)}(\x')\ .
\eea
Obviously $\omega_{M,L}(\x)\leq \omega^*_{M,L}(\x)$.

\begin{lemma}\label{Lem5.6}
Let $f\in\SS^{\prime}(\R^{d})$, let $M>0$, let $m$ be the order of cancellation of $\widetilde \psi_{i,\ell_{i}}$, and let $\delta$ be the constant from Lemma \ref{Lem5.5}. There is a constant $C>0$ depending only on $M,m,\delta$ so that 
\beas
\omega_{M,L}(\x)\leq C \sum\nolimits_{L'\in\LL_L^+}2^{-m\delta|L-L'|}\big|f*\Psi_{L'}^{(-L')}\big|*\chi_M^{(-L)}(\x).
\eeas
\end{lemma}

\begin{proof}
From Lemma \ref{Lem5.5}, we have
\beas
\Big|f*\Psi_{L}^{(-L)}(\x-\x')\Big|&\chi_M^{(-L)}(\x')
\leq C\sum_{L'\in\LL_L^+}\!\!2^{-m\delta|L-L'|}\!\!\int \big|f*\Psi_{L'}^{(-L')}(\x-\y)\big|\chi_M^{(-L)}(\y-\x')\chi_M^{(-L)}(\x')d\y .
\eeas
From the general inequality 
\bea\label{5.5}
\chi_M^{(-L')}(\y-\x')\chi_M^{(-L)}(\x')\leq C\, 2^{\sum_i q_i\ell'_i}\chi_M^{(-L)}(\y)
\eea
for $L'\in\LL_L^+$ applied with $L'=L$, we obtain 
$$
\chi_M^{(-L)}(\x')\int \big|f*\Psi_{L'}^{(-L')}(\x-\y)\big|\chi_M^{(-L)}(\y-\x')\,dy\leq C2^{\sum_i q_i\ell_i} \int \big|f*\Psi_{L'}^{(-L')}(\x-\y)\big|\chi_M^{(-L)}(\y)\,d\y\ .
$$
The right-hand side is independent of $\x'$, completing the proof.
\end{proof}

\begin{lemma}\label{Lem5.7}
For $r\leq 1$ and $rM>1$, we have the inequality
$$
\omega^*_{M,L}(\x)^r\leq C_M \sum\nolimits_{L'\in\LL_L^+}2^{-rM|L-L'|}\big|f*\Psi_{L'}^{(-L')}\big|^r*\chi_{M}^{(-L)}(\x)\ .
$$
\end{lemma}

\begin{proof}
From Lemma \ref{Lem5.5} and equation \eqref{5.5}
\beas
2^{-M|L'-L|}&\big|f*\Psi_{L'}^{(-L')}(\x-\x')\big|\chi_M^{(-L)}(\x')\\
&\leq C 2^{-M|L'-L|} \sum\nolimits_{L''\in\LL_{L'}^+}2^{-m\delta |L''-L'|}\Big(\big|f*\Psi_{L''}^{(-L'')}\big|*\chi_M^{(-L')}\Big)(\x-\x')\chi_M^{(-L)}(\x')\\
&\leq C' 2^{-M|L'-L|}2^{\sum_{i}q_{i}\ell_{i}}\sum\nolimits_{L''\in\LL_{L'}^+}2^{-m\delta |L''-L'|}\Big(\big|f*\Psi_{L''}^{(-L'')}\big|*\chi_M^{(-L)}\Big)(\x)\ .
\eeas
We majorize the convolution as follows:
\beas
\Big(\big|&f*\Psi_{L''}^{(-L'')}\big|*\chi_M^{(-L)}\Big)(\x)=\int \big|f*\Psi_{L''}^{(-L'')}(\x-\x')\big|\chi_M^{(-L)}(\x')\,d\x'\\
&\leq \Big(\sup_{\x'}\big|f*\Psi_{L''}^{(-L'')}(\x-\x')\big|\chi_M^{(-L)}(\x')\Big)^{1-r}\int \big|f*\Psi_{L''}^{(-L'')}(\x-\x')\big|^r\chi_M^{(-L)}(\x')^r\,d\x'\\
&\leq \Big( 2^{\sum_iq_i\ell_i}2^{M|L''-L|}\omega^*_{M,L}(\x)\Big)^{1-r}\int \big|f*\Psi_{L''}^{(-L'')}(\x-\x')\big|^r\chi_M^{(-L)}(\x')^r\,d\x'\\
&=2^{(1-r)M|L''-L|}\omega^*_{M,L}(\x)^{1-r}\Big(\big|f*\Psi_{L''}^{(-L'')}\big|^r*\chi_{rM}^{(-L)}\Big)(\x)\ ,
\eeas
where we have used the identity $\big(\chi_M^{(-L)}\big)^r=2^{-(1-r)\sum_iq_i\ell_i}\chi_{rM}^{(-L)}$. Putting this into the previous formula we obtain
\beas
2^{-\sum_iq_i\ell_i}&2^{-M|L'-L|}\big|f*\Psi_{L'}^{(-L')}(\x-\x')\big|\chi_M^{(-L)}(\x')\\
&\lesssim 2^{-M|L'-L|}\omega^*_{M,L}(\x)^{1-r}\sum\nolimits_{L''\in\LL_{L'}^+}2^{-\mu|L''-L'|}2^{(1-r)M|L''-L|}\Big(\big|f*\Psi_{L''}^{(-L'')}\big|^r*\chi_{rM}^{(-L)}\Big)(\x)\ .
\eeas
Taking the supremum over $\x'\in\R^d$ and $L'\in\LL_L^+$, the left-hand side gives $\omega^*_{M,L}(\x)$. The right-hand side does not depend on $\x'$ can be further majorized by exchanging the order of the supremum in~$L'$ and the sum to give
\beas
\omega^*_{ML}(\x)^r&\lesssim \sum_{L''\in\LL_L^+}2^{(1-r)M|L''-L|}\Big(\big|f*\Psi_{L''}^{(-L'')}\big|^r*\chi_{rM}^{(-L)}\Big)(\x) \sup_{\substack{L'\in \LL\\L\leq L'\leq L''}}\Psi_{L'}^{(-L')}2^{-M|L-L'|}2^{-\mu|L''-L'|}\ .
\eeas
Since 
$
2^{-M|L'-L|}2^{-\mu|L''-L'|}=2^{-M(|L'|-|L|)}2^{-\mu(|L''|-|L'|)}=2^{M|L|}2^{-\mu|L''|} 2^{(\mu-M)|L'|}$, taking $m$ sufficiently large, the sup is reached at $L'=L''$, so that
\beas
\omega^*_{M,L}(x)^r\lesssim \sum\nolimits_{L''\in\LL_L^+}2^{-rM|L''-L|}\Big(\big|f*\Psi_{L''}^{(-L'')}\big|^r*\chi_{rM}^{(-L)}\Big)(\x).\qquad\qedhere
\eeas
\end{proof}

\begin{proof}[Proof of Theorem \ref{Thm5.1}] To estimate of $S_{\Psi^{\prime}}f(\x)=\big[\sum_{P\in \LL}\big\vert f*\big(\Psi_{P}^{\prime}\big)^{(-P)}\big\vert^{2}\big]^{\frac{1}{2}}$, the starting point is the identity $f*(\Psi_{P}^{\prime})^{(-P)}= \sum\nolimits_{L\in \LL}\Big(f*\Psi_{L}^{(-L)}\Big)*\Big((\Psi_{P}^{\prime})^{(-P)}*\widetilde\Psi_{L}^{(-L)}\Big)$ from Proposition \ref{Prop4.3}. Then
\bea\label{5.6}
\big|f*{\Psi_{P}^{\prime}}^{(-P)}(\x)\big| &\leq
\sum\nolimits_{L\in\LL}\int_{\R^{d}}\big|f*\Psi_{L}^{(-L)}(\x-\x')\big|\,\,\big|{\Psi_{P}^{\prime}}^{(-P)}*\widetilde\Psi_{L}^{(-L)}(\x')\big|\,d\x'\\
&\leq C \sum\nolimits_{L\in\LL}\omega_{M,L}(\x)\int\big|{\Psi_{P}^{\prime}}^{(-P)}*\widetilde\Psi_{L}^{(-L)}(\x')\big|\chi_{M}(\x')\,d\x'\\
&\leq C\sum\nolimits_{L\in \LL}2^{-\delta m|P-L|}\omega_{M,L}(\x).
\eea
Using Young's inequality on $\bZ^n$ it follows that 
\bea\label{5.7}
\Big(\sum\nolimits_{P\in\LL}\big|f*{\Psi'_P}^{(-P)}(\x)\big|^2\Big)^\half&\leq C \bigg(\sum\nolimits_{L\in\LL}\Big(\sum\nolimits_{L'\in\LL}2^{-\eps|L-L'|}\omega_{M,L'}(\x)\Big)^2\bigg)^\half\\
&\leq C \Big(\sum\nolimits_{L\in\LL}\omega_{M,L}^2(\x)\Big)^\half\ .
\eea
Thus it suffices to show that $\big\| \big(\sum_{L\in\LL}\omega_{M,L}^2\big)^\half\big\|_1\leq C\|S_\Psi f\|_1$. From Lemma \ref{Lem5.7}, we obtain, for every $M>1/r$, the inequality
\be\label{5.8}
\omega_{M,L}(\x)^r\leq C_M\sum\nolimits_{L'\in\LL_L^+}2^{-M|L'-L|}\cM_s\big(|f*\Psi_{L'}^{(-L')}|^r\big)(\x)\ ,
\ee
where $\cM_s$ is the strong maximal operator
\be\label{Ms}
\cM_sf(\x)=\sup\nolimits_{0<r_i\le1}\Big(\prod\nolimits_{i=1}^n r_i^{-q_i}\int_{|\y_{i}|\leq r_i}\big|f(\x-\y)\big|\,d\y\Big)\ .
\ee
 Going back to \eqref{5.7}, it follows from \eqref{5.8} and Young's inequality that
\beas
\Big(\sum\nolimits_{L\in\LL}\omega_{M,L}^2(x)\Big)^\half&\leq C_M\bigg(\sum\nolimits_{L\in\LL} \Big(\sum\nolimits_{L'\in\LL_L^+} 2^{-M|L'-L|}\cM_s\big(|f*\Psi_{L'}^{(-L')}|^r\big)\Big)^\frac2r\bigg)^\half\\
&\leq C_M\Big(\sum\nolimits_{L\in\LL} \big(\cM_s(|f*\Psi_{L}^{(-L)}|^r)\big)^\frac2r\Big)^\half\ .
\eeas
Finally,
\beas
\big\|\big(\sum\nolimits_{L\in\LL}\omega_{M,L}^2\big)^\half\big\|_1&\leq C_M \big\|\big(\sum\nolimits_{L\in\LL}\cM_s\big(|f*\Psi_{L}^{(-L)}|^r\big)^\frac2r\big)^\frac r2 \big\|_{1/r}^{1/r}\\
&\leq C'_M \big\|\big(\sum\nolimits_{L\in\LL}\big|f*\Psi_{L}^{(-L)}|^2\big)^\frac r2 \big\|_{1/r}^{1/r}=C'_M \big\|\big(\sum\nolimits_{L\in\LL}\big|f*\Psi_{L}^{(-L)}\big|^2\big)^\half \big\|_1\ ,\\
\eeas
by the vector-valued strong maximal theorem. This completes the proof . 
\end{proof}

Theorem \ref{Thm5.1} establishes the equivalence of $S_{\Psi}f\in L^{1}(\R^{d})$ and $S_{\Psi'}f\in L^{1}(\R^{d})$ if the families $\Psi$ and $\Psi'$ are constructed via tensor products of reproducing formulas on each $\R^{d_{i}}$. It is useful to observe that the single inequality \eqref{5.1.5} also holds with weaker conditions imposed on $\Psi'$. Thus suppose $\Theta=\big\{\Theta_{L}:L\in\LL\big\}\subseteq \SS(\R^{d})$ is a family uniformly bounded in every Schwartz norm satisfying the cancellation conditions
\bea\label{5.9}
\int_{\R^{d_{i}}}\Theta_{L}(\ldots,\x_{i},\ldots)\,d\x_{i}&=0&&\text{for all $i\in D_{L}$}.
\eea
Note that $\Theta_{L}$ is \textbf{not} required to be a tensor product. Let $S_{\Theta}f=\big(\sum\nolimits_{L\in\LL}\big|f*\Theta_{L}^{(-L)}\big|^2\big)^\half$. Without further assumptions on $\Theta$ we have the following one-sided inequality.

\begin{proposition}\label{Prop5.8}
Let $\Psi=\big\{\Psi_{L}:L\in\LL\big\}$ be a family given by tensor products as in Theorem \ref{Thm5.1}. There is a constant $C>0$ depending only on the Schwartz norms of $\Theta$ so that if $f\in \SS'(\R^{d})$ and if $S_{\Psi}f\in L_{1}(\R^{d})$ then $S_{\Theta}f\in L^{1}(\R^{d})$ and $\|S_{\Theta}f\|_{L^{1}}\leq C \|S_{\Psi}f\|_{L^{1}}$.
\end{proposition}

The proof is contained in that of Theorem \ref{Thm5.1}, once the following reduction to tensor products is established. 

\begin{lemma}\label{Lem5.9}
Let $F\in\cS(\R^d)$, let $D\subseteq \{1, \ldots, n\}$, and suppose $\int_{\R^{d_{i}}}F(\ldots,\x_{i}, \ldots)\,d\x_{i}=0$ for $i\in D$. For each $m\in \N$ there are functions $f_{i,m}\in \SS(\R^{d_{i}})$, $1\leq i \leq n$, such that
\begin{enumerate}[\rm(1)]
\item $ \int_{\R^{d_{i}}}f_{i,m}(\x_{i})\,d\x_{i}=0$ if $i\in D$;

\smallskip

\item the Schwartz norms of $f_{i,m}$ decay rapidly in $m$; precisely, for every $N,M\in \N$ there exist $N'\in \N$ and $C_{M,N}>0$ such that $\|f_{i,m}\|_{(N)}\leq C_{N,M}(1+m)^{-M}\|F\|_{(N')}^{1/n}$;

\smallskip

\item
$F(\x_{1}, \ldots, \x_{n})=\sum_{m\in \N}\left(\prod_{i=1}^{n}f_{i,m}(\x_{i})\right)$.
\end{enumerate}
\end{lemma}

\begin{proof}
Begin by ignoring the cancellations imposed on $F$. Let $\chi\in\CC^{\infty}_{0}(\R^{d})$ be supported on the cube $Q=[-1,1]^d$ and satisfy $\sum_{\k\in\bZ^d}\chi(\x-\k)\equiv 1$. Then $F_\k(\x)=F(\x+\k)\chi(\x)$ is supported in $Q$, and $\|F_{\k}\|_{(N)}\leq C_{N,M}\big(1+|\k|\big)^{-M}\|F\|_{(N\vee M)}$ for every $N,M$. Let $F_\k^{\rm per}$ denote the periodic extension of $F_{\k}$ (restricted to $2Q$) having period 4 in each variable. $F_{\k}^{\rm per}$ has a Fourier expansion 
$
F_\k^{\rm per}(\x)=\sum_{\h\in\bZ^d}a_{\k,\h}e^{i\frac\pi2\sum_{j=1}^d h_jx_j}$. Choose $\eta\in C^\infty_{0}(\R)$ supported on $[-2,2]$ and equal to 1 on $[-1,1]$ so that $\eta^\#(\x)=\prod_{j=1}^d\eta(x_j)$ is supported on $2Q$ and equal to 1 on $Q$. Set $ \eta^{\#}_{i}(\x_{i})=\prod_{j\in E_{i}}\eta(x_j)$ and $\h_i=(h_j)_{j\in E_{i}}$. Putting $g_{i,\k,\h}(\x_{i})=a_{\k,\h} ^{1/n}\eta^\#(\x_{i})e^{i\pi \h_i\cdot \x_{i}}$ we have
\beas
F_\k(\x)\eta^\#(\x)=F_\k^{\rm per}(\x)\eta^\#(\x)
=\sum\nolimits_{\h\in\bZ^d}a_{\k,\h}\prod\nolimits_{i=1}^n\big(\eta^\#(\x_{i})e^{i\frac\pi2 \sum\nolimits_{j\in E_{i}}h_jx_j}\big)
=\sum\nolimits_{\h\in\bZ^d}\prod_{i=1}^ng_{i,\k,\h}(\x_{i}),
\eeas 
and so $F(\x)=\sum\nolimits_{\k\in\bZ^d}F_\k(\x-\k)=\sum\nolimits_{\k,\h\in\bZ^d}\prod\nolimits_{i=1}^ng_{i,\k,\h}(\x_{i}-\k_i)$. For every $P\in\bN$, the coefficients $a_{\k,\h}$ satisfy the estimate $\big|a_{\k,\h}\big|\leq C'_P\big(1+|\h|)^{-P}\|F_\k\|_{(P)}
\leq C'_{M,P}\big(1+|\h|)^{-P}\big(1+|\k|\big)^{-M}\|F\|_{(P\vee M)}$. Thus $\big\|g_{i,\k,\h}\big\|_{(N)}\leq C_N \big(1+|\k|\big)^N\|g_{i,\k,\h}\|_{\CC^N}
\leq C_{N,P,M}|\h|^N\big(1+|\h|)^{-\frac Pn}\big(1+|\k|\big)^{-\frac Mn}\|F\|_{(P\vee M)}^\frac1n$ for every $N\in\bN$. This gives the estimates $\|f_{i,m}\|_{(N)}\leq C_{N,M} (m+1)^{-M}\|F\|_{(N')}^{1/n}$ up to well-ordering the pairs $(\k,\h)$ by increasing modulus and choosing $P$ appropriately. 

Now suppose that $F$ satisfies the cancellation conditions. It follows from Lemma \ref{Lem4.1} that $F=\sum\nolimits_{q\in \prod_{i=1}^nE_{i}}\de_{y_{q_1}}\cdots\de_{y_{q_n}}F_q$ with $F_q\in\cS(\R^d)$ and $\|F_q\|_{(N)}\leq C_N\|F\|_{(N)}$ for every $q$. It is then sufficient to apply the previous result to $F_q$.
\end{proof}

\begin{proof}[Proof of Proposition \ref{Prop5.8}]
Given $\Theta=\big\{\Theta_{L}:L\in \LL\big\}$ satisfying \eqref{5.9}, write $\Theta_{L}=\sum_{m\in \N}\Theta_{m,L}$ where $\Theta_{m,L}(\x_{1}, \ldots, \x_{n})=\prod_{i=1}^{n}\theta_{i,m,L}(\x_{i})$ is a tensor product. Then by Minkowski's inequality,
\beas
\SS_{\Theta}f(\x)&=\Big[\sum_{L\in \LL}|f*\Theta_{L}(\x)|^{2}\Big]^{\frac{1}{2}}=\Big[\sum_{L\in \LL}\Big|\sum_{m\in \N}f*\Psi_{m,L}(\x)\Big|^{2}\Big]^{\frac{1}{2}} = \Big\vert \Big\vert \Big\{\sum_{m\in \N}f*\Psi_{m,L}(\x)\Big\}_{L\in \LL} \Big\vert \Big\vert_{\ell^{2}(\LL)}\\
&\leq
\sum_{m\in \N}\Big\vert \Big\vert \Big\{f*\Psi_{m,L}(\x)\Big\}_{L\in \LL} \Big\vert \Big\vert_{\ell^{2}(\LL)}
=
\sum_{m\in \N}S_{\Psi_{m}}f(\x).
\eeas
Repeating the proof of Theorem \ref{Thm5.1}, we obtain that $\big\vert\big\vert \SS_{\Theta_{m}}f\big\vert\big\vert_{L^{1}}\leq \sum_{m\in \N}\|\Theta_{m}\|_{(M)}\big\vert\big\vert S_{\Psi}f\big\vert\big\vert_{L^{1}}$ and the proof is completed applying Lemma \ref{Lem5.9}.
\end{proof}

\section{ $L^1$-equivalence of multi-norm square functions of convolution and tensor type}\label{Sec6}
\medskip

In this section we study the multi-norm square functions of convolution type defined in Section~\ref{Sec4.4}. The first step is to consider two families $\Sigma=\big\{\Sigma_{L}, \Rho_{L}, :L\in \LL\big\}$ and $\Sigma'=\big\{\Sigma'_{L}, \Rho'_{L}, :L\in \LL\big\}$  constructed via convolution from $\big\{\rho_{i,\ell}\big\}$ and $\big\{\rho'_{i,\ell}\big\}$ respectively. We first prove that any two  square functions of convolution type are $L^1$-equivalent and then prove $L^1$- equivalence between specific square functions of different types.

\begin{remark}{\rm
In general, the  study of the functions $\Sigma_L$ cannot be reduced to the study of the functions $\Psi_L$ appearing in square functions of tensor type because they do not necessarily have cancellation in single coordinates $\x_{i}$. Note however that if we assume that the Fourier transform of the factors $\sigma_{i,\ell_i}^{(-\ell_i)}$ is supported near the sets $B^i_{\ell_i}$ introduced in Section \ref{Sec2.1}, then the Fourier transform of $\Sigma_L^{(-L)}$ is supported near the annular set $B_L$ in \eqref{2.5},  and hence 
$\Sigma_L^{(-L)}$ gains cancellation in each variable $\x_i$ with $i\in D_L$ (see the discussion in Section \ref{Sec3.4} and  the proof of Proposition \ref{Prop6.5} below).
This indicates that also for general $\sigma_{i,\ell_i}\in\SS(\R^d)$ with integral zero, the convolution $\mathop*_{i\in D_L}\sigma_{i,\ell_i}^{(-\ell_i)}$ generates some implicit  cancellation in the single variables that are not immediately  apparent from the formulas.}
\end{remark}

\subsection{Independence of the choice of $\big\{\Sigma_{L}:L\in \LL\big\}$}\label{Sec6.1} 
\quad
\smallskip

We show that the condition $S'_\Sigma f\in L^{1}(\R^{d})$ does not depend on the choice of $\{\Sigma_L:L\in \LL\}$.

\begin{theorem}\label{Thm6.2}
Let $\Sigma,\Sigma'$ be two families as in Section \ref{Sec4.3}, and let $S'_\Sigma$, $S'_{\Sigma'}$ be the corresponding square functions \eqref{6.1}. Given $f\in \cS'(\R^d)$ the conditions $S'_\Sigma f\in L^{1}$ and $S'_{\Sigma'}f\in L^{1}$ are equivalent.
\end{theorem}

This statement will be a direct consequence of the following analogue of Proposition \ref{Prop5.8}. 
\begin{proposition}\label{Prop6.3}
Let $\Sigma$ be the family of functions introduced in Section \ref{Sec4.3}, and for $1\leq i \leq n$ and $L\in \LL$ let $\{\gamma_{i,L}\}$ be a family of Schwartz functions on $\R^d$, uniformly bounded in every Schwartz norm such that, if $L\in\LL_D$, then $\int_{\R^d}\gamma_{i,L}(\x)\,d\x=0$ for all $i\in D$.
Define $\Gamma_L$ so that $\Gamma_L^{(-L)}=\mathop*_{i=1}^n\gamma_{i,L}^{(-\ell_i)_i}$.
Then $\big\|\big(\sum_{L\in\LL}\big|f*\Gamma_L^{(-L)}\big|^2\big)^\half\big\|_1\leq C \big\|S'_\Sigma f\big\|_1$ where $C>0$ depending only on the bounds for Schwartz norms of $\gamma_{i,L}$.
\end{proposition}

With a single addition the proof follows the same lines as that of Theorem \ref{5.1}. 

\begin{proof}
 The only point that requires some comment is Lemma \ref{5.3}. We prove that required estimate remains valid in the form: 
\beas
\int_{\R^d}\Big(\prod_{i=1}^n\big(1+2^{\ell'_i}|\x_{i}|\big)\Big)^M\big|\Gamma_L^{(-L)}*\widetilde\Sigma_{L'}^{(-L')}(\x)\big|\,d\x\leq C2^{-\eps|L-L'|}\ ,
\eeas
uniformly in $L,L'\in\LL$, provided $\widetilde\Sigma_{L'}$ has cancellations of sufficiently high order relative to $M$ and $\eps$ is sufficiently small. For each convolution $\gamma_{i,\ell_i}^{(-\ell_i)_i}*\widetilde\sigma_{i,\ell'_i}^{(-\ell'_i)_i}$, we apply Lemma \ref{5.2} on $\R^d$ to obtain that, with $m$ equal to the order of cancellation of $\widetilde\sigma_{i,\ell'_i}$,
\be\label{6.2}
\gamma_{i,\ell_i}^{(-\ell_i)_i}*\widetilde\sigma_{i,\ell'_i}^{(-\ell'_i)_i}=\begin{cases}2^{-m\bar\la_i(\ell'_i-\ell_i)}h_{i,L,L'}^{(-\ell_i)_i}&\text{ if }\ell_i\leq \ell'_i\\
2^{-\bar\la_i(\ell_i-\ell'_i)}h_{i,L,L'}^{(-\ell'_i)_i}&\text{ if }\ell'_i< \ell_i\ ,\end{cases}
\ee
where $h_{i,L,L'}\in\cS(\R^d)$ with Schwartz norms controlled by those of $\gamma_{i,\ell_i}$ and $\widetilde\sigma_{i,\ell'_i}$.
Hence, with $\bar\ell_i=\min\{\ell_i,\ell'_i\}$ and $\mu_i(L,L')$ equal to the exponent of 2 in \eqref{6.2},
$$
\Gamma_L^{(-L)}*\widetilde\Sigma_{L'}^{(-L')}=2^{-\sum_i\mu_i(L,L')}\Big(\mathop*_{i=1}^nh_{i,L,L'}^{(-\bar \ell_i)_i}\Big)\ .
$$
We apply Lemma \ref{Lem4.5} and Lemma \ref{Lem3.8} (iii) to obtain that the functions $k_{L,L'}=\big(\mathop*_{i=1}^nh_{i,L,L'}^{(-\widetilde \ell_i)_i}\big)^{(\widetilde L)}$, with $\widetilde L\in\LL$ closest to $\bar L=L\wedge L'$, are uniformly bounded in every Schwartz norm. We obtain, with $N$ sufficiently large relative to $M$,
\beas
\int_{\R^d}\Big(\prod_{i=1}^n\big(1+2^{\ell'_i}|\x_{i}|\big)\Big)^M&\big|\Gamma_L^{(-L)}*\widetilde\Sigma_{L'}^{(-L')}(\x)\big|\,d\x\\
&=
2^{-\sum_i\mu_i(L,L')}\int_{\R^d}\Big(\prod_{i=1}^n\big(1+2^{\ell'_i}|\x_{i}|\big)\Big)^M\big|k_{L,L'}^{(-\widetilde L)}(\x)\big|\,d\x\\
&\leq C_N2^{\sum_i(q_i\widetilde\ell_i-\mu_i(L,L'))}\int_{\R^d}\frac{\big(\prod_{i=1}^n\big(1+2^{\ell'_i}|\x_{i}|\big)\big)^M}{\prod_{i=1}^n \big(1+\widehat N_i(2^{\widetilde\ell_1}\x_1,\dots,2^{\widetilde\ell_n}\x_n)\big)^{-N}}\,dx\\
&\leq C_N2^{-\sum_i(q_i\widetilde\ell_i-\mu_i(L,L'))}\int_{\R^d}\frac{\big(\prod_{i=1}^n\big(1+2^{\ell'_i}|\x_{i}|\big)\big)^M}{\prod_{i=1}^n \big(1+2^{\widetilde\ell_i}|\x_{i}|\big)^{-N'}}\,dx\ ,
\eeas
where \eqref{4.7} is used in last inequality and the $C_N$ subsumes the norms in $\|h\|_{(N)}\leq C_N\prod_{i=1}^n\|f_i\|_{(N')}$ of the single factors. We then repeat the proof of Lemma \ref{Lem5.3}.
\end{proof}

\subsection{$L^{1}$-equivalence with the square functions of tensor type}\label{Sec6.2}
\quad
\smallskip

\begin{theorem}\label{Thm6.4}
Let $\Psi=\{\Psi_L\}_{L\in\LL}$ be a family as in \eqref{3.4ee} of Section \ref{Sec4.2} constructed via tensor products, and  let $\Sigma=\{\Sigma_L\}_{L\in\LL}$ be a family as in Section \ref{Sec4.3} constructed via convolution. Let $S_\Psi$, $S'_\Sigma$ be the corresponding square functions \eqref{5.1} and \eqref{6.1}. Then, given $f\in \cS'(\R^d)$ the two conditions $S_\Psi f\in L^{1}$ and $S'_\Sigma f\in L^{1}$ are equivalent.
\end{theorem}

Having Theorem \ref{Thm6.2} at our disposal, it suffices to prove that, for appropriately chosen pairs of families $(\Psi,\Sigma)$ and $(\Psi',\Sigma')$  we have the implication $S_\Psi f\in L^{1}\Longrightarrow S'_\Sigma f\in L^{1}$ and the implication $S'_{\Sigma'} f\in L^{1}\Longrightarrow S_{\Psi'} f\in L^{1}$.

\begin{proposition}\label{Prop6.5}
Let $\big\{\sigma_{i,L}:i=1,\dots,n\,,\,L\in\LL\big\}\subset \cS(\R^d)$ be uniformly bounded in every Schwartz norm and assume that, if $L\in\LL_D$ with $D\subseteq\{1,\dots,n\}$, then 
\beas
\supp\widehat{\sigma_{i,L}}\subset\begin{cases} \big\{\xib: 2^{-3/4}\leq \widehat N_i(\xib)\le2^{3/4}\big\}&\text{ if }i\in D\\
\big\{\xib: \widehat N_i(\xib)\leq 2^{3/4}\big\}&\text{ if }i\not\in D\ \end{cases}.
\eeas
If $\Sigma_L^{(-L)}=\mathop{*}_{i=1}^n \sigma_{i,L}^{(-\ell_i)}$ then $\big\|\big(\sum_{L\in\LL}\big|f*\Sigma_L^{(-L)}\big|^2\big)^\half\big\|_1\leq C\big\|S_\Psi f\big\|_1$, with $C$ depending on the bounds for the $\sigma_{i,L}$ in an appropriate Schwartz norm, and this establishes  the implication $S_\Psi f\in L^{1}\Longrightarrow S'_\Sigma f\in L^{1}$.

\end{proposition}

\begin{proof}
We prove that the functions $\Sigma_L$ are uniformly bounded in $\SS(\R^{d})$ and satisfy the cancellation conditions in \eqref{5.9}. In fact Lemma \ref{Lem4.5} gives uniform bounds on Schwartz norms of the $\Sigma_L$. In order to verify the cancellation conditions, we need a detailed description of the support of $\FF[\Sigma_L^{(-L)}]=\prod_{i=1}^n \FF[\sigma_{i,L}^{(-\ell_i)}]$. 
We have
\beas
\supp \widehat{\Sigma_L^{(-L)}}&\subseteq\bigcap_{i\in D}\big\{\xib: 2^{-\ell_i-3/4}\leq \widehat N_i(\xib)\leq 2^{-\ell_i+3/4}\big\}
\cap\bigcap_{i\not\in D}\big\{\xib: \widehat N_i(\xib)\leq 2^{-\ell_i+3/4}\big\} 
\\
&\subset \big\{\xib:2^{\ell_i-3/4}\le|\xib_i|\leq 2^{\ell_i+3/4}\text{ if }i\in D\,,\,|\xib_i|\leq 2^{\ell_i+3/4}\text{ if }i\not\in D\big\}.
\eeas
Since $\ell_i\ge1$ for $i\in D$ by \eqref{3.6}, the support of $\widehat{\Sigma_L^{(-L)}}$ does not intersect the coordinate subspaces $\xib_i=\0$ for $i\in D$, and this gives the required cancellations.
\end{proof}

We now turn to the second implication. For each $i=1,\dots,n$,  fix functions $\ph_{i,\ell}$ on $\R^{d_i}$, so that $\widehat{\ph_{i,\ell}}$ is 1 on the set $\{\xib_i:|\xib_i|\leq \frac{1}{2}\}$ and is supported in $\{\xib_i:|\xib_i|\le2\}$. Then, for $\ell>0$, the functions $\psi_{i,\ell}$ in \eqref{4.1} have Fourier transforms supported in the set $\{\xib_i:1/4\le|\xib_i|\le2\}$.
From these functions we construct $\Psi_L$, for $L\in\LL$ according to \eqref{3.4ee}. 

\begin{lemma}\label{Lem6.6}
For every $L\in\LL$, there exists a family $\big\{\gamma_{i,L}:i=1,\dots,n\,,\,L\in\LL\big\}$ satisfying the hypothesis of Proposition \ref{Prop6.3} and such that
$
\Psi_L^{(-L)}=\mathop*_{i=1}^n\gamma_{i,L}^{(-\ell_i)_i}.
$
\end{lemma}

\begin{proof}
The Fourier transform of $\Psi_L^{(-L)}$ is supported, for $L\in\LL_D$, on the set
$$
E_L=\Big(\prod\nolimits_{i\in D}\{\xib_i:2^{\ell_i-2}\le|\xib_i|\le2^{\ell_i+1}\}\Big)\times\Big(\prod\nolimits_{i\not\in D}\{\xib_i:|\xib_i|\le2^{\ell_i+1}\}\Big)\ .
$$
We claim that $E_L\subset \big\{\xib:A2^{\ell_i}\leq \widehat N_i(\xib)\leq B2^{\ell_i}\big\}$ for appropriate constants $B>A>0$ and $i=1,\dots,n$. In fact, for $\xib\in E_L$, Lemma \ref{Lem3.8} (i) gives the upper bound $\widehat N_i(\xi)=\max_j|\xi_j|^\frac1{e(j,i)}\leq c \max_j2^{\frac{\ell_j}{e(j,i)}}\leq c2^{\ell_i+C\kappa}$. The lower bound is obvious if $i\in D$, because
$\widehat N_i(\xib)\ge|\xib_i|\ge 2^{\ell_i-2}$. If $i\not\in D$, we use Lemma \ref{Lem3.8} (i) again and take $\t\in F(D)$ at distance at most $\kappa$ from $L$. Then there is $k\in D$ such that $t_k=e(k,i)t_i$. Then
$
\widehat N_i(\xib)\ge|\xib_k|^\frac1{e(k,i)}\ge 2^{\frac{\ell_k-2}{e(k,i)}}\ge2^{\frac{t_k-\kappa-2}{e(k,i)}}\ge2^{t_i-\frac{\kappa+2}{e(k,i)}}\ge2^{\ell_i-C'}.
$
For $i=1,\dots,n$, fix a smooth function $\eta_i$, equal to 1 on the set $\big\{\xi:A\leq \widehat N_i(\xi)\leq B\big\}$ and supported in $\big\{\xi:A/2\leq \widehat N_i(\xi)\leq 2B\big\}$. Then define functions $\gamma_{i,L}$ by $\widehat{\gamma_{i,L}}=\widehat{\psi_{i,\ell_i}}\eta_i$. Recalling that $\widehat{\psi_{i,\ell_i}}$ only depends on the variable $\xi_i$, it follows that $\widehat{\gamma_{i,L}}(2^{-\ell_i}\cdot_i\xi)=\widehat{\psi_{i,\ell_i}}(2^{-\ell_i}\xi_i)\eta_i(2^{-\ell_i}\cdot_i\xi)$, and therefore
\beas
\prod_{i=1}^n\widehat{\gamma_{i,L}^{(-\ell_i)_i}}(\xi)&=\Big(\prod_{i=1}^n\widehat{\psi_{i,\ell_i}}(2^{-\ell_i}\xi_i)\Big)\Big(\prod_{i=1}^n\eta_i(2^{-\ell_i}\cdot_i\xi)\Big)
=\widehat{\Psi_L^{(-L)}}(\xi)\ .
\eeas
Notice that $\supp \gamma_{i,L}$ contains $\0$ only if $L=0$, so that the cancellations required by Proposition~\ref{Prop6.3} are satisfied. 
\end{proof}

It now follows from Proposition \ref{Prop6.3} that if $\big\{\Psi_L\}_{L\in\LL}$ as in Lemma \ref{Lem6.6} and if $\big\{\Sigma_L\}_{L\in\LL}$ as in in Section \ref{Sec4.3} then  $\big\|S_\Psi f\big\|_1\leq C\big\|S'_\Sigma f\big\|_1$. This gives the second implication $S'_{\Sigma'} f\in L^{1}\Longrightarrow S_{\Psi'} f\in L^{1}$, and this completes the proof of Theorem \ref{Thm6.4}.

\subsection{Square functions of convolution type with parameters varying in $\N^{n}$}\label{Sec6.3}
\quad
\medskip

Keeping the same kind of family $\big\{\sigma_{i,\ell}:i=1,\dots,n\,,\,\ell\in\bN\big\}\subset\cS(\bR^d)$ appearing in \eqref{4.5}, we define the following square function: 
\be\label{6.3}
S''_\sigma f=\Big(\sum_{L\in\N^{n}}\big|f*\big(\mathop*_{i=1}^n\sigma_{i,\ell_i}^{(-\ell_i)_i}\big)\big|^2\Big)^\half\ .
\ee

In this section we first prove independence of the condition $S''_\sigma f\in L^{1}(\R^d)$ from the choice of the family $\big\{\sigma_{i,\ell}\big\}$, and then its equivalence with the same condition on the previously considered square functions.
In analogy with the arguments for the previously considered square functions, we prove the following.

\begin{proposition}\label{Prop6.8}
Let $\{\eta_{i,L}\}_{i=1,\dots,n\,,\,L\in\N^{n}}$ be a family of functions in $\cS(\R^d)$, uniformly bounded in every Schwartz norm and such that
$
\int_{\R^d}\eta_{i,L}(x)\,dx=0\ .
$
for all pairs $(i,L)$ for which $\ell_i>0$.
Then
\be\label{6.4}
\Big\|\Big(\sum_{L\in\N^{n}}\big|f*\big(\mathop*_{i=1}^n\eta_{i,L}^{(-\ell_i)_i}\big)\big|^2\Big)^\half\Big\|_1\leq C \big\|S''_\sigma f\big\|_1\ ,
\ee
with $C$ depending on the uniform bound on the $\eta_{i,L}$ in an appropriate Schwartz norm.
\end{proposition}

\begin{proof}
By Lemma \ref{Lem4.5}, given $L\in\N^{n}$, we can express any convolution $\mathop*_{i=1}^nf_i^{(-\ell_i)_i}$ as $h_L^{(-\t_L)}$ with $\t_L\in\Gamma(\mbE)$ and $h_L$ satisfying the inequalities of Lemma \ref{Lem4.5}. By Lemma \ref{Lem3.8} (ii), there is a map $\N^{n}\ni L\longrightarrow \widetilde L\in \LL$ so that $\widetilde L$ is at uniformly bounded distance from $\t_L$ and $\mathop*_{i=1}^nf_i^{(-\ell_i)_i}=k_L^{(-\widetilde L)}$ with $\|k_{L}\|_{(N)}\leq C_N\prod_{i=1}^n\|f_i\|_{(N')}$. Applying this procedure with $f_i$ replaced by $\eta_{i,\ell_i}$, respectively $\sigma_{i,\ell'_i}, \widetilde\sigma_{i,\ell'_i}$, we denote by $\eta^*_L$, respectively $\sigma^*_{L'}, \widetilde\sigma^*_{L'}$, the functions, uniformly bounded in every Schwartz norm, such that
\bea\label{6.5}
\mathop*_{i=1}^n\eta_{i,\ell_i}^{(-\ell_i)_i}=\eta^*_L{}^{(-\widetilde L)}\ ,\qquad
\mathop*_{i=1}^n\sigma_{i,\ell'_i}^{(-\ell'_i)_i}=\sigma^*_{L'}{}^{(-\widetilde{L'})}\ ,\qquad
\mathop*_{i=1}^n\widetilde\sigma_{i,\ell'_i}^{(-\ell'_i)_i}=\widetilde\sigma^*_{L'}{}^{(-\widetilde{L'})}\ .
\eea
Using $\del_\0=\sum\nolimits_{L\in\N^{n}}\Big(\mathop{*}_{i=1}^n\sigma_{i,\ell_i}^{(-\ell_i)_i}\Big)*\Big(\mathop{*}_{i=1}^n\widetilde\sigma_{i,\ell_i}^{(-\ell_i)_i}\Big)$
 write $f*\eta^*_L{}^{(-\widetilde L)}=\sum_{L'\in\N^{n}}\big(f*\sigma^*_{L'}{}^{(-\widetilde{L'})}\big)*\big(\eta^*_L{}^{(-\widetilde L)}*\widetilde\sigma^*_{L'}{}^{(-\widetilde{L'})}\big)$. We need the following analogue of Lemma \ref{Lem5.3}.

\begin{lemma}\label{Lem6.9}
Let $L^\#=(L\wedge L')\,\widetilde{}=\widetilde L\wedge\widetilde{L'}=(\ell^\#_1,\dots,\ell^\#_n)$. Then $L^\#\in\LL$ and, given $M>0$, there is $\eps>0$ such that the inequality $\int_{\R^d}\Big(\prod_{i=1}^n\big(1+2^{\widetilde\ell'_i}|x_{i}|\big)\Big)^M\big|\eta^*_L{}^{(-\widetilde L)}*\widetilde\sigma^*_{L'}{}^{(-\widetilde{L'})}(x)\big|\,dx\leq C_M2^{-\eps|L-L'|}$ holds uniformly in $L,L'\in\N^{n}$, provided the order $m$ of cancellations of the $\widetilde\sigma_{i,\ell'_i}$ is large enough relative to $M$.
\end{lemma}

\begin{proof}
The equality $(L\wedge L')\,\widetilde{}=\widetilde L\wedge\widetilde{L'}$ is obvious.
Let $\bar L=L\wedge L'=(\bar\ell_{1}, \ldots, \bar\ell_{n})$. Then for each $i$, 
$\eta_{i,L}^{(-\ell_i)_i}*\widetilde\sigma_{i,\ell'_i}^{(-\ell'_i)_i}=2^{-\gamma_i|\ell'_i-\ell_i|}\zeta_{i,L,L'}^{(-\bar\ell_i)_i}$, 
with $\zeta_{i,L,L'}$ uniformly bounded in every Schwartz norm with respect to $i,L,L'$. According to Lemma \ref{Lem5.2}, 
$
\gamma_i=\begin{cases}\bar\la_im&\text{ if }\ell'_i>\ell_i\\ \bar\la_i&\text{ if }\ell'_i\leq \ell_i\ .\end{cases}
$
Also from Lemma \ref{Lem4.5} (i),
$
\eta^*_L{}^{(-\widetilde L)}*\widetilde\sigma^*_{L'}{}^{(-\widetilde{L'})}=2^{-\sum_i\gamma_i|\ell'_i-\ell_i|}\zeta^*_{L,L'}{}^{(-L^\#)}$, with the $\zeta^*_{L,L'}$ uniformly bounded in every Schwartz norm.
For any $M$, we then have $\big|\eta^*_L{}^{(-\widetilde L)}*\widetilde\sigma^*_{L'}{}^{(-\widetilde{L'})}(x)\big|\leq C_M2^{\sum_iq_i\ell_i^\#-\sum_i\gamma_i|\ell'_i-\ell_i|}\prod_{i=1}^n\big(1+2^{\ell_i^\#}|x_{i}|\big)^{-2M}$, and 
\beas
\int_{\R^d}\Big(\prod_{i=1}^n\big(1+2^{\widetilde\ell'_i}|x_{i}|\big)\Big)^M&\big|\eta^*_L{}^{(-\widetilde L)}*\widetilde\sigma^*_{L'}{}^{(-\widetilde{L'})}(x)\big|\,dx\\
&\leq C_M 2^{\sum_iq_i\ell_i^\#-\sum_i\gamma_i|\ell'_i-\ell_i|}\prod_{i=1}^n\int_{\R^{d_i}}\frac{\big(1+2^{\widetilde\ell'_i}|x_{i}|\big)^M}{\big(1+2^{\ell_i^\#}|x_{i}|\big)^{2M}}\,dx_{i}\\
&=C_M 2^{-\sum_i\gamma_i|\ell'_i-\ell_i|}\prod_{i=1}^n\int_{\R^{d_i}}\frac{\big(1+2^{\widetilde\ell'_i-\ell_i^\#}|x_{i}|\big)^M}{\big(1+|x_{i}|\big)^{2M}}\,dx_{i}\\
&\leq C'_M 2^{\sum_i\big(M(\widetilde\ell'_i-\ell_i^\#)-\gamma_i|\ell'_i-\ell_i|\big)}.
\eeas
It suffices to prove that, for some $\eps>0$ and $C''_M$ independent of $\ell,\ell'$ we have
\be\label{total}
\sum\nolimits_i\big(M(\widetilde\ell'_i-\ell_i^\#)-\gamma_i|\ell'_i-\ell_i|\big)\leq C''_M- \eps\sum\nolimits_i|\ell'_i-\ell_i|.
\ee

If $\widetilde\ell'_i\leq \widetilde\ell_i$, the $i$-th term in the sum gives the inequality $M(\widetilde\ell'_i-\ell_i^\#)-\gamma_i|\ell'_i-\ell_i|=-\gamma_i|\ell'_i-\ell_i|$. In order to estimate the terms for which $\widetilde\ell'_i> \widetilde\ell_i$, i.e., for which $\ell_i^\#=\widetilde \ell_i$, it is necessary to group terms appropriately.\footnote{The following argument is the same used in \cite{MR3862599} Section 3.1 to obtain the ``dominance principle'' allowing to associate a marked partition to each point in the unit ball in the $x$-space. } 
Let $\t_L\in\Gamma(\mathbf E)$ have entries $t_{L,i}=\min_{j=1,\dots,n}e(i,j)\ell_j$. Suppose that, for some $i$, $t_{L,i}=e(i,k)\ell_k<\ell_i$ (with $k\ne i$).
Using the basic assumptions, it is easy to verify that $t_{L,k}=\ell_k$. 
If we select the indices $k_1,\dots,k_s$ for which $t_{L,k_r}=\ell_{k_r}$, we can then partition $\{1,\dots,n\}$ into subsets $A_r$ with the property that $t_{L,i}=e(i,k_r)\ell_{k_r}$ for all $i\in A_r$ (this partition may not be unique if some equality $e(i,j)\ell_j=e(i,j')\ell_{j'}$ occurs).
Let $A'_r$ be the set of $i\in A_r$ such that $\widetilde\ell'_i> \widetilde\ell_i$, and
$
S'_r=\sum_{i\in A'_r}\big(M(\widetilde\ell'_i-\widetilde\ell_i)-\gamma_i|\ell'_i-\ell_i|\big)\ .
$
For $i\in A'_r$ we have
\beas
\widetilde\ell'_i-\widetilde\ell_i\leq t_{L',i}-t_{L,i}+1\leq e(i,k_r)\big(t_{L',k_r}-\ell_{k_r}\big)+1+\frac{e(i,k_r)}2\leq e(i,k_r)(\ell'_{k_r}-\ell_{k_r})+C\ ,
\eeas
with $C$ depending only on $\mathbf E$. Therefore
$
S'_r\leq \sum_{i\in A'_r}\big(Me(i,k_r)(\ell'_{k_r}-\ell_{k_r})+MC-\gamma_i|\ell'_i-\ell_i|\big).
$
The terms $Me(i,k_r)(\ell'_{k_r}-\ell_{k_r})$ can be neglected if $\ell'_{k_r}\leq \ell_{k_r}$.
If $\ell'_{k_r}> \ell_{k_r}$, then the constant $\gamma_{k_r}$ can be made $>M\sum_{i\in A'_r}\big(e(i,k_r)$ by choosing $m$ large enough. Summing over $r$, we obtain \eqref{total} as required, and this proves Lemma \ref{Lem6.9}.
\end{proof}

We now continue with the proof of Proposition \ref{Prop6.8}. In analogy with \eqref{5.4}, we define 
\beas
\widetilde\omega_{M,L}(x)=\sup_{x'\in\R^d}\frac{\big|f*\sigma^*_L{}^{(-\widetilde L)}(x-x')\big|}{\prod_i\big(1+2^{\widetilde\ell_i}|x'_i|\big)^M}
\eeas
for $M>0$, $L\in\N^{n}$, and $\sigma^{*}_{L}$ as defined in \eqref{6.5},
We then have
\beas
\big|f*\eta_L^*{}^{(-\widetilde L)}(x)\big|&\leq \!\!\sum_{L'\in\N^{n}}\int \big|f*\sigma^*_L{}^{(-\widetilde L)}(x-x')\big|\big|\eta^*_L{}^{(-\widetilde L)}*\widetilde\sigma^*_{L'}{}^{(-\widetilde{L'})}(x') \big|\,dx'
\leq \!\sum_{L'\in\N^{n}}2^{-\eps|L-L'|}\widetilde\omega_{M,L'}(x)
\eeas
and so $\Big(\sum_{L\in\N^{n}}\big|f*\eta^*_L{}^{(-\widetilde L)}(x)\big|^2\Big)^\half\leq \Big(\sum_{L\in\N^{n}}\Big(\sum_{L'\in\N^{n}}2^{-\eps|L-L'|}\widetilde\omega_{M,L'}\Big)^2\Big)^\half
\leq C \Big(\sum_{L\in\N^{n}}\widetilde\omega_{M,L}^2\Big)^\half$ 
by Young's inequality on $\bZ^n$. We need the following analogue of Lemma \ref{Lem5.6} with $(\Psi^\#_{L,L'})^{(-L')}$ replaced by $\sigma^{**}_{L,L'}{}^{(-\widetilde L')}=\Big(\mathop*_{\{i:\ell'_i>\ell_i\}}\sigma_{i,\ell'_i}^{(-\ell'_i)_i}\Big)*\Big(\mathop*_{\{i:\ell'_i=\ell_i\}}\rho_{i,\ell'_i}^{(-\ell'_i)_i}\Big)$, and similarly for $(\widetilde\Psi^\#_{L,L'})^{(-L')}$.

\begin{lemma}\label{Lem6.10}
Let $f\in \cS'(\R^d)$ and $M>0$. Then, for every positive $\mu$,
\beas
\widetilde\omega_{M,L}(x)\leq C_{M,\mu} \sum_{L'\ge L}2^{-\mu|L-L'|}\big|f*\sigma^*_{L'}{}^{(-\widetilde L')}\big|*\chi_M^{(-\widetilde L)}(x)\ .
\eeas
\end{lemma}
We omit the proof, which follows the same lines and is actually simpler.
The proof of Proposition \ref{Prop6.8} now proceeds as in Theorem \ref{Thm5.1} with the obvious modifications.
\end{proof}

\begin{theorem}\label{Thm6.11}
Let $\{\sigma_{i,\ell}\}_{i\in\{1,\dots,n\}\,,\,\ell\in\bN}$ be a family as in Section \ref{Sec4.3}, $S'_\Sigma$ as in \eqref{6.1} and $S''_\sigma$ as in \eqref{6.3}. Then, given $f\in \cS'(\R^d)$ the two conditions $S'_\Sigma f\in L^{1}$ and $S''_\sigma f\in L^{1}$ are equivalent.
\end{theorem}

\begin{proof}
By Proposition \ref{Prop6.8}, it suffices to prove the equivalence for an appropriately chosen family~$\{\sigma_{i,\ell}\}$.
We then take $\sigma_{i,\ell}$ with Fourier transform supported in $\big\{\xi:\widehat N_i(\xi)\leq 1+\del\big\}$ if $\ell=0$ and in $\big\{\xi:\frac{1}{2}-\del\leq \widehat N_i(\xi)\leq 1+\del\big\}$ if $\ell>0$, with $\del>0$ sufficiently small.
It follows from Lemma~\ref{Lem3.8}~(iv) that $\sigma^*_L=0$ if $|L-L'|>\kappa+1$ for all $L'\in\LL$ and therefore, for $L\in\LL_D$,
\beas
\Sigma_L^{(-L)}&=\Big(\mathop*_{i\in D}\sigma_{i,\ell_i}^{(-\ell_i)_i}\Big)*\Big(\mathop*_{i\not\in D}\rho_{i,\ell_i}^{(-\ell_i)_i}\Big)
=\Big(\mathop*_{i\in D}\sigma_{i,\ell_i}^{(-\ell_i)_i}\Big)*\Big(\mathop*_{i\not\in D}\sum_{\ell'_i\leq \ell_i}\sigma_{i,\ell'_i}^{(-\ell'_i)_i}\Big)
=\sum_{|L'-L|\leq A}\sigma^*_{L'}{}^{(-L')} .
\eeas

Then the conclusion is a consequence of the following two remarks.
\begin{enumerate}[1.]
\item The square function $S''_\sigma$ can be regarded as the square function associated to a family $\{\gamma_{i,L}\}$ as in Proposition \ref{Prop6.3}.
\item the square function $S'_\Sigma$ can be regarded as a finite sum of square functions in the left-hand side of \eqref{6.4}, each associated to a family $\{\eta_{i,L}\}$ as in Proposition \ref{Prop6.8}. \qedhere
\end{enumerate}
\end{proof}

\section{Dyadic cubes and rectangles in $\R^d_\x$\\ and square functions of Plancherel-P\'olya type}\label{Sec7.1}

We shall need the definition of dyadic cubes and rectangles in our multi-norm context. Since several families of dilations are involved, the definition requires some care. We will only focus our attention on cubes and rectangles at scales smaller than unit scale.

\subsection{Dyadic cubes}\label{Sec8.1s}\quad
  
\smallskip

The dilations $\delta_{t}$ from \eqref{1.1} act on each factor of $\R^{d}=\prod_{i=1}^{n}\R^{d_{i}}$; if $\x_{i}=(x_{h})_{h\in E_{i}}$ then $\delta_{t}(\x_{i})=\big(t^{\lambda_{h}}x_{h}\big)_{h\in E_{i}}$. 
We define dyadic cubes $\bfQ_{i}^{\ell_{i}}\subseteq\R^{d_{i}}$ at scale $2^{-\ell_{i}}$ for $\ell_{i}\geq 0$ which respect this family of dilations. This is done by induction on $\ell_{i}$. For $\ell_{i}=0$ the dyadic cubes at scale $2^{0}=1$ are the unit cubes $\bfQ_{i,\p}^{0}=\big\{(x_{h})_{h\in E_{i}}:p_h\leq x_{h}<p_h+1\big\}$ where $\p=(p_{h})_{h\in E_{i}}\in\Z^{d_{i}}$. We denote this family by $\D_{i}^{0}$. The cubes at scale $2^{-\ell_{i}}$ should then have side lengths comparable to $\big\{2^{-\ell_{i}\lambda_{h}}:h\in E_{i}\big\}$ and be comparable to translates of $\big\{(x_{h})_{h\in E_{i}}:0\leq x_{h} < 2^{-\ell_{i}\lambda_{h}}\big\}$.  
We cannot simply divide the sides of dyadic cubes by the same number at each stage unless the exponents $\lambda_{h}$ are rational multiples of each other. 

To deal with this problem, for $h\in E_{i}$ we say that a sequence $\big\{\mu(\ell_{i},h)\in \N, \ell_{i}\geq 0\big\}$ is  {\it admissible } if it is non-decreasing, if $\mu(0,h)=0$, and if there exists $a>0$ so that 
\bea\label{mu ell}
2^{\ell_{i}\lambda_{h}-a}&\leq 2^{\mu(\ell_{i},h)}\leq2^{\ell_{i}\lambda_{h}+a},&&\text{ or equivalently } & |\mu(\ell_{i},h)-\ell_{i}\la_h|&\leq a.
\eea
Given an admissible sequence, suppose that the family of cubes $\D_{i}^{\ell_{i}-1}=\big\{\bfQ_{i,\p}^{\ell_{i}-1}\big\}$ has been defined so that the $h^{th}$ side of $\bfQ_{i,\p}^{\ell_{i}-1}$ has length $2^{-\mu(\ell_{i}-1,h)}$. Then we divide the $h^{th}$-side of each cube  into $2^{\mu(\ell_{i},h)-\mu(\ell_{i}-1,h)}$ equal pieces. The result of this division is the family $\bD_i^{\ell_{i}}=\big\{\bfQ_{i,\p}^{\ell_{i}}:\p\in \Z^{d_{i}}\big\}$ where 

\bea\label{3.11}
 \bfQ_{i,\p}^{\ell_{i}}=\Big\{(x_{h})_{h\in E_{i}}:2^{-\mu(\ell_{i},h)}p_{h}\leq x_{h}<2^{-\mu(\ell_{i},h)}(p_{h}+1)\Big\}.
\eea
The side lengths of these rectangular sets are not equal, but we call them {\it dyadic cubes in $\R^{d_{i}}$} because their scale depends on a single dyadic parameter $2^{-\ell_i}$ compatible with the given dilations. The following properties of dyadic cubes are then clear.
\begin{enumerate}[(1)]
\item\label{Dy1}
$\R^{d_{i}}=\bigcup_{\p\in \Z^{d_{i}}}{\bfQ_{i,\p}^{\ell_{i}}}$ is a disjoint union.

\smallskip
\item \label{Dy2}
The $h^{th}$ side of $\bfQ_{i,\p}^{\ell_{i}}$ has length $2^{-\mu(\ell_{i},h)}$.

\smallskip
\item\label{Dy3}
If $0\leq m \leq \ell_{i}$ and if $\bfQ_{i,\p}^{\ell_{i}}\in \bD_i^{\ell_{i}}$ there is a unique dyadic cube $\bfQ_{i,\p'}^{\ell_{i}-m}\in \bD_i^{\ell_{i}-m}$ at scale $2^{m-\ell_{i}}$ containing $\bfQ_{i,\p}^{\ell_{i}}$. We use the notation
\bea\label{ancestor}
2^{m}\bfQ_{i,\p}^{\ell_{i}}=\bfQ_{i,\p'}^{\ell_{i}-m}.
\eea
It may happen that $2^{m}\bfQ_{i,\p_i}^{\ell_{i}}$ denotes the same set for more than one value of $m$ since the coefficients $\mu(\ell_{i},h)$ are not assumed to be strictly increasing in $\ell_{i}$.

\item\label{Dy5}
Each dyadic cube at scale $2^{-\ell_{i}}$ is a disjoint union of $\prod_{h\in E_{i}}2^{\mu(\ell_{i}+1,h)-\mu(\ell_{i},h)}$ cubes in $\bD_i^{\ell_{i}+1}$.

\smallskip
\item\label{Dy6}
If $\bfQ_{i,\p}^{\ell_{i}}$ and $\bfQ_{i,\p'}^{\ell_{i}'}$ are cubes with 
$\bfQ_{i,\p}\cap\bfQ_{i,\p'}^{\ell_{i}'}\neq\emptyset$ then 
$\begin{cases}
\bfQ_{i,\p}^{\ell_{i}}=\bfQ_{i,\p'}^{\ell_{i}'}&\text{if $\ell_{i}=\ell_{i}'$}\\
\bfQ_{i,\p}^{\ell_{i}}\subseteq\bfQ_{i,\p'}^{\ell_{i}'}&\text{if $\ell_{i}'<\ell_{i}$}
\end{cases}$.
\end{enumerate}
\label{Dypage}
\smallskip

\begin{remarks}{\rm \quad
\begin{enumerate}[1.]
\item
In this definition of dyadic cubes, the precise choice of the admissible sequence $\mu(\ell_{i},h)$ is not important. The most natural choice is to take $\mu(\ell_{i},h)=\lfloor \ell_{i}\lambda_{h}\rfloor$, with $\lfloor x\rfloor$ denoting the largest integer less than or equal to $x\in \R$, but other choices are used, for example, in the proof of Proposition \ref{R=Q}.
\smallskip

\item This construction of dyadic cubes can be done more generally in any vector space with a family of dilations. We do this in Section \ref{Sec7.2a} below.
\end{enumerate}}
\end{remarks}
 
\subsection{Dyadic rectangles in the product $\R^{d_{1}}\times\cdots\times\R^{d_{n}}$ }\label{Sec7.2}
\quad
\smallskip

{\it Dyadic rectangles} in $\R^{d}=\R^{d_{1}}\times\cdots\times\R^{d_{n}}$ are by definition the product of dyadic cubes in each factor. Thus if $L=(\ell_1,\dots,\ell_n)\in\N^{n}$ and $\bfP=(\p_{1},\dots,\p_{n})\in\Z^{d_1}\times\cdots\times\Z^{d_n}$ then  $\bfR_{\bfP}^{L}=\prod\nolimits_{i=1}^{n}\bfQ_{i,\p_{i}}^{\ell_{i}}$ is a typical dyadic rectangle at scale $2^{-L}$. We use the following notation.
\bea\label{4.6xyz}
\D^{L}&=\Big\{\bfR_{P}^{L}:P\in\Z^{d}\Big\},&&\text{the set of all dyadic rectangles at scale $2^{-L}$;}\\
\D\phantom{^{L}}&=\bigcup\nolimits_{L\in \N^{n}}\D^{L},&&\text{the set of all dyadic rectangles at scales $\le1$}; \\
\D_\LL&=\bigcup\nolimits_{L\in \LL}\bD^{L},&&\text{the set of all dyadic rectangles at scale $2^{-L}$ with $L\in \LL$};\\
\D_{S}&=\bigcup\nolimits_{L\in \LL_{S}}\bD^{L},&&\text{the set of all dyadic rectangles at scale $2^{-L}$ with $L\in \LL_S$}.
\eea
Recall from  the end of Section \ref{Sec2.1}} that $\SS^*_\bfE=\SS_{\bfE}\cup \{\emptyset\}$. Then $\D_\LL=\bigcup\nolimits_{S\in\SS^*_\bfE}\bD_S$ where  $\D_\emptyset=\D^0$.

\subsection{Dyadic rectangles in  the product $\R^{A_{1}}\times\cdots\times\R^{A_{s}}$}\label{Sec7.2a}
\quad
\smallskip

Let $S=\big\{(A_{1},k_{1}), \ldots, (A_{s},k_{s})\big\} \in\SS_{\bfE}$. If $L=(\ell_{1}, \ldots, \ell_{n})\in\LL_S$ and $\bfR=\prod_{i=1}^{n}\bfQ_{i}^{\ell_{i}}\in\D_S$, we have the coarser decomposition $\bfR=\prod_{r=1}^{s}\bfR_{A_{r}}$ where $\bfR_{A_{r}}=\prod_{i\in A_{r}}\bfQ_{i}^{\ell_{i}}\subseteq\R^{A_{r}}$. We show that the {\it dyadic rectangle} $\bfR_{A_{r}}$ is actually the same as a {\it dyadic cube} relative to an appropriate family of dilations in  the subspace $\R^{A_{r}}$. Recall from \eqref{2.5a} $\delta_{t}^{i}(\x)=\big(t^{\frac{1}{e(i,1)}}\x_{1}, \ldots, t^{\frac{1}{e(i,n)}}\x_{n}\big)$.  We introduce the $s$-parameter dilations $\delta_{\t}^{S}$ on $\R^{d}$ by setting

\bea\label{3.11.6}
\del^S_\t\x&=\big(\del^{k_1}_{t_1}\x_{A_1},\dots,\del^{k_s}_{t_s}\x_{A_s}\big)&&\text{ where }&\delta_{t_r}^{k_{r}}(\x_{A_{r}})&=\big(t_r^{\frac{1}{e(k_{r},i)}}\x_{i}\big)_{i\in A_{r}}.
\eea

\begin{proposition}\label{R=Q}
Let $S=\big\{(A_{1},k_{1}), \ldots, (A_{s},k_{s})\big\} \in\SS_{\bfE}$, let $L=(\ell_{1}, \ldots, \ell_{n})\in\LL_S$ and let $\bfR=\prod_{i=1}^{n}\bfQ_{i,\p_i}^{\ell_{i}}\in\D_S$, where each $\bfQ_i^{\ell_i}$ is constructed according to \eqref{3.11} using the admissible sequences $\big\{\mu(\ell_{i},h)\in \N, \ell_{i}\geq 0\big\}$ for $i\in E_i$. Then for each $1 \leq r\leq s$ the dyadic rectangle $\bfR_{A_{r}}=\prod_{i\in A_{r}}\bfQ_{i,\p_i}^{\ell_{i}}$ is a dyadic cube in $\R^{A_{r}}$ relative the to dilations $\delta_{t_r}^{k_{r}}$ restricted to $\R^{A_{r}}$, constructed as in \eqref{3.11} using the admissible sequences $\Big\{\nu(\ell_{k_{r}},h)=\mu\Big(\Big\lfloor\frac{\ell_{k_r}}{e(k_r,i)}\Big\rfloor,h\Big):\ell_{k_{r}}\geq0\Big\}$ for  $h\in E_i$, $i\in A_r$.
\end{proposition}

\begin{proof}
On the one hand 
\beas
\bfR_{A_{r}}&=\prod_{i\in A_{r}}\bfQ_{i}^{\ell_{i}}
=
\prod_{i\in A_{r}}\Big\{(x_{h})_{h\in E_{i}}\in\R^{d_{i}}:2^{-\mu(\ell_{i},h)}p_{h}\leq x_{h}<2^{-\mu(\ell_{i},h)}(p_{h}+1)\text{ for all }h\in E_i\Big\}.
\eeas
Since $L\in\LL_{S}$  we have $\ell_{i}=\left\lfloor\frac{\ell_{k_{r}}}{e(k_{r},i)}\right\rfloor$ for  all $i\in A_{r}$ and so
\be\label{8.4yy}
\bfR_{A_{r}}=
\prod_{i\in A_{r}}\prod_{h\in E_i}\Big\{x_{h}:2^{-\mu\left(\lfloor\frac{\ell_{k_r}}{e(k_r,i)}\rfloor,h\right)}p_{h}\leq x_{h}<2^{-\mu\left(\lfloor\frac{\ell_{k_r}}{e(k_r,i)}\rfloor,h\right)}(p_{h}+1)\Big\}.
\ee
Thus scales on $\R^{A_{r}}$ are indexed by the integer $\ell_{k_{r}}$. 

On the other hand, as in Section \ref{Sec8.1s}, we can use the family of dilations $\delta_{t}^{k_{r}}$  to construct dyadic cubes in $\R^{A_{r}}$. 
If $\big\{\nu(\ell_{k_{r}},h):\ell_{k_{r}}\geq0\big\}$  is an admissible sequence for $h\in \bigcup_{i\in A_{r}}E_{i}$, then a dyadic cube in $\R^{A_{r}}$ relative to these dilations  has the form
\bea\label{8.5yy}
\bfQ_{A_{r}}^{\ell_{k_{r}}}=\prod_{i\in A_r}\prod_{h\in E_i}\big\{x_{h}:2^{-\nu(\ell_{k_{r}},h)}p_{h}\leq x_{h}<2^{-\nu(\ell_{k_{r}},h)}(p_{h}+1)\big\}\quad\text{where $\p\in \Z^{A_{r}}$}.
\eea
Comparing equations \eqref{8.4yy} and \eqref{8.5yy}, we see that if $\nu(\ell_{k_{r}},h)=\mu\left(\left\lfloor\frac{\ell_{k_r}}{e(k_r,i)}\right\rfloor,h\right)$ then $\bfR_{A_{r}}=\bfQ_{A_{r}}^{\ell_{k_{r}}}$. It remains to show that with this definition, $\big\{\nu(\ell_{k_{r}},h):\ell_{k_{r}}\geq0\big\}$ is an admissible sequence. But 
\beas
\Big|\nu(\ell_{k_{r}},h)-\frac{\ell_{k_r}\la_h}{e(k_r,i)}\Big|
&=
\Big|\mu_{\lfloor\frac{\ell_{k_r}}{e(k_r,i)}\rfloor,h}\!-\!\frac{\ell_{k_r}\la_h}{e(k_r,i)}\Big|
\le \Big|\mu_{\lfloor\frac{\ell_{k_r}}{e(k_r,i)}\rfloor,h}\!-\!\lfloor\frac{\ell_{k_r}}{e(k_r,i)}\rfloor\la_h\Big|+ \Big|\lfloor\frac{\ell_{k_r}}{e(k_r,i)}\rfloor\la_h-\frac{\ell_{k_r}\la_h}{e(k_r,i)}\Big|\\
&\le a+\la_h
\leq a+\sup\{\lambda_{h}:1\leq h \leq d\},
\eeas
and this shows that $\{\nu(\ell_{k_{r}},h):\text{$h\in E_{i}$ and $i\in A_{r}$}\}$ is an admissible sequence for $\R^{\A_{r}}$ if $L\in \LL_{S}$. This completes the proof.
\end{proof}

The last statement justifies the following alternative notation for $\bfR_{A_r}$.

\begin{definition}\label{8.5zz}
For $L\in\LL_S$, $\bfR=\prod_{i=1}^{n}\bfQ_{i}^{\ell_{i}}\in\D_L$ and $r=1,\dots,s$, we define $\bfQ_{A_r}^{\ell_{k_r}}=\prod_{i\in A_r}\bfQ_{i}^{\ell_{i}}$.
\end{definition}

This allows  us to write $\bfR=\prod_{r=1}^s\bfQ_{A_r}^{\ell_{k_r}}$, a notation that only refers to the coarser decomposition 
$\R^d=\R^{A_{1}}\times\cdots\times\R^{A_{s}}$. The basic properties of dyadic cubes are satisfies by the $\bfQ_{A_r}^{\ell_{k_r}}$. In particular, two such cubes are either disjoint or one contained in the other.

\subsection{Enlargements of dyadic rectangles}\label{Sec7.2b}
\quad
\smallskip

In this section we refer to a fixed marked partition $S=\big\{(A_{1},k_{1}), \ldots, (A_{s},k_{s})\big\} \in\SS_{\bfE}$, with the associated decomposition 
$\R^d=\R^{A_{1}}\times\cdots\times\R^{A_{s}}$ and the dilations $\del^S_{\t}$ in \eqref{3.11.6}.  We will use two kinds of ``enlargements'' of dyadic  rectangles $\bfR=\prod_{r=1}^s\bfQ_{A_r}$. Though we will ultimately be only interested with rectangles of less that unit size, it is not relevant to impose such a restriction at this stage.

\subsubsection{Dyadic enlargements}\label{Sec7.4.1}\quad

\smallskip

 If $\bfQ_{i,\p}^{\ell_{i}}$ is a dyadic cube and $m\in\N$, the dyadic enlargement $2^m\bfQ_{i,\p}^{\ell_{i}}$ was defined in \eqref{ancestor}. If $\m=(m_{1}, \ldots, m_{n})\in \N^{n}$ and $\bfR=\prod_{i=1}^{n}\bfQ_{\i,\p}^{\ell_{i}}$, then $2^{\m}\bfR=\prod_{i=1}^{n}2^{m_{i}}\bfQ_{i,\p}^{\ell_{i}}$. As in Section \ref{Sec7.2a}, if $S=\big\{(A_{1},k_{1}), \ldots, (A_{s},k_{s})\big\}\in\SS_{\bfE}$ and if $\bfQ_{A_{r}}=\prod_{i\in A_{r}}\bfQ_{i,\p}^{\ell_{i}}$ then $2^{m}\bfQ_{A_{r}}=\prod_{i\in A_{r}}2^{m}\bfQ_{i,\p}^{\ell_{i}}$  for $m\in \N$. If $\m=(m_{1}, \ldots, m_{s})\in\N^{s}$ then $2^{\m}\prod_{r=1}^{s}\bfQ_{A_{r}}=\prod_{r=1}^{s}2^{m_{r}}\bfQ_{A_{r}}$. 
If $B\subset\{1,\dots,s\}$ and $\m_B\in\N^B$, we use the notation $2^{\m_B}\bfR$, with a slight abuse of notation, to denote an enlargement only in the variables in $B$.
To denote enlargement in a single variable $\x_r$  let  $\e_r=(\dots,0,1,0,\dots)$ be the unit vector with the 1 in $r$-th position, and write $2^{m\e_r}\bfR=2^{m_r}\bfQ_{A_r}\times \Big(\prod\nolimits_{r'\ne r}\bfQ_{A_{r'}}\Big)$.

\begin{remarks}{\rm\quad 
\begin{enumerate}[(1)]
\item Denoting the homogeneous dimension of $\R^{A_{r}}$ by $q_{A_{r}}=\sum_{i\in A_{r}}\frac{q_{i}}{e(k_{r},i)}$ and letting $a$ be the constant in \eqref{mu ell}, then $\big|2^{m\e_r}\bfR\big|\ge c 2^{mq_{A_r}}|\bfR|$.

\smallskip

\item
If $\bfR\in \D_S$, its dyadic enlargements need not be in $\D_S$ because the second kind of restrictions in \eqref{3.6} are not preserved in general.
\end{enumerate}}
\end{remarks}

\subsubsection{Enlargements by rescaling}\quad

\smallskip

We can also enlarge a cube $\bfQ^{\ell_i}_{i,\p_i}$ be rescaling about its center $\c$.

\begin{definition}\label{Def7.5}{\rm
If  $\bfQ^{\ell_i}_{i,\p_i}=\big\{(x_{h})_{h\in E_i}:2^{-\mu_{\ell_i,h}}p_{h}\leq x_{h}<2^{-\mu_{\ell_i,h}}(p_{h}+1)\big\}\in\D$ and $\rho>0$ write
$
2^{\rho}\circ\bfQ^{\ell_{i}}_{i,\p_{i}}=
\Big\{(x_{h})_{h\in E_i}:c_{i}-2^{\rho}2^{-\mu_{\ell_{i},h}} \leq x_{h}<c_{i}+2^{\rho}2^{-\mu_{\ell_{i},h}}\Big\}$.}
\end{definition}

The following properties are satisfied.

\begin{lemma}\label{Lem11.4pqr}
\quad
\begin{enumerate}[\rm(a)]
\item Given $\bfQ_{i,\p_i}^{\ell_i}$ as above,  the dyadic enlargement $2^m  \bfQ_{i,\p_i}^{\ell_i}$ and the enlargement $2^m\circ  \bfQ_{i,\p_i}^{\ell_i}$  by the dilations \eqref{1.1} have comparable measures and there is $\del>0$, depending only on the exponents $\lambda_h$ and the constant $a$ in \eqref{mu ell},  such that $2^m  \bfQ_{i,\p_i}^{\ell_i}\subset 2^{m+\del}\circ  \bfQ_{i,\p_i}^{\ell_i}$.
\item If $\bfQ_{i,\p_i}^{\ell_i}\subset \bfQ_{i,\p'_i}^{\ell'_i}$ and $m\in\N$, then $2^m\circ\bfQ_{i,\p_i}^{\ell_i}\subset 2^m\circ\bfQ_{i,\p'_i}^{\ell'_i}$.
\end{enumerate}
\end{lemma}

\begin{proof}
(a) The side of $2^m  \bfQ_{i,\p_i}^{\ell_i}$ in any direction $x_h$ with $h\in E_i$ has length $2^{-\mu_{\ell_i-m,h}}$ and includes the interval $\big[2^{-\mu_{\ell_i,h}}p_h,2^{-\mu_{\ell_i,h}}(p_{h}+1)\big]$. Therefore it must be contained in the larger interval 
 $$
 \big[2^{-\mu_{\ell_i,h}}(p_{h}+1)-2^{-\mu_{\ell_i-m,h}},2^{-\mu_{\ell_i,h}}p_h+2^{-\mu_{\ell_i-m,h}}\big].
 $$ 
 Defining $
Q_i^*=\prod_{h\in E_i}\big[2^{-\mu_{\ell,h}}(p_{h}+1)-2^{-\mu_{\ell_i-m,h}},2^{-\mu_{\ell_i,h}}p_h+2^{-\mu_{\ell_i-m,h}}\big]$,
we can notice that it has the same center $\c_i$ as $\bfQ_{i,\p_i}^{\ell_i}$. Since $|\mu_{\ell_i,h}-\mu_{\ell_i-m,h}-m\la_h|\le 2a$, the scaling factor in the variable $x_h$ can be estimated by
$
\Big(\frac{2^{-\mu_{\ell_i,h}}+2^{-\mu_{\ell_i-m,h}+1}}{2^{-\mu_{\ell_i,h}+1}}\Big)^\frac1{\la_h}=\Big(\frac12+2^{\mu_{\ell_i,h}-\mu_{\ell_i-m,h}}\Big)^\frac1{\la_h}\le  2^{m+(2a+1)\gamma_0}.
$
Setting $\gamma_{0}=\max_{1\le h\le d}\lambda_{h}^{-1}$,
we then have $2^m\bfQ_{i,\p_i}^{\ell_i}\subset Q_i^*\subset2^{m+\del}\circ\bfQ_{i,\p_i}^{\ell_i}$ with $\del=(2a+1)\gamma_0$. With $q_i$ denoting the homogeneous dimension of $\R^{d_i}$, this gives an inequality $|2^m\bfQ_{i,\p_i}^{\ell_i}|\le 2^{\del q_i}|2^{m}\circ\bfQ_{i,\p_i}^{\ell_i}|$. Moreover 
\bea\label{11.5pqr}
|2^m\bfQ_{i,\p_i}^{\ell_i}|&=\prod_{h\in E_i}2^{-\mu_{\ell_i-m,h}+1} \ge\prod_{h\in E_i}2^{-\mu_{\ell_i,h}+m\la_h-2a}\\
&\ge c_a 2^{mq_i}
|\bfQ_{i,\p_i}^{\ell_i}|=c_a |2^m\circ\bfQ_{i,\p_i}^{\ell_i}|.
\eea
(b) For every $h\in E_i$ we have the inequalities $\mu(\ell'_i,h)<\mu(\ell_i,h)$ and
$$
2^{-\mu(\ell'_i,h)}p'_h\le 2^{-\mu(\ell_i,h)}p_h\le2^{-\mu(\ell_i,h)}(p_h+1)\le 2^{-\mu(\ell'_i,h)}(p'_h+1).
$$
Simple calculations show that these inequalities imply the following ones:
\beas
2^{-\mu(\ell'_i,h)}\Big(p'_h+\frac12-2^{m-1}\Big)&\le 2^{-\mu(\ell_i,h)}\Big(p_h+\frac12-2^{m-1}\Big)\\&\le 2^{-\mu(\ell_i,h)}\Big(p_h+\frac12+2^{m-1}\Big)\le2^{-\mu(\ell'_i,h)}\Big(p'_h+\frac12+2^{m-1}\Big),
\eeas
which give the required inclusion.
\end{proof}

By Proposition \ref{R=Q}, Lemma \ref{Lem11.4pqr} also applies to the  cubes $\bfQ^{\ell_{k_r}}_{A_r,\p_r}$ in $\R^{A_r}$ endowed with the dilations~$\del^{k_r}$. Passing to a dyadic rectangle $\bfR=\prod_{r=1}^s\bfQ_{A_r}\in\D_S$ and $\rhob=(\rho_1,\dots,\rho_s)\in\R^s_+$ we define 
$$
2^\rhob\circ\bfR=\prod_{r=1}^s2^{\rho_r}\bfQ_{A_r}.
$$
We also use the notation $2^{\m+\tau}\circ\bfR=\prod_{r=1}^s2^{m_r+\tau}\bfQ_{A_r}$ with $\m\in\N^s$ and $\tau>0$.

\medskip

\subsection{Multi-norm square functions of Plancherel-P\'olya type}\label{Sec7}
\quad
\smallskip

The notion of Plancherel-P\'olya type square functions was introduced by Frazier and Jawerth in~\cite{MR1070037}. Once we have defined dyadic rectangles in the multi-norm context, we can define Plancherel-P\'olya variants of the multi-norm square functions defined in Section \ref{Sec4.4} and prove $L^{1}$-equivalence between the new and the old ones. We limit ourselves to the Plancherel-P\'olya variants of the multi-norm square functions $S_\Psi$ of tensor type which we will use in later sections, leaving the other cases to the interested reader.

For all $\x\in\R^{d}$ and all $L\in \LL$ denote by $\bfR^L_\x$ the unique dyadic rectangle $\bfR_{P}^{L}$ containing $\x$. With $\Psi_L$ given in \eqref{3.4ee} and for $\x\in\R^{d}$, 
define the Plancherel-P\'olya norm and the Plancherel-P\'olya type square function by
\bea\label{8.1}
\big|f*\Psi_L^{(-L)}\big|_{\rm PP}(\x)&=\sup\nolimits_{\y\in \bfR_{\x}^{L}} \big|f*\Psi_L^{(-L)}(\y)\big|\ ;\\
S^{\rm PP}_\Psi f(\x)&=\Big(\sum\nolimits_{L\in\LL}\big|f*\Psi_L^{(-L)}\big|_{\rm PP}(\x)^2\Big)^{\frac{1}{2}}.
\eea

\begin{theorem}\label{Thm7.1}
For $f\in \cS'(\R^d)$, the square functions $S_\Psi f$ and $S^{\rm PP}_\Psi f$ have comparable $L^{1}$-norms. The same holds true for the square function $S'_\Sigma$ defined in \eqref{6.1} and its Plancherel-P\'olya-type analogue ${S'_\Sigma}^{\!\!\rm PP}$.
\end{theorem}

\begin{proof}
By the trivial point-wise inequality $\big|f*\Psi_L^{(-L)}(\x)\big|\leq \big|f*\Psi_L^{(-L)}\big|_{\rm PP}(\x)$, it suffices to prove $\|S^{\rm PP}_\Psi f\|_1\leq C\|S_\Psi f\|_1$.
We start from the identity $f*\Psi_L^{(-L)}=\sum_{L'\in\LL}(f*\Psi_{L'}^{(-L')})*(\Psi_L^{(-L)}*\widetilde\Psi_{L'}^{(-L')})$. Applying \eqref{5.6} with $\x'$ replaced by $\x'-\y$, we obtain

\beas
\big|f*\Psi_L^{(-L)}(\x)\big|&\leq \sum\nolimits_{L'\in\LL}\int\nolimits_{\R^{d}}\big|f*\Psi_{L'}^{(-L')}(\x-\x'+\y)\big|\big|\Psi_L^{(-L)}*\widetilde\Psi_{L'}^{(-L')}(\x'-\y)\big|\,d\x'\\
&\leq C \sum\nolimits_{L'\in\LL}\omega_{M,L'}(\y)\int\nolimits_{\R^{d}} \prod\nolimits_{i=1}^n\big(1+2^{\ell'_i}|\x_{i}-\x'_i|\big)^M\big|\Psi_L^{(-L)}*\widetilde\Psi_{L'}^{(-L')}(\x'-\y)\big|\,d\x'\\
&= C \sum\nolimits_{L'\in\LL}\omega_{M,L'}(\y)\int\nolimits_{\R^{d}} \prod\nolimits_{i=1}^n\big(1+2^{\ell'_i}|\x_{i}-\x'_i-\y_i|\big)^M\big|\Psi_L^{(-L)}*\widetilde\Psi_{L'}^{(-L')}(\x')\big|\,d\x'\ .
\eeas
Fix now $\bfR_{P}^{L}\in\D^L$, with $\D^L$ defined in \eqref{4.6xyz}, 
and assume $\x,\y\in \bfR_{P}^{L}$, so that $|\x_{i}-\y_i|\leq C2^{-\ell_i}$ for all $i=1,\dots,n$. Calling $\r_i=\x_{i}-\y_i$, it follows that 
$$
\int\nolimits_{\R^{d_{i}}}\big(1+2^{\ell'_i}|\r_i-\x'_i|\big)^M\big|\psi_{i,\ell_i}^{(-\ell_i)}*\widetilde\psi_{i,\ell'_i}^{(-\ell'_i)}(\x'_i)\big|\,d\x'_i\leq C 2^{-\gamma_i|\ell_i-\ell'_i|}2^{q_i\bar\ell i}\int\nolimits_{\R^{d_{i}}}\frac{\big(1+2^{\ell'_i}|\r_i-\x'_i|\big)^M}{\big(1+2^{\bar\ell_i}|\x'_i|\big)^{2M}}\,d\x'_i\ .
$$
Since
$$
1+2^{\ell'_i}|\r_i-\x'_i|\leq 1+2^{\ell'_i-\ell_i}+2^{\ell'_i}|\x'_i| \leq \begin{cases}1+2^{\ell'_i}|\x'_i|\qquad\quad \ =1+2^{\bar \ell_i}|\x'_i|&\text{ if }\ell'_i\leq \ell_i\\
2^{\ell'_i-\ell_i}\big(1+2^{\ell_i}|\x'_i|\big)=2^{\ell'_i-\ell_i}\big(1+2^{\bar\ell_i}|\x'_i|\big)&\text{ if }\ell'_i>\ell_i\ ,
\end{cases}
$$
we can continue as before and conclude that $\big|f*\Psi_L^{(-L)}(\x)\big|\leq C\sum_{L'\in\LL}2^{-\eps|L-L'|}\omega_{M,L'}(\y)$. Taking the supremum over $\x\in \bfR_{P}^{L}$ and summing over $M$ and $L$, we obtain that
\beas
\Big(\sum\nolimits_{L\in\LL}\Big(\big|f*\Psi_L^{(-L)}\big|_{\rm PP}(\y)\Big)^2\Big)^\half&\leq C \Big(\sum\nolimits_{L\in\LL}\Big( \sum\nolimits_{L'\in\LL}2^{-\eps|L-L'|} \omega_{M,L'}(\y)\Big)^2\Big)^\half\\
&\lesssim \Big(\sum\nolimits_{L\in\LL}\omega_{M,L}(\y)^2\Big)^\half
\eeas
by Young's inequality. The result for $S^{\rm PP}_\Psi$ follows from the proof of Theorem \ref{Thm5.1}. The proof for ${S'_\Sigma}^{\!\!\rm PP}$ is similar and we omit the details.
\end{proof}

\section{The multi-norm Hardy space $\h^1_{\bfE}(\bR^d)$}\label{Sec8}

We define $\h^1_\mbE(\bR^d)$  to be the space of distributions $f\in\cS'(\bR^d)$ such that $S_{\Xi}f\in L^{1}(\R^{d})$ where $S_{\Xi}$ is any of the square functions introduced in Sections \ref{Sec5}, \ref{Sec6} and \ref{Sec7}. We set $\|f\|_{h^{1}_{\bfE}}=\|S_{\Xi}f\|_{L^{1}}$. Different choices of square function result in equivalent norms. In what follows, the particular square function may vary from one occurrence to another, and will be specified when necessary. 
We recall from Section \ref{Sec7.2b} the two notions of dyadic enlargement and enlargement by dilations of a given dyadic rectangle.

\subsection{Continuous inclusion of $ h^1_\mbE(\bR^d)$ in $L^{1}(\bR^d)$}\label{Sec8.1}

\begin{theorem} \label{Thm8.1}
$\h^1_\mbE(\bR^d)\subseteq L^{1}(\R^{d})$ and there is a constant $C$ so that $\|f\|_{L^{1}}\leq C\|f\|_{h^{1}_{\bfE}}$. \end{theorem}

The proof is a direct consequence of Lemmas   \ref{Lem8.2} and   \ref{Lem8.3} below.

\begin{lemma}\label{Lem8.2}
$\h^1_\mbE(\R^d)\cap L^2(\bR^d)$ is dense in $\h^1_\mbE(\bR^d)$.
\end{lemma}

\begin{proof}
Fix a square function $S_\Psi f=\big(\sum_{L\in\LL}\big|f*\Psi_L^{(-L)}\big|^2\big)^\half$ as in \eqref{3.4ee}, with the Fourier transform of $\Psi_L^{(-L)}$ supported on a small neighborhood of $\pi\inv(T_L)$ where $\pi$ is the log-map in \eqref{2.10} and $T_L$ the tube in \eqref{3.8}. Then $\Psi_L^{(-L)}*\Psi_{L'}^{(-L')}=0$ for $|L-L'|>\kappa$ where $\kappa$ is the constant in Lemma \ref{Lem3.8} (iv). This implies that $\Phi_L^{(-L)}=\sum_{L'\leq L}\Psi_{L'}^{(-L')}$ has Fourier transform supported on $\{\xib\in \R^{d}:|\xib_i|\leq 2^{\ell_i+\kappa}\ \forall i\}$ and equal to 1 on $\{\xib:|\xib_i|\leq 2^{\ell_i-\kappa}\ \forall i\}$. Let $f\in \h^1_\mbE(\bR^d)$ so $f\in\cS'(\bR^d)$ and $S_\Psi f\in L^{1}(\R^{d})$. We prove that $f*\Phi_L^{(-L)}$ is in $L^2\cap \h^1_\mbE$ , for every $L\in\LL$, and $\|f*\Phi_L^{(-L)}-f\|_{\h^1_\mbE}\to0$ as $L\to\infty$. 
First observe that $f*\Phi_L^{(-L)}\in L^{1}(\bR^d)$. Since $\int_{\R^{d}}\big|f*\Phi_L^{(-L)}(\x )\big|\,d\x \leq \sum_{L'\leq L}\int\big|f*\Psi_{L'}^{(-L')}(\x )\big|\,d\x $ it follows that $\int_{\R^{d}}\big|f*\Phi_L^{(-L)}(\x )\big|\,d\x  \leq C|L|^n\big(\sum_{L'\leq L}\int\big|f*\Psi_{L'}^{(-L')}(\x )\big|^2\,d\x \big)^\half \leq C_L\|f\|_{\h^1_\mbE}$. Next, since $f*\Phi_L^{(-L)}=f*\Phi_L^{(-L)}*\Phi_{2L}^{(-2L)}$ we have $\|f*\Phi_L^{(-L)}\|_2\leq \|f*\Phi_L^{(-L)}\|_1\|\Phi_{2L}^{(-2L)}\|_2\leq C'_L\|f\|_{\h^1_\mbE}$. Thus $f*\Phi_L^{(-L)}\in L^2(\bR^d)$. Next we prove $\|f*\Phi_L^{(-L)}-f\|_{\h^1_\mbE}\to0$ as $L\to\infty$. (This implicitly includes the fact that $f*\Phi_L^{(-L)}\in \h^1_\mbE$.) We have 
\beas
\|f*\Phi_L^{(-L)}-f\|_{\h^1_\mbE}=\int_{\bR^d}\big(\sum\nolimits_{L'\in\LL}\big|f*\Phi_L^{(-L)}*\Psi_{L'}^{(-L')}(\x )-f*\Psi_{L'}^{(-L')}(\x )\big|^2\big)^\half\,d\x .
\eeas
We split the index set $\LL$ into three subsets depending on $L$: $E_L^0=\{L':\ell'_i<\ell_i-2\kappa\ \forall i\}$,\,\, $E_L^\infty=\{L':\ell'_i>\ell_i+2\kappa\ \forall i\}$, and $E_L^{1}=\{L':\exists i \text{ such that } \ell_1-2\kappa\leq \ell'_i\leq \ell_i+2\kappa\}$.
Because of the support properties, the terms with $L'\in E_L^0$ are equal to zero. Those with $L'\in E_L^\infty$ are equal to $\big|f*\Psi_{L'}^{(-L')}(\x )\big|^2$, and $\lim_{L\to\infty}\int_{\bR^d}\big(\sum_{L'\in E_L^\infty}\big|f*\Psi_{L'}^{(-L')}(\x )\big|^2\big)^\half\,d\x =0$ by monotone convergence. It remains to consider the terms with $L'\in E_L^{1}$. The number of these terms grows polynomially in $L$, so that it suffices to prove that each term is controlled by $2^{-\del|L|}$ for some $\del>0$.
We have

\beas
\int_{\bR^d}\big|f*\Phi_L^{(-L)}&*\Psi_{L'}^{(-L')}(\x )-f*\Psi_{L'}^{(-L')}(\x )\big|\,d\x \\
&\leq \int_{\bR^d}\int_{\bR^d}\big|f*\Psi_{L'}^{(-L')}(\x -\mby)-f*\Psi_{L'}^{(-L')}(\x )\big|\,\big|\Phi_L^{(-L)}(\mby)\big|\,d\mby\,d\x \\
&=\int_{\bR^d}\Big(\int_{\bR^d}\big|f*\Psi_{L'}^{(-L')}(\x -2^{-L'}\mby)-f*\Psi_{L'}^{(-L')}(\x )\big|\,d\x \Big)\big|\Phi_L(\mby)\big|\,d\mby\\
&=\int_{\bR^d}\Big(\int_{\bR^d}\big|f*\Omega_{L',\mby}^{(-L')}(\x )\big|\,d\x \Big)\big|\Phi_L(\mby)\big|\,d\mby\ ,
\eeas
where we have set $\Omega_{L',\mby}(\x )=\Psi_{L'}(\x -2^{-L'}\mby)-\Psi_{L'}(\x )$. 
We split the integral in $d\mby$ into two parts. 

If $|\mby|>2^{|L'|/2}$ we use the inequality
$
\int_{\bR^d}\big|f*\Omega_{L',\mby}^{(-L')}(\x )\big|\,d\x \leq 2\int_{\bR^d}\big|f*\Psi_{L'}^{(-L')}(\x )\big|\,d\x \leq 2\|f\|_{\h^1_\mbE}\ ,
$
to conclude that $\int_{|\mby|>2^{|L'|/2}}\big(\int_{\bR^d}\big|f*\Omega_{L',\mby}^{(-L')}(\x )\big|\,d\x \big)\big|\Phi_L(\mby)\big|\,d\mby$ decays faster than $2^{-N|L'|}$ for every~$N$. In order to replace $|L'|$ with $|L|$, observe that, since $L'\in E_L^{1}$, there is a component $i$ such that $|\ell'_i-\ell_i|\leq \kappa$. Using the basic assumptions, for every component $j$ we obtain that $\ell_j\leq c\ell'_j$ with $c>0$. 
This implies that this part of the integral decays faster than $2^{-N|L|}$ for every~$N$. 
If $|\mby|\le2^{|L'|/2}$ we use instead the inequalities
$$
\big|\de^\al\Omega_{L',\mby}(\x )\big|=\big|\de^\al\Psi_{L'}(\x -2^{-L'}\mby)-\de^\al\Psi_{L'}(\x )\big|\leq C_{\al,N}\frac{|2^{-L'}\mby|^\eps}{(1+|\x |)^N}\leq C_{\al,N}\frac{2^{-\eps|L'|/2}}{(1+|\x |)^N}
$$
for every $N$, with $\eps>0$ depending on the exponents of the basic dilations on $\bR^d$. Hence the Schwartz norms of $\Omega_{L',\mby}$ are bounded by a constant times $2^{-\eps|L'|/2}$ uniformly in $\mby$.
Observing that $\Omega_{L',\mby}$ inherits the cancellations of $\Psi_{L'}$, we can then apply Theorem 4.1 with the family $\{\Theta_L\}_{L\in\LL}$ defined by $\Theta_L=\Omega_{L',\mby}$ if $L=L'$ and $\Theta_L=0$ otherwise. It follows that
$
\int_{\bR^d}\big|f*\Omega_{L',\mby}^{(-L')}(\x )\big|\,d\x \leq C2^{-\eps' |L'|/2} \|S_\Psi f\|_1=C2^{-\eps' |L'|/2}\|f\|_{\h^1_\mbE}$, 
uniformly in $\mby$. As above, $|L'|$ can be changed to $|L|$, correcting the coefficient $\eps'$.
\end{proof}

\begin{lemma}\label{Lem8.3}
There is a constant $C>0$ such that if $f \in L^2(\R^d) \cap \h^1_{\mathbf{E}}(\R^d)$ then $ \|f\|_{1} \leq C\|f\|_{\h^1_{\mathbf{E}}}$.
\end{lemma}

\begin{proof}
It is now convenient to choose a square function $S_\Psi^{\rm PP}$ of Plancherel-P\'olya type, as in \eqref{8.1}, and assume that the terms $\Psi_L,\widetilde\Psi_L$ in the corresponding Calder\'on reproducing formula $ f=\sum_{L \in\LL}f*\Psi_{L}^{(-L)}*\widetilde{\Psi}_{L}^{(-L)}$ are localized in space; \textit{i.e.} supported on a common compact neighborhood $U$ of $\0$ rather than in frequency space, cf. Lemma \ref{Lem4.2}. For $f \in L^2 \cap \h^1_{\mathbf{E}}(\bR^d)$, the series $\sum_{L \in\LL}f*\Psi_{L}^{(-L)}*\widetilde{\Psi}_{L}^{(-L)}$ is convergent in the $L^2$-norm. For $i \in \mathbb{Z}$ we define 
\bea\label{9.1}
\Omega^{i}&=\Big\{\x \in \mathbb{R}^{d}\;:\;S_{\Psi}^{{\rm PP}}f(\x )>2^{i}\Big\},\\
\mathcal{B}_i&=\Big\{\bfR_{P}^{L} \in \mathbb{D}\;:\;|\bfR_{P}^{L} \cap\Omega^{i}|>\frac{1}{2}|\bfR_{P}^{L}|,\; |\bfR_{P}^{L} \cap \Omega^{i+1}| \leq \frac{1}{2} |\bfR_{P}^{L}|\Big\}.
\eea
Each $\bfR_{P}^{L}\in \D^{L}$ belongs to a unique $\mathcal B_{i}$. Therefore, we have 
\bea\label{8.2}
f(\x ) &= \sum_{L \in\LL}\Psi_{L}^{(-L)}*\widetilde{\Psi}_{L}^{(-L)}*f(\x )\\
&= \sum_{i \in \mathbb{Z}}\sum_{L \in\LL}\sum_{\bfR_{P}^{L} \in \D^{L}\cap \mathcal{B}_i}\int_{\R^d}\widetilde{\Psi}_L^{(-L)}(\x -\mby)\Psi_L^{(-L)}*f(\mby)\chi_{\bfR_{P}^{L}}(\mby) \, d\mby
\eea
in $L^2$-sense. Then, it suffices to show there is a constant $C>0$ such that for all $i \in \mathbb{Z}$ we have 
\be\label{9.3}
\Big\|\sum\nolimits_{L \in\LL}\sum\nolimits_{\bfR_{P}^{L} \in \bD^{L}\cap \mathcal{B}_i}\int_{\mathbb{R}^d}\widetilde{\Psi}_L^{(-L)}(\cdot-\mby)\Psi_L^{(-L)}*f(\mby)\chi_{\bfR_{P}^{L}}(\mby) \, d\mby\Big\|_{1} \leq C\,2^{i}\,|\Omega_i|. 
\ee
Indeed, combining~\eqref{9.1},~\eqref{8.2}, and~\eqref{9.3} we obtain $ \|f\|_{1} \leq C\sum_{i \in \mathbb{Z}}2^{i}|\Omega_i| \leq C\|S^{\rm PP}_\Psi f\|_1 \leq C\|f\|_{\h^1_{\mathbf{E}}}$. Now we prove~\eqref{9.3}. Since the functions $\{\widetilde\Psi_L\}$ are supported on $U$, there is a constant $\tau_0>0$, independent of the scale $L$ and $K$, such that

\bea\label{9.4}
\supp \int_{\mathbb{R}^d}\widetilde{\Psi}_L^{(-L)}(\cdot-\mby)\Psi_L^{(-L)}*f(\mby)\chi_{\bfR_{P}^{L}}(\mby) \, d\mby \subseteq \bfR_{P}^{L}+2^{-L}U\subseteq 2^{\tau_{0}}\circ\bfR_{P}^{L} .
\eea
Therefore, the function $f_i:=\sum_{L \in\LL}\sum_{\bfR_{P}^{L} \in \bD^{L}\cap\mathcal{B}_i}\int_{\mathbb{R}^d}\widetilde{\Psi}_L^{(-L)}(\cdot-\mby)\Psi_L^{(-L)}*f(y)\chi_{\bfR_{P}^{L}}(\mby) \, d\mby
$
is supported on $\bigcup_{\bfR_{P}^{L} \in \mathcal{B}_i} 2^{\tau_{0}}\circ\bfR_{P}^{L} $. Therefore, it follows from the definition \eqref{9.1} of $\mathcal{B}_i$ that $f_i$ is supported on $\Omega^{i}_{(\tau_0)}$ as well, where $\Omega^{i}_{(\tau_0)}:=\Big\{\x \in \mathbb{R}^{d}\;:\;\MM_{s}[\chi_{\Omega_i}](\x )>\frac{1}{(10\cdot2^{\tau_0})^d}\Big\}$, where $\MM_{s}$ is the strong maximal function in $\eqref{Ms}$. Hence, by Cauchy--Schwarz inequality together with the strong maximal theorem applied to $\chi_{\Omega_i}$,
\begin{equation}\label{9.5}
\|f_i\|_1 \leq |\Omega^{i}_{\tau_0}|^{\frac{1}{2}}\|f_i\|_2 \leq C|\Omega_i|^{\frac{1}{2}}\|f_i\|_2\ .
\end{equation}
By a duality argument,
\beas
\|f_i\|_2&=\sup_{h \in L^2, \; \|h\|_2=1}\int_{\mathbb{R}^d}h(\x )f_i(\x )\,d\x 
=\sum_{L \in\LL}\sum_{\bfR_{P}^{L} \in\D^{L}\cap\mathcal{B}_i}\int_{\mathbb{R}^d}\widetilde{\Psi}_L^{(-L)}*h(\mby) \Psi_L^{(-L)}*f(\mby)\chi_{\bfR_{P}^{L}}(\mby) \, d\mby\ .
\eeas
Using Cauchy--Schwarz and the $L^2$-boundedness of square function $ h \longmapsto \Big(\sum_{L \in\LL}|\widetilde{\Psi}_{L}^{(-L)}*h|^2\Big)^{\frac{1}{2}}$ we conclude that $ \|f_i\|_2 \leq C\big\|\big(\sum_{L \in\LL}\sum_{\bfR_{P}^{L} \in \D^{L}\cap\mathcal{B}_i}\max_{\mby' \in R_K^{L}}|\Psi_L^{(-L)}*f(\mby')|^2\chi_{\bfR_{P}^{L}}\big)^{\frac{1}{2}}\big\|_{2}$. Note that, by the definition of $\mathcal{B}_i$, for all $\x \in {\bfR_{P}^{L} \in \D^{L}\cap\mathcal{B}_i}$ we have $\MM_{s}(\chi_{(\Omega^{i}_{(\tau_0)} \setminus \Omega^{i+1}) \cap \bfR_{P}^{L}})(\x )>\frac{1}{2}$, which implies that, for all $\x \in \mathbb{R}^{d}$, we have $\chi_{\bfR_{P}^{L}}({\x }) = \chi_{\bfR_{P}^{L}}({\x })^2 \leq 4 \MM_{s}(\chi_{(\Omega^{i}_{(\tau_0)} \setminus \Omega^{i+1}) \cap \bfR_{P}^{L}})(\x )^2$. Hence, by the Fefferman--Stein vector-valued strong maximal inequality
\beas
\big\|\big(\sum_{L \in\LL}&\sum_{\bfR_{P}^{L} \in\D^{L}\cap \mathcal{B}_i}\max_{\mby' \in R_K^{L}}|\Psi_L^{(-L)}*f(\mby')|^2\chi_{\bfR_{P}^{L}}\big)^{\frac{1}{2}}\big\|_{2}\\
&\qquad \leq C\sum_{L \in\LL}\sum_{\bfR_{P}^{L} \in \cap\mathcal{B}_i}\int_{\mathbb{R}^d} \max_{\mby' \in \bfR_{P}^{L}}|\Psi_{L}^{(-L)}*f(\mby')|^2 \MM_{s}\chi_{(\Omega^{i}_{(\tau_0)}\setminus \Omega^{i+1}) \cap \bfR_{P}^{L}}(\x )^2\,d\x \\
&\qquad \leq C\int_{\Omega^{i,\tau}_{D} \setminus \Omega^{i+1}}S_{\Psi}^{{\rm PP}}(\x )^2\,d\x .
\eeas
Recall that, by \eqref{9.1}, $S_{\Psi}^{{\rm PP}}f(x) \leq 2^{i+1}$ for $x \in\Omega^{i}_{(\tau_0)} \setminus \Omega^{i+1}$, so that, by \eqref{9.4}, 
\be\label{9.7}
\big\|\big(\sum_{L \in\LL}\sum_{\bfR_{P}^{L} \in \D^{L}\cap\mathcal{B}_i}|\Psi_{L}^{(-L)}*f|^2\chi_{\bfR_{P}^{L}}\big)^{\frac{1}{2}}\big\|_{L^2}^2 \leq C2^{2i}|\Omega^{i,\tau}_{D}|.
\ee
Finally,~\eqref{9.3} is a consequence of~\eqref{9.5} and~\eqref{9.7}.
\end{proof}

\subsection{Locality of $\h_{\bfE}^{1}(\R^d)$}\label{Sec8.2}\quad

\smallskip

\begin{proposition}\label{Prop8.4} Let $f\in h_{\bfE^{1}}$ and let $\eta\in \CC^{\infty}_{c}(\R^{d})$ be supported in a ball of radius $2^{\tau}$ with $\tau>0$. Then the product $\eta f\in \h_{\bfE}^{1}(\R^d)$ and $\|\eta f\|_{\h_{\bfE}^{1}}\leq C\|f\|_{\h_{\bfE}^{1}}$ where $C$ depends on $\tau$ and on some $\CC^{k}$-norm of $\eta$.
\end{proposition}

\begin{proof}
We may assume that $\eta$ is supported in $2^{\tau}\circ B(1)$. We use the square function $S_{\Psi}$ in \eqref{5.1} with each $\Psi_{L}$ supported in $B(\rho)$ and uniformly bounded in every Schwartz norm. Then $S_{\psi}[\eta f]$ is supported in $2^{\tau}\circ B(1) +B(\rho)\subseteq B(\mu)$ for some $\mu>1$ depending only on $\tau$, $\rho$, and the exponents $\{\lambda_{h},\,1\leq h\leq d\}$ from \eqref{1.1}. By choosing $\rho$ sufficiently small, we can make $\mu$ arbitrarily close to $2^{\tau}$. We have
$$
\big((f\eta)*\Psi_L^{(-L)}\big)(\x)=\eta(\x)(f*\Psi_L^{(-L)})(\x)+\int_{\bR^d}f(\x-\y)\big[\eta(\x-\y)-\eta(\x)\big]\Psi_L^{(-L)}(\y)\,d\y\ .
$$
There are $\del,\del'>0$ such that, for $\y\in\supp\Psi_L^{(-L)}$ and $\x\in B(\mu)$, then $|\eta(\x-\y)-\eta(\x)|\lesssim |\y|^\del\lesssim 2^{-\del'|L|}$, with implicit constants depending only on $\tau$ and the $C^1$-norm of $\eta$. Then
$$
\big|(f\eta)*\Psi_L^{(-L)}(\x)\big|\le \big|f*\Psi_L^{(-L)}(\x)\big|+C2^{-\del'|L|}\big(|f|*|\Psi_L^{(-L)}|\big)(\x)\ ,
$$
and
\beas
\|S_\Psi(f\eta)\|_1&\le \|S_\Psi f\|_1+C\int\Big(\sum_{L\in\LL}\big(2^{-\del'|L|}\big(|f|*|\Psi_L^{(-L)}|\big)(\x)\big)^2\Big)^{1/2}\d\x\\
&\le \|S_\Psi f\|_1+C\sum_{L\in\LL}2^{-\del'|L|}\int\big(|f|*|\Psi_L^{(-L)}|\big)(\x)\d\x\\
&\lesssim \|S_\Psi f\|_1+C'\|f\|_1\lesssim\|S_\Psi f\|_1\ ,
\eeas
by Theorem \ref{Thm8.1}.
\end{proof}

\begin{corollary}\label{Cor8.5}
Let $\tau>0$ and $\eta\in C^\infty_c(\R^d)$ be supported on $ 2^{\tau}\circ B(1)$  such that $\sum_{P\in\Z^n}\eta(\x-P)=1$. Calling $\eta_P(\x)=\eta(\x-P)$, there is $A>1$, depending only on $\tau$ and some $C^k$-norm of $\eta$, such that $\|S_\Psi f\|_1\le\sum_{P\in\Z^n}\|S_\Psi(f\eta_P)\|_1\le A\|S_\Psi f\|_1$ for every $f\in L^{1}(\bR^d)$.\end{corollary}

\begin{proof}
The first inequality is trivial and the second follows from Proposition \ref{Prop8.4} and the bounded overlapping property of the supports of the $S_\Psi(f\eta_P)$.
\end{proof}

\subsection{$\h^1_\mbE$-boundedness of multi-norm singular integrals}\label{Sec8.3}\quad

\smallskip

The notions of multi-norm Mihlin-H\"ormander multiplier and multi-norm singular kernel have been introduced in Theorem \ref{Thm3.9}. 
We prove the $\h^1_\mbE$-boundedness theorem for multi-norm singular integral operators.

\begin{theorem}\label{Thm8.6}
Let $\cK$ be a multi-nom singular kernel. Then the operator $Tf=f*\cK$ is bounded on $\h^1_\mbE(\bR^d)$, and $\|Tf\|_{h^{1}_{\bfE}}\leq C\,\|f\|_{h^{1}_{\bfE}}$ where $C$ depends on the norm of $\KK$ in $\PP_{0}(\bfE)$.
\end{theorem}

\begin{proof}
For $i=1,\dots,n$ we fix functions $\ph_{i,\ell_i}\in\cS(\bR^{d_i})$ as in Section \ref{Sec4.1}, imposing the conditions $\supp\widehat{\ph_{i,\ell_i}}\subset \big\{\xib_i:|\xib_i|\leq 2^{3/4}\big\}$ and $\widehat{\ph_{i,\ell_i}}(\xib_i)=1 \text{ for }|\xib_i|\leq 2^{-3/4}$. We define $\psi_{i\ell_i}=\ph_{i,\ell_i}-\ph_{i,\ell_i-1}^{(1)}$ and, as in Section \ref{Sec4.2}, we construct, for $L\in\LL_D$,
$
\Psi_{L}(\x )=\big(\prod_{i\in D_{L}}\psi_{i,\ell_i}(\x _{i})\big)\big(\prod_{i\notin D_{L}}\ph_{i,\ell_i}(\x _{i})\big).
$
For $L\in\LL$ we have $\del_\0=\sum_{L\in \LL}\Psi_{L}^{(-L)}$ and 
\bea\label{9.8}
\supp \widehat{\Psi_L^{(-L)}}&\subset \Big(\prod_{i\in D_L}\big\{\xib_i: 2^{-\ell_i-3/4}\le|\xib_i|\leq 2^{-\ell_i+3/4}\big\}\Big)
\times\Big(\prod_{i\not\in D_L}\big\{\xib_i: |\xib_i|\leq 2^{-\ell_i+3/4}\big\}\Big).
\eea

Let $K$ be a multi-norm singular kernel and $m=\widehat K$ the corresponding multiplier.
For $L\in\LL$, we define $ m_L(\xib)=m(2^L\xib) \widehat{\Psi_L}(\xib)$ and $\eta_L=\cF\inv(m_L)$. It is easy to deduce from \eqref{9.8} that the $m_L$ verify the conditions in Theorem \ref{Thm3.9} (ii) and the $\eta_L$ the conditions in Theorem \ref{Thm3.9} (iv). Given $f\in\cS(\bR^d)$ we have $Tf=f*K=\sum_{L\in\LL}f*\eta_L^{(-L)}$. With $\|f\|_{\h^1_\mbE}=\|S_\Psi\|_1$, we have
$
\|Tf\|_{\h^1_\mbE}=\big\|\big(\sum_{L\in\LL}\big|Tf*\Psi_L^{(-L)}\big|^2\big)^\half\big\|_1\ ,
$
where
$
Tf*\Psi_L^{(-L)}=\sum_{L'\in\LL}f*\eta_{L'}^{(-L')}*\Psi_L^{(-L)}.
$
However, the convolution $\eta_{L'}^{(-L')}*\Psi_L^{(-L)}$ is zero if $L$ and $L'$ are not close enough. To see this, consider
the supports of the Fourier transforms of the two factors and their images under $\tau$ in $\bR^n_+$. By \eqref{9.8},
the two images are contained in the $1/4$-enlargements $\widetilde T_L$ of $T_L$ and $\widetilde T_{L'}$ $T_{L'}$ respectively, i.e.,
$$
\widetilde T_L=\Big\{\mbt\in \bR^{n}_{+}:\ell_{j}-\frac34\leq t_{j}\leq \ell_{j}+\frac34\text{ if }j\in D_{L}\text{ and }0\leq t_{j}\leq \ell_{j}+\frac34\text{ if }j\notin D_{L}\Big\}\ ,
$$
and similarly for $T_{L'}$. The intersection $\widetilde T_L\cap\widetilde T_{L'}$ is nonempty if and only if $T_L$ and $T_{L'}$ have an edge in common. We can then apply Lemma \ref{Lem3.8} (iv) to conclude that $|L-L'|\leq \kappa$. Therefore, 
\beas
\|Tf\|_{\h^1_\mbE}&=\Big\|\Big(\sum_L\big|Tf*\Psi_L^{(-L)}|^2\Big)^\half\Big\|_1
=\Big\|\Big(\sum_L\big|\sum_{|L'-L|\leq \kappa}f*\eta_{L'}^{(-L')}*\Psi_L^{(-L)}|^2\Big)^\half\Big\|_1\\
&\leq \sum_{|\widetilde L|\leq \kappa}\Big\|\Big(\sum_L\big|f*\eta_{L+\widetilde L}^{(-L-\widetilde L)}*\Psi_L^{(-L)}|^2\Big)^\half\Big\|_1\ .
\eeas

For every $\widetilde L$ in the above formula, the functions $\Theta_L=\eta_{L+\widetilde L}^{(-\widetilde L)}*\Psi_L$ are uniformly bounded in every Schwartz norm and satisfy the other conditions in Proposition \ref{Prop5.8}. This gives the inequality $\|Tf\|_{\h^1_\mbE}\leq C\|f\|_{\h^1_\mbE}$.
\end{proof}

\subsection{Characterization of $\h^1_\mbE(\R^d)$ by Riesz transforms}\label{Sec8.4}\quad
\smallskip

By \cite{MR0757952}, for any family $\{\del_r\}_{r>0}$ of dilations on a space $\bR^\nu$, it is possible to construct a system of $\nu$ smooth, homogeneous Calder\'on-Zygmund kernels $\cK_1,\dots,\cK_\nu$ with the property that a function $f\in L^1(\R^\nu)$ belongs to the non-isotropic Hardy space $H^1_\del$ if and only if $f*\KK_j\in L^1(\R^\nu)$ for all $j=1,\dots,\nu$. One construction of the $\cK_j$ goes as follows. In $\bR^\nu$ with scalar coordinates $(x_1,\dots,x_\nu)$ consider the Euclidean sphere $S^{\nu-1}=\big\{\x :\sum_{j=1}^\nu x_j^2=1\big\}$ and define the $\del$-homogeneous norm $|\cdot|$ by the condition
$
|\x |=r\ \Longleftrightarrow \del_{r\inv}\x \in S^{\nu-1}.
$
It follows from the implicit function theorem that this norm is smooth away from the origin.
We define $\cK_j$ by the condition
$
\widehat{\cK_j}(\xi)=\frac{\xi_j}{|\xib|^{\la_j}}\ ,
$
where $\la_j$ is the exponent of the dilations in the $j$-th entry.

\begin{lemma}\label{Lem8.7}
If $f\in L^{1}(\bR^\nu)$ then $f\in H^1_\del(\bR^\nu)$ if and only if $f*\cK_j\in L^{1}(\bR^d)$ for all $j=1,\dots,\nu$.
\end{lemma}

\begin{proof}
The sufficient condition (*) in \cite[p. 547]{MR0757952} requires that, for any vector $\v=(v_0,\dots,v_\nu)\in S^\nu$ there exist multipliers $\mu_0,\mu_1,\dots,\mu_\nu$ such that $\sum_{j=0}^\nu m_j(\xib)\mu_j(\xib)=1$ and $ \sum_{j=0}^\nu v_i\mu_j(\xib)=0$. Borrowing formulas from \cite[p. 572]{MR0757952}, we define functions $\mu_{j}$, homogeneous of degree zero, by $\mu_0(\xib)=\sum_{j=1}^\nu v_j^2+v_0\sum_{j=1}^\nu v_j\xi_j$ and $\mu_j(\xib)=-v_0v_j-v_0^2\xi_j$ for $1\leq j \leq \nu$ when $\xib\in S^{\nu-1}$. 
\end{proof}

Modifying the $\cK_j$ constructed above, we can obtain local Calder\'on-Zygmund operators characterizing the local Hardy space~$\h^1_\del(\bR^\nu)$.

\begin{proposition} \label{Prop8.8}
For $j=1,\dots,\nu$, let $\cK^\flat_j$ be the local Calder\'on-Zygmund kernel such that $\widehat {\cK^\flat_j}(\xib)=\widehat {\cK_j}(\xib)\big(1-\theta(\xib)\big)$, where $\theta\in C^\infty_c(\bR^\nu)$ is supported where $|\xib|\leq 2$ and is equal to $1$ for $|\xib|\leq 1$. Then a function $f\in L^{1}(\bR^\nu)$ belongs to $\h^1_\del(\bR^\nu)$ if and only if $f*\cK^\flat_j\in L^{1}(\bR^\nu)$ for all $j=1,\dots,\nu$.
\end{proposition}

\begin{proof}
The ``only if'' part of the proof is standard. We then assume that $f\in L^{1}(\bR^\nu)$ and $f*\cK^\flat_j\in L^{1}(\bR^\nu)$ for every $j$. Keeping notation and assumptions of Section \ref{Sec4.1}, take $\ph_\ell=\ph,\psi_\ell=\psi=\ph-\ph^{(1)}$, with $\widehat\ph(\xib)=1$ for $|\xib|\leq \frac{1}{2}$, $\widehat\ph(\xib)=0$ for $|\xib|\ge1$. Decompose $f\in L^{1}(\bR^\nu)$ as $f_0+f_1+f_\infty$ where $ f_0=f*\ph\ $, $ f_1=f*\psi^{(-1)}$, and $f_\infty=f-f_0-f_1=\sum_{\ell\ge2}f*\psi^{(-\ell)}$. Then $f_0*\cK^\flat_j=0$, $f_1*\cK^\flat_j\in L^{1}(\bR^\nu)$ and $f_\infty*\cK^\flat_j=f_\infty*\cK_j$. Then $f_\infty*\cK_j\in L^{1}(\bR^\nu)$ for every $j$ and, by Lemma \ref{Lem8.7}, $f_\infty\in H^1_\del(\bR^\nu)\subset \h^1_\del(\bR^\nu)$.
To prove that also $f_0$ and $f_1$ are in $\h^1_\del(\bR^\nu)$ we consider the square function
$
S_{\rm loc}g(\x )=\Big(\big|g*\ph(\x )\big|^2+\sum_{\ell\ge1}\big|g*\psi^{(-\ell)}(\x )\big|^2\Big)^\half$.
Taking $g=f_0$ or $g=f_1$ only two values of $\ell$ at most give a non-zero term, so that both $S_{\rm loc}f_0, S_{\rm loc}f_1$ are in $L^{1}(\bR^\nu)$.
\end{proof}

After these preliminaries for general dilations, we go back to $\R^d$ with the $n$-tuple of dilations $\widehat\del^i$ in \eqref{2.5a}. For each $i=1,\dots,n$ we denote by $\cK^\flat_{i,j}$, $j=1,\dots,d$ the kernels introduced in Proposition \ref{Prop8.8} relative to the dilations $\widehat\del^i$. 
We also define $\KK^\flat_{i,0}=\del_\0$, the Dirac delta at the origin.

For $J\in\{0,\dots,d\}^n$, the {\it multinorm Riesz kernels} 
$
R_J=\cK^\flat_{1,j_1}*\cdots*\cK^\flat_{n,j_n}\ .
$
Notice that, since one or more factors can be $\del_\0$, this family of kernels also contains convolutions of a smaller number of factors, including the single $\KK^\flat_{i,j}$ and even $\del_\0$.
It follows from Theorem \ref{Thm3.10} that the distributions $R_J$ are in $\cP_0(\mbE)$. By Theorem \ref{Thm8.6}, each $R_J$ maps $\h^1_\mbE(\R^d)$ into itself, hence into $L^{1}(\bR^d)$.
We prove the converse statement using a randomization argument as in \cite{Cowling-etal}.
\begin{theorem}\label{Thm8.9}
Let $f\in\SS'(\R^d)$. If $f*R_J\in L^{1}(\bR^d)$ for every $J\in\{0,\dots,d\}^n$ then $f\in \h^1_\mbE(\R^d)$, and $\|S_\Psi f\|_1\sim\sum_J\|f*R_J\|_1$. 
\end{theorem}

\begin{proof}
Define the $\h^1_\mbE$-norm using the square function \eqref{6.3},
$
S'' f=\Big(\sum_{L\in\N^{n}}\big|f*\big(\mathop*_{i=1}^n\sigma_{i,\ell_i}^{(-\ell_i)_i}\big)\big|^2\Big)^\half
$
where the $\sigma_{i,\ell_i}$ are as in \eqref{4.5}. Let $f$ be as in the hypothesis.
For $k\in \bN$ we denote by $r_k(s)$ the $k$-th Rademacher function, defined for $s\in[0,1]$. For any $\mbs=(s_1,\dots,s_n)$, we consider the kernel
\beas
\cH_\mbs&=\sum_{L\in\N^{n}}\Big(\prod_{i=1}^nr_{\ell_i}(s_i)\Big)\Big(\mathop*_{i=1}^n\sigma_{i,\ell_i}^{(-\ell_i)_i}\Big)
=\mathop*_{i=1}^n\Big(\sum_{\ell_i\in\bN}r_{\ell_i}(s_i)\sigma_{i,\ell_i}^{(-\ell_i)_i}\Big)
=\cH_{s_1}*\cdots *\cH_{s_n},
\eeas 
where each $\cH_{s_i}$ is a local Calder\'on-Zygmund kernel adapted to the dilations $\del^i$. By Theorem \ref{Thm8.1} and the uniform boundedness of the functions $\sigma_{i,\ell_i}$ in every Schwartz norm, we have uniform bounds
$
\|g*\cH_{s_i}\|_1\leq \|g*\cH_{s_i}\|_{\h^1_{\del^i}}\leq C\|g\|_{\h^1_{\del^i}}\leq C'\sum_{j=1}^d\|g*\cK^\flat_{i,j}\|_1
$. 
Then
\beas
\|f*H_\mbs\|_1&=\big\|(f*\cH_{s_1}*\cdots*\cH_{s_{n-1}})*\cH_{s_n}\big\|_1
=\big\|(f*\cH_{s_1}*\cdots*\cH_{s_{n-1}})*\cH_{s_n}\big\|_1\\&\leq C\sum_{j=1}^d\big\|(f*\cH_{s_1}*\cdots*\cH_{s_{n-1}})*\cK^\flat_{n,j}\|_1\ .
\eeas
In this way we have replaced, in each summand, the term $\cH_{s_{n}}$ by one of the $\cK^\flat_{n,j}$. We can now continue the chain of inequalities by iteration of the above argument. At the next step, we majorize each term $\big\|(f*\cH_{s_1}\cdots*\cH_{s_{n-1}})*\cK^\flat_{n,j}\|_1$ with $\sum_{j'\ge0}\big\|(f*\cH_{s_{1}}*\cdots*\cH_{s_{n-2}})*\cK^\flat_{n-1,j'}*\cK^\flat_{n,j}\|_1$.
Ultimately we obtain the inequality
$
\|f*H_\mbs\|_1\leq C\sum_{J\in\{1,\dots,d\}^n}\|f*R_J\|_1\ ,
$
uniformly in $\mbs$. It remains to apply Khintchin's inequality to obtain that
$
\|S''f\|_1\leq C\sum\nolimits_{J\in\{1,\dots,d\}^n}\|f*R_J\|_1\ .\qquad\qquad\qquad\qedhere
$
\end{proof}

\begin{corollary}\label{Cor8.10}
If $f\in L^{1}(\R^{d})$ and $\|\KK*f\|_{L^{1}}\leq C\|f\|_{L^{1}}$ for every multi-norm kernel $\KK\in \PP_{0}$ where $C$ depends only on the constants $C_{\alpha,N}$ from Definition \ref{Def3.1}, then $f\in \h_{\bfE}^{1}(\R^d)$ and $\|f\|_{\h_{\bfE}^{1}}\lesssim \|f\|_{L^{1}}$.
\end{corollary}

\medskip

\section{Atomic decomposition of functions in $\h_{\bfE}^{1}(\R^d)$}\label{Sec9}

In this section we define {\it multi-norm atoms} and prove that every function $f\in \h_{\bfE}^{1}(\R^d)$ can be decomposed $f=\sum_{i}\lambda_{1}a_{i}$ where the functions $\{a_{i}\}$ are \textit{atoms} and $\sum_{i}|\lambda_{i}|\leq C\|f\|_{\h_{\bfE}^{1}}$. Atoms are defined in Section \ref{Sec9.1}, and the decomposition is proved in Section \ref{Sec9.2}. 

\subsection{Multi-norm atoms}\label{Sec9.1}
\quad

\smallskip

Multi-norm atoms are functions $a$ whose support involves non-isotropic dyadic rectangles and whose size is controlled by certain $L^{2}$ estimates. Since $\h_{\bfE}^{1}(\R^d)$ is a local  Hardy space, we require that atoms be supported in a ball of (essentially) unit radius. Each atom $a$ is associated to a parameter $\tau >0$ and to a marked partition  $S\in \SS_{\bfE}^{*}=\SS_{\bfE}\cup\{\emptyset\}$. For $S\in\SS_{\bfE}$ this determines the admissible scales of $a$ in each variable $\x_i$ and the required cancellation properties. For $S=\emptyset$ atoms have unit size and no required cancellation.  

Let $S=\big\{(A_{1},k_{1}), \ldots, (A_{s},k_{s})\big\}\in \SS_{\bfE}$.  Recall from \eqref{4.6xyz} that  $\D_{S}$ is the set of dyadic rectangles $\bfR^{L}=\prod_{r=1}^{s}\bfQ_{A_r}^{\ell_{k_r}}$ with multi-scale $L=(\ell_{1}, \ldots, \ell_{n})\in\LL_{S}$, and if $L\in \LL_{S}$ then $\D_{L}$ is the set of dyadic rectangles with multi-scale $L$. If $\Omega\subseteq\R^{d}$ is a bounded set let $\D_{S}(\Omega)=\big\{\bfR\in\D_{S}:\bfR\subseteq\Omega\big\}$.

\begin{definition}\label{Def9.1}\rm{
A \textit{multi-norm $\h_{\bfE}^{1}$-atom} is a function $a:\R^{d}\to\C$ associated with a pair $(S,\tau)$ where $S\in \SS_{\bfE}^{*}$ and $\tau>0$. There are three kinds, depending on the number $s\geq 0$ of marked elements in $S$. 

\begin{enumerate}[\rm (1)]

\item \label{A1} 
If $s=0$ then $S=\emptyset$ and an \textit{$(\emptyset,\tau)$-atom} is a function $a$ with $\|a\|_{L^{2}}\leq 1$ supported on $2^{\tau} \circ\mathbf Q_{P}$ for some unit cube $\mathbf \bfQ_{P}=\big\{\x\in\R^{d}:p_{j}\leq x_{j}<p_{j}+1,\,1\leq j \leq d\big\}$ where $(p_{1}, \ldots, p_{n})\in \Z^{d}$. No cancellation property of $a$ is required. 

\smallskip

\item \label{A2} 
If $s=1$ and $k\in \{1, \ldots, n\}$ is the only dotted entry for $S$ then an \textit{$(S,\tau)$-atom} is a function $a$ supported on $2^{\tau} \circ\bfR$ for some $\bfR\in\D_{S}$ with $\|a\|_{L^{2}}\leq |\bfR|^{-\frac{1}{2}}$ and cancellation in $\x_{k}$; \textit{i.e.} $\int_{\R^{d_{k}}}a( \ldots, \x_{k},\ldots )\,d\x_{k}=0$, identically in the variables $\x_i$ with $i\ne k$.

\smallskip

\item \label{A3} 
If $s\geq 2$ and $S=\big\{(A_{1},k_{1}), \ldots, (A_{s},k_{s})\big\}$ then an \textit{$(S,\tau)$-atom} 
is a function $a$, an associated open set $\Omega$ contained in some rescaled unit cube $2^{\tau\circ\bfQ_{P}}$, and a decomposition $a=\sum_{\bfR\in \D_{S}(\Omega)}a_{\bfR}$ into {\it pre-atoms} satisfying: 

\smallskip

\begin{enumerate}[(a)]
\item \label{atomsC}

\goodbreak 
$a_{{\bfR}}$ is supported in $2^{\tau}\circ\bfR$;
\smallskip

\item \label{atomsD}
$a_{{\bfR}}$ has cancellation in each dotted variable $\x_{k_{r}}$, \textit{i.e.}
$ \int_{\R^{d_{k_{r}}}}a_{\bfR}(\ldots,\x_{k_{r}}, \ldots)\,d\x_{k_{r}}=0$, identically in the other variables;

\smallskip

\item\label{atomsE}
if $\D_{S}(\Omega)=\bigcup_{j\in\Z}\DD_{j}$ is a partition and if $a_{\mathcal{D}_j}=\sum_{\bfR \in \mathcal{D}_j}a_{\bfR}$ then $\sum_{j \in \mathbb{Z}}\|a_{\mathcal{D}_j}\|_{L^2}^{2} \leq |\Omega|^{-1}$; in particular, $\|a\|_{L^2} \leq |\Omega|^{-\frac{1}{2}}$.
\end{enumerate}

\end{enumerate}
}
\end{definition}

\begin{remarks}\label{Rem11.2}\quad
{\rm
\begin{enumerate}[1.]

\item The $(\emptyset,\tau)$-atoms having no cancellation appear in the atomic decomposition in \textit{any} local Hardy spaces. See for example \cite{MR523600}.

\smallskip

\item When $s=1$ and $\x_{k}$ is the only dotted variable, the dyadic rectangles $\bfR\in\D_S$ are dyadic cubes in $\R^d$ relative to the dilations $\del^k$ in \eqref{2.5aa}, and the atoms defined in \eqref{A2} are {\it special} Coifman-Weiss atoms for the associated local non-isotropic Hardy space. The adjective ``special'' is due to the cancellation in the single variable $\x_k$  assumed here, which is stronger than cancellation in all variables together, which required in the standard Coifman-Weiss theory. See for example \cite{MR447954}.

\smallskip

\item \label{Rem11.2.3}If $s=n$, the atoms defined in \eqref{A3} are {\it standard} Chang-Fefferman atoms for the product Hardy space relative to the decomposition $\bR^d=\bigoplus_{i=1}^n\bR^{d_i}$ with dilations \eqref{1.1}, with the restriction that only pre-atoms at scale $2^{-L}$ with $L\in \LL_{S_p}$ are admitted. See for example \cite{MR584078} and \cite{MR658542}.

\smallskip

\item\label{Rem11.2.4} If $S=\big\{(A_{1},k_{1}), \ldots, (A_{s},k_{s})\big\}$ with $2\leq s\leq d-1$, the atoms defined in \eqref{A3} are {\it special} Chang-Fefferman atoms for the product Hardy space relative to the coarser decomposition $\bR^d=\prod_{r=1}^s\bR^{A_r}$ with $\bR^d$ endowed of the dilations $\del^S$ defined in \eqref{3.11.6}. In addition to the restriction in the admissible scales, cancellation of pre-atoms is requested in each dotted variable $\x_{k_r}$ as opposed to cancellation in each full block of variables $\x_{A_r}$, as requested to product atoms.
\item Notice that for $2\le s\le n-1$ (which forces $n\ge3$) the definition of $(S,\tau)$-atom actually depends on the marked partition $S$ and not simply on the set $D$ of dotted entries.
\smallskip

\item The proof of Theorem \ref{Thm9.3} provides an atomic decomposition of a function $f\in\h^1_\bfE(\R^d)$ where the atoms of type \eqref{A3} are sums of pre-atoms depending only on {\it maximal} rectangles contained in $\Omega$. This is only apparently a stronger property than that required in Definition \ref{Def9.1}, because, given an atom $a$ as in Definition \ref{Def9.1}\eqref{A3}, it is always possible to associate to each pre-atom $a_\bfR$  a same maximal rectangle $\bfR^\sharp$ containing it and regard the sum of the pre-atoms associated to the same $\bfR^\sharp$ as a unique atom $a'_{\bfR^\sharp}$, so that $a=\sum a'_{\bfR^\sharp}$. The possibility of allowing only maximal rectangles may be useful in certain proofs where non-comparability of the different rectangles produces a simplification.
\end{enumerate}
}
\end{remarks}

\subsection{Atomic decomposition in $\h_{\bfE}^{1}(\R^d)$}\label{Sec9.2}\quad

\smallskip
In this section we prove that every $f\in \h_{\bfE}^{1}(\R^d)$ admits an {\it atomic decomposition}. The precise statement is  given in the following theorem. 
\begin{theorem}\label{Thm9.3}
Let $\tau>0$. There exists $C>0$ so that if $f\in \h_{\bfE}^{1}(\R^d)$ the for all $S\in\SS_{\bfE}^{*}$, there exists a sequence $\big\{a_{j,S}:j\in \Z\big\}$ of $(S,\tau)$-atoms, and a sequence $\big\{\lambda_{j,S}\in\C:j\in \Z\big\}$ so that 
\begin{enumerate}[\rm(a)]
\item\label{Thm9.3a}
$f=\sum_{S\in \SS_{\bfE}^{*}}\sum_{j\in \Z}\lambda_{j,S}\,a_{j,S}$ with convergence in the distributional sense;

\smallskip

\item\label{Thm9.3b} 
$\sum_{S\in \SS_{\bfE}^{*}}\sum_{j\in \Z}|\lambda_{j,S}|\leq C\|f\|_{\h^1_\bfE}$.
\end{enumerate}
\end{theorem}
\label{pagetau}
The proof makes use of reproducing formulas from Section \ref{Sec4}, constructed via tensor products from functions $\big\{\varphi_{i,\ell_i}:\ell_{i}\in \N\big\}\subseteq\SS(\R^{d_{i}})$ that are uniformly bounded in every Schwartz norm and have integral~$1$. We shall also suppose that $\varphi_{i,\ell_{i}}$ is supported in $\epsilon\circ B_{i}(1)\subseteq \R^{d_{i}}$ where $\epsilon>0$ depends on $\tau$ is chosen so that condition \eqref{tau2} of the following are satisfied.\begin{enumerate}[(A)]
\item\label{tau1} 
If $\Psi_L, \widetilde\Psi_{L}$ are as in equation \eqref{3.4ee}, then
by Lemma \ref{Lem4.2}\eqref{Lem4.2f} the supports of $\Psi_{L}^{(-L)}$ and $\widetilde \Psi_{L}^{(-L)}$ are also contained supported in an $\eps$-neighborhood of the origin in $\R^d$;

\smallskip

\item \label{tau2}
If $\Psi_L, \widetilde\Psi_{L}$ are as in equation \eqref{3.4ee} and if $f$ is supported in $2^\eps B(1)$ then $
\widetilde\Psi_{L}^{(-L)}*\Psi_{L}^{(-L)}*f$ is supported in $2^{\tau}\circ B(1)$ for every $L\in\LL$. If $f$ is supported in a dyadic rectangle $\bfR$ then $
\widetilde\Psi_{L}^{(-L)}*\Psi_{L}^{(-L)}*f$ is supported in $2^{\tau}\circ \bfR$. 
\end{enumerate}
The existence of such an $\eps$ only depends on continuity of homogeneous norms.

Using the partition $\big\{\LL_{S}:S\in \SS_{\bfE}^{*}\big\}$ of $\LL$, decompose $f\in \h_\bfE^{1}(\R^d)$ as 
\be\label{9.2}
f=\sum_{S\in \SS_{\bfE}^{*}}f_S \qquad\text{where}\quad
\begin{cases}
f_\emptyset=\Psi_{0}*\widetilde{\Psi}_{0}*f,&\\ 
f_S=\sum_{L \in\LL_{S}}\Psi_{L}^{(-L)}*\widetilde{\Psi}_{L}^{(-L)}*f&\text{ for }S\in\SS_{\bfE}\ .\end{cases}
\ee 
We then show that for each $S$ there is a decomposition $f_S=\sum_{j\in \Z}\lambda_{j,S}\,a_{j,S}$ with $(S,\tau)$-atoms $a_{j,S}$ and $\sum_{j\in \Z}|\lambda_{j,S}|\leq C\|f\|_{\h^1_\bfE}$. Since there are finitely many marked partitions, this gives the inequality in part \eqref{Thm9.3b} of the Theorem. With the only excetion of the case $S=\emptyset$, the proof of Theorem \ref{Thm9.3} makes use of the characterization of $\h_{\bfE}^{1}(\R^d)$ using the Plancherel-P\'olya type square function $S_{\Psi}^{{\rm PP}}f(\x)=\big(\sum_{L\in \LL}\big\vert f*\Psi_{L}^{(-L)}\big\vert_{{\rm PP}}(\x)^{2}\big)^{1/2}$ defined in \eqref{8.1}, with $\Psi_L$ satisfying \eqref{tau1} and \eqref{tau2}. 

\begin{proof}[Proof of Theorem \ref{Thm9.3}]

Let $f\in \h_{\bfE}^{1}(\R^d)$. By Corollary \ref{Cor8.5} we may assume that $f$ is supported in $2^{\epsilon}\circ B(1)$ with $\eps>0$ satisfying \eqref{tau2} above, so that every $f_S$ in \eqref{9.2} and $S_{\Psi}^{{\rm {\rm PP}}}f$ are supported in $2^{\tau} \circ B(1)$. For $i\in\Z$ define  $\Omega^{i}=\big\{\x :S_{\Psi}^{{\rm PP}}f(\x)>2^{i}\big\}\subseteq 2^{\tau}\circ B(1)$ and $\Omega^{i}_{s}=\big\{\x :\MM_{s}[\1_{\Omega^{i}}](\x)>1/2 \big\}$, $\MM_s$ being the local maximal operator in \eqref{Ms}.
The differences $\Omega^{i}\setminus\Omega^{i+1}$ partition $\supp (S_{\Psi}^{{\rm PP}}f)$ and 
\be\label{10.1.5}
 \|S_{\Psi}^{{\rm PP}}f\|_{L^{1}}\sim\sum\nolimits_{i\in \Z}2^{i}\,|\Omega^{i}\setminus\Omega^{i+1}|\sim \sum\nolimits_{i\in \Z}2^{i}\,|\Omega^{i}|\sim\sum\nolimits_{i\in \Z}2^{i}\,|\Omega^{i}_s|,
\ee
the last equivalence following by $L^2$-boundedness of the maximal operator. 
For $S\in \SS_{\bfE}^{*}$, we let $s=0$ if $S=\emptyset$ and $s$ equal to the number of dotted entries if $S\in \SS_{\bfE}$. 

\medskip

\noindent\textbf{Case 1: $s=0$}.  In this case $\LL_{\emptyset}=\big\{0\big\}$, $S=\emptyset$, and $\D_{\emptyset}=\big\{\bfR_{P}^{0}:P\in \Z^{d}\big\}$ is the set of dyadic unit cubes in $\R^d$. We have $f_\emptyset(\x)=\Psi_{0}*\widetilde\Psi_{0}*f$. Write $\Psi_{0}*\widetilde\Psi_{0}*f=\la_0 \,a_0$ where 
$\lambda_0=\|f\|_1$ and  $a_0=\|f\|_1^{-1}\Psi_{0}*\widetilde\Psi_{0}*f$.
Then $a_0$ is a multiple (depending only on $\tau$) of a $(\emptyset,\tau)$-atom. In fact, $\supp a_0\subset 2^{\tau}\circ B(1)$ by \eqref{tau2}, and $\|a_0\|_{L^{2}}\leq \|\widetilde\Psi_0*\Psi_0\|_2\le C$. Moreover
$0\le\lambda_0\le\|f\|_{\h^1_\bfE}$.
This completes the proof in case $s=0$.

\medskip

\noindent \textbf{Case 2: $s=1$.}  Adapting the notation in \eqref{9.1}, we define
\bea\label{BBiS}
\BB_{S}^{i}&=\Big\{\bfR\in \mathbb{D}_{S}:\text{$|\bfR\cap\Omega^{i}|>\frac{1}{2}|\bfR|$ and $|\bfR\cap \Omega^{i+1}| \leq \frac{1}{2} |\bfR|$}\Big\},\\
\BB_{S}^{i,\sharp}&=\big\{\bfR:\text{$\bfR$ is maximal in $\BB_{S}^{i}$ under inclusion}\big\},\\
 \BB_{S}^{i,\flat}(\bfR)&=\big\{\bfR'\in \BB_{S}^{i}:\bfR'\subseteq\bfR\big\}\qquad\text{ for }\bfR\in\BB_{S}^{i,\sharp}.
\eea 
If $\bfR\in\BB^i_S$, then $\bfR\subseteq \Omega^i_s$. Moreover, each $\bfR\in\D_{S}$ belongs to a unique $\BB_{S}^{i}$, so that the $\BB_{S}^{i}$ form a partition of $\D_{S}$. Moreover, for $s=1$ with $k$ denoting the unique dotted entry, the rectangles $\bfR\in \D_{S}$ are dyadic cubes of $\R^d$ relative to the dilations $\del^k$. Because of the intersection properties of dyadic cubes, the elements of $\BB_{S}^{i,\sharp}$ are pairwise disjoint sets. Hence the $\BB_{S}^{i,\flat}(\bfR)$ are disjoint subsets of of $\BB_{S}^{i,\sharp}$ and form a partition of $\BB_{S}^{i}$. Denoting by $L_{\bfR'}$ the element of $\LL_S$ such that $\bfR'\in\D^{L_{\bfR'}}$, we obtain the following identity:
\bea\label{11.1}
f_S(\x)&=\sum_{L \in \LL_{S}}\Psi_{L}^{(-L)}*\widetilde{\Psi}_{L}^{(-L)}*f(\x)
= \sum_{L \in \LL_{S}}\sum_{\bfR\in\D_{L}}\int_{\bfR}\widetilde{\Psi}_{L}^{(-L)}(\x-\y)(\Psi_{L}^{(-L)}*f)(\y)\,d\y\\
&=
\sum_{i\in \Z}\sum_{\bfR\in\BB_{S}^{i,\sharp}}\sum_{\bfR'\in\BB_{S}^{i,\flat}(\bfR)}
\int_{\bfR' }\widetilde{\Psi}_{L_{\bfR'}}^{(-L_{\bfR'})}(\x-\y)\Psi_{L_{\bfR'}}^{(-L_{\bfR'})}*f(\y)\,d\y.
\eea

\begin{lemma}\label{Lem10.4X}
Let $\bfR\in \BB_{S}^{i}$. Then $\1_{\bfR}(\x)\leq 2\MM_{s}[\1_{\Omega_{s}^{i}\setminus\Omega^{i+1}}](\x)$.
\end{lemma}
\begin{proof}
Since $\bfR\subseteq\Omega_{s}^{i}$ we have 
$\bfR=(\bfR\cap\Omega^{i+1}) \cup \big(\bfR\cap (\Omega_{s}^{i}\setminus\Omega^{i+1})\big)$. Then
$$
|\bfR|=\big|\bfR\cap\Omega^{i+1}\big| +\big|\bfR\cap (\Omega_{s}^{i}\setminus\Omega^{i+1})\big|< \frac{1}{2}|\bfR| +\big|\bfR\cap (\Omega_{s}^{i}\setminus\Omega^{i+1})\big|
$$
and so $\big|\bfR\cap (\Omega_{s}^{i}\setminus\Omega^{i+1})\big|>\frac{1}{2}|\bfR|$. This means that if $\x\in \bfR$ then $\MM_{s}[\1_{\Omega_{s}^{i}\setminus \Omega^{i+1}}](\x)>\frac{1}{2}$.
\end{proof}

 For $i\in \Z$ and $\bfR\in \BB_{S}^{i,\sharp}$ put
\bea\label{10.4}
\lambda_{i,\bfR}&= \big\vert\bfR\big\vert^{\frac{1}{2}}\Big\vert\Big\vert\Big(\sum\nolimits_{\bfR'\in \BB_{S}^{i,\flat}(\bfR)}\big\vert\Psi_{L_{\bfR'}}^{(-L_{\bfR'})}*f\big\vert^2 \,\1_{\bfR'}\Big)^{\frac{1}{2}}\Big\vert\Big\vert_{2},\\
a_{i,\bfR}(\x)&= \lambda_{i,\bfR}^{-1}\sum\nolimits_{\bfR'\in\BB_{S}^{i,\flat}(\bfR)}\widetilde\Psi_{L_{\bfR'}}^{(-L_{\bfR'})}*\left[\1_{\bfR'}\big(\Psi_{L_{\bfR'}}^{(-L_{\bfR'})}*f\big)\right](\x).
\eea
We show that $a_{i,\bfR}$ is a multiple of a $(S,\tau)$-atom, where the multiple depends only on $\tau$ and on the $L^{2}$ bound of the square function operator $h \longmapsto \big(\sum_{L' \in \LL_{S}}|\widetilde{\Psi}_{L'}^{(-L')}*h|^2\big)^{\frac{1}{2}}$. It follows from properties of $\epsilon$ and $\tau$ in \eqref{tau1} and \eqref{tau2} on page \pageref{pagetau} that $a_{i,\bfR}$ is supported in $2^{\tau}\bfR$, since the rectangles $\bfR'$ are contained in $\bfR$. The cancellation condition follows since $\widetilde{\Psi}_{L}$ has cancellation in the variable $\x_{k}$ if $L\in\LL_S$. To show that $\|a_{i,\bfR}\|_{L^2}\leq C|\bfR|^{-\frac{1}{2}}$ we use duality and the Cauchy-Schwarz inequality as follows. Let $\big[\widetilde{\Psi}_{L_{\bfR'}}^{(-L_{\bfR'})}\big]^\vee(\y)=\widetilde{\Psi}_{L_{\bfR'}}^{(-L_{\bfR'})}(-\y)$. Then

\bea\label{10.5}
\lambda_{i,\bfR}\|a_{i,\bfR}&\|_{L^2}
=
\sup_{\|h\|_{2} \leq 1} \Big|\int_{\mathbb{R}^d} h(\x)\Big(\sum_{\bfR'\in \BB_{S}^{i,\flat}(\bfR)}\int_{\bfR' }\widetilde{\Psi}_{L_{\bfR'}}^{(-L_{\bfR'})}(\x-\y)\Psi_{L_{\bfR'}}^{(-L_{\bfR'})}*f(\y)\,d\y\Big)\,d\x\Big|\\
&=
\sup_{\|h\|_{2} \leq 1} \Big|\sum_{\bfR'\in\BB_{S}^{i,\flat}(\bfR)}\int_{{\bfR' }}\big[\widetilde{\Psi}_{L_{\bfR'}}^{(-L_{\bfR'})}\big]^\vee*h(\y)\,\,\Psi_{L_{\bfR'}}^{(-L_{\bfR'})}*f(\y)\,d\y\Big\vert
\\
&\leq 
\Big\|\Big(\sum_{\bfR'\in \BB_{S}^{i,\flat}(\bfR)}|\Psi_{L_{\bfR'}}^{(-L_{\bfR'})}*f|^2\1_{\bfR'}\Big)^{\frac{1}{2}}\Big\|_{2}
\sup_{\|h\|_{2} \leq 1}\Big\|\Big(\sum_{ \bfR'\in\BB_{S}^{i,\flat}(\bfR)}|\widetilde{\Psi}_{L_{\bfR'}}^{(-L_{\bfR'})}*h|^2\1_{\bfR'}\Big)^{\frac{1}{2}}\Big\|_{2}
\\
&=
\lambda_{i,\bfR}\big\vert\bfR\big\vert^{-\frac{1}{2}}\sup_{\|h\|_{2} \leq 1}\Big\|\Big(\sum_{L_{\bfR'}\in \LL_{S}}|\Psi_{L_{\bfR'}}^{(-L_{\bfR'})}*h|^2\Big)^{\frac{1}{2}}\Big\|_{2}
\leq C\,\lambda_{i,\bfR}\big\vert\bfR\big\vert^{-\frac{1}{2}}.
\eea

To complete the case $s=1$ it remains to show that $\sum_{i\in\Z}\sum_{\bfR\in\BB_{S}^{i,\sharp}}|\lambda_{i,\bfR}|\leq C \|f\|_{\h_\bfE^{1}}$. 
We use the inequality $\sum_{\bfR\in\BB_{S}^{i,\sharp}} |\bfR|\le|\Omega^i_s|$, due to the fact that maximal rectangles are essentially disjoint. From this, Lemma \ref{Lem10.4X}, the fact that $\big|\Psi_{L_{\bfR'}}^{(-L_{\bfR'})}*f\big|_{\rm PP}$ is constant, and the $L^2$-boundedness of the maximal operator, we obtain, for every $i$, 
\bea\label{10.5.5}
\sum_{\bfR\in\BB_{S}^{i,\sharp}} |\lambda_{i,\bfR}|&
\leq\Big(\sum_{\bfR\in\BB_{S}^{i,\sharp}} |\bfR| \Big)^{\frac{1}{2}}
\Big(\sum_{\bfR'\in\BB_{S}^{i}}\int_{\bfR' }\,\big|\Psi_{L_{\bfR'}}^{(-L_{\bfR'})}*f(\x)\big|^2\,d\x\Big)^\frac12\\
&\le
|\Omega^i_s|^\frac12
\Big(\sum_{\bfR'\in\BB_{S}^{i}}\int_{\bfR' }\,\big|\Psi_{L_{\bfR'}}^{(-L_{\bfR'})}*f\big|_{\rm PP}^2(\x)\,d\x\Big)^\frac12\\
&\leq 2|\Omega^i_s|^\frac12
\Big( \sum_{\bfR'\in\BB_{S}^{i}}\int_{\bfR'}\bfM_{S}[\1_{\Omega^{i}_{s}\setminus\Omega^{i+1}}](\x)\,\big|\Psi_{L_{\bfR'}}^{(-L_{\bfR'})}*f\big|_{\rm PP}^2(\x)\,d\x\Big)^{\frac{1}{2}}\\
&\le 2|\Omega^i_s|^\frac12
\Big( \sum_{\bfR'\in\BB_{S}^{i}}\int_{\bfR'}\Big(\bfM_{S}\big[\1_{\Omega^{i}_{s}\setminus\Omega^{i+1}}\big|\Psi_{L_{\bfR'}}^{(-L_{\bfR'})}*f\big|_{\rm PP}\big](\x)\Big)^2\,d\x\Big)^{\frac{1}{2}}\\
&\le C|\Omega^i_s|^\frac12
\Big( \sum_{\bfR'\in\BB_{S}^{i}}\int_{\Omega^{i}_{s}\setminus\Omega^{i+1}}\big|\Psi_{L_{\bfR'}}^{(-L_{\bfR'})}*f\big|_{\rm PP}^2(\x)\,d\x\Big)^{\frac{1}{2}}
\eea
For every $L\in\LL_S$, almost every $\x\in\R^d$ is contained in at most one rectangle $\bfR'\in \D_{L}$, so that $\sum_{\bfR'\in\BB_{S}^{i}}\big|\Psi_{L_{\bfR'}}^{(-L_{\bfR'})}*f\big|_{\rm PP}^2(\x)\le (S_\Psi^{\rm PP}f)^2(\x)$. Since $S_\Psi^{\rm PP}f(\x)\le 2^{i+1}$ for $\x\not\in\Omega^{i+1}$, we conclude that
$ |\lambda_{i,\bfR}|\le C|\Omega^i_s|2^i$.
Summing over $i\in\Z$ the conclusion in the case $s=1$ follows from \eqref{10.1.5}.

\medskip

\noindent \textbf{Case 3: $s\geq 2$.}  Given a marked partition $S$ with $s\ge2$, we will write $f_{S}=\sum_{i\in \Z}\lambda_{i,S}\,a_{i,s}$ where the atom $a_{i,S}$ is associated to the set $\Omega_{s}^{i}$ and can be decomposed as in Definition \ref{Def9.1} \eqref{A3}. To do this we modify the definitions in \eqref{BBiS}. Again let $\BB_{S}^{i}=\big\{\bfR\in \mathbb{D}_{S}:\text{$|\bfR\cap\Omega^{i}|>\frac{1}{2}|\bfR|$ and $|\bfR\cap \Omega^{i+1}| \leq \frac{1}{2} |\bfR|$}\big\}$, but instead of using $\BB_{S}^{i,\sharp}$ (the maximal rectangles in $\BB_{S}^{i}$), we use $\widetilde\BB_{S}^{i,\sharp}=\D_{S}^{\sharp}(\Omega_{s}^{i})$, the maximal rectangles in $\Omega_{s}^{i}$, ordered by inclusion. A rectangle of $\ \BB_{S}^{i}$ may be contained in more than one element of $\D_{S}^{\sharp}(\Omega_{s}^{i})$, but we can partition $\BB_{S}^{i}$ into disjoint subsets $\widetilde\BB_{S}^{i,\flat}(\bfR)$ indexed by rectangles $\bfR\in \D_{S}^{\sharp}(\Omega_{s}^{i})$, so that $\bfR'\in \widetilde\BB_{S}^{i,\flat}(\bfR)\Longrightarrow \bfR'\subseteq \bfR$. Then the decomposition corresponding to \eqref{11.1} is
\bea\label{9.10z}
f_{S}(\x)=\sum_{i\in \Z}\sum_{\bfR\in \D_{S}^{\sharp}(\Omega_{s}^{i})}\sum_{\bfR'\in \widetilde\BB_{S}^{i,\flat}(\bfR)}\int_{\bfR' }\widetilde{\Psi}_{L_{\bfR'}}^{(-L_{\bfR'})}(\x-\y)\Psi_{L_{\bfR'}}^{(-L_{\bfR'})}*f(\y)\,d\y.
\eea

Let
\be\label{eq:lambda}
\lambda_{i,S}=|\Omega^{i}_s|^{\frac{1}{2}}\,\Big\|\Big(\sum_{\bfR'\in\BB_{S}^{i}}|\Psi_{L_{\bfR'}}^{(-L_{\bfR'})}*f|^2\1_{\bfR'}\Big)^{\frac{1}{2}}\Big\|_2,
\ee
\be\label{eq:a}
a_{i,S}(\x)=\lambda_{i,S}^{-1}\sum_{\bfR'\in\BB_{S}^{i}}\int_{{\bfR'}}\widetilde{\Psi}_{L_{\bfR'}}^{(-L_{\bfR'})}(\x-\y)\Psi_{L_{\bfR'}}^{(-L_{\bfR'})}*f(\y)\,d\y\ .
\ee
so that, by \eqref{9.10z},
 $\sum_{L \in\LL_{S}}\Psi_{L}^{(-L)}*\widetilde{\Psi}_{L}^{(-L)}*f(\x)
=\sum_{i\in\Z}\lambda_{i,S} \,a_{i,S}(\x)$. 
We prove that each $a_{i,S}$ is a $(S,\tau)$-atom associated to $\Omega^i_s$. Define

\be\label{10.6}
a_{i,S,\bfR}=\lambda_{i,S}^{-1}\sum_{\bfR'\in\widetilde\BB_{S}^{i,\flat}(\bfR)}\int_{{\bfR'}}\widetilde{\Psi}_{L_{\bfR'}}^{(-L_{\bfR'})}(\x-\y)\Psi_{L_{\bfR'}}^{(-L_{\bfR'})}*f(\y)\,d\y
\ee
so that $a_{i,S}(\x)=\sum_{\bfR\in \D_{S}^{\sharp}(\Omega_{s}^{i})}a_{i,S,\bfR}(\x)$.
We prove that the $a_{i,S,\bfR}$ satisfy the requirements for being pre-atoms of $a_{i,S}$.
Condition \eqref{atomsC} in Definition~\ref{Def9.1} on the support follows from the formula in \eqref{eq:lambda} and the cancellation condition \eqref{atomsD} is inherited from the factors $\widetilde\Psi_{L_{\bfR'}}$, which have cancellation in the dotted variables in $S$ because $L_{\bfR'}\in\LL_S$. To prove property \eqref{atomsE} we use a duality argument as in \eqref{10.5}. Given a partition $\{\DD_j\}_{j\in J}$ of $\D_{S}^{\sharp}(\Omega_{s}^{i})$, we have
\bea\label{10.9}
\lambda_{i,S}\Big\|\sum_{\bfR\in\DD_j}a_{i,S,\bfR}\Big\|_2&=\sup_{\|h\|_{2} \leq 1} \Big|\sum_{\bfR\in\DD_j}\sum_{\bfR'\in\widetilde \BB^{i}_S(\bfR)}\int_{\bfR'}\big[\widetilde{\Psi}_{L_{\bfR'}}^{(-L_{\bfR'})}\big]^\vee*h(\y)\Psi_{L_{\bfR'}}^{(-L_{\bfR'})}*f(\y)\,d\y\Big|\\
&\leq 
C\,\bigg[\int_{\R^d}\sum_{\bfR\in\DD_j}\sum_{\bfR'\in\widetilde \BB^{i}_S(\bfR)}|\Psi_{L_{\bfR'}}^{(-L_{\bfR'})}*f(\x)|^2\,d\x\bigg]^\frac12 \ ,
\eea
and so
\bea\label{10.8}
\sum_{j\in J}\Big\|\sum_{\bfR\in\DD_j}a_{i,S,\bfR}\Big\|_2^2
&\leq
C^{2}\lambda_{i,S}^{-2}\sum_{\bfR\in\D_{S}^{\sharp}(\Omega_{s}^{i})}\sum_{\bfR'\in\widetilde \BB^{i}_S(\bfR)}\int_{\bfR'}|\Psi_{L_{\bfR'}}^{(-L_{\bfR'})}*f
(\y)|^{2}d\y\\
&=C^{2}\lambda_{i,S}^{-2}\sum_{\bfR'\in \BB^{i}_S}\int_{\bfR'}|\Psi_{L_{\bfR'}}^{(-L_{\bfR'})}*f(\y)|^{2}d\y
\le C^{2}\big|\Omega^{i}\big|^{-1}\ .
\eea
Thus each $a_{i,S,\bfR}$ satisfies \eqref{atomsE}.  To finish the proof of the theorem in Case 3 we must show that $ \sum_{i\in\Z}|\lambda_{i,S}| \leq C\|S_{\Psi}^{{\rm PP}}f\|_{L^{1}}$. Starting from the formula
$
\la_{i,S}^2=|\Omega^{i}_s|\sum_{\bfR'\in\BB_{S}^{i}}\int_{\bfR'}|\Psi_{\bfR'}^{(-L_{\bfR'})}*f(\x)|^2\,d\x
$
and repeating the same steps in \eqref{10.5.5} we deduce that $\la_{i,S}\le C2^i|\Omega^i_s|$
and derive the conclusion summing over $i\in\Z$.
\end{proof}

\section{A Journ\'e-type covering lemma}\label{journe}

In this section we obtain geometric results related to Journ\'e's covering lemma. This was first stated in \cite{MR0826486} and then further developed by Pipher \cite{MR0860666}, Carbery and Seeger \cite{MR1072104}, and Cabrelli {\it et.al.} \cite{MR2247912}. The content of this section is motivated by these previous works and, in particular, our results are essentially equivalent to those in the geometric part of  \cite{MR1072104}. The natural setting for these results is a general product structure and for this reason  we temporarily abandon our multi-norm setting. The results we obtain will be used in Section \ref{Sec13} where the relevant  product structures on $\R^d$ will depending on a marked partition $S=\big\{(A_{1},k_{1}), \ldots, (A_{s},k_{s})\big\}$, with factors $\R^{A_1}\times\cdots\times\R^{A_s}$ and  dilations $\del^S_t$ defined in~\eqref{3.11.6}. In principle we could also remove the locality assumption, but this is not so relevant and we prefer to focus on the local setting.
 
To establish the basic terminology in a general product structure, we denote the ambient space by $\R^\nu=\R^{\nu_1}\times\cdots\times\R^{\nu_s}$ and by $\{\del_t\}_{t>0}$ a family of dilations  leaving each subspace $\R^{\nu_r}$ invariant. Following the definitions in Section \ref{Sec7.1}, we denote by $\bfQ^{\ell_r}_{r}$, or simply $\bfQ_r$, a dyadic cube in $\R^{\nu_r}$ and by $\bfR^L=\prod_{r=1}^s\bfQ^{\ell_r}_{r}$ or simply $\bfR=\prod_{r=1}^s\bfQ_r$, a dyadic rectangle in $\R^\nu$. For $\m\in\N^s$ we denote by $2^\m\bfR$ the $2^\m$ dyadic enlargement of $\bfR$ as defined in Section \ref{Sec7.4.1} .

\subsection{Enlargements and embeddedness factors}\label{Sec11.1ww}\quad

\smallskip

Let $\Omega\subseteq\R^{\nu}$ be a set of finite measure.  The {\it standard enlargements} $\big\{\Omega^{(j)}:j\geq 0\big\}$ of $\Omega$ are defined inductively. For $j=0$ set $\Omega^{(0)}=\Omega$. Let $\MM_{s}$ denote the dyadic strong maximal function. Then for $j\geq 1$ set
\be\label{Omegaj}
\Omega^{(j)}=\big\{\x\in \R^{\nub}:\MM_{s}[\1_{\Omega^{(j-1)}}](\x)>1/2
\big\}.
\ee
It follows from the $L^{2}$-boundedness of $\MM_{s}$ that there is a constant $C$  so that $\big\vert\Omega^{(j)}\big\vert\leq C^{j}\big\vert\Omega\big\vert$.

Next let $\D(\Omega)$ be the family of all dyadic rectangles contained in $\Omega$. Let $\bfR\in\D(\Omega)$ and  write $2^{m\,\e_r}\bfR=2^m\bfQ_r\times\big(\prod_{r'\ne r}\bfQ_{r'}\big)$ as in Section \ref{Sec7.4.1}. The {\it embeddedness factor of $\bfR$ in $\Omega$ in the variable $\x_r$} is
\bea\label{5.2q}
\Emb_{\Omega}^{r}&=2^{\mu} &&\text{ where }& \mu&=\max_{m\in \N}\Big\{m: 
\big\vert 2^{m\,\e_r}\bfR\cap\Omega\big\vert >\frac12\big\vert 2^{m\,\e_r}\bfR\big\vert\Big\}.
\eea
Note that if $\bfR\subseteq \Omega^{(j)}$ and $\Emb_{\Omega^{(j)}}^{r}(\bfR)=2^\mu$ then $2^{\mu\e_r}\bfR\subseteq \Omega^{(j+1)}$. 

\smallskip

We now introduce enlargements of  dyadic rectangles. Let $N\in \N$ and let $\r=(r_{1}, \ldots, r_{N})\in \big\{1, \ldots, s\big\}^{N}$  so that each $r_{j}\in \{1, \ldots, s\}$. (We call $\r$ an {\it $N$-string}.) Given $\bfR\in \D(\Omega)$ and $0\leq j \leq N$ we construct enlargements $\bfR^{(j)}$ of $\bfR$  and  multi-indices $\m^{(j)}(\bfR)=\big(m_1^{(j)}(\bfR), \dots,m_s^{(j)}(\bfR)\big)\in\N^s$. This is done by induction. For $j=0$ set $\bfR^{(0)}=\bfR$ and $\m^{(0)}(\bfR)=(0, \ldots, 0)$.  If $\bfR^{(j-1)}$ and $m^{(j-1)}(\bfR)$ have been defined, let $2^{\mu_j(\bfR)}=\Emb^{r_j}_{\Omega^{(j)}}(\bfR^{(j-1)})$ and put
\bea\label{10.3z}
m^{(j)}_{r_j}(\bfR)&=
\begin{cases}
m^{(j-1)}_{r_{j}}(\bfR)+\mu_j(\bfR)& \text{if $r_{j}\in \r$}\\
m^{(j)}_{r'}=m^{(j-1)}_{r'}& \text{if  $r'\neq r_{j}$} 
\end{cases}&&\text{ and }&
\bfR^{(j)}&=2^{\m^{(j)}(\bfR)}\bfR=2^{\mu_j(\bfR)\e_{r_j}}\bfR^{(j-1)}.
\eea

\subsection{Subfamilies of $\D(\Omega)$ and factorizations of $\R^\nu$ depending on $B\in\mathfrak P_{s}$ }\label{Sec11.2ww}\quad

\smallskip

Let $\UU\subseteq\D(\Omega)$ be a family of dyadic rectangles and let $\r=(r_{1}, \ldots, r_{N})$ be an $N$-string. Then $\bfR^{(j)}$ and $\m^{(j)}(\bfR)$ are defined in \eqref{10.3z} for $\bfR\in\UU$ and $0\leq j \leq N$. Let $B$ be a non-empty subset of $\{1, \ldots, s\}$ with cardinality $|B|=b$. 

We define certain subfamilies of $\UU$ depending on $B$ and $\r$. To do this we write $\R^\nu=\R^B\times\R^{B^c}=\big(\prod\nolimits_{r\in B}\R^{\nu_r}\big)\times\big(\prod\nolimits_{r\notin B}\R^{\nu_r}\big)$.  $\bfR'$ and $\bfR''$ will denote dyadic rectangles in $\R^{B}$ and $\R^{B^{c}}$, and if $\bfR\in\D(\Omega)$ we write $\bfR=\bfR_{B}\times\bfR_{B^{c}}\subseteq \R^{B}\times\R^{B^{c}}$. $\m'$ and $\m''$ will denote multi-indices in $\N^{B}$ and $\N^{B^{c}}$ and if $\m\in \N^{s}$ we write $\m=(\m',\m'')$.  Given $\m'\in\N^{B}$ and $0\leq j \leq N$ define
\bea\label{UUVV}
\UU^{j,\m'}&=\big\{\bfR\in\UU:\m_B^{(j)}(\bfR)=\m'\big\}\\
\UU_B^{j,\m'}&=\big\{\bfR':\text{exists }\bfR'' \text{ such that }\bfR'\times\bfR''\in\UU^{j,\m'}\big\}\\
\UU^{j,\m'}(\bfR')&=\big\{\bfR\in\UU^{j,\m'}:\bfR_{B}=\bfR'\big\}\\
\VV^{j,\m'}(\bfR')&=\big\{\bfR'':\bfR'\times\bfR''\in\UU^{j,\m'}\big\}.
\eea
If $|B|=s$ only the first two lines of \eqref{UUVV} make sense and define the same object. 
With $|B|$-tuples $\m'$ ordered component-wise, we also define 
\be\label{barUUVV}
\overline \UU^{j,\m'}=\bigcup\nolimits_{\n'\le\m'} \UU^{j,\n'}\quad \text{and similarly for } \overline \UU_B^{j,\m'}\ \text {and }\ \overline \VV^{j,\m'}(\bfR').
\ee
The sets in \eqref{UUVV} and \eqref{barUUVV} depend on the choice of an N-string $\r$ and the subset $B$. 

\begin{definition}\label{Def10.1}{\rm\quad

\begin{enumerate}[(a)]
\item
Let $\r=(r_1,\dots,r_N)\in\{1, \ldots, s\}^{N}$  be an $N$-string. 
If $B\subseteq\{1, \ldots, s\}$ then {\it $\r$ contains $B$} if there exists a segment of $\r$,   $\r_B=(r_k,r_{k+1},\dots,r_{k+b-1})$, which is a list, without repetitions, of all elements of $B$. $\r$ is {\it complete} if it contains every non-empty subset of $\{1, \ldots, s\}$.  

\smallskip

\item
If $\AA$ is a family of sets,  the {\it shadow} of $\AA$ is the set $\sh(\AA)=\bigcup_{A\in\AA}A$.
\end{enumerate}}
\end{definition}
\noindent An example of complete string is the one defined in the proof of Lemma~3.3 of \cite{MR1072104} for general $s$. An earlier example for $s=3$ is $\r=(1,2,3,1)$, implictly used in \cite{MR0860666}.
\smallskip

\noindent

We now have the following proposition, which is essentially contained in   \cite[Lemmas~2.3,~3.3]{MR1072104}.

\begin{proposition}\label{CarberySeeger}
Let $\UU$ be a family of dyadic rectangles contained in $\Omega$, $B\subseteq\{1, \ldots, s\}$ a non-empty subset, and $\r=(r_1,\dots,r_N)$ a string containing $B$. Then there exist $C,p>0$ such that, for every $\m'\in\N^B$,
\be\label{abc}
\sum\nolimits_{\bfR'\in\overline\UU^{N,\m'}_B}|\bfR'|\big\vert\sh\big(\overline\VV^{N,\m'}(\bfR')\big)\big\vert\le C\big(1+|\m'|\big)^p|\Omega|.
\ee
\end{proposition}

\begin{proof}
We may assume that $B=\{1,2,\dots,b\}$ and 
$
\r=(r_1,\dots,r_\ell,1,2,\dots,b,r_{\ell+b+1},\dots,r_N),
$
 the issue being the validity of \eqref{abc} independently on the location of $\r_B$ inside $\r$. 
Lemma 2.3 in \cite{MR1072104} proves \eqref{abc} in the case $\r=(1,2,\dots,b)$, and Lemma 3.3 essentially contains the proof in the case where $(1,2,\dots,b)$ is the terminal segment of $\r$, that is $N=\ell+b$. It remains to discuss the case where $N>\ell+b$.
But, if $\bfR\in\UU$ then $\m^{(\ell+b)}(\bfR)\le\m^{(N)}(\bfR)$, hence $\overline\UU^{N,\m'}\subseteq\overline\UU^{\ell+b,\m'}$ and
$\overline\VV^{N,\m'}(\bfR')\subseteq \overline\VV^{\ell+b,\m'}(\bfR')$. Then, applying Lemma 3.3 in \cite{MR1072104},
\beas
\sum\nolimits_{\bfR_B\in\overline\UU^{N,\m'}_B}|\bfR'|\big\vert\sh\big(\overline\VV^{N,\m'}(\bfR')\big)\big\vert&\le \sum\nolimits_{\bfR'\in\overline\UU^{\ell+b,\m'}_B}|\bfR'|\big\vert\sh\big( \overline\VV^{\ell+b,\m'}(\bfR')\big)\big\vert
\le C\big(1+|\m'|\big)^{p}|\Omega|.\qedhere
\eeas
\end{proof}

For completeness we add the estimates on the total measure of rectangles enlarged either dyadically or by dilations, as defined in Section \ref{Sec7.2b}.  Let $\UU,\Omega,\r,N,B$ be as above.

\begin{lemma}\label{CSmeasure}
 Let $\overline\Omega=\bigcup\nolimits_{\bfR\in\UU}2^{\m^{(N)}(\bfR)}\bfR$ and $\widehat\Omega=\bigcup\nolimits_{\bfR\in\UU}2^{\m^{(N)}(\bfR)}\circ\bfR$.
Then, with appropriate constants independent of $\UU$, $|\overline\Omega|\sim|\Omega|\sim|\widehat\Omega|$.
\end{lemma}

\begin{proof} By the choice of $\del$ in \eqref{Omegaj} we have $2^{\m^{(N)}(\bfR)}\bfR\subset\Omega^{(N+1)}$. Then the first comparison follows since $|\Omega^{(N+1)}|\leq C^{N+1}|\Omega|$. 
By Lemma \ref{Lem11.4pqr} the sets $2^{\m^{(N)}(\bfR)}\bfR$ and $2^{\m^{(N)}(\bfR)}\circ\bfR$ have comparable measures and this implies that the (ordinary) strong maximal function of $\1_{\overline\Omega}$ is greater than a constant $\gamma>0$ on $\widehat\Omega$, hence $|\widehat\Omega|\le C|\overline\Omega|$.
\end{proof}

\section{A convolution estimates for multi-norm kernels}\label{Sec12gg}

We return to the multi-norm setting with $\bfE$ a standard $n\times n$ matrix and $\R^{d}=\R^{d_{1}}\times\cdots\times\R^{d_{n}}$.  Let  $\KK\in \PP_{0}(\bfE)$ be a multi-norm kernel on $\R^{d}$ as in Definition \ref{Def3.1} and let $f\in L^{2}(\R^{d})$ satisfy appropriate support and cancellation conditions. The main result of this section is Proposition \ref{Lem10.11} below which gives estimates for certain integrals of $|f*\KK|$.

 As usual, if $S=\big\{(A_{1},k_{1}), \ldots, (A_{s},k_{s})\big\}\in\SS_{\bfE}$   we write $\R^{d}=\R^{A_{1}}\times\cdots\times\R^{A_{s}}$. If $\x\in \R^{d}$ we write $\x=(\x_{A_{1}}, \ldots, \x_{A_{s}})$ with $\x_{A_{r}}=(\x_{i})_{i\in A_{r}}$. If $\bfR=\prod_{i=1}^{n}\bfQ_{i}^{\ell_{i}}$ is a dyadic rectangle we write $\bfQ_{A_{r}}=\prod_{i\in A_{r}}\bfQ_{i}^{\ell_{i}}$. (Recall from Proposition \ref{R=Q}  that $\bfQ_{A_r}$ is a dyadic cube in $\R^{A_r}$ for the dilations $\del^{k_r}$.) We use notation and results from Sections \ref{Sec7.1} and \ref{journe}  with $\R^{\nu_r}=\R^{A_r}$.

Let $\mathfrak P_{s}$ denote the set of all non-empty subsets of $\{1, \ldots, s\}$. Given $B\in\mathfrak P_{s}$, let  $\R^B=\prod_{r\in B}\R^{A_{r}}$ and $\R^{B^{c}}=\prod_{r\notin B}\R^{A_{r}}$. If $\x\in \R^{d}$,  write $\x=(\x_B,\x_{B^{c}})\in \R^{B}\times\R^{B^{c}}$. In the same spirit, if $\bfR\subset\R^{d}$ is a dyadic rectangle, write $\bfR=\bfR_B\times\bfR_{B^{c}}$. 
The Fourier transform of $\phi\in L^{1}(\R^{d})$ is the function $\FF[\phi](\xib)=\int_{\R^{d}}e^{-2\pi i\langle\xib,\x\rangle}\phi(\x)\,d\x$. Then $\FF_{B}$ and $\FF_{B^{c}}$ are the partial transforms in the variables $\x_{B}$ and $\x_{B^c}$ respectively: $\FF_{B}[\phi](\xib_{B},\x_{{B^{c}}})=\int_{\R^{B}}e^{-2\pi i\langle\xib_{B},\x_{B}\rangle}\phi(\x_{B},\x_{{B^{c}}})\,d\x_{B}$ and $\FF_{{B^{c}}}[\phi](\x_{B},\xib_{{B^{c}}})=\int_{\R^{B^{c}}}e^{-2\pi i\langle\xib_{{B^{c}}},\x_{{B^{c}}}\rangle}\phi(\x_{B},\x_{{B^{c}}})\,d\x_{{B^{c}}}$.

\smallskip

To state the main result, fix $S=\big\{(A_{1},k_{1}), \ldots, (A_{s},k_{s})\big\}\in\SS_{\bfE}$, fix a subset $B\in \mathfrak P_{s}$ with cardinality $|B|$, and fix a parameter $\tau>0$. Let $\bfR=\bfR_{B}\times\bfR_{{B^{c}}}\in\D_S$ with $\bfR=\prod_{i=1}^{n}\bfQ_{i}^{\ell_{i}}$, and let $\m_B=(m_r)_{r\in B}\in \N^{|B|}$ with $1\le m_r\le \ell_{k_r}$ for $r\in B$. Then put 
\bea\label{11.1z}
\widehat\bfR_{B,\m}&=2^{\m+\tau}\circ\bfR_{B}=\prod\nolimits_{r\in B}2^{m_r+\tau}\circ\bfQ_{A_r}&&\text{ and }&
\widehat\bfR^{cc}_{B,\m}&=\prod\nolimits_{r\in B}(2^{m_r+\tau}\circ\bfQ_{A_r})^c.
\eea
It is important to note that $\widehat\bfR^{cc}_{B,\m}\subsetneqq (\widehat\bfR_{B,\m})^{c}$ if $s\ge2$.

Moreover we observe that, keeping the notation introduced at the end of Section \ref{Sec7.2b}, the enlargement by rescaling $2^{\sigma}\circ\bfQ_{A_r}$ in each factor subspace $\R^{A_r}$ is made with respect to the dilations $\del^{k_r}$. Since the scalar coordinates $x_h$ of points $\x_{A_r}$ will intervene in the next proofs, it is convenient to introduce, for $h=1,\dots,d$, a new notation, $\la'_h$ (depending on the marked partition $S$), for the exponent of $x_h$ in these dilations. Explicitly, if $h\in E_i$ with $i\in A_r$, then 
\be\label{lambda'}
\la'_h=\frac{\la_h}{e(k_r,i)}.
\ee

\begin{proposition}\label{Lem10.11} 
Let $S=\big\{(A_{1},k_{1}), \ldots, (A_{s},k_{s})\big\}$, let $\tau>0$, $B\in\PP_s$, $\bfR=\bfR_{B}\times\bfR_{{B^{c}}}\in\D_S$, and let $f\in L^{2}(R^{d})$ be supported on $(2^{\tau}\circ\bfR_{B})\times\R^{B^{c}}$, with cancellation in the coordinate $\x_{k_{r}}$ for each $r\in B$. Let $V\subseteq\R^{B^{c}}$ be a set of finite measure. Given $\m_B=(m_r)_{r\in B}$ with $1\le m_r\le \ell_{k_r}$ for $r\in B$, let
\beas
\widehat\bfR_{B}=2^{\m+\tau}\circ\bfR_{B}=\prod\nolimits_{r\in B}2^{m_r+\tau}\circ\bfQ_{A_r},\qquad 
\widehat\bfR^{cc}_{B}=\prod\nolimits_{r\in B}(2^{m_r+\tau}\circ\bfQ_{A_r})^c
\eeas
Then there exists $\delta>0$ such that, if $\KK\in \PP_{0}(\bfE)$ is a multi-norm kernel, there is a constant $C_{\KK}$ depending on only finitely many kernel constants of $\KK$ so that 
\be\label{estimate}
\int_{\widehat\bfR^{cc}_{B}\times V}\big|f*\KK(\x)\big|\,d\x\le C_\KK|\bfR_{B}|^\frac{1}{2}| V|^\frac{1}{2}2^{-\del |\m_B|}\|f\|_2\ .
\ee
If $B=\{1,\dots,s\}$ there is no set $ V$, and no factor $| V|^\frac{1}{2}$ in \eqref{estimate}.
\end{proposition}

We begin with some preliminary lemmas.

\begin{lemma}\label{Lem11.4}
Let $\gamma_{0}=\max_{1\leq h \leq d}(\lambda'_{h})^{-1}$ where $\{\lambda'_{h}\}$ are the exponents in \eqref{lambda'}. 
Let $\bfQ_{A_r}\subset\R^{A_r}$ be a dyadic cube with center $\c_{A_r}$, and let $\tau>0$. If $\y_i\in 2^{\tau}\circ\bfQ_{A_r}$ and $\x_{A_r}\notin 2^{\tau+1}\circ\bfQ_{A_r}$, then $2^{-\gamma_{0}}|\x_{A_r}-\c_{A_r}|\le|\x_{A_r}-\y_{A_r}|\le 2^{\gamma_{0}}|\x_{A_r}-\c_{A_r}|$.
\end{lemma}

\begin{proof} Suppose $\bfQ_{A_r}$ has scale $2^{-\ell_r}$. With $E_{A_r}=\bigcup_{i\in A_r}E_i$, write $\c_{A_r}=(c_h)_{h\in E_{A_r}}$, $\x_{A_r}=(x_{h})_{h\in E_{A_r}}$ and $\y_{A_r}=(y_h)_{h\in E_{A_r}}$. Then for each $h\in E_{A_r}$ we have $|y_h-c_h|\le\frac12 2^{\tau\la'_h-\mu_{\ell_i,h}}$ and $|x_{h}-c_h|\ge 2^{\tau\la'_h-\mu_{\ell_i,h}}>2|y_h-c_h|$, so $|x_{h}-c_h|<2|x_{h}-y_h|$ and $|x_{h}-y_h|<2|x_{h}-c_h|$. Since $|\x_{A_r}|=\max_{h\in E_{A_r}}|x_{h}|^{1/\la'_h}$, the lemma follows.
\end{proof}

\begin{lemma}\label{Lem3.17}
Let $\bfQ_{A_r}$ and $\bfQ'_{A_r}$ be dyadic cubes with centers $\c_{A_r}$, $\c_{A_r}'$ and scales $2^{-\ell_r}$, $2^{-\ell'_r}$ such that $2^{\tau+1}\circ\bfQ_{A_r}\subseteq 2^{\tau}\circ\bfQ'_{A_r}$. 
Let $\psi\in\SS(\R^{A_r})$. There is a constant $\eps>0$ and for every $N\in\N$ a constant $C_N$ so that if $\y_{A_r}\in 2^{\tau}\circ\bfQ_{A_r}$, $\x_i\notin 2^{\tau}\circ \bfQ'_{A_r}$ and $m_r\in\N$, then
\beas
\big|\psi\big(2^{m_r}(\x_{A_r}-\y_{A_r})\big)&-\psi\big(2^{m_r}(\x_{A_r}-\c_{A_r})\big)\big|\\
&\le C_N\min\big\{2^{-\eps(\ell_r-m_r)},2^{-\eps(m_r-\ell'_r)}\big\}\big[1+2^{m_r}|\x_{A_r}-\y_{A_r}|\big]^{-N}\\
&\le C_N2^{-\frac\eps2(\ell_r-\ell'_r)}\big[1+2^{m_r}|\x_r-\y_r|\big]^{-N}\ .
\eeas
The constants $C_N$ only depend only on a Schwartz norm of $\psi$, of order $M_N$ depending only on $N$, and the exponents $\la'_h$ in \eqref{lambda'}.
\end{lemma}
\begin{proof}
It follows from Lemma \ref{Lem11.4} that $|\x_{A_r}-\c_{A_r}|\sim|\x_{A_r}-\y_{A_r}|$. 
The rapid decay of $\psi$ then gives $\big|\psi\big(2^{m_r}(\x_{A_r}-\y_{A_r})\big)-\psi\big(2^{m_r}(\x_{A_r}-\c_{A_r})\big)\big|\le C_{N,\psi}[1+2^{m_r}|\x_{A_r}-\y_{A_r}|]^{-N}$. 
Since $\y_{A_r}\in 2^{\tau}\circ\bfQ_{A_r}$ and $\x_{A_r}\notin 2^{\tau}\circ\bfQ'_{A_r}$, the non-isotropic mean value theorem (see \cite{FoSt82})  gives an additional gain of $C_{\psi}\big|2^{m_r}(\x_{A_r}-\y_{A_r})\big|^{\sigma}\leq C_{\psi}2^{-\epsilon(\ell_{r}-m_{r})}$. On the other hand we can use a power of $2^{m_{r}}|\x_{A_r}-\y_{A_r}|$ in the denominator to give an additional gain of $C_{\psi}\big|2^{m_{r}}(\x_{A_r}-\y_{A_r})\big|^{-\sigma}\leq C_{\psi}2^{-\epsilon(m_{r}-\ell'_{r})}$. This gives the first estimate. The second inequality comes from comparison between minimum and geometric mean.
\end{proof}
\begin{lemma}\label{Lem10.6}
Let $f$, $B$, $\bfR$, $\m_B$ be as in Proposition \ref{Lem10.11} (in particular,  $\bfR\in \D_{L}$ with $L\in\LL_S$ in the notation of \eqref{4.6xyz}). Let also $\c_{B}$ be the center of $\bfR_{B}$, $\eta\in \CC^{\infty}(\R^{d})$ be supported in the unit ball $B(1)$ and $J\in\LL$.
There exist constants $\mu>1$ and $\nu>0$ depending only on the exponents $\lambda_{h}$ and the entries of the matrix $\bfE$ so that if $m_{r}>\mu$ for all $r\in B$ and if $f*\eta^{(-J)}$ does not vanish identically on $\widehat\bfR_{B}^{cc}\times\R^{B^{c}}$ then $j_{i}\leq \ell_{i}-\nu\,m_{r}$ for all $r\in B$ and  $i\in A_{r}$.
\end{lemma}
\begin{proof}
Since $J$ need not be in $\LL_S$, we must work with the single components $\x_i$ of $\x_{A_r}$ with $r\in B$. For $i\in A_r$ we also let $m_i=m_r/e(k_r,i)$, not necessarily an integer.
Let $\x=(\x',\x'')\in \widehat\bfR_{B}^{cc}\times\R^{B^{c}}$ such that $f*\eta_{J}^{(-J)}(\x)=\int_{\R^{d}}f(\y)\eta_{J}^{(-J)}(\x-\y)\,d\y\neq 0$.
Since $\x'\in \widehat\bfR_{B}^{cc}$, for every $r\in B$ there exists $i_{r}\in A_{r}$ such that $|\x_{i_{r}}-\c_{i_{r}}|^{e(k_r,i_r)}\geq 2^{\tau+m_{r}-\ell_{k_{r}}}$, that is $|\x_{i_{r}}-\c_{i_{r}}|\geq C_\tau 2^{m_{i_r}-\frac{\ell_{k_{r}}}{e(k_r,i_r)}}\geq C'_\tau2^{m_{i_r}-\ell_{i_{r}}}$, since $L\in\LL_S$. On the other hand there exists $\y'\in 2^{\tau}\circ\bfR_{B}$ so that $2^{J_B}(\x'-\y')\in 2^{\tau}\circ B'(1)$, $B'(1)$ being the unit ball in $\R^B$. Thus 
\beas
h\in E_{i_{r}}&\Longrightarrow |y_{h}-c_{h}|^{\frac{1}{\lambda_{h}}}\leq |\y_{i_{r}}-\c_{i_{r}}|\leq 2^{\tau-\ell_{i_{r}}}\text{\quad and\quad}|x_{h}-y_{h}|^{\frac{1}{\lambda_{h}}} \leq 2^{\tau-j_{i_{r}}}\\
&\Longrightarrow 
|y_{h}-c_{h}|\leq 2^{\lambda_{h}(\tau-\ell_{i_{r}})}\text{\quad and\quad}|x_{h}-y_{h}|\leq 2^{\lambda_{h}(\tau-j_{i_{r}})}\\
&\Longrightarrow |x_{h}-c_{h}|\leq \big[2^{\lambda_{h}(\tau-\ell_{i_{r}})}+2^{\lambda_{h}(\tau-j_{i_{r}})}\big]\\
&\Longrightarrow |\x_{i_{r}}-\c_{i_{r}}|\leq \big[2^{\lambda_{h}(\tau-\ell_{i_{r}})}+2^{\lambda_{h}(\tau-j_{i_{r}})}\big]^{\frac{1}{\lambda_{h}}}\leq C_\tau\big(2^{-\ell_{i_{r}}}+2^{-j_{i_{r}}}\big),
\eeas
where $C$ depends only on the exponents $\lambda_{h}$. Combining the two inequalities we have
\bea\label{10.7}
2^{m_{i_r}-\ell_{i_{r}}}\leq |\x_{i_{r}}-\c_{i_{r}}| 
\leq 
\begin{cases}
2^{-\ell_{i_{r}}+c_\tau}&\text{if $\ell_{i_{r}}\leq j_{i_{r}}$}\\
2^{-j_{i_{r}}+c_\tau}&\text{if $\ell_{i_{r}}> j_{i_{r}}$}.
\end{cases}
\eea
Let $\mu>1+2e(k_r,i_r)c_\tau$ for all $r\in B$. We can exclude that for some $r$, we have $m_r>\mu$ and $\ell_{i_{r}}\leq j_{i_{r}}$, because in this case it would follows from \eqref{10.7} that $m_{i_r}\le c_\tau$ and then $m_r=e(k_r,i_r)m_{i_r}\leq e(k_r,i_r)c_\tau<\mu$, a contradiction. Assume therefore that $m_r>\mu$ for $r\in B$. Then, for $r\in B$, $\ell_{i_{r}}> j_{i_{r}}$ and it follows from \eqref{10.7} that $
j_{i_{r}}-\ell_{i_{r}}\leq -m_{i_r}+c_\tau< \frac{-2m_{r}+\mu-1}{2e(k_r,i_r)}< -\frac{m_{r}}{2e(k_r,i_r)}$. Now let $i\in A_{r}$. This proves the thesis for $i=i_r$ with $\nu=\big(2e(k_r,i_r)\big)^{-1}$.

In order to extend the inequality to all $i\in A_r$, assume first that $i=k_r$, in case $k_r\ne i_r$. Since $J\in \LL$ we have $j_{k_r}\leq e(k_r,i_{r})j_{i_{r}}+\kappa$ where $\kappa$ is the constant in Lemma~\ref{Lem3.8}.  Since  $L\in\LL_{S}$ we have $\ell_{i_{r}}=\big\lfloor\frac{\ell_{k_{r}}}{e(k_{r},i_{r})}\big\rfloor \leq \frac{\ell_{k_{r}}}{e(k_{r},i_{r})}$. Hence $j_{k_r}-\ell_{k_r}\le e(k_r,i_{r})(j_{i_{r}}-\ell_{i_r})+\kappa\le-\frac{m_r}2+\kappa$. With the further condition $\mu>4\kappa$ we obtain
$j_{k_r}-\ell_{k_r}\le -\frac{m_r}4$.
Finally, for general $i\in A_r$, using the inequalities $j_{i}\leq e(i,k_r)j_{k_{r}}+\kappa$
and $\ell_i=\big\lfloor\frac{\ell_{k_{r}}}{e(k_{r},i)}\big\rfloor\ge \frac{\ell_{k_{r}}}{e(k_{r},i)}-1$, the argument is very similar.
 \end{proof}

 Now let $\KK\in\PP_0(\bfE)$, expressed as a dyadic sum $\KK=\sum_{J\in \LL}\eta_{J}^{(-J)}$ according to Theorem \ref{Thm3.9}\eqref{Thm3.9.(4)}. Without loss of generality, we may assume that $\eta_0=0$ and that all the other $\eta_J$ are supported in the unit ball $B(2^\tau)$. 
 We split indices $J\in\LL$ as $(J',J'')$ according to the splitting $\N^n=\N^B\times\N^{{B^{c}}}$. For better clarity, we denote the entries of $J'$ as $j'_i$ for $i\in B$ and the entries of $J''$ as $j''_i$ for $i\notin B$. If $L\in \LL$ let
\bea\label{11.7}
\LL_{L}&=\Big\{J=(J',J'')\in \LL:\text{$f*\eta_{J}^{(-J)}\not\equiv0$ on $\widehat\bfR_{B}^{cc}\times\R^{ B^{c}}$}\big\},\\
 \JJ_{L}'&=\big\{J':\exists\, J=(J',J'')\in\LL_L\big\},\\
\JJ''_{L,J'}&=\big\{J'':(J',J'')\in\LL_L\big\}\text{ \,if $J'\in \JJ_{L}'$},\\
\KK_{L,J'}&=\sum_{J''\in \JJ''_{L,J'}}\eta_{(J',J'')}^{(-(J',J''))}.
\eea
\begin{remarks}\label{Rem10.7}\quad
{\rm
\begin{enumerate}[(1)]
\item\label{Rem10.71}
It follows from Lemma \ref{Lem10.6} that $\LL_{L}$ is a finite set, and its cardinality grows at most polynomially in $|L|$. Consequently the sets $\JJ_{L}'$ and $\JJ_{L,J'}''$ are also finite sets.

\smallskip

\item \label{Rem10.72}
Since $\KK_{L,J'}$ is a finite sum, it is a $\CC^{\infty}$-functions with compact support.

\smallskip

\item \label{Rem10.73}
If $\KK_{L,J'}(\x',\x'')\neq 0$ then $|\x_{A_r}|\leq 2^{-j_{k_r}}$ for each $r\in B$.
\end{enumerate}
 }
\end{remarks}
\begin{lemma}\label{Lem10.8} 
 Let $\h\in F(B)$ and let $\AA_{\h,J',\x'}(\x'')=\partial^{b}_\h\KK_{L,J'}(\x',\x'')$.
Then $\AA_{\h,J',\x'}$ is a multi-norm kernel on the space $\R^{B^{c}}$ relative to the matrix $\bfE_{B}= \big(e(j,k)\big)_{j,k\in B}$, with the constants in equations \eqref{3.9} and \eqref{3.9.5} replaced by $C_{\alpha,N}2^{\sum_{r\in B}j'_{k_{r}}\lambda_{h_{r}}+\sum_{i\in  B}j'_{i}q_{i}}$ so that, for all derivatives~$\de^{\al''}_{\x''}$,
\beas
\big\vert\partial_{\x''}^{\alpha''}\AA_{\h,J',\x'}(\x'')\big|&\leq C_{\alpha,N}2^{\sum_{r\in B}j'_{k_{r}}(\lambda_{h_{r}}+q_{k_{r}})}\prod\nolimits_{i\in  B^{c}}N_{i}(\x'')^{-q_{i}-\[\alpha_{i}\]}\big(1+\|\x''\|\big)^{-N},
\eeas
and 
$$
\big\vert \FF''[\AA_{\h,J',\x'}](\xib'')\big|\leq C\,2^{\sum_{r\in B}j'_{k_{r}}(\lambda_{h_{r}}+q_{k_{r}})}.
$$
\end{lemma}

\begin{proof}

\beas
\AA_{\h,J',\x'}(\x'')
&=
\sum\limits_{J''\in \JJ''_{L,J'}}\partial^b_{\h}\big[\eta_{(J',J'')}^{(-J',-J'')}(\x',\x'')\big]\\
&=
\sum\limits_{J''\in \JJ''_{L,J'}}2^{\sum_{r\in B}j'_{k_{r}}\lambda_{h_{r}}} \big(\partial^b_{\h}\eta_{(J',J'')}\big)^{(-J'-,J'')}(\x',\x'')\\
&=
\sum\limits_{J''\in \JJ''_{L,J'}}2^{\sum_{r\in B}j_{k_{r}}\lambda_{h_{r}}+\sum_{i\in  B}j_{i}q_{i}}\big(\partial^b_{\h}\eta_{(J',J'')}\big)^{(0,-J'')}(2^{J'}\x',\x'')\\
\eeas
Since the functions 
$ \partial^b_{\h}\eta_{(J',J'')}(2^{J'}\x',\cdot)$, depending on the parameters $\h,J',\x'$, are uniformly bounded in every Schwartz norm  and have the required cancellations, the conclusion follows easily (notice that the factor $2^{J'}$ plays in favour of uniformity in $\x'$).
\end{proof}

\begin{proof}[Proof of Proposition \ref{Lem10.11}]
We can assume $b=|B|<s$ since, as will be clear from the proof, the case $B=\{1,\dots,s\}$ is simpler. 
The cancellations of $f$ allow us to express it as a sum of derivatives. Lemma \ref{Lem10.8} shows that for each $\h\in F(B)$ there exists $g_{\h}\in L^{2}(\R^{d})$ supported on $2^{\tau}\bfR_{ B}\times\R^{ B^{c}}$ so that 
\beas
f(\x',\x'')&=\sum\nolimits_{\h\in F}\de^b_\h g_\h(\x',\x'')&&\text{ with }&&\|g_{\h}\|_{2}&\leq C\prod\nolimits_{r\in B}2^{-\lambda_{h_{r}}\ell'_{k_{r}}}\|f\|_{2}.
\eeas
To prove the Proposition we can assume that $f$ consists of a single term in the above sum: thus assume
there exists $g\in L^{2}(\R^{d})$ supported on $2^{\tau}\bfR_{ B}\times\R^{ B^{c}}$, and for each $r\in B$ there is an index $h_{r}\in E_{k_{r}}$, so that
\bea\label{11.8}
f(\x',\x'')&=\de^b_\h g(\x',\x'')=\Big(\prod\nolimits_{r\in B}\partial_{x_{h_{r}}}\Big)g(\x',\x''),\\
\|g\|_{2}&\leq C\,2^{-\sum_{r\in B}\lambda_{h_{r}}\ell'_{k_{r}}}\|f\|_{2}.
\eea

 If $\x'\in \widehat\bfR_{B}^{cc}=\prod_{r\in B}\big(2^{m_{r}+\tau}\circ \bfQ_{A_r}\big)^{c}$ then for each $r\in B$ there exists $i_{r}\in A_{r}$ so that $\x_{i_{r}}\notin 2^{m_{r}+\tau}\circ\bfQ^{\ell_{i}}$, and therefore $|\x_{i_{r}}|>2^{\tau+m_{r}-\ell_{i_{r}}-1}$. Now let $(\x',\x'')\in \widehat\bfR_{B}^{cc}\times V$. Using \eqref{11.7} and \eqref{11.8} and integration by parts, we have 
\bea\label{11.9}
f*\KK(\x',\x'')&=\sum_{J\in \LL_{L}}f*\eta_{J}^{(-J)}(\x',\x'')
=\sum_{J'\in \JJ_{L}'}f*\KK_{L,J'}(\x',\x'')\\
&=\sum_{J'\in\JJ'_L}\de^b_\h g*\KK_{L,J'}(\x',\x'')
=
\sum_{J'\in\JJ'_L} g*\de^b_\h \KK_{L,J'}(\x',\x'').
\eea
Thus from \eqref{11.9}

\beas
\phantom{.}&\int_{\widehat\bfR^c_{B}\times V}\big\vert f*\KK(\x',\x'')\big\vert\,d\x'\,d\x''
\leq
\sum_{J'\in \JJ_{L}'}\int_{\widehat\bfR^c_{B}\times V}\big\vert g *\de^b_\h \KK_{L,J'}(\x',\x'')\big\vert\,d\x'\,d\x''\\
&\leq
\sum_{J'\in \JJ_{L}'}\int_{2^{\tau}\bfR_{ B}}\int_{\widehat\bfR_{B}^{cc}}\Big[\int_{ V}\Big\vert g (\y',\y'')\de^b_\h \KK_{L,J'}(\x'-\y',\x''-\y'')\big\vert d\x''\Big]\,d\x'\,d\y'\\
&\leq
\big| V\big|^{\frac{1}{2}}\sum_{J'\in \JJ_{L}'}\int_{2^{\tau}\bfR_{ B}}\int_{\widehat\bfR_{B}^{cc}}\Big[\int_{ V}\Big\vert g (\y',\y'')\de^b_\h \KK_{L,J'}(\x'-\y',\x''-\y'')\Big\vert^{2} d\x''\Big]^{\frac{1}{2}}\,d\x'\,d\y'\\
&=\big| V\big|^{\frac{1}{2}}\sum_{J'\in \JJ_{L}'}\int_{2^{\tau}\bfR_{ B}}\int_{\widehat\bfR_{B}^{cc}}\big\|g(\y',\,\cdot\,)*''\de^b_\h \KK_{L,J'}(\x'-\y',\,\cdot\,)\big\|_{L^{2}(\R^{ B^{c}})}\,d\x'\,d\y',
\eeas
where $*''$ is convolution in the $\x''$-variables. Let $\AA_{\h,J',\x'-\y'}(\x'')=\partial^b_{\h}\KK_{L,J'}(\x'-\y',\x'')$. According to Lemma \ref{Lem10.8}, $\AA_{\h,J',\x'-\y'}$ is a multi-norm kernel on $\R^{ B^{c}}$ with constants independent of $\x'-\y'$ and equal to those of $\KK$ multiplied by the factor $2^{\sum_{r\in B}j_{k_{r}}\lambda_{h_{r}}+\sum_{i\in B}j_{i}'q_{i}}$. Also $\AA_{\h,J', \x'-\y'}(\x'')=0$ unless $|\x_{A_r}'-\y_{A_r}'|\leq 2^{\tau-j'_{k_r}}$ for all $r\in B$. Then using the Cauchy-Schwarz inequality, Remark \ref{Rem10.7}\eqref{Rem10.73}, and the $L^{2}(\bfR^{ B^{c}})$-boundedness of $\AA_{\h,J',\x'-\y'}$, we have
\beas
\phantom{.}&\int_{\widehat\bfR^c_{B}\times V}\big\vert f*\KK(\x',\x'')\big\vert\,d\x'\,d\x''\\
&\leq
\big| V\big|^{\frac{1}{2}}\sum_{J'\in \JJ_{L}'}\int_{2^{\tau}\bfR_{ B}}\int_{\R^{ B}}\big\|g(\y',\,\cdot\,)*''\AA_{\h,J',\x'-\y'}(\,\cdot\,)\big\|_{{L^{2}(\R^{ B^{c}})}}\1_{2^{-J'}\circ B(1)}(\x'-\y')\,d\x'\,d\y'\\
&\leq
C_{\KK}\big| V\big|^{\frac{1}{2}}\sum_{J'\in \JJ_{L}'}2^{\sum_{r\in B}j_{k_{r}}'\lambda_{h_{r}}+\sum_{i\in  B}j_{i}'q_{i}}\int_{2^{\tau}\bfR_{ B}}\big\|g(\y',\cdot)\big\|_{L^{2}(\R^{ B^{c}})}\Big[\int_{\R^{ B}}\1_{2^{-J'}\circ B(1)}(\x'-\y')\,d\x'\Big]\,d\y'\\
&\leq
C_{\KK,\tau}\big| V\big|^{\frac{1}{2}}\sum_{J'\in \JJ_{L}'}2^{\sum_{r\in B}j_{k_{r}}'\lambda_{h_{r}}}\int_{2^{\tau}\bfR_{ B}}\big\|g(\y',\cdot)\big\|_{L^{2}(\R^{ B^{c}})}\,d\y'\\
&\leq
C_{\KK,\tau}\big\vert \bfR_{ B}\big\vert^{\frac{1}{2}}
\big| V\big|^{\frac{1}{2}}\sum_{J'\in \JJ_{L}'}2^{\sum_{r\in B}j_{k_{r}}'\lambda_{h_{r}}}\Big[\int_{\R^{ B}}\big\|g(\y',\cdot)\big\|^{2}_{L^{2}(\R^{ B^{c}})}\,d\y'\Big]^{\frac{1}{2}}\\
&=
C_{\KK,\tau}\big\vert \bfR_{ B}\big\vert^{\frac{1}{2}}
\big| V\big|^{\frac{1}{2}}\sum_{J'\in \JJ_{L}'}2^{\sum_{r\in B}j_{k_{r}}'\lambda_{h_{r}}}\|g\|_{L^{2}(\R^{d})}\\
&\leq
C_{\KK,\tau}\big\vert \bfR_{ B}\big\vert^{\frac{1}{2}}
\big| V\big|^{\frac{1}{2}}\|f\|_{2}
\sum_{J'\in \JJ_{L}'}2^{-\sum_{r\in B}\lambda_{h_{r}}(\ell'_{k_{r}}-j'_{k_{r}})}\\
&\leq
C_{\KK,\tau}\big\vert \bfR_{ B}\big\vert^{\frac{1}{2}}
\big| V\big|^{\frac{1}{2}}\|f\|_{2}
\sum_{J'\in \JJ_{L}'}2^{-\eps|L'-J'|},
\eeas
for some $\eps>0$. We have used that, since $J\in\LL$ and $L\in\LL_S$, $\ell'_i-j'_i\le(\ell'_{k_r}- j'_{k_r})/e(k_r,i)$ for every $i\in A_r$. By Lemma \ref{Lem10.6}, for $J'\in \JJ_{L}'$ we have $|L'-J'|\ge c|\m_B|$ and this concludes the proof.
\end{proof}

\medskip

\section{Atoms belong to $\h_{\bfE}^{1}(\R^{d})$}\label{Sec13}

In this section we  show that the multi-norm atoms $a$ in Definition \ref{Def9.1} belong to the Hardy space $ \h_{\bfE}^{1}(\R^d)$. We remark that if $S=\big\{(A_{1},k_{1}),\ldots,(A_{s},k_{s})\big\}$ then an $(S,\tau)$-atom  is a Chang-Fefferman product atom relative to the decomposition $\R^{d}=\prod_{r=1}^{s}\R^{A_r}$ with dilations $\del^{k_r}$ in the $A_r$-factor (or a Coifman-Weiss atom if $s=1$, see Remarks \ref{Rem11.2}.) It follows that the `partial' square function $S_{\Psi,S}a=\big(\sum_{L\in \LL_{S}}|a*\Psi_{L}^{(-L)}|^{2}\big)^{\frac{1}{2}}\in L^{1}(\R^{d})$. It is thanks to the extra restrictions on a  multi-norm atom (supporting rectangles of pre-atoms in $\mathbb{D}_S$ and stronger cancellation assumptions) that we obtain  $S_{\Psi, S'}a\in L^{1}(\R^{d})$ also for the other marked partitions $S'$.

\begin{theorem}\label{Thm11.7}
There is a constant $C>0$ so that $\|a\|_{\h_{\bfE}^{1}}\leq C$ for all multi-norm atoms $a$.
\end{theorem}

By Corollary \ref{Cor8.10}, Theorem \ref{Thm11.7} is a direct consequence of the following Proposition.

\begin{proposition}\label{Prop11.9}
Let $S{=\big\{(A_{1},k_{1}),\ldots,(A_{s},k_{s})\big\}}\in \mathcal{S}_{\bfE}^{*}$ and let $\tau >1$. If $\KK\in\PP_{0}(\bfE)$ is a multi-norm kernel, there is a constant $C_{\KK}>0$ so that $\|a*\KK\|_1\leq C_{\KK}$ for any $(S,\tau)$-atom $a$. $C_{\KK}$ depends on finitely many of the kernel constants of $\KK$ in \eqref{3.9} and \eqref{3.9.5}.

\end{proposition}

\begin{proof}
\smallskip

Let $a$ be an $(S,\tau)$-atom with $a=\sum_{\bfR\in \UU}a_{\bfR}$ and $\Omega=\sh(\UU)\subseteq2^\tau\circ B(1)$.  The argument depends on the number $s$ of dotted entries in $S$ and is straightforward if $s\leq 1$. If $s=0$,  the $L^2$-boundedness of convolution with $\KK$ and the fact that $a*\KK$ is supported on a ball of bounded radius shows that $\|a*\KK\|_1\le C\|a*\KK\|_2\le C'\|a\|_2\le C'$. If $s=1$ and $a$ is supported in $2^{\tau}\circ\bfR$ with $\bfR\in\D_S$,
split $\int_{\R^d}\big|a*\KK(\x)\big|\,d\x$ into the sum of the integral over $2^{1+\tau}\circ\bfR$ and the integral over its complement. For the first integral, again apply the $L^2$-boundedness of convolution by $\KK$. For the integral over the complement it is sufficient to use Proposition \ref{Lem10.11} with $\delta=0$ and $m=1$  to obtain $\int_{(2^{1+\tau}\circ\bfR)^c}\big|a*\KK(\x)\big|\,d\x\le C_\KK|\bfR|^\frac12 \|a\|_2\le C_\KK$.

\smallskip

Now  let $S=\big\{(A_{1},k_{1}), \ldots, (A_{s},k_{s})\big\}\in\SS_{\bfE}$ with $s\geq 2$ and let $\tau>0$.  An $(S,\tau)$-atom $a$ is associated with an open set $\Omega\subset 2^{\tau}\circ\bfQ$ where $\bfQ$ a unit cube in $\R^{d}$, and $a(\x)=\sum_{\bfR\in\UU}a_\bfR(\x)$ where $\UU\subseteq\D_S(\Omega)$ is a family of dyadic rectangles. The functions $\{a_{\bfR}:\bfR\in \UU\}$ are pre-atoms as in part \eqref{A3} of Definition \ref{Def9.1}.  Fix a complete N-string $\r=(r_{1}, \ldots,r_{N})$ as Definition \ref{Def10.1}. If $\bfR\in\UU$ let $\bfR^{(j)}$ and $\m^{(j)}(\bfR)=\big(m_1^{(j)}(\bfR), \dots,m_s^{(j)}(\bfR)\big)\in\N^s$ for $0\leq j \leq N$ be the sequence of enlargements and multi-indices defined in equation \eqref{10.3z}. Let $\mathfrak P_{s}$ be the family of all non-empty subsets of $\{1, \ldots, s\}$. If $\bfR=\prod_{i=1}^{n}\bfQ_{i}\in\UU$  then $\bfQ_{A_{r}}=\prod_{i\in A_{r}}\bfQ_{i}\subseteq\R^{A_{r}}$ is a dyadic cube with respect to the dilations $\del^{k_r}$. Let
\bea
\widehat \bfQ_{A_{r}}&=2^{m_{r}^{(N)}(\bfR)}\circ\bfQ_{A_r},\\
\widehat \bfR&=2^{\m^{(N)}(\bfR)}\circ\bfR=\prod\nolimits_{r=1}^{s}\widehat\bfQ_{A_{r}},\\
\widehat\Omega &= \bigcup\nolimits_{\bfR\in\UU}2^{\m^{(N)}(\bfR)}\circ\bfR=\bigcup\nolimits_{\bfR\in\UU}\widehat\bfR.
\eea
From  Lemma \ref{CSmeasure} we have $|\widehat\Omega|\leq C|\Omega|$. Using the Cauchy-Schwarz inequality, the $L^{2}$-boundedness of convolution with $\KK$,  and the $L^2$ normalization of atoms, we  have 
\beas
\int_{\widehat\Omega}\big|a*\KK(\x)\big|\,d\x\leq \|a*\KK\|_{L^{2}}|\widehat\Omega|^{\frac{1}{2}}\leq C_{\KK} |\widehat\Omega|^\frac{1}{2}\|a\|_2\leq C |\widehat\Omega|^{\frac{1}{2}}|\Omega|^{\frac{1}{2}}
\leq  C.
\eeas 

Thus to prove the proposition when $s\geq 2$ we  must show that $\int_{\widehat\Omega^c}\big|a*\KK(\x)\big|\,d\x\leq C_{\KK}$, and to proceed we partition $\widehat \Omega^{c}$  into regions indexed by sets $B\in \mathfrak P_{s}$. We recall the following notation from Section \ref{Sec11.2ww}. If $B\in \mathfrak P_{s}$ write $\R^{d}=\R^{B}\times\R^{B^{c}}=\big(\prod_{r\in B}\R^{A_{r}}\big)\times\big(\prod_{r\notin B}\R^{A_{r}}\big)$, and if $\bfR\in\UU$ write $\bfR=\bfR_{B}\times\bfR_{B^{c}}\subseteq \R^{B}\times\R^{B^{c}}$. We define
\bea
\widehat\bfR_{B}&=\prod_{r\in B}\widehat\bfQ_{A_{r}}&&\text{ and }&(\widehat\bfR_{B})^{cc}&=\prod_{r\in B}\big(\widehat\bfQ_{A_{r}}\big)^{c}.
\eea
The symbols $\bfR'$ and $\bfR''$ denote dyadic rectangles in $\R^{B}$ and $\R^{B^{c}}$, and we use similar notation for  multi-indices $\m'\in\N^{B}$ and $\m''\in\N^{B^{c}}$. Given $\m'\in\N^{B}$ and $0\leq j \leq N$ let
\bea\label{12.3q}
\UU^{j,\m'}&=\Big\{\bfR\in\UU:\m_B^{(j)}(\bfR)=\m'\Big\},\\
\UU_B^{j,\m'}&=\Big\{\bfR'\subseteq \R^{B}:\exists \bfR'' \subseteq\R^{B^{c}}\text{ with }\bfR'\times\bfR''\in\UU^{j,\m'}\Big\},
\eea
and if $\bfR'\subseteq\R^{B}$ let
\bea\label{12.4q}
\UU^{j,\m'}(\bfR')&=\Big\{\bfR\in\UU^{j,\m'}:\bfR_{B}=\bfR'\Big\}\\
\VV^{j,\m'}(\bfR')&=\Big\{\bfR'':\bfR'\times\bfR''\in\UU^{j,\m'}\Big\}.
\eea
With $|B|$-tuples $\m'\in\N^{B}$ ordered component-wise, we also define 
\bea\label{12.5q}
\overline \UU^{j,\m'}&=\bigcup\nolimits_{\n'\le\m'} \UU^{j,\n'},&&&
\overline \UU_B^{j,\m'}&=\bigcup_{\n'\leq\m'}\UU_B^{j,\m'},&&& 
\overline \VV^{j,\m'}(\bfR')&=\bigcup_{\n'\leq\m'}\VV^{j,\m'}(\bfR').
\eea
With this notation in place we can now proceed with the proof.

\smallskip

The inclusion-exclusion formula gives 
$
\1_{\R^{d}}(\x)=\1_{\widehat\bfR}(\x)+\sum\nolimits_{B\in \mathfrak P_{s}}(-1)^{|B|+1}\prod\nolimits_{r\in B}\1_{(\widehat\bfQ_{A_{r}})^{c}}(\x_{A_{r}})$. Then since $\widehat\bfR \subseteq \widehat\Omega$ we have
\bea\label{in-ex}
\1_{\widehat\Omega^{c}}(\x)&=\sum\nolimits_{B\in\mathfrak P_{s}}(-1)^{|B|+1}\prod\nolimits_{r\in B}\1_{(\widehat\bfQ_{A_r})^c}(\x_{A_{r}})\1_{\widehat\Omega^{c}}(\x)\\
&=\sum\nolimits_{B\in \mathfrak P_{s}}(-1)^{|B|+1}\1_{(\widehat\bfR_{B})^{cc}\times\R^{B^{c}}}(\x)\1_{\widehat\Omega^{c}}(\x).
\eea
Thus if $\displaystyle I_{B}=\sum\nolimits_{\bfR\in\UU}\int_{[(\widehat\bfR_{B})^{cc}\times\R^{B^{c}}]\cap\widehat\Omega^c}a_\bfR*\KK(\x)\,d\x$ it follows that 
\beas
\int_{\widehat\Omega^c}a*\KK(\x)\,d\x&=\sum_{\bfR\in\UU}\int_{\widehat\Omega^c}a_\bfR*\KK(\x)\,d\x
=\sum_{B\in\mathfrak P_{s}}(-1)^{|B|+1}\sum\nolimits_{\bfR\in\UU}\int_{[(\widehat\bfR_{B})^{cc}\times\R^{B^{c}}]\cap\widehat\Omega^c}a_\bfR*\KK(\x)\,d\x\\
&=\sum\nolimits_{B\in\mathfrak P_{s}}(-1)^{|B|+1}I_B,
\eeas
and  to prove the proposition it suffices to show that $|I_{B}|\leq C$ for every $B\in \mathfrak P_{s}$.

We argue by ``reverse induction'' on the cardinality of $B$, and we consider first the case of $B_0=\{1, \ldots, s\}$. From \eqref{12.3q} we have the partition $\UU=\bigcup_{\m\in\N^{s}}\UU^{N,\m}$.  Using the inclusion $\widehat\bfR\subseteq \widehat\Omega$, the Cauchy-Schwarz inequality, and Propositions \ref{Lem10.11}and \ref{CarberySeeger} we have 
\beas
|I_{B_0}|&=\Big|\sum_{\bfR\in\UU}\int_{(\widehat\bfR)^{cc}\cap\widehat\Omega^c}a_\bfR*\KK(\x)\,d\x\Big|
\leq\sum_{\bfR\in\UU}\int_{\widehat\bfR^{cc}}\big|a_\bfR*\KK(\x)\big|\,d\x
=\sum_{\m\in \N^{s}}\sum_{\bfR\in \UU^{N,\m}}\int_{\widehat\bfR^{cc}}\big\vert a_{\bfR}*\KK(\x)\big\vert\,d\x\\
&\leq 
\sum_{\m\in \N^{s}}2^{-\delta|\m|}\sum_{\bfR\in  \UU^{N,\m}}\big\vert\bfR\big\vert^{\frac{1}{2}}\|a_{\bfR}\|_{L^{2}}
\leq
\sum_{\m\in \N^{s}}2^{-\delta|\m|}\Big[\sum_{\bfR\in \UU^{N,\m}}\big\vert\bfR\big\vert\Big]^{\frac{1}{2}}\Big[\sum_{\bfR\in\UU^{N,\m}}\|a_{\bfR}\|^{2}\Big]^{\frac{1}{2}}\\
&\leq
C\sum_{\m\in \N^{s}}2^{-\delta|\m|}\big(1+|\m|\big)^{p/2}|\Omega|^{\frac{1}{2}}\big\vert\Omega\big\vert^{-\frac{1}{2}}\lesssim 1.
\eeas

It remains to prove  the estimate $|I_{B}|\lesssim1$ for proper subsets of $\{1, \ldots,s\}$.  
For every $B\in\mathfrak P_{s}$ we have the partition 
$\UU=\bigcup\nolimits_{\m'\in \N^{B}}\UU^{N,\m'}=\bigcup\nolimits_{\m'\in\N^{B}}\bigcup\nolimits_{\bfR'\in \UU^{N,\m'}_{B}}  \UU^{N,\m'}(\bfR')$,
which we use to group together the $a_\bfR$. Define
$
a_{\m',\bfR'}=\sum\nolimits_{\bfR\in \UU^{N,\m'}(\bfR')}a_{\bfR},
$
so that 
$
a=\sum\nolimits_{\m'\in \N^{B}}\sum\nolimits_{\bfR'\in\UU^{N,\m'}_{B}}\,a_{\m',\bfR'}$. 
For $\m'\in \N^{B}$ fixed, if $\bfR_1,\,\bfR_2\in \UU^{N,\m'}(\bfR')$ then the two dyadic enlargements $2^{\m(\bfR_1)}\bfR_1,\, 2^{\m(\bfR_2)}\bfR_2$ have the same $B$-component $2^{\m'}\bfR'$, and the same is true for the enlargement by dilations. Thus $(\widehat\bfR_1)_B=(\widehat\bfR_2)_B$. We denote this last set by $\widehat{\bfR'}$. If now we put
\bea\label{12.8q}
S^{B}(\m',\bfR')=\int_{({\bfR'})^{cc}\times\R^{B^{c}}}\big| a_{\m',\bfR'}*\KK(\x',\x'')\big|\,d\x'\,d\x''
\eea
then
\bea\label{S^B}
|I_{B}|&\leq\sum_{\m'\in \N^{B}}\sum_{\bfR'\in\UU^{N,\m'}_{B}}
\int_{({\bfR'}^*)^{cc}\times\R^{B^{c}}}\big| a_{\m',\bfR'}*\KK(\x)\big|\,d\x
=
\sum_{\m'\in \N^{B}}\sum_{\bfR'\in\UU^{N,\m'}_{B}}S^{B}(\m',\bfR').
\eea
We show that the sum $\sum_{\m',\bfR'}S^B(\m',\bfR')$ in \eqref{S^B} is bounded by a constant depending only on $\KK$, $\UU$, $\Omega$ 
for sets $B$ with $|B|=b<s$ assuming the inductive hypothesis that the same holds for all $H$ with $|H|>b$.

Now suppose $b<s$ and that $\sum_{\m',\bfR'}S^{H}(\m',\bfR')$ is bounded for all subsets $H\in \mathfrak P_{s}$ with $|H|>b$. Fix a pair $(\m',\bfR')$ with $\bfR'\in\UU^{N,\m'}_B$.  We split the integral in \eqref{12.8q} into two parts by decomposing $\R^{B^{c}}$ into two pieces. Replacing $\UU$ with $\VV^{N,\m'}$ we \ apply the definitions in \eqref{12.3q} and \eqref{12.4q} and introduce enlargements $\widetilde\VV=\VV^{N,\m'}(\bfR')$ and a family of rectangles $\widetilde\Omega=\big\{\bfR'':\bfR''\in\widetilde\VV\big\}$.  We also fix a complete $N'$-string $\widetilde\r\in(B^c)^{N'}$ and define $\widehat{\widetilde\Omega}\subset\R^{B^c}$ as in Lemma \ref{CSmeasure}. Then
\bea\label{S^B_1,S^B_2}
S^{B}(\m',\bfR') &= \int_{\big((\widehat{\bfR'})^{cc}\times\widetilde\Omega^{(j)}}\big|a_{\m',\bfR'}*\KK(\x)\big|\,d\x+\int_{(\widehat{\bfR'})^{cc}\times(\widetilde\Omega^{(j)})^c}\big|a_{\m',\bfR'}*\KK(\x)\big|\,d\x\\
&=
S^B_1(\m',\bfR') + S^B_2(\m',\bfR').
\eea
The inequality \eqref{abc}  holds true, {\it a fortiori}, with $\UU^{N,\m'}_B$, $\VV^{N,\m'}(\bfR_B)$ in place of  \,$\overline\UU^{N,\m'}_B$, \,$\overline\VV^{N,\m'}(\bfR')$.

We first deal with $\sum_{\m',\bfR'}S^B_1(\m',\bfR')$.
Applying Proposition \ref{Lem10.11} with $f=a_{\m',\bfR'}$ and using the inequality  $\big|\widehat{\widetilde\Omega}\big|\le C|\widetilde\Omega|=C|\VV^{N,\m'}(\bfR')|$ we obtain
\beas
S^B_1(\m',\bfR')&\leq 
C_\KK|\bfR'|^\frac{1}{2}\big|\widehat{\widetilde\Omega}\big|^\frac{1}{2}2^{-\delta |\m'|}\big\|a_{\m',\bfR'}\big\|_2
= C_{\KK}|\bfR'|^\frac{1}{2}\big|\sh\big(\VV^{N,\m'}(\bfR')\big)\big|^\frac{1}{2}2^{-\delta |\m'|}\big\|a_{\m',\bfR'}\big\|_2.
\eeas
Summing over $\m',\bfR'$ and using Proposition \ref{CarberySeeger}, we have 
\beas
\sum_{\m'\in \N^{B}}&\sum_{\bfR'\in \UU^{N,\m'}_{B}}S^B_1(\m',\bfR')
\leq
\sum_{\m'\in \N^{B}}2^{-\delta |\m'|}\sum_{\bfR'\in\UU^{N,\m'}_{B}}|\bfR'|^\frac{1}{2}\big|\sh\big(\VV^{N,\m'}(\bfR')\big|^\frac{1}{2}\big\|a_{\m',\bfR'}\big\|_2\\
 &
 \le C\Big(\sum_{\m'\in \N^{B}}2^{-2\delta |\m'|}\sum_{\bfR'\in  \UU^{N,\m'}_{B}}|\bfR'|\big|\sh\big(\VV^{N,\m'}(\bfR')\big|\Big)^\frac{1}{2}
 \Big(\sum_{\m'\in \N^{B}}\sum_{\bfR'\in  \UU^{N,\m'}_{B}}\big\|a_{\m',\bfR'}\big\|_2^2\Big)^\frac{1}{2}\\
 &\leq
 C\Big(\big\vert\Omega\big\vert\sum_{\m'\in \N^{B}}2^{-2\delta|\m'|} \big(1+|\m'|)^p\Big)^{\frac{1}{2}}\Big(\sum_{\m'\in \N^{B}}\sum_{\bfR'\in  \UU^{N,\m'}_{B}}\big\|a_{\m',\bfR'}\big\|_2^2\Big)^\frac{1}{2}
\le C|\Omega|^{\frac{1}{2}}|\Omega|^{-\frac{1}{2}}\lesssim 1.
\eeas
This takes care of the sum  $\sum_{\m',\bfR'}S_{1}^{B}(\m',\bfR')$.

To deal with $\sum_{\m',\bfR'}S^B_2(\m',\bfR')$ we decompose integration over $\big(\widehat{\widetilde\Omega}\big)^c$ into a sum of terms, depending on nonempty subsets $G$ of $B^c$ and analogous to the $I_B$ in \eqref{S^B}. We denote by $\PP_{B^c}$ the collection of these sets $G$. For $G\in \PP_{B^{c}}$, we set $H=B\cup G$ and  $H^{c}=\{1,\dots,s\}\setminus H=B^{c}\setminus G$. We write $\R^{d}=\R^{B}\times\R^{G}\times\R^{H^c}$, and $\x\in \R^{d}$ as $\x=(\x_{B},\x_{B^{c}})=(\x_{B},\x_{G},\x_{H^{c}})$ (if $G=B^c$ the third component is not present). We also denote rectangles in $\R^G$ by $\bfR''$ and $G$-tuples of nonegative integers by $\m''$. Using the inclusion-exclusion formula \eqref{in-ex} in the $\R^{B^c}$ variables to obtain the inequality
$
\big|a_{\m',\bfR'}*\KK (\x)\big|\1_{\big(\widehat{\widetilde\Omega}\big)^c}(\x_{B^{c}})
\le \sum\nolimits_{G\in\PP_{B^c}}\big|a_{\m',\bfR'}*\KK (\x)\big|\1_{(\widehat\bfR_G)^{cc}}(\x_G).
$
Since $\UU^{N,\m'}(\bfR')$ decomposes as the disjoint union
\beas
\UU^{N,\m'}(\bfR')=\bigcup\nolimits_{\m''\in\N^G}\bigcup\nolimits_{\bfR'\times\bfR''\in\UU^{N,(\m',\m'')}_H}\UU^{N,(\m',\m'')}(\bfR'\times\bfR''),
\eeas
we can write
$
a_{\m',\bfR'}=\sum\nolimits_{\m''\in\N^G}\sum\nolimits_{\R''\in\UU^{N,(\m',\m'')}_H}\,\,a_{(\m',\m''),\bfR'\times\bfR''}.
$
Thus
\bea\label{12.10}
S_{2}^B(\m',\bfR')\leq \sum_{G\in\PP_{B^c}}&\sum_{\m''\in\N^G}\sum_{\bfR'\times\bfR''\in\UU^{N,(\m',\m'')}_{B\cup G}}
 \int_{(\widehat{\bfR'\times\bfR''})^{cc}\times\R^{H^c}}\big|a_{(\m',\m''),\bfR'\times\bfR''}*\KK(\x)\big|\,d\x.
\eea
Summing over $\m',\bfR'$ we obtain

\beas
\sum\nolimits_{\m'\in \N^{B}}&\sum\nolimits_{\bfR'\in\UU^{N,\m'}_{B}}S_2^{B}(\m',\bfR')\\
&\le\sum\nolimits_{H\supsetneqq B}\sum\nolimits_{(\m',\m'')\in\N^H}\sum\nolimits_{\bfR'\times\bfR''\in \UU^{N,(\m',\m'')}_H}       \int_{(\widehat\bfR_H)^{cc}\times\R^{H^c}}\big|a_{(\m',\m''),\bfR'\times\bfR''}*\KK(\x)\big|\,d\x\\
&=\sum\nolimits_{H\supsetneqq B}\Big(\sum\nolimits_{\m'''\in\N^H}\sum\nolimits_{\bfR'''\in \UU^{N,(\m''')}_H}    S_2^H(\m''',\bfR''')\Big).
\eeas
 Each expression in parentheses is uniformly bounded by the inductive hypothesis, and this completes the proof.
\end{proof}

\section{Further comments}\label{Sec12}

The study of square functions and Hardy spaces developed in this paper is far from being complete. This is due in part to a choice, in order to limit the size of the paper, in part to the fact that certain aspects of the theory would require further thoughts.
\smallskip

\begin{enumerate}[1.]

\item

The multi-norm Littlewood-Paley decomposition presented in Sections \ref{Sec2} and \ref{Sec3} is local for a good reason. If we extend the content of Section \ref{Sec3} to take into account all scales, we immediately realize that large and small scales, either in space or in frequency, behave differently and independently one from the other. To begin with, a dyadic annulus $\widehat N_i(\xib)\sim 2^p$ with $p>0$ cannot intersect another dyadic annulus $\widehat N_j(\xib)\sim 2^{q}$ with $q<0$. Moreover the rules that determine which intersections of dyadic annuli are different in the two ranges. Though the ``far-from-local'' part of the theory may also have some interest, we focus on the local part, which is more natural also because, in principle, it could be extended to non-translation invariant contexts, such as in \cite[Ch. 12, 13]{MR3862599}.

\smallskip

\item We believe that an extension of $L^{1}$-equivalence to $L^p$-equivalence for $p<1$ and close to 1 is possible using the techniques of Section \ref{Sec5}.
\smallskip

\item We do not consider characterizations of $\h^1_\mbE(\bR^d)$ via maximal operators. Natural candidates as ``multi-norm maximal operators'' are those of the form
$
M_\varphi f(\x )=\sup_{\mbt\in\Gamma(\mbE)}\big|f*\varphi^{(-\mbt)}(\x )\big|\ ,
$
with $\int\varphi=1$, as well as grand-maximal analogues.
\smallskip

\item We do not know whether the results of Section \ref{Sec4.2} still hold removing the assumption that the $\Psi_L$ are tensor products, but keeping the other conditions, that is, uniform boundedness in Schwartz norms, cancellation in each variable $\x _{i}$ with $i\in D_S$ and $\sum_{L\in\LL}\Psi_L^{(-L)}=\del_\0$.
More generally, we do not know whether it is possible to obtain, in a simple form, a larger class of mutually $L^{1}$-equivalent square functions which includes those of both tensor- and convolution-type. 
\smallskip

\item In this paper we do not consider continuous-parameter square functions that are natural counterparts of those studied Sections \ref{Sec5}- \ref{Sec7}. In some cases, it is not hard to obtain their $L^{1}$-equivalence with those in this paper, adapting the arguments of Section \ref{Sec5}. This is the case with a $g$-function of tensor type depending on the choice, for $i=1,\dots,n$, of $\ph_i\in \cS(\bR^{d_i})$ with $\int\ph_i=1$ and of the form
$
g(f)(\x )=\big(\big|f*\Psi_0(\x )\big|^2+\sum_{S\in\cS'_\mbE}\int_{F_S}\big|f*\Psi_S^{(-\mbt)}(\x )\big|^2\,d\sigma_S(\mbt)\big)^\half$. 
Here $F_S$ is the face \eqref{3.6*} associated to $S\in\cS'_\mbE$ and $\sigma_S$ is the Hausdorff measure on $F_S$. Moreover, $\Psi_0(\x )=\prod_{i=1}^n\ph_i(\x _{i})$ and $\Psi_S(\x )=\Big(\prod_{i\in D_S}\psi_i(\x _{i})\Big)\Big(\prod_{i\not\in D_S}\ph_i(\x _{i})\Big)$ where $\psi_i(\x _{i})=(d/dt)_{|_{t=1}}\ph^{(-t)}(\x _{i})$.
\smallskip

\item 
There are close relations between multi-norm and flag structures. The paper \cite{MR3862599} was initially motivated by the remarkable observation that a singular kernel that is simultaneously a flag kernel for two ``opposite'' flags is singular only at the origin and locally satisfies the conditions of a multi-norm kernel. Heuristically, flag structures can be regarded as ``limits'' of (special) multi-norm structures. Considering, for instance, the flag $\bR_{x_1}\subset \bR^2_{x_1,x_2}\subset\cdots\bR^n$ endowed with isotropic dilations, the corresponding flag-seminorms $N_i=\max\big\{|x_i|,|x_{i+1}|,\dots,|x_n|\big\}$ and $\widehat N_j=\max\big\{|\xi_1|,|\xi_2|,\dots,|\xi_j|\big\}$ are the limits as $\nu\to\infty$ of the norms \eqref{2.5a} and associated to the standard matrices
\beas
\mbE_\nu=\begin{pmatrix} 1&1&1&1&\cdots\\\nu&1&1&1&\cdots\\\nu^2&\nu&1&1&\cdots\\\nu^3&\nu^2&\nu&1&\cdots\\\vdots&\vdots&\vdots&\vdots&\ddots\end{pmatrix}\ \qquad (\nu>1)\ ,
\eeas

The Littlewood-Paley decomposition of the flag frequency space, constructed in analogy with Section \ref{Sec3.4}, consists of rectangles depending on parameters  $L=(\ell_1,\dots,\ell_n)\in\bZ^n$ 
with $\ell_1\leq \ell_2\leq \cdots\leq \ell_n$ . Fixing one such $L$ and grouping together the consecutive entries that are equal, we obtain a string $\{I_1,\dots,I_s\}$ of consecutive subintervals of $\{1,\dots,n\}$. Setting $k_r=\min I_r$ (the dotted entry in $I_r$), the rectangle at scale $L$ consists of the points $\xib\in\bR^n$ such that, for all $r=1,\dots,s$,
$
|\xi_{k_r}|\sim 2^{\ell_{k_r}}\ \text{ and }\ |\xi_j|\lesssim 2^{\ell_j}=2^{\ell_{k_r}}\text{ if }j\in I_r\setminus\{k_r\}\ .
$
We believe that analysis of flag structure with arbitrary number of factors can be developed along the lines of the present paper, extending results that are known with two factors and, at the same time, acquiring a more structured point of view.
\smallskip

\item It is likely that large part of our results can be extended to homogeneous nilpotent groups, with a standard matrix $\mbE$ that is doubly monotone according to \cite[Def. 9.2]{MR3862599}. The main obstacle to face is the lack of commutativity and the fact that not all dilations can be automorphisms. Indeed, matters are much simpler in special situations, like in the Phong-Stein case of the Heisenberg group $H_n\cong\bC^n\times\bR$ endowed with isotropic and parabolic dilations, provided the square functions involve convolution with Schwartz functions $\Psi_L$ that are radial in the $\bC^n$ variable. 
This last context has been studied by various authors in connection with flag structures~\cite{arXiv:2102.07371, MR3293443}.

\end{enumerate}

\bibliographystyle{plain}

\bibliography{Ricci.bib}

\end{document}